\documentclass[12pt]{amsart}

\usepackage{graphicx}
\usepackage{mathtools}
\usepackage{enumitem}
\usepackage{scrextend}
\usepackage{pinlabel}
\usepackage{mathrsfs}

\usepackage{geometry}

\usepackage{amsmath, amssymb, amsthm, graphics}
\usepackage{etoolbox}

\usepackage[linktocpage=true]{hyperref}

\usepackage{cleveref}
\usepackage{thmtools}
\usepackage{thm-restate}

\newcommand {\mathsym}[1]{{}}
\newcommand {\unicode}[1]{{}}

\newtheorem{thm}{Theorem}[section]
\newtheorem{prop}[thm]{Proposition}
\newtheorem{lem}[thm]{Lemma}
\newtheorem{cor}[thm]{Corollary}
\newtheorem{conj}{Conjecture}
\newtheorem{question}{Question}

\theoremstyle{definition}
\newtheorem{definition}[thm]{Definition}

\newtheoremstyle{case}{}{}{}{}{}{:}{ }{}
\theoremstyle{case}

\newenvironment{claimproof}[1]{\par\noindent\textit{Proof of Claim:}\space#1}{\hfill $\blacksquare$ \vspace{0.5cm} \par }

\newtheorem{claim}{\textit{Claim}}
\AtEndEnvironment{proof}{\setcounter{claim}{0}}

\theoremstyle{remark}

\newtheorem{remark}[thm]{Remark}

\numberwithin{equation}{section}

\newcommand{\R}{\mathbb{R}}  
\newcommand{\N}{\mathbb{N}}  
\newcommand{\Z}{\mathbb{Z}}  
\newcommand{\C}{\mathbb{C}}  

\newcommand{\norm}[1]{\left\lVert#1\right\rVert}

\begin{document}
\title{Eguchi-Hanson singularities in U(2)-invariant Ricci flow}
\author{Alexander Appleton}
\address{Department of Mathematics, UC Berkeley, CA 94720, USA}
\email{aja44@berkeley.edu}
\maketitle

\begin{abstract}
We show that a Ricci flow in four dimensions can develop singularities modeled on the Eguchi-Hanson space. In particular, we prove that starting from a class of asymptotically cylindrical $U(2)$-invariant initial metrics on $TS^2$, a Type II singularity modeled on the Eguchi-Hanson space develops in finite time. Furthermore, we show that for these Ricci flows the only possible blow-up limits are (i) the Eguchi-Hanson space, (ii) the flat $\R^4 /\Z_2$ orbifold, (iii) the 4d Bryant soliton quotiented by $\Z_2$, and (iv) the shrinking cylinder $\R \times \R P^3$. As a byproduct of our work, we also prove the existence of a new family of Type II singularities caused by the collapse of a two-sphere of self-intersection $|k| \geq 3$.
\end{abstract}
\setcounter{tocdepth}{2}
\makeatletter
\def\l@subsection{\@tocline{2}{0pt}{2.5pc}{5pc}{}}
\makeatother

\tableofcontents

\section{Introduction}
The main result of this paper is to show 
\noindent 
\vspace{1em}
\begin{center}
\parbox{0.8\linewidth}{
\noindent \textit{A Ricci flow on a four dimensional non-compact manifold may develop a Type II singularity modeled on the Eguchi-Hanson space in finite time.} 
} 
\end{center}
\vspace{1em}  
The Eguchi-Hanson space is diffeomorphic to the cotangent bundle of the two-sphere and asymptotic to the flat cone $\R^4 / \Z_2$. It is the simplest example of a Ricci flat asymptotically locally euclidean (ALE) manifold and in the physics literature known as a gravitational instanton. The Eguchi-Hanson singularities constructed in this paper are the first examples of orbifold singularities in Ricci flow, and are also the first examples of singularities with Ricci flat blow-up limits. As a byproduct of our work we also show that
\noindent 
\vspace{1em}
\begin{center}
\parbox{0.8\linewidth}{
\noindent \textit{A Ricci flow on a four dimensional non-compact manifold may collapse an embedded two-dimensional sphere with self-intersection $k \in \Z$ to a point in finite time and thereby produce a singularity.} 
} 
\end{center}
\vspace{1em} 
The singularities we construct when $|k| \geq 3$ are of Type II and the author conjectures that their blow-up limits are homothetic to the steady Ricci solitons found in \cite{A17}.

\subsection{Background}
A family of time-dependent metrics $g(t)$ on a manifold $M$ is called a Ricci flow if it solves the equation 
\begin{equation}
\label{RF}
\partial_t g(t) = -2 Ric_{g(t)}.
\end{equation}
Here $Ric_{g(t)}$ is the Ricci tensor of the metric $g(t)$. In local coordinates the Ricci flow equation can be written as a coupled system of second order non-linear parabolic equations. Heuristically speaking, the Ricci flow smoothens the metric $g(t)$, while simultaneously shrinking positively curved and expanding negatively curved directions at each point of the manifold. 

Ricci flow was introduced by Hamilton \cite{Ham82} in 1982 to prove that a closed three dimensional manifold admitting a metric of positive Ricci curvature also admits a metric of constant positive sectional curvature. This success demonstrated the power of Ricci flow and ignited much research in this area, culminating in Perelman's proof of the Poincar\'e and Geometrization Conjectures for three dimensional manifolds. 

Even though every complete Riemannian manifold $(M,g)$ of bounded curvature admits a short-time Ricci flow starting from $g$, singularities may develop in finite time. Understanding their geometry is central to the study of Ricci flow and has topological implications. For instance, Perelman proved the Geometrization Conjecture by analyzing the singularity formation in three dimensional Ricci flow and showing that a Ricci flow nearing its singular time exhibits one of the following two behaviors:
\begin{itemize}
\item Extinction: The manifold becomes asymptotically round before shrinking to a point
\item (Degenerate or non-degenerate) Neckpinch: A region of the shape of a small cylinder $\R \times S^2$ develops
\end{itemize}
Based on this knowledge Perelman was able to construct a Ricci flow with surgery, which performs the decomposition of a three manifold into pieces corresponding to the eight Thurston geometries, yielding a proof of the Geometrization Conjecture.

In order to understand the formation of singularities in Ricci flow it is very useful to take blow-up limits. Roughly speaking one zooms into the region in which the singularity is forming by parabolically rescaling space and time. The resulting blow-up limit is an ancient Ricci flow called the \textbf{singularity model}. It encapsulates most of the geometric information of the singularity. Note that a Ricci flow is called \textbf{ancient} if it can be extended infinitely into the past. To date all known singularity models are either shrinking or steady \textbf{Ricci solitons}. These are self-similar solutions to the Ricci flow equation that, up to diffeomorphism, homothetically shrink or remain steady, and can be understood as a generalization of Einstein manifolds of positive or zero scalar curvature, respectively. Hamilton distinguishes between \textbf{Type I and Type II singularities}, depending on the rate at which the curvature blows up to infinity as one approaches the singular time. It has been proven that Type I singularities are modeled on shrinking Ricci solitons \cite{EMT11}, however it is unknown whether all Type II singularity models are steady Ricci solitons. In three dimensions the only Type I singularity models are $S^3$, $\R \times S^2$ and their quotients.  

As three dimensional singularity formation is now well understood the next step is to analyze the four dimensional case, where currently very little is known other than that the possibilities are far more numerous. Below we list all Type I singularity models known in four dimensions:
\begin{enumerate}
\item $S^3 \times \R $ and its quotients
\item $S^2 \times \R^2$ and its quotients
\item Einstein manifolds of positive scalar curvature (e.g. $S^4$, $\C P^2$, etc.)
\item Compact gradient shrinking Ricci solitons that are not Einstein
\item The FIK shrinker \cite{FIK03}
\end{enumerate}
Note that (1) and (2) are just products of a three dimensional Type I singularity model with the real line. As Einstein manifolds in four dimensions remain to be classified, list item (3) may contain a very large set of manifolds. As for (4), to date only few examples of compact shrinking Ricci solitons are known and a list of these can be found in Cao's survey \cite{Cao10}. The FIK shrinker is a non-compact $U(2)$-invariant shrinking K\"ahler-Ricci soliton, which is diffeomorphic to the blow-up of $\C^2$ at the origin. It is an open question whether there are other non-flat one-ended shrinking Ricci solitons in four dimensions. Maximo proved that Type I singularities modeled on the FIK shrinker may occur in $U(2)$-invariant K\"ahler-Ricci flow \cite{M14}.

The FIK shrinker models an interesting singularity in four dimensional Ricci flow --- namely the collapse of an embedded two-dimensional sphere with non-trivial normal bundle. Topologically, real rank 2 vector bundles over the two-dimensional sphere are classified by their Euler class, which is an integer multiple of the generator of $H^2(S^2, \Z)$. We call this multiple the twisting number and denote it by $k$. Recall that the self-intersection of an embedded two-dimensional sphere in a four-dimensional manifold is equal to the twisting number of its normal bundle. Unlike in K\"ahler geometry, where there is a canonical choice for the generator of $H^2(S^2, \Z)$ and the sign of the self-intersection number is crucial, in the Riemannian case only its absolute value affects the geometry and behavior of embedded two-spheres under Ricci flow. Heuristically speaking, the larger $|k|$, the more negative curvature there is in the vicinity of the sphere and the less likely it collapses to a point. In the list above $S^2 \times \R^2$ and the FIK shrinker model the collapse of two dimensional spheres with self-intersection equal to $0$ and $-1$, respectively. The main goal of this paper is to show that embedded spheres of self-intersection number $|k| \geq 2$ may also collapse in finite time. To explain this in greater detail we give an overview of our setup below.

\subsection{Overview of setup}
Let $M_k$, $k \geq 1$, be diffeomorphic to the blow-up of $\C^2/\Z_k$ at the origin, and denote by $S^2_o$ the two-sphere stemming from the blow-up. Alternatively one can also view $M_k$ as a plane bundle over $S^2_o$. Fix an arbitrary point o, for 'origin', on $S^2_o$. Note that $S^2_o$, with respect to the orientation inherited from $\C^2$, has self-intersection $-k$. Then equip $M_k$ with an $U(2)$-invariant metric $g$. It turns out that with help of the Hopf fibration 
$$\pi: S^3/\Z_k \rightarrow S^2$$
such $U(2)$-invariant metrics can be conveniently written as a warped product metric of the form 
\begin{equation}
\label{metric-intro-s}
g = ds^2 + a^2(s) \omega \otimes \omega + b^2(s) \pi^{\ast} g_{S^2(\frac{1}{2})},
\end{equation}
on the open dense subset $\R_{> 0 } \times S^3/\Z_k \subset M_k$. $\omega$ is the 1-form dual to the vertical directions of the Hopf fibration and $s$ is a parametrisation of the $\R_{>0}$ factor. Note that $g$ pulls back to a Berger sphere metric on the cross-sections $S^3/\Z_k$. One can complete $g$ to a smooth metric on all of $M_k$ by requiring 
\begin{align}
\label{boundary-cond-intro-s}
a(0) &= 0 \\ \nonumber
a_s(0) &= k \\ \nonumber
b(0) &> 0
\end{align}
and that $a(s)$ and $b(s)$ can be extended to an odd and even function around $s=0$, respectively. Via the boundary condition $a_s(0) = k$ is how topology enters the analysis of the Ricci flow equation. We would like to mention here that throughout this paper we often take the warping functions $a$ and $b$ to be functions of points $(p,t)$ in spacetime rather than of $(s,t)$. This will always be clear from context.

An upshot of writing the metric $g$ in the form (\ref{metric-intro-s}) is that the Ricci flow equation (\ref{RF}) reduces to a $(1+1)$-dimensional system of parabolic equations for the warping functions $a$ and $b$, which simplifies the analysis greatly. In addition to this, both the FIK shrinker, which is diffeomorphic to $M_1$, and the Eguchi-Hanson space, which is diffeomorphic to $M_2$, are $U(2)$-invariant, and therefore their metrics can be written in the form (\ref{metric-intro-s}). In this paper we only study Riemannian manifolds diffeomorphic to $M_k$, $k \in \N$, equipped with a $U(2)$-invariant metric of the form (\ref{metric-intro-s}).

We will consider numerous \emph{scale-invariant} quantities, the most fundamental and important of which we introduce here:
\begin{align*}
Q &\coloneqq \frac{a}{b} \\
x &\coloneqq a_s + Q^2 - 2 \\
y &\coloneqq b_s - Q 
\end{align*}
The quantity $Q$ measures the `roundness' of the cross-sectional $S^3/\Z_k$. That is, when $Q =1$ the metric on the cross-section is round. The quantities $x$ and $y$ are more interesting, as they measure the deviation of the metric $g$ from being K\"ahler. In particular, when 
$$y = 0$$
the manifold $(M_k,g)$, $k\geq 1$, is K\"ahler with respect to the standard complex structure induced from $\C^2$. Moreover, when 
$$x = y = 0$$
the manifold $(M_k, g)$ is hyperk\"ahler, and as we show in section \ref{kahler-E-H-section}, homothetic to the Eguchi-Hanson space. 

\subsection{Overview of results}
The first main result of this paper is to show that 

\vspace{1em}
\begin{center}
\parbox{0.8\linewidth}{
\noindent \textit{For a large class of $U(2)$-invariant asymptotically cylindrical initial metrics on $M_k$, $k \geq 2$, the Ricci flow develops a Type II singularity in finite time, as the area of $S^2_o$ decreases to zero. When $k=2$ the blow-up limit of the singularity is homothetic to the Eguchi-Hanson space.} 
} 
\end{center}
\vspace{1em}

We define the class of metrics for which this result holds in subsection \ref{subsec:preciseresults}. Note that in the $k=2$ case the Eguchi-Hanson space is the first example of a Ricci flat singularity model. Based on numerical simulations the author believes that the Type II singularities in the $k \geq 3$ case are modeled on the steady solitons found in \cite{A17}. A paper on the numerical results is in preparation. 

The above result should be contrasted with the behavior of a Ricci flow starting from a K\"ahler metric. It is well known that the K\"ahler condition is preserved by Ricci flow, and that for such a flow the area of a complex submanifold evolves in a fixed manner. In particular, if $(M, g)$ is a K\"ahler manifold with K\"ahler form $\omega$, then under Ricci flow the K\"ahler class evolves by 
$$[\omega(t)] = [\omega(0)] - 4 \pi t c_1(M),$$
where $c_1(M)$ is the first Chern class of $M$. If we integrate the above equation over a complex curve $\Sigma$ in $M$ then one sees that
$$ |\Sigma|_t = |\Sigma|_0 - 4 \pi t \langle c_1(M), [\Sigma] \rangle,$$
where $|\Sigma|_t$ denotes the area of $\Sigma$ at time $t$. In the case that $M \cong M_k$, $\Sigma = S^2_o$ and $g$ is K\"ahler, it was shown in \cite[Proof of Lemma 1.2]{FIK03} that
$$\langle c_1(M_k), [S^2_o] \rangle = 2 - k$$
and hence
\begin{equation}
\label{area-evol-kahler}
|S^2_o|_t = |S^2_o|_0 - 4 \pi t \left( 2 - k \right).
\end{equation}
This shows that for a K\"ahler-Ricci flow $(M_k, g(t))$, $k \in \N$ the two sphere $S^2_o$ can only collapse to a point when $k = 1$. In fact, when $k=2$ the area of $S^2_o$ is stationary and for $k\geq 3$ increasing. Maximo in \cite{M14} uses the K\"ahler condition and (\ref{area-evol-kahler}) to show that an embedded sphere of self-intersection $-1$ may collapse to a point in finite time under Ricci flow. Note that in our construction the metrics are not assumed to be K\"ahler and hence we cannot rely on (\ref{area-evol-kahler}).

The second main result of this paper is the classification of all possible blow-up limits in the $k=2$ case, including those at larger distance scales from the tip of $M_2$.  In particular, we show that

\vspace{1em}
\begin{center}
\parbox{0.8\linewidth}{
\noindent \textit{For a large class of $U(2)$-invariant asymptotically cylindrical initial metrics on $M_2$ any sequence of blow-ups subsequentially converges to one of the following spaces:
	\begin{enumerate}[label=(\roman*)]
	\item The Eguchi-Hanson space
	\item The flat $\R^4 / \Z_2$ orbifold
	\item The 4d Bryant soliton quotiented by $\Z_2$
	\item The shrinking cylinder $\R \times \R P^3$
	\end{enumerate}
}
} 
\end{center}
\vspace{1em}

The blow-up limits (ii) and (iii) show that the Eguchi-Hanson singularity results in the formation of an orbifold point, which to our knowledge the first concrete example of such in four dimensional Ricci flow.

We expect that many of our methods generalize to the analysis of Ricci flow on other cohomogeneity one manifolds. These are manifolds that admit an action by isometries of a compact Lie group $G$ for which the quotient is one dimensional. The author believes that this work could contribute towards a complete picture of Ricci solitons and ancient Ricci flows on cohomogeneity one manifolds in four dimensions.

\subsection{Precise statement of results}
\label{subsec:preciseresults}
Before presenting the main theorems of this paper, we list the definition of a class $\mathcal{I}$ of metrics needed to state our results.

\vspace{1em}
\noindent \textbf{Definition 7.2.}\textit{ 
For $K>0$ let $\mathcal{I}_K$ be the set of all complete \emph{bounded curvature} metrics of the form (\ref{metric-intro-s}) on $M_k$, $k \geq 1$, with \emph{positive injectivity radius} that satisfy the following scale-invariant inequalities:
\begin{align*}
 Q &\leq 1 \\
a_s, b_s &\geq 0 \\
y &\leq 0 \\
\sup a_s &< K \\
\sup |b b_{ss}| &< K
\end{align*}
Denote by $\mathcal{I}$ the set of metrics $g$ such that for sufficiently large $K>0$ we have $g\in \mathcal{I}_K$.
}
\vspace{1em}

For any $k \in \N$ the set $\mathcal{I}$ of metrics on $M_k$ is non-empty, as for example the metric on $M_k$ defined by
\begin{align*}
a(s) &= Q = \tanh(k s) \\
b(s) &= 1
\end{align*}
is contained in $\mathcal{I}$. Moreover, as we prove in Lemma \ref{I-preserved}, the class $\mathcal{I}_K$ of metrics is preserved by the Ricci flow for sufficiently large $K>0$. In our paper we will mainly consider Ricci flows $(M_k, g(t))$, $t \in [0,T)$, starting from an initial metric $g(0) \in \mathcal{I}$.

Now we list the precise statements of the main results of this paper.

\vspace{1em}
\noindent \textbf{Theorem 9.1}\hspace{0.5em}(Type II singularities)\textbf{.}\textit{ 
Let $(M_k, g(t))$, $k \geq 2$, be a Ricci flow starting from an initial metric $g(0) \in \mathcal{I}$ (see Definition \ref{def:I}) with
\begin{equation*}
\sup_{p \in M_2} b(p,0) < \infty.
\end{equation*}
Then $g(t)$ encounters a Type II curvature singularity in finite time $T_{sing}>0$ and 
\begin{equation*}
\sup_{0 \leq t < T_{sing}} \left(T_{sing} - t\right) b^{-2}(o,t) = \infty.
\end{equation*}
Furthermore, there exists a sequence of times $t_i \rightarrow T_{sing}$ such that the following holds:
Consider the rescaled metrics 
\begin{equation*}
g_i(t) = \frac{1}{b^2(o,t_i)} g\left( t_i + b^2(o,t_i) t \right), \quad t \in \big[-b(o, t_i)^{-2} t_i, b(o, t_i)^{-2}\left(T_{sing} - t_i\right) \big).
\end{equation*}
Then $(M_k, g_i(t), o)$ subsequentially converges, in the pointed Gromov-Cheeger sense, to an eternal Ricci flow $(M_k, g_\infty(t), o)$, $t \in (-\infty, \infty)$. When $k=2$ the metric $g_\infty(t)$ is stationary and homothetic to the Eguchi-Hanson metric.
}
\vspace{1em}

\begin{remark}
We would like to make the following remarks:
\begin{enumerate}
\item In Theorem 9.1, case $k=2$, we only prove that there exists a blow-up sequence which converges to the Eguchi-Hanson space. In Theorem 12.1 below we extend this result and show that in fact any blow-up around the tip of $M_2$ is homothetic to the Eguchi-Hanson space.
\item The initial metric $g(0) \in \mathcal{I}$ with $\sup_{p \in M} b(p,0)<\infty$ is asymptotic to $\R \times S^3 / \Z_k$, where $S^3/\Z_k$ is equipped with a squashed Berger sphere metric. This is because metrics in $\mathcal{I}$ satisfy $a_s, b_s \geq 0$ and $Q \leq 1$. 
\end{enumerate}

\end{remark}

The second main result of our paper is the classification of all possible blow-up limits in the $k=2$ case:

\vspace{1em}
\noindent \textbf{Theorem 12.1}\hspace{0.5em}(Blow-up limits)\textbf{.}\textit{ 
Let $(M_2, g(t))$, $t \in [0, T_{sing})$, be a Ricci flow starting from an initial metric $g(0) \in \mathcal{I}$ (see Definition \ref{def:I}) with $\sup_{p \in M_2} b(p,0) < \infty$. Let $(p_i,t_i)$ be a sequence of points in spacetime with $b(p_i, t_i) \rightarrow 0$. Passing to a subsequence, we may assume that we are in one of the four cases listed below. 
\begin{enumerate}[label=(\roman*)]
\item $\lim_{i \rightarrow \infty} \frac{b(p_i,t_i)}{b(o,t_i)} < \infty$
\item $\lim_{i \rightarrow \infty} \frac{b(p_i,t_i)}{b(o,t_i)} = \infty$ and $\lim_{i \rightarrow \infty} b_s(p_i, t_i) = 1$
\item $\lim_{i \rightarrow \infty} \frac{b(p_i,t_i)}{b(o,t_i)} = \infty$ and $\lim_{i \rightarrow \infty} b_s(p_i, t_i) \in (0,1)$
\item $\lim_{i \rightarrow \infty} \frac{b(p_i,t_i)}{b(o,t_i)} = \infty$ and $\lim_{i \rightarrow \infty} b_s(p_i, t_i) = 0$
\end{enumerate}
Consider the dilated Ricci flows 
$$g_i (t) = \frac{1}{b^2(p_i,t_i)} g\left(t_i + b^2(p_i, t_i) t \right), \quad t \in [- b(p_i,t_i)^{-2} t_i , 0].$$ 
Then $(M_2, g_i(t), p_i)$, $t \in [- b(p_i,t_i)^{-2} t_i , 0]$, subsequentially converges, in the Cheeger-Gromov sense, to an ancient Ricci flow $(M_\infty, g_\infty(t), p_\infty)$, $t \in (-\infty, 0]$. Depending on the limiting property of the sequence $(p_i, t_i)$ we have:
\begin{enumerate}[label=(\roman*)]
	\item $M_\infty \cong M_2$ and $g_\infty(t)$ is stationary and homothetic to the Eguchi-Hanson metric
	\item $M_\infty \cong \R^4\setminus\{0\} / \mathbb{Z}_2$ and $g_\infty(t)$ can be extended to a smooth orbifold Ricci flow on $\R^4/\Z_2$ that is stationary and isometric to the flat orbifold $\R^4/\Z_2$
	\item $M_\infty \cong \R^4\setminus\{0\} / \mathbb{Z}_2$ and $g_\infty(t)$ can be extended to a smooth orbifold Ricci flow on $\R^4/\Z_2$ that is homothetic to the 4d Bryant soliton quotiented by $\Z_2$
	\item $M_\infty \cong \R \times \R P^3$ and $g_\infty(t)$ is homothetic to a shrinking cylinder
\end{enumerate}
}
\vspace{1em}

\begin{remark}
Note that in Theorem 12.1 we do not prove that all blow-up limits (i)-(iv) occur. In fact, it may turn out that the Eguchi-Hanson singularity is isolated, in which case we would only see blow-up limits (i) and (ii).
\end{remark}

As a byproduct of our work we also prove the following two theorems, which are of independent interest. Firstly, we exclude shrinking Ricci solitons in a large class of metrics.

\vspace{1em}
\noindent \textbf{Theorem 6.1} \hspace{0.5em}(No shrinker)\textbf{.}\textit{
On $M_k$, $k \geq 2$, there does not exists a complete $U(2)$-invariant shrinking Ricci soliton of bounded curvature satisfying the conditions
\begin{enumerate}
\item $\sup_{p \in M_k} |b_s| < \infty$ 
\item $T_1 = a_s + 2 Q^2 - 2 > 0$ for $s > 0$
\item $Q = \frac{a}{b} \leq 1$
\end{enumerate}
}
\vspace{1em}

As we show in section 5 the inequalities $T_1 > 0$ for $s >0$ and $Q \leq 1$ are preserved by a Ricci flow $(M_k, g(t))$, $k \geq 2$, with $g(t) \in \mathcal{I}$. For this reason Theorem 6.1 can be used to exclude Type I singularities for such flows. 

Secondly, we prove a uniqueness result for ancient Ricci flows on $M_2$.

\vspace{1em}
\noindent \textbf{Theorem 11.1} \hspace{0.5em}(Unique ancient flow)\textbf{.}\textit{
Let $\kappa > 0$ and $\left(M_2,g(t)\right)$, $t \in (-\infty,0]$, be an ancient Ricci flow that is $\kappa$-non-collapsed at all scales and $g(t) \in \mathcal{I}$, $t \in (-\infty,0]$ (see Definition \ref{def:I}). Then $g(t)$ is stationary and homothetic to the Eguchi-Hanson metric. 
}
\vspace{1em}

We rely heavily on this result when we investigate all possible blow-up limits of a Ricci flow $(M_2, g(t))$ encountering a singularity at $S^2_o$.

\subsection{Outline of paper and proofs}
\label{subsec:outline-paper}
Our paper is organized by sections. Section 2 is preliminary and its goal is to set up in more detail the manifolds and metrics considered in this paper. Here we also derive the full curvature tensor and Ricci flow equation for $U(2)$-invariant metrics. In section 3 we prove a maximum principle for degenerate parabolic differential equations on $M_k$. Beginning from section 4 we present new results. Below we outline the main results of those sections and their proofs.

\vspace{1em}
\noindent \textbf{Outline of section 4.}
A key ingredient in our paper are the scale-invariant quantities
$$ x = a_s + Q^2 - 2$$
and
$$ y = b_s - Q,$$
that measure the deviation of a $U(2)$-invariant metric from being K\"ahler with respect to two fixed complex structures $J_1$ and $J_2$ on $M_k$, $k \geq 1$ (see section \ref{kahler-E-H-section} for the precise definition of $J_1$ and $J_2$). In particular, a metric is K\"ahler with respect to $J_1$ whenever $y=0$ and with respect to $J_2$ whenever $x=y=0$.

Interestingly, a $U(2)$-invariant metric of the form (\ref{metric-intro-s}) is K\"ahler with respect to $J_2$ if, and only if, the underlying manifold is diffeomorphic to $M_2$ and the metric is homothetic to the Eguchi-Hanson metric, as we show in Lemma \ref{unique-Kahler-lem}. Therefore the quantities $x$ and $y$ can be used to measure how much a metric on $M_2$ deviates from the Eguchi-Hanson metric --- a tool that is indispensable to our analysis. In the later sections we develop methods to control the behavior of $x$ and $y$ under the Ricci flow. This will allow us to prove that certain singularities of Ricci flows $(M_2, g(t))$ are modeled on the Eguchi-Hanson space.

In Lemma \ref{E-H-properties-lem} of this section we also derive various properties of the Eguchi-Hanson metric. These are frequently used throughout the paper.

\vspace{1em}
\noindent \textbf{Outline of section 5.} The goal of this section is to derive various \emph{scale-invariant} inequalities that are conserved by Ricci flow.  We say that on a Riemannian manifold $(M,g)$ a geometric quantity $T_g: M \rightarrow \R$  is scale-invariant if for every point $p \in M$, we have $T_g(p) = T_{\lambda g}(p)$ for all $\lambda > 0$. The scale-invariance of the inequalities derived is crucial, as it ensures that they pass to blow-up limits and thus also constrain their geometry. 

We construct these inequalities from the scale-invariant quantities $a_s$, $b_s$ and $Q := \frac{a}{b}$, where $a$ and $b$ are the warping functions of the metric $g$ of the form (\ref{metric-intro-s}). Note that subscript $s$ denotes the derivative with respect to $s$. The key observation is that the evolution equation of the \emph{scale-invariant} quantity
$$T_{(\alpha, \beta, \gamma)} = \alpha a_s + \beta Q b_s + \gamma Q^2, \quad \alpha, \beta, \gamma \in \R,$$
can be written in the form 
$$ \partial_t T_{(\alpha, \beta, \gamma)} = \left[T_{(\alpha, \beta, \gamma)}\right]_{ss} + \left( 2\frac{b_s}{b} -\frac{a_s}{a} \right)\left[T_{(\alpha, \beta, \gamma)}\right]_{s}  + \frac{1}{b^2} C_{(\alpha, \beta, \gamma)},$$
where $C_{(\alpha, \beta, \gamma)}$ is a function of $a_s$, $b_s$ and $Q$. For certain choices of $\alpha$, $\beta$, $\gamma$ and $\delta \in \R$ one can determine the sign of $C_{(\alpha, \beta, \gamma)}$ at a local extremum at which $T_{(\alpha, \beta, \gamma)} = \delta$. Depending on the sign, this allows one to prove, via the maximum principle, that either
$$T_{(\alpha, \beta, \gamma)} \geq \delta$$
or 
$$T_{(\alpha, \beta, \gamma)} \leq \delta$$
is a conserved inequality. One of the conserved inequalities of this form is 
$$ x \leq 0,$$
however we derive many others. 

In this section we also find conserved inequalities not of the above form. For instance, we show that each of the inequalities listed below are conserved by the Ricci flow:
\begin{itemize}
\item $Q \leq 1$ 
\item $y \leq 0$
\item $a_s, b_s \geq 0$
\end{itemize}
The proof is carried out by applying the maximum principle to their evolution equations or, in the case of $a_s , b_s \geq 0$, to their system of evolution equations. The conserved inequalities $Q \leq 1$, $y \leq 0$ and $a_s, b_s \geq 0$ are especially important, as they are part of the definition of the class of metrics $\mathcal{I}$ mentioned above, and constitute the first step in showing that $\mathcal{I}$ is preserved by the Ricci flow.

\vspace{1em}
\noindent \textbf{Outline of section 6.} The main result of section 6 is Theorem \ref{thm:no-shrinker}, which rules out shrinking solitons on $M_k$, $k \geq 2$, within a large class of $U(2)$-invariant metrics. Before we outline the proof, note that from the evolution equation (\ref{b-evol}) of $b$ under Ricci flow it follows by an application of L'H\^opital's rule that at $s=0$
\begin{equation}
\label{outline:b0-evol}
 \partial_t b(0,t)^2 = 4 \left( by_s + k-2 \right).
\end{equation}
This formula is a generalization of (\ref{area-evol-kahler}) to the non-K\"ahler case, as the area of $S^2_o$ at time $t$ equals $b(o,t)^2 \pi$. Hence a shrinking soliton must satisfy
$$ \partial_t b(0,t)^2 < 0,$$
which for $k \geq 2$ implies that $y_s < 0$ at $s = 0$. 

For the proof of Theorem \ref{thm:no-shrinker} we have to rely on the inequality
$$ T_1 = a_s + 2 Q^2 - 2 \geq 0,$$
which by Lemma \ref{T1-preserved-lem} is conserved by the Ricci flow.  In particular, we show that amongst metrics of the form (\ref{metric-intro-s}) on $M_k$, $k \geq 2$, satisfying $Q \leq 1$, $T_1 > 0$ when $s> 0$, and $\sup_{p\in M_k} |b_s| < \infty$ there are no shrinking solitons. We briefly sketch the proof here: First we show in Lemma \ref{dQ-Lemma} that $Q_s \geq 0$ for shrinking solitons. This follows from the Ricci soliton equation, which for metrics of the form (\ref{metric-intro-s}) reduces to a system of ordinary differential equations. Then we consider the evolution equation
\begin{equation}
\label{outline:y-evol}
\partial_t y = y_{ss} + \frac{a_s}{a} y - \frac{y}{a^2} G,
\end{equation}
of $y$, where $G$ is a function of $a_s$, $b_s$ and $Q$. In Lemma \ref{Gpos-lem} we show that whenever $Q_s, T_1 > 0$ we have $G > 0$. This shows that under Ricci flow satisfying these inequalities a negative minimum of $y$ is strictly increasing and a positive maximum is strictly decreasing. However, since $y$ is a scale-invariant quantity, and a shrinking Ricci soliton, up to diffeomorphism, homothetically shrinks under Ricci flow, we see that the maximum or minimum of $y$ must remain constant throughout the flow. We conclude that $y = 0$ everywhere, excluding a shrinking soliton. In the proof of Theorem \ref{thm:no-shrinker}, rather than working with the evolution equation (\ref{outline:y-evol}) of $y$, we use the corresponding ordinary differential equation on a Ricci soliton background.

\vspace{1em}
\noindent \textbf{Outline of section 7.}
The goal of this section is to prove Theorem \ref{curv-bound}, which states that for a Ricci flow $(M_k, g(t))$, $k \geq 1$, starting from an initial metric $g(0) \in \mathcal{I}$ with $\sup_{p \in M_k} b(p,0) < \infty$ there exists a $C_1 > 0$ such that the curvature bound
$$ |Rm_{g(t)}|_{g(t)} \leq \frac{C_1}{b^{2}}$$
holds. The proof is carried out by a contradiction/blow-up argument: Assume there exists a sequence of numbers $D_i \rightarrow \infty$ and points $(p_i, t_i)$ in spacetime such that 
$$K_i := |Rm_{g(t_i)}|_{g(t_i)}(p_i) = \frac{D_i}{b(p_i,t_i)^{2}}.$$ 
Consider the rescaled metrics
$$g_i(t) = K_i g\left( t_i + \frac{t}{K_i}\right), \quad t \in [-K_i t_i, 0],$$ 
normalized such that $|Rm_{g_i(0)}|_{g_i(0)}(p_i) = 1$. Then Perelman's no-local-collapsing theorem shows that $(M_k, g_i(t), p_i)$ subconverges to an ancient non-collapsed Ricci flow $(M_\infty, g_\infty(t), p_\infty)$, $t \leq 0$. As $D_i \rightarrow \infty$ the warping functions $b_i$ corresponding to the metrics $g_i(t)$ satisfy $b_i(p_i,0) \rightarrow \infty$. Recalling that the warping function $b$ describes the size of the base manifold $S^2$ in the Hopf fibration of the $S^3/\Z_k$ cross-sections, one can see that $(M_\infty, g_\infty(t))$ splits as $M_\infty = \R^2 \times N$, where $\R^2$ is equipped with the flat euclidean metric and the restriction of $g_\infty(t)$ to $N$ is a 2d non-compact $\kappa$-solution. However, the only $\kappa$-solutions in 2d are either the shrinking sphere or its $\Z_2$ quotient, both of which are compact. This is a contradiction and the proof of the curvature bound follows. 

In Corollary \ref{cor:curv-bound-ancient} we show that ancient Ricci flows in $\mathcal{I}$, which are $\kappa$-non-collapsed at all scales, also satisfy the curvature bound  
$$ |Rm_{g(t)}|_{g(t)} \leq \frac{C_1}{b^{2}}.$$
This curvature bound will be important in section \ref{E-H-unique-ancient-section}.

\vspace{1em}
\noindent \textbf{Outline of section 8.}
In this section we prove various local and global compactness results for $U(2)$-invariant Ricci flows in the class of metrics $\mathcal{I}$. To state the results we need to first introduce the following notation for a $U(2)$-invariant Riemannian manifold $(M,g)$:
\begin{itemize}
\item Let $\Sigma_p \subset M$ denote the orbit of $p$ under the $U(2)$-action. 
\item Let $$C_g(p,r) \coloneqq \left\{ q\in M \: \Big | \: d_g(q, \Sigma_p) < r \right\}$$
\end{itemize}
One sees that $C_g(p,r)$ is the tubular neighborhood of `radial width' $r$ of the orbit $\Sigma_p$ of $p$ under the $U(2)$-action. See Definition \ref{def:C} for more details.

 The main result of this section is Theorem \ref{thm:local-compactness}, which states under which conditions a sequence of $U(2)$-invariant Ricci flows of the form $(C_{g_i(0)}(p_i, r), g_i(t), p_i)$, $t \in [-\Delta t, 0]$, $\Delta t > 0 , r > 0$, subsequentially converges, in the Cheeger-Gromov sense, to a limiting $U(2)$-invariant Ricci flow $(\mathcal{C}_\infty, g_\infty(t), p_\infty)$, $t \in [-\Delta, 0]$. Amongst other conditions, we require that $g_i(t)$ is
\begin{itemize}
\item $\kappa$-non-collapsed at some scale $\rho > 0$ at the point $(p_i, 0)$ in spacetime
\item In the class $\mathcal{I}$
\item Normalized such that $b(p_i, 0) = 1$ 
\item Of uniformly bounded curvature in $C_{g_i(0)}(p_i, r) \times [-\Delta t, 0]$
\end{itemize}
We also show that after choosing suitable coordinates the warping functions of the metrics $g_i(t)$ subsequentially converge to the corresponding warping functions of $g_\infty(t)$. The compactness result of Theorem \ref{thm:local-compactness} is used frequently throughout the paper, especially its variation, stated in Proposition \ref{blow-up-prop}.

\vspace{1em}
\noindent \textbf{Outline of section 9.}
The main goal of this section is to constrain the geometry of ancient Ricci flows $(M_k, g(t))$, $k \in \N$, $t \in (-\infty, 0]$ in the class of metrics $\mathcal{I}$ that are $\kappa$-non-collapsed at all scales. This is achieved by proving that various scale-invariant inequalities hold. For instance, in Theorem \ref{T-ancient} we prove that three inequalities of the form $T_{(\alpha, \beta, \gamma)} \geq 0$, as in introduced in the outline of section 5 above, hold on such ancient flows. Furthermore, we prove in Theorem \ref{kahler-ancient-thm} that an ancient Ricci flow on $M_2$ in $\mathcal{I}$ which is K\"ahler with respect to $J_1$, i.e. $y=0$ everywhere, is stationary and homothetic to the Eguchi-Hanson space. This result will be used in section \ref{E-H-sing-section}, where we construct an eternal blow-up limit of a Ricci flow on $M_2$ that is homothetic to the Eguchi-Hanson space.

The proof of these theorems is via a \textbf{contradiction/compactness argument} frequently employed throughout the paper. We briefly sketch the method here: Assume we want to prove that a scale-invariant inequality $T \geq 0$ holds on $M_k \times (-\infty, 0]$. We argue by contradiction and assume that 
$$\iota :=\inf_{M_k \times (-\infty, 0]} T < 0.$$ 
We then take a sequence of points $(p_i, t_i)$ in spacetime such that $T(p_i, t_i) \rightarrow \iota$ as $i \rightarrow \infty$, and consider the dilated metrics 
$$g_i(t) = \frac{1}{b(p_i,t_i)^2} g\left( t + t_i b(p_i,t_i)^2\right), \quad t \in [-\Delta t, 0],$$
on the tubular neighborhoods $C_{g_i(0)}(p_i, \frac{1}{2})$ (see Definition \ref{def:C}) for some small $\Delta t > 0$. By the compactness results of section 8, in particular Proposition \ref{blow-up-prop}, the Ricci flows $(C_{g_i(0)}(p_i, \frac{1}{2}), g_i(t), p_i)$, $[-\Delta t, 0]$, subsequentially converges to a Ricci flow $(\mathcal{C}_\infty, g_\infty(t), p_\infty)$, $[-\Delta t, 0]$, where 
$$T(p_\infty, 0) = \inf_{\mathcal{C}_\infty \times [-\Delta t, 0]} T = \iota < 0$$
by the scale invariance of $T$. If, however, the evolution equation of $T$ precludes a negative infimum from being attained, we have arrived at a contradiction and proven the desired result.

\vspace{1em}
\noindent \textbf{Outline of section 10.} The goal of this section is to prove Theorem \ref{E-H-sing-thm}, which states that a Ricci flow $(M_k, g(t))$, $k \geq 2$, starting from an initial metric $g(0) \in \mathcal{I}$ with $\sup_{p \in M_k} b(p,0) < \infty$ encounters a Type II singularity in finite time at the tip of $M_k$ as the area of $S^2_o$ decreases to zero. In the $k=2$ case we show that such a singularity possesses a blow-up limit that is stationary and homothetic to the Eguchi-Hanson space. We do not further investigate the $k \geq 3$ case, however the author conjectures that their blow-up limits are homothetic to the steady Ricci solitons found in \cite{A17}.

The proof is carried out in multiple steps. First we show in Lemma \ref{sing-time-finite} that $g(t)$ encounters a singularity in finite time $T_{sing}\in (0, \infty)$ and $b(o,t) \rightarrow 0$ as $t \rightarrow T_{sing}$. This shows that the two-sphere $S^2_o$ at the tip of $M_k$ collapses to a point in finite time and thereby produces a singularity.

In the second step, we rely on the results of section 6 to show that a blow-up limit around $o \in S^2_o$ cannot be a shrinking Ricci soliton when $k \geq 2$.
As all Type I singularities are modeled on shrinking Ricci solitons we deduce that the singularity is of Type II.

In the third step we borrow a trick due to Hamilton to pick a sequence of times $t_i \rightarrow T_{sing}$ such that the following holds: Take the rescaled metrics
$$g_i(t) = \frac{1}{b(o,t_i)^2} g\left(t_i + b^2(o,t_i) t \right), \quad t \in \big[-b(o, t_i)^{-2} t_i, b(o, t_i)^{-2}\left(T_{sing} - t_i\right) \big),$$
where we recall that $o \in S^2_o$. Then $(M_k, g_i(t), o)$ subsequentially converges to an eternal Ricci flow $(M_\infty, g_{\infty}(t), o)$, $t \in (-\infty, \infty)$, where $M_\infty$ is diffeomorphic to $M_k$.

In the final step we analyze the geometry of $M_\infty$ when $k=2$. It turns out that for the choice of times $t_i$ it follows that
$$ \partial_t b(o,0) = 0$$
on $M_\infty$ background. By the evolution equation (\ref{outline:b0-evol}) of $b$ at $o$ this implies 
$$ y_s(o,0) = 0.$$
Applying a strong maximum principle we deduce that $y=0$ everywhere. By the results of section 9 it then follows that $g_\infty(t)$ is stationary and homothetic to the Eguchi-Hanson metric. 

We mention here that the $k=2$ case of Theorem \ref{E-H-sing-thm} is superseded by Corollary \ref{E-H-blowup} of Theorem \ref{E-H-ancient-thm}. However, since the proof of Theorem \ref{E-H-sing-thm} is simpler we present it here.

\vspace{1em}
\noindent \textbf{Outline of section 11.} The goal of this section is to show that an ancient Ricci flow $(M_2, g(t)), t\in (-\infty,0]$, which is $\kappa$-non-collapsed at all scales and satisfies $g(t) \in \mathcal{I}$, is stationary and homothetic to the Eguchi-Hanson space. The most important consequence of this is that in Theorem \ref{E-H-sing-thm} in fact \emph{any} blow-up of the singularity forming at the tip of $M_2$ is homothetic to the Eguchi-Hanson space, whereas we had previously only proven that there exists \emph{a} blow-up sequence that converges to the Eguchi-Hanson space. 

The proof idea, which we call \textbf{successive constraining}, is to find a continuously varying family of preserved inequalities $Z_{\theta} \geq 0$, $\theta \in [0,1]$, for which $Z_0 \geq 0$ on $M_2 \times (-\infty, 0]$ implies that $g(t)$ is homothetic to the Eguchi-Hanson metric. For our choice of conserved inequalities $Z_\theta \geq 0$, $\theta \in [0,1]$, it follows from the work of section 9 that $Z_1 \geq 0$ on $M_2 \times (-\infty,0]$. Then we deform the inequality $Z_1 \geq 0$ along the path $Z_\theta \geq 0$, $\theta \in [0,1]$, to the inequality $Z_0 \geq 0$ with help of the strong maximum principle applied to the evolution equation of $Z_\theta$. This allows us to deduce that $g(t)$ is stationary and homothetic to the Eguchi-Hanson space. In subsection \ref{E-H-ancient-thm-outline} we give a more detailed outline of the proof of Theorem \ref{E-H-ancient-thm}.

\vspace{1em}
\noindent \textbf{Outline of section 12.} The main result of this section is Theorem \ref{blow-up-thm}, which characterizes all the possible blow-up limits of a Ricci flow $(M_2, g(t))$ starting from an initial metric $g(0) \in \mathcal{I}$ with $\sup_{p \in M_2} b(p,0) < \infty$. We show that the only possible blow-up limits are (i) the Eguchi-Hanson space, (ii) the flat orbifold $\R^4 / \Z_2$, (iii) the 4d Bryant soliton quotiented by $\Z_2$ and (iv) the shrinking cylinder $\R \times \R P^3$. 

Below we give a brief outline of the proof of Theorem \ref{blow-up-thm}: Assume we are given a sequence of points $(p_i, t_i)$ in spacetime with $b(p_i, t_i) \rightarrow 0$. Consider the rescaled metrics
$$g_i(t) = \frac{1}{b(p_i,t_i)^2} g( t_i + b(p_i, t_i)^2 t), \quad t \in [- b(p_i, t_i)^{-2} t_i, 0].$$
By passing to a subsequence we may assume that either 
$$\text{ (I) } \sup_i \frac{b(p_i,t_i)}{b(o,t_i)} < \infty \quad \text{or} \quad \text{ (II) } \lim_{i \rightarrow \infty} \frac{b(p_i,t_i)}{b(o,t_i)} = \infty.$$ 
By section 11 we already know that in case (I) we converge to the Eguchi-Hanson space. Therefore we only need to investigate the behavior in case (II), i.e. at scales larger than the forming Eguchi-Hanson singularity. For this we need to divide case (II) into three subcases: By passing to a subsequence we may assume that
$$ \text{(II.a) } b_s(p_i, t_i) \rightarrow 1 \: \text{ or } \:  \text{ (II.b) } b_s(p_i,t_i) \rightarrow \eta \in (0,1) \: \text{ or } \: \text{ (II.c) } b_s(p_i,t_i) \rightarrow 0.$$ 
For (II.a) and (II.c) we show in Lemma \ref{flat-orbifold-blow-up-lem} and Lemma \ref{cylinder-blow-up-lem} that $(M_2, g_i(t), p_i)$ subsequentially converges to the flat orbifold $\R^4 / \Z_2$ and the shrinking cylinder $\R \times \R P^3$, respectively. The proof of these lemmas is relatively easy. Proving in Lemma \ref{lem:orbifold-blowup} that the blow-up limit in case (II.b) is homothetic to the 4d Bryant soliton quotiented by $\Z_2$ is trickier. Here we rely on Lemma \ref{QT2bddb-limit-lem}, which characterizes the geometry of the high curvature regions of $g(t)$ at distance scales larger than the Eguchi-Hanson singularity away from the tip of $M_2$. In subsection \ref{blow-up-thm-outline} we give a more detailed outline of the proof of Theorem \ref{blow-up-thm}.

\subsection{Further questions and conjectures}
\label{subsec:conjectures}
In this section we collect some conjectures and further questions that arise from our results. The central open question remaining in this paper is whether or not the Eguchi-Hanson singularity of Theorem \ref{blow-up-thm} is isolated. By isolated we mean that the only blow-up limits are the Eguchi-Hanson space and its asymptotic cone, the flat orbifold $\R^4/\Z_2$. We conjecture that 

\begin{conj}
The Eguchi-Hanson singularity of Theorem \ref{blow-up-thm} is not isolated and all four blow-up limits (i)-(iv) occur. In particular, it is accompanied by a Type I singularity modeled on the shrinking cylinder $\R \times S^3/\Z_2$.
\end{conj}

An affirmative answer to this conjecture would provide evidence for a longstanding conjecture in Ricci flow stating that a Type II singularity is always accompanied by a Type I singularity in its vicinity. The author has an argument showing that if the Eguchi-Hanson singularity were isolated, the curvature would blow up at a rate faster than $(T_{sing}-t)^{-\lambda}$, where $\lambda$ is any positive constant.

Although we have not analyzed the blow-up limits of a $U(2)$-invariant Ricci flow $(M_k, g(t))$, $t \in [0,T_{sing})$, in the $k\neq2$ case, we believe that for each $k\in \N$ there exists a unique blow-up limit of the singularity arising from the collapse of the two sphere $S^2_o$ at the tip of $M_k$. In collaboration with Jon Wilkening the author has already conducted numerical simulations confirming this, and a paper is in preparation \cite{AW19}. In particular, we conjecture that

\begin{conj}
\label{conj:sing}
Let $(M_k, g(t))$ be a $U(2)$-invariant Ricci flow encountering a singularity at the tip of $M_k$, as the area of $S^2_o$ decreases to zero. Then the following picture holds: 

\vspace{0.5em}
\begin{center}
\begin{tabular}{| c| c |c| c| c |}
\hline
  $k$ & Blow-up limit at $o \in S^2_o$ & Type & Isolated  \\
  \hline
  $1$ & FIK shrinker & Type I & Yes \\
  $2$ & Eguchi-Hanson space & Type II & No\\
  $\geq 3$ & Steady Ricci solitons of \cite{A17} & Type II & No\\
  \hline
\end{tabular}
\end{center}
\vspace{0.5em}
\end{conj}

By isolated we mean that the singularity is not accompanied by a Type I singularity in its vicinity. For instance, in the case $k\geq 2$ we expect a singularity caused by the collapse of the two-sphere $S^2_o$ at the tip of $M_k$ to always be accompanied by a Type I singularity modeled on the shrinking cylinder $\R \times S^3/\Z_k$ and therefore not to be isolated. If for each $k\geq 3$ the corresponding steady Ricci soliton of \cite{A17} is in fact the unique blow-up limit at the tip of $M_k$, then these singularities are necessarily accompanied by a Type I singularity modeled on $S^3/\Z_k$, because these solitons are asymptotically cylindrical.

Another interesting question is the following:
\begin{question}
Can the Eguchi-Hanson singularity occur on a closed four dimensional manifold?
\end{question}
The author conjectures that the answer is yes, however only non-generically. The simplest model on which to investigate this question is $M = M_2 \#_{\R P^3} M_2 \cong S^2 \times S^2$ equipped with an $U(2)$-invariant metric. One could carry out a construction as follows: Vary between an initial metric that encounters a $\R \times \R P^3$ neckpinch singularity and an initial metric that leads to the collapse of one of the $S^2$ factors of $M$. On the path between these two metrics there should be a metric whose Ricci flow evolution forms an Eguchi-Hanson singularity in finite time. 

This paper has not touched upon the behavior of a general non-$U(2)$-invariant metric on $TS^2$. A first question would be:
\begin{question}
Does the picture of Theorem \ref{blow-up-thm} also hold for Ricci flows starting from non-$U(2)$-invariant perturbations of asymptotically cylindrical $U(2)$-invariant metrics on $TS^2$?
\end{question}
And a final big question mark is the following:
\begin{question}
Are there other four dimensional Ricci flat ALE spaces that can occur as blow-up limits in Ricci flow?
\end{question}
So far all known Ricci flat ALE spaces in four dimensions are hyperk\"ahler and it is not known whether non-hyperk\"ahler examples exist. Kronheimer classified all hyperk\"ahler ALE spaces \cite{KronI89}, \cite{KronII89}. These spaces have one end that is asymptotic to the cone $\R^4 / \Gamma$, where $\Gamma \subset U(2)$ is a certain finite group ---  a binary dihedral, tetrahedral, octahedral or icosahedral group. 
In the case that $\Gamma = \Z_k$ is cyclic, Gibbons and Hawking \cite{GH78}, \cite{GH79} discovered a closed form $(3k  - 6)$-parameter family of such metrics. In the physics literature these metrics are known as multi-center Eguchi Hanson spaces. It would be interesting to see whether the results of this paper can be generalized to prove the existence of singularities modeled on these spaces. 

\subsection{Acknowledgments}
The author would like to thank his PhD advisors Richard Bamler and Jon Wilkening for their constant support and encouragement. This work was supported by a GSR fellowship, which was funded by NSF grant DMS-1344991.

\section{Preliminaries}
\subsection{Notation}
Here we collect some of the notation used throughout the paper.

\begin{itemize}
\item $M_k$, $k \in \N$: a manifold diffeomorphic to the blow-up of $\C^2/ \Z_k$, $k \geq 1$, at the origin.
\item $S^2_o$:  the two-sphere added during the blow-up of $\C^2/ \Z_k$.
\item $\xi$: the radial coordinate on $M_k$ or the parametrization of the $\R_{>0}$ factor in the product $\R_{>0} \times S^3/\Z_k \subset M_k$. 
\item $o$: a fixed point on $S^2_o$.
\item $\Sigma_{p}$: denotes the orbit of $p$ under the $U(2)$-action. For instance if $p \in S^2_o \subset M_k$ we have $\Sigma_{p} = S^2_o$ and when $p \in M_k \setminus S^2_o$ we have $\Sigma_{p}\cong S^3 /\Z_k$.
\item $s$: the geodesic distance from $S^2_o$, and often considered as a function of $\xi$ and $t$.
\item origin: refers to the point $o$.
\item $g$: a metric of the form (\ref{metric1}) or (\ref{metric2}) unless otherwise stated
\item $d_g$: the metric distance induced by $g$.
\item $g(t)$: a time dependent family of metrics of the form (\ref{metric1}) or (\ref{metric2}).
\item $u$, $a$, $b$:  the warping functions of the metric (\ref{metric1}). Depending on context these will be considered as functions of $(\xi,t)$, $(s,t)$ or $(p,t)$, where $p$ is a point on $M_k$.
\item $Q\coloneqq \frac{a}{b}$. 
\item $B_g(p,r)$: the ball centered at $p$ of radius $r$ with respect to the metric $g$.
\item $C_g(p, r)$, $ r>0$: the subset of a cohomogeneity one $U(2)$-invariant Riemannian manifold $(M,g)$ defined by 
$$C_g(p, r) = \left\{ q \in M \Big| \: d_g(q, \Sigma_{p}) < r \right\}.$$ The set $C_g(p, r)$ is diffeomorphic to either $M_k$ or $\R \times S^3/\Z_k$.
\item $\overline{C}_g(p, r)$: the closure of $C_g(p, r)$.
\item $T_{sing}$: the singular time of a Ricci flow.
\item $C^\infty_{U(2)}(M_k \times [0,T])$: The space of smooth $U(2)$-invariant functions $u: M_k \times [0,T] \rightarrow \R$.
\item $x, y$: K\"ahler quantities introduced in section \ref{kahler-E-H-section}.
\end{itemize}

\subsection{The manifold and metric}
\label{manifold-metric-subsec}
For $k \in \N$ let $M_k$ be diffeomorphic to the blow-up of $\C^2 /\Z_k$ at the origin. Denote by $S^2_o$ the embedded two-sphere in $M_k$ stemming from the blow-up, and fix some point $o$ for `origin' on $S^2_o$. 

We now describe the $U(2)$-invariant metrics on $M_k$, $k \geq 1$, studied in this paper. Let $z_1, z_2$ be the standard coordinates on $\C^2$ and let $U(2)$ act on $\C^2$ by left multiplication. This action descends to $M_k$, $k \in \N$. Note that $M_k$ can be seen as the total space of the complex line bundle $O(-k)$ via 
\begin{align*}
\pi: M_k &\longrightarrow S^2_o \\
(z_1, z_2) &\mapsto [z_1, z_2]
\end{align*}
Then $U(2) \cong U(1) \times SU(2)$ acts on $M_k$ in the following way: The action of $U(1)$ rotates the fibres of $\pi$ and $SU(2)$ acts on the base $S^2_o$ via rotations. Now introduce the Hopf coordinates
\begin{align*}
z_1 &= \xi \sin\eta \, e^{i( \psi + \phi)} = x_1 + i y_1 \\
z_2 &= \xi \cos\eta \, e^{i( \psi - \phi)} = x_2 + i y_2
\end{align*}
on $\C^2_\ast$, where $\xi > 0$, $\eta \in [0, \pi/2]$ and $\psi, \phi \in [0, 2 \pi)$. These coordinates descend to $M_k$. In particular, this allows us to endow $M_k$ with the radial coordinate $\xi: M_k \rightarrow \R_{\geq 0}$, by continuously extending $\xi$ to $S^2_o$ by taking $\xi = 0$ on $S^2_o$. Note that the coordinate $\xi$ is only smooth on $M_k \setminus S^2_o$.

A computation shows that the standard euclidean metric
$$g_{eucl} = dx_1^2 + dy_1^2 + dx_2^2 + dy_2^2$$
may be written as
$$g_{eucl} = d\xi^2 + \xi^2 \left( d\eta^2 + \sin^2(2\eta) d\phi^2 + \left[d \psi - \cos(2\eta)d\phi \right]^2 \right)$$
in Hopf coordinates. The 1-form 
$$\omega \coloneqq d \psi - \cos(2\eta)d\phi$$
is dual to the Hopf fibre directions, or equivalently dual to the vector field generated by the $U(1)$ action. Furthermore 
\begin{equation}
\label{FS-pullback}
d\eta^2 + \sin^2(2\eta) d\phi^2
\end{equation}
is the pull-back of the Fubini-Study metric $g_{FS}$ on $\C P^1$, normalized to have constant sectional curvature equal to $\frac{1}{4}$. 

From the above we see that the warped-product metric
\begin{equation}
\label{metric1}
g = u(\xi)^2 d\xi^2 + a(\xi)^2 \omega \otimes \omega + b(\xi)^2 \pi^\ast( g_{FS})
\end{equation}
is the most general $U(2)$-invariant metric on $\C_\ast^2$ and descends to a $U(2)$-invariant metric on the open dense set $\C^2_\ast /\Z_k \subset M_k$. It will be useful to introduce the change of coordinates defined by 
$$ ds = u(\xi) d\xi$$
and $s = 0 $ at $\xi = 0$. Then for $p \in M_k$ the quantity
$s(p) = d_{g}(p, S^2_o)$ describes the radial distance of $p$ from $S^2_o$. In these coordinates the metric becomes
\begin{equation}
\label{metric2}
g = ds^2 + a(s)^2 \omega \otimes \omega + b(s)^2 \pi^\ast( g_{FS}),
\end{equation}
where in a slight abuse of notation we consider $a$ and $b$ as functions of $s$. Depending on the context we will consider $a$ and $b$ either as functions of $s$ or $\xi$. 

The metric $g$ can be extended to a metric on all of $M_k$ by taking $a(0) = 0$ and $b(0) > 0$. In other words we shrink the Hopf fibre directions to zero as $s \rightarrow 0$ or equivalently as we approach $S^2_o$. Note that for every $p \in S^2_0$
$$ds^2 + a(s)^2 \omega \otimes \omega$$
is the pull-back of $g$ onto the fibre $\pi^{-1}(p)$. As $U(1)$ acts on the fibre $\pi^{-1}(p)$, we see that $\pi^{-1}(p)$ is a union of $S^1$ orbits and $p$. Furthermore, such a $S^1$ orbit in $\pi^{-1}(p) \subset M_k$ is parametrized by $0\leq \psi < \frac{2\pi}{k}$ and, by the form of the metric (\ref{metric2}), such an $S^1$ orbit at radial distance $s$ from $S^2_o$ has a circumference of length $\frac{2\pi}{k}a(s)$. Because
$$\frac{2\pi}{k}a(s) = \frac{2\pi}{k} a_s(0) s + O(s^2) \: \text{ as } \:  s \rightarrow 0$$
we must require that $a_s(0) = k$ in order to avoid a conical singularity at $S^2_o$. This is how the topology of the manifold enters the analysis of the Ricci flow equation. Additionally requiring that $a(s)$ and $b(s)$ can be extended to an odd and even function, respectively, around $s=0$ is a sufficient condition for the metric $g$ to be smoothly extendable to all of $M_k$ \cite{VZ18}. In the rest of the paper all metrics considered will be of the form (\ref{metric1}) or equivalently (\ref{metric2}). In Figure \ref{fig:manifold} the manifold $M_k$ and its metric close to the two sphere $S^2_o$ is schematically depicted.

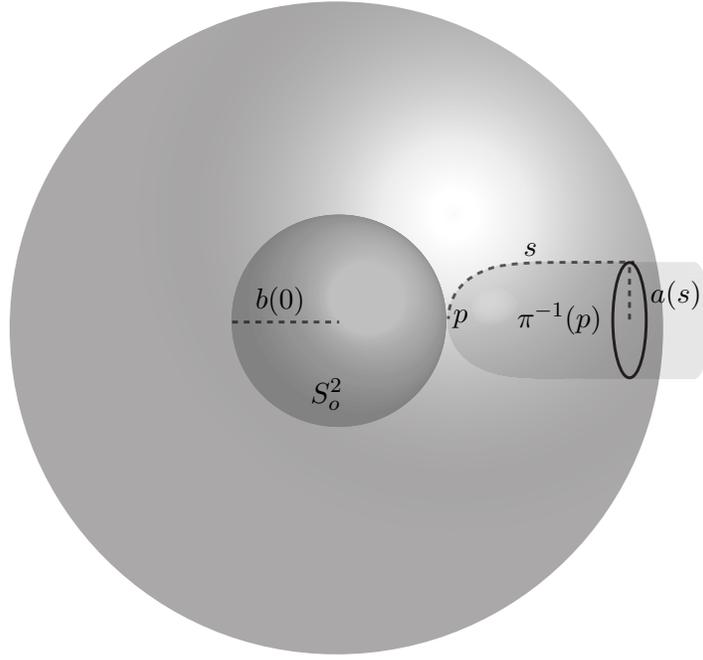
\begin{figure}[h]
\labellist
\small\hair 2pt
\pinlabel $S^2_o$ at 65 55 
\pinlabel $b(0)$ at 57 71
\pinlabel $s$ at 100 80 
\pinlabel $p$ at 88 68 
\pinlabel $\pi^{-1}(p)$ at 105 68 
\pinlabel $a(s)$ at 125 72   
\endlabellist
\centering
\includegraphics[width=0.7\textwidth]{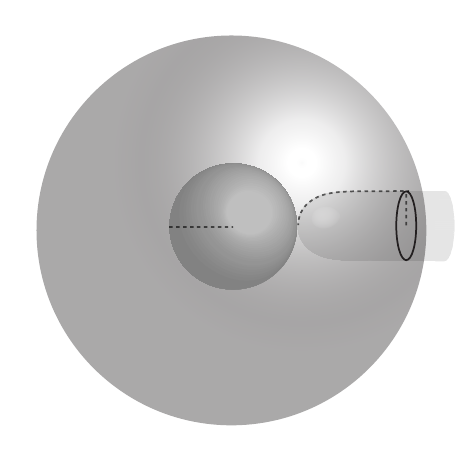}
\caption{Diagram of the manifold $M_k$ close to the tip}
\label{fig:manifold}
\end{figure}

$(M_k, g)$, $k \in \N$, are cohomogeneity one manifolds, meaning that the generic orbits of the $U(2)$ action are of codimension 1. The generic orbit is also called the principal orbit. The non-generic orbits are called non-principal orbits. In the case of $M_k$ the principal orbits are diffeomorphic to $S^3/\Z_k$ and the single non-principal orbit is $S^2_o$ and of codimension 2. Below we introduce some notation that we frequently employ throughout the paper:
\begin{definition}
\label{def:C}
Assume $(M,g)$ is a $U(2)$-invariant cohomogeneity one manifold with principal orbit $S^3/\Z_k$ for some fixed $k \in \N$ and $g$ is a metric of the form (\ref{metric2}). Let $p \in M$ and $r >0$. Then
\begin{itemize}
\item Let $\Sigma_p \subset M$ denote the orbit of $p$ under the $U(2)$-action. 
\item Let $\Sigma^+_p$ be the set of all points $q \in M$ that can be joined via path $\tau$ to $p$ with $g\left(\dot{\tau},\frac{\partial}{\partial s}\right) \geq 0$.
\item Let $$C_g(p,r) \coloneqq \left\{ q\in M \: \Big | \: d_g(q, \Sigma_p) < r \right\}$$
\item Let $$C^+_g(p,r) \coloneqq \left\{ q\in \Sigma^+_p \: \Big | \: d_g(q, \Sigma_p) < r \right\}$$
\end{itemize}
\end{definition}
Note that we have $\Sigma_p \cong S^3/\Z_k$ if $p$ lies on a principal orbit and $\Sigma_p \cong S^2$ if $p$ lies on a non-principal orbit.

\subsection{The connection, Laplacian and curvature tensor}
\label{con-lap-cur-subsec}
We now compute the connection, Laplacian and curvature tensor for metrics of the form (\ref{metric2}). To obtain the corresponding expressions for the metric (\ref{metric1}) use the relation 
$$ \frac{\partial}{\partial s} = \frac{1}{u}\frac{\partial}{\partial \xi}.$$
Take the orthonormal basis
$$e^0 = ds \qquad e^1 = a \left[ d\psi - \cos(2\eta) d \phi \right] \qquad e^2 = b d \eta \qquad e^3 = b \sin(2 \eta ) d \phi$$
of $T^\ast M$. Let $e_i$, $i = 1 , 2, 3, 4$, be the corresponding dual basis of $T_\ast M$. Define the connection 1-forms $\theta_i^j$ by $\nabla e_i = \theta_i^j e_j$ and the curvature 2-forms $\Omega_i^j$ by $R( \cdot, \cdot) e_i = \Omega_i^j e_j$. With help of Cartan's structure equations 
\begin{align*}
\theta_i^j &= - \theta_j^i \\
de^i &= - \theta_j^i \wedge e^j \\
\Omega_i^j &= d \theta_i^j + \theta^j_k \wedge \theta^k_i
\end{align*}
one can compute the connection 1-forms and curvature 2-forms. First note that
\begin{align*}
de^0 &= 0 \\
de^1 &= \frac{a_s}{a} e^0 \wedge e^1 + \frac{2a}{b^2} e^2 \wedge e^3 \\
de^2 &= \frac{b_s}{b} e^0 \wedge e^2 \\
de^3 &= \frac{b_s}{b} e^0 \wedge e^3 + \frac{2}{b} \cot(2 \eta) e^2 \wedge e^3.
\end{align*}
Hence we obtain the connection 1-forms $\theta^i_j$:
\begin{align*}
\theta^1_0 &= \frac{a_s}{a} e^1 && \theta^1_2 = \frac{a}{b^2} e^3 \\
\theta^2_0 &= \frac{b_s}{b} e^2 && \theta^2_3 =  - \frac{a}{b^2} e^1 - \frac{2}{b} \cot(2\eta) e^3 \\
\theta^3_0 &= \frac{b_s}{b}	e^3 && \theta^3_1 =  \frac{a}{b^2} e^2
\end{align*}
Therefore
$$ \nabla_{e_0} e_0 = 0 \qquad \nabla_{e_1} e_1 = - \frac{a_s}{a} e_0 \qquad \nabla_{e_2} e_2 = -\frac{b_s}{b} e_0 \qquad \nabla_{e_3} e_3 = - \frac{b_s}{b} e_0$$
from which we can derive the expression for the Laplacian of a $U(2)$-symmetric function $f(s)$ on $M_k$.
\begin{equation}
\label{laplacian}
\Delta f = \sum_{i=0}^3 \nabla^2_{e_i,e_i} f = f_{ss} + \left(\frac{a_s}{a} + 2 \frac{b_s}{b} \right) f_s.
\end{equation}
Finally, we may compute the components 
\begin{equation*}
R_{ijkl} = g\left(R(e_k,e_l)e_j,e_i\right).
\end{equation*}
of the curvature tensor. Below we list its non-zero components 
\begin{align*}
R_{0101} &= -\frac{a_{ss}}{a}= K_1 \\
R_{0202} &= -\frac{b_{ss}}{b} = K_2 \\
R_{0303} &= -\frac{b_{ss}}{b}= K_3 \\
R_{0123} &= -\frac{2}{b^2} \left( a_s - Q b_s\right) = M_1 \\
R_{0231} &= \frac{1}{b^2} \left( a_s - Q b_s\right) = M_2 \\
R_{0312} &= \frac{1}{b^2} \left( a_s - Q b_s\right) = M_3 \\
R_{1212} &= \frac{a^2}{b^4}- \frac{a_sb_s}{ab} = H_{12} \\
R_{2323} &= \frac{4}{b^2} - 3 \frac{a^2}{b^4} - \left(\frac{b_s}{b}\right)^2 = H_{23} \\
R_{3131} &= \frac{a^2}{b^4}- \frac{a_sb_s}{ab} = H_{31}.
\end{align*}
All other components are either determined by the standard symmetries of the curvature tensor or are zero. 

\subsection{The Ricci flow equation}
\label{ricci-flow-equations-sec}
With help of the above list of curvature components one can check that the Ricci tensor is diagonal and hence the form of the metric (\ref{metric1}) is preserved by Ricci flow. Allowing the warping functions $a$, $b$ and $p$ to vary in time, the Ricci flow equation (\ref{RF}) in $(\xi, t)$ coordinates can be expressed as a system of coupled parabolic equations in $a$, $b$ and $u$.
\begin{align}
\label{u-evol}\partial_t u &= \frac{1}{a} \partial_{\xi} \left( \frac{a_{\xi}}{u} \right) + \frac{2}{b} \partial_{\xi} \left( \frac{b_{\xi}}{u} \right) \\
\label{a-evol} \partial_t a &= \frac{1}{u} \partial_{\xi} \left( \frac{a_{\xi}}{u} \right) - 2 \frac{a^3}{b^4} + 2 \frac{a_{\xi}b_{\xi}}{bu^2} \\
\label{b-evol} \partial_t b &= \frac{1}{u} \partial_{\xi} \left( \frac{b_{\xi}}{u} \right) - \frac{4}{b} +  2\frac{a^2}{b^3} + \frac{a_{\xi}b_{\xi}}{au^2} + \frac{b^2_{\xi}}{bu^2}
\end{align}
Define the time dependent radial distance function $s = s(\xi, t)$ by 
\begin{equation*}
ds = u(\xi, t) d\xi
\end{equation*}
Then
\begin{equation}
\label{s-def}
s(\xi, t) = \int_0^{\xi} u(\xi,t) \, \mathrm{d}\xi
\end{equation}
and
\begin{equation}
\label{s-deriv}
\frac{\partial}{\partial s} = \frac{1}{u} \frac{\partial}{\partial \xi}.
\end{equation}
Furthermore the commutation relation
\begin{equation*}
[\partial_t, \partial_s] = -\frac{\partial_t u}{u}\partial_s
\end{equation*}
holds. In terms of $s$ we can use (\ref{s-deriv}) to rewrite the Ricci flow equation in a slightly simpler form
\begin{align}
\label{u-evol} \frac{\partial_t u}{u}  &= \frac{a_{ss}}{a} + 2 \frac{b_{ss}}{b} \\
\label{a-evol} \partial_t a  &= a_{ss} - 2 \frac{a^3}{b^4} + 2 \frac{a_sb_s}{b} \\
\label{b-evol} \partial_t b  &= b_{ss} - \frac{4}{b} +  2\frac{a^2}{b^3} + \frac{a_sb_s}{a} + \frac{b^2_s}{b}.
\end{align}
Note that the dependence of the right hand side of this system of equations on $\xi$ is hidden in the variable $s = s(\xi, t)$. However we can write the equations in terms of $(s,t)$ by introducing the functions
\begin{align*}
\tilde{a}(s,t) &= a(\xi, t) \\
\tilde{b}(s,t) &= b(\xi, t)
\end{align*}
and noting that
\begin{align*}
\partial_t a \big|_{\xi} &= \partial_t \tilde{a} \big|_{s} + \partial_s \tilde{a} \big|_{t} \frac{\partial s}{\partial t}\big|_{\xi} \\
\partial_t b \big|_{\xi} &= \partial_t \tilde{b} \big|_{s} + \partial_s \tilde{b} \big|_{t} \frac{\partial s}{\partial t}\big|_{\xi}.
\end{align*}
By slight abuse of notation, however, we will drop the tilde and consider the warping functions $a$,$b$ and $u$ as functions of either $(p,t)$, $p \in M_k$ or $(\xi, t)$ or $(s,t)$, depending on context. In $(s,t)$ coordinates the Ricci flow equation reads
\begin{align}
\label{a-evol-s-coord} \partial_t a \big |_{s} &= a_{ss} - 2 \frac{a^3}{b^4} + 2 \frac{a_sb_s}{b} - a_s \frac{\partial s}{\partial t} \\
\label{b-evol-s-coord} \partial_t b \big |_{s} &= b_{ss} - \frac{4}{b} +  2\frac{a^2}{b^3} + \frac{a_sb_s}{a} + \frac{b^2_s}{b} - b_s \frac{\partial s}{\partial t}
\end{align}
where
\begin{align}
\label{dsdt}
\frac{\partial s}{\partial t}\big|_{\xi} = \int_0^s \frac{a_{ss}}{a} + 2 \frac{b_{ss}}{b} \, \mathrm{d}s
\end{align}

Whenever we differentiate a function $f: M_k \times [0,T]$ with respect to time, unless stated otherwise, assume that the point on the manifold $M_k$ is held fixed. If we differentiate with respect to time while holding $s$ fixed we will denote the partial derivative by $\partial_t f |_{s}$ to avoid confusion. Because $s$ is a function of $(\xi,t)$, in general for fixed $s_0 > 0$ the set $\{ s = s_0\} \subset M_k$ is dependent on time. Therefore holding $s$ or $\xi$ fixed during partial differentiation produces very different results.

This following property of the warping functions $a$ and $b$ will be used throughout the paper.
\begin{lem}
\label{parity-lem}
Let $(M_k, g(t))$, $t \in [0,T)$, be a smooth Ricci flow solution. Then for all $t \in [0,T]$ the warping functions $a(s ,t)$ and $b(s ,t)$ can be extended to an odd and even function, respectively, on $\R$.
\end{lem}
\begin{proof}
Note that a necessary condition for a metric $g$ of the form (\ref{metric2}) to be smooth is that its corresponding warping functions $a$ and $b$ are extendable to odd and even functions, respectively, on $\R$. Therefore the desired result follows. Alternatively, notice that if the warping functions of the initial data $a(s, 0)$ and $b(s, 0)$ can be extended to an odd and even function, respectively, on $\R$, we can also extend the equations (\ref{a-evol-s-coord}), (\ref{b-evol-s-coord}) and (\ref{dsdt}) to all of $\R$. An inspection of these equations shows that the parity of $a$ and $b$ is preserved under the flow.
\end{proof}

\subsection{Recap of blow-up limits of singularities}
As mentioned above, every complete Riemannian manifold $(M,g)$ of bounded curvature admits a short-time Ricci flow starting from $g$, however singularities may develop in finite time. Similar to the study of other nonlinear equations, it is very useful to consider blow-up limits of singularities. We briefly sketch the idea here: Assume $(M, g(t))$, $t \in [0,T_{sing})$, is a Ricci flow encountering a curvature singularity as $t \rightarrow T_{sing}$. Let $(p_i, t_i)$ with $p_i \in M$ and $t_i \rightarrow T_{sing}$ be a sequence of points in spacetime such that
$$K_i := |Rm_{g(t_i)}|_{g(t_i)}(p_i) = \sup_{t \leq t_i} |Rm_{g(t)}|_{g(t)}$$
and
$$K_i \rightarrow \infty \: \text{ as } \: i \rightarrow \infty.$$
Take the rescaled metrics
$$ g_i(t) = K_i g\left( t_i + \frac{t}{K_i} \right), \quad t \in[-K_i t_i, 0].$$
Then $(M_k, g_i(t), p_i)$ subsequentially converge, in the Cheeger-Gromov sense, to a pointed ancient Ricci flow solution $(M_\infty, g_\infty(t), p_\infty), t\in(-\infty, 0]$ (see \cite[Theorem 6.68]{ChI} for more details). Note that in general $M_\infty \neq M$. A Ricci flow is called ancient if it can be extended to a time interval of the form $(-\infty, T)$, $T \in \R$. The blow-up limit $(M_\infty, g_\infty(t), p_\infty)$ is called the singularity model and yields important geometrical information on the shape of the singularity. Hamilton \cite{Ham95} distinguishes between Type I and II singularities, depending on the rate of curvature blow-up, i.e. for Type I
$$\sup_{M \times [0,T)} \left(T_{sing}- t\right)|Rm_{g(t)}|_{g(t)} < \infty$$ 
and for Type II
$$\sup_{M \times [0,T)} \left(T_{sing}- t\right)|Rm_{g(t)}|_{g(t)} = \infty.$$
By the work of Naber \cite{N10}, and Enders, M\"uller and Topping \cite{EMT11} it is known that Type I singularities are modeled on shrinking Ricci solitons. One hopes --- although it has not been proven --- that all Type II singularities are modeled on steady solitons, as to date all known examples are.

\section{The maximum principle}
\label{sec:maximum-principle}
Assume we are given a Ricci flow $(M_k, g(t))$, $t \in [0,T]$, $k \geq 1$. We make the following definition:
\begin{definition}
\label{def:smoothU2}
Let $C^\infty_{U(2)}(M_k \times [0,T])$ be the space of smooth $U(2)$-invariant functions 
$$u: M_k \times [0,T] \rightarrow \R.$$
\end{definition}
In this section we prove a maximum principle for operators 
$$P: C^\infty_{U(2)} ( M_k \times [0,T]) \rightarrow C^\infty_{U(2)} ( M_k \times [0,T])$$ 
that in $(\xi,t)$ coordinates and away from the non-principal orbit $S^2_o$ of $M_k$ can be written in the form 
\begin{equation}
\label{Pop}
P[u] = \partial_{ss} u + \left( m \frac{a_s}{a} + n \frac{b_s}{b}\right) u_s + c u - \partial_t u, \quad m,n \in \R,
\end{equation}
where $c\in C^\infty_{U(2)}(M_k \times [0,T])$. Recall from section \ref{ricci-flow-equations-sec} that we are interpreting the $s$ derivative as
$$\frac{\partial}{\partial s} = \frac{1}{u} \frac{\partial}{\partial \xi}.$$
It is useful to work in $(s,t)$ coordinates, in which case the operator $P[u]$ can be expressed as 
\begin{equation}
\label{Pop-st}
P[u] = \partial_{ss} u + \left( m \frac{a_s}{a} + n \frac{b_s}{b} - \frac{\partial s}{\partial t} \right) u_s + c u - \partial_t \Big|_s u,
\end{equation}
where we recall the expression $(\ref{dsdt})$ for $\frac{\partial s}{\partial t}$. Note $P[u]$ is degenerate at the origin $s=0$ as 
$$m \frac{a_s}{a} + n \frac{b_s}{b} - \frac{\partial s}{\partial t}  \sim \frac{m}{s} \text{ for } 0<|s|\ll 1.$$
However, in $(s,t)$ coordinates the smoothness of $u$ and $c$ is equivalent to saying that $u(s,t), c(s,t)$ can be extended to smooth even functions around $s=0$ by defining $u(s,t) = u(-s,t)$ and $v(s,t) = v(-s,t)$ for $s \leq 0$. Hence we see that $u_s = c_s = 0$ at $s=0$ and via L'H\^opital's Rule we obtain the following representation of $P[u]$ on the principal orbit $S^2_0$:
$$P[u] = (m+1)\partial_{ss} u + c u - \partial_t u.$$
The maximum principle derived for $P$ below depends on the sign of $(m+1)$:

\begin{thm}
\label{maximum-principle}
Let $(M_k, g(t))$, $k \in \N$, $t\in[0,T]$ be a Ricci flow with bounded curvature. Let $P$ be an operator of the form (\ref{Pop}) and $u \in C^\infty_{U(2)}(M_k \times [0,T])$. If
\begin{equation*}
P[u] \leq 0 \: \text{ in } \: M_k \times [0,T]
\end{equation*}
and there exist constants $M, \sigma > 0$ such that the growth conditions
\begin{align*}
u(s,t) & \geq - M\exp(- \sigma s^2) \\
c(s,t) & \leq M\left( |s|^2 +1 \right)
\end{align*}
are satisfied, then the following holds true:
\begin{description}
   \item[Case 1]($1+m \leq 0$) If  
\begin{align*}
u(s,0) &\geq 0 \: \text{ for } \: s \geq 0 \\
u(0,t) &\geq 0 \: \text{ for } \: t \in [0,T]
\end{align*}
then $u(s,t) \geq 0$ on $[0,\infty) \times [0,T]$. Furthermore, if $u = 0$ somewhere on $(0,\infty)\times(0,T]$ then $u=0$ everywhere.
   \item[Case 2]($1+m > 0$)
 If  
\begin{align*}
u(s,0) &\geq 0 \: \text{ for } \: s \geq 0
\end{align*}
then $u(s,t) \geq 0$ on $[0,\infty) \times [0,T]$. Furthermore, if $u = 0$ somewhere on $[0,\infty) \times(0,T]$ then $u=0$ everywhere.
\end{description}
\end{thm}

Before proving Theorem \ref{maximum-principle} we need to derive some bounds on $\frac{a_s}{a}$, $\frac{b_s}{b}$ and $\frac{1}{b^2}$ for metrics $g$ of bounded curvature. This will allow us to bound the coefficients appearing in the expression (\ref{Pop-st}) of the operator $P[u]$.

\begin{lem}
\label{b-bound}
Let $(M_k ,g)$, $k \in \N$, and $K > 0$ such that $|Rm_g|_g \leq K \: \text{ on } \: M_k$. Then everywhere on $M_k$ we have 
\begin{enumerate}
\item $b^2 \geq \frac{1}{K}$ 
\item $\left(\frac{b_s}{b}\right)^2 \leq 5 K$
\item $ \frac{Q^2}{b^2} \leq \frac{5}{3} K $
\end{enumerate}
\end{lem}
\begin{proof}
From the curvature components derived in subsection \ref{con-lap-cur-subsec} we see that
\begin{align}
\label{exp1} b^2 H_{23} &= 4 - 3Q^2 - b_s^2 \\
\label{exp2} b^2 H_{12} &= Q^2 - \frac{a_s b_s}{Q}
\end{align}
At a local minimum $b_s =0$ we thus have
\begin{equation*}
b^2 = \frac{4}{3H_{12} + H_{23}} \geq \frac{1}{K}
\end{equation*}
Now we argue by contradiction. Assume that there exists a $s^\ast>0$ and $\delta > 0$ such that at $s= s^\ast$ we have $b^2 < \frac{1-\delta}{K}$. From above it follows that $b_s < 0$ for $s \geq s^{\ast}$ and hence, because $b >0$ everywhere, $\lim_{s \rightarrow \infty} b_s = 0$. Equation (\ref{exp1}) then shows that 
\begin{align*}
Q^2 &= \frac{1}{3} \left( 4 - b^2 H_{23} - b_s^2 \right) \\ 
	&\geq 1 + \frac{\delta}{3} - \frac{b_s^2}{3} \\
	&\geq 1 + \frac{\delta}{4}
\end{align*}
for $s$ sufficiently large. Then (\ref{exp2}) implies that eventually
\begin{equation*}
a_s \frac{b_s}{Q} \geq \frac{5}{4}\delta.
\end{equation*}
Dividing by $\frac{b_s}{Q}$ shows that
\begin{equation*}
a_s \rightarrow -\infty \: \text{ as } \: s\rightarrow \infty
\end{equation*}
contradicting $a \geq 0$. This proves the first bound. To prove the second bound note that
\begin{equation*}
\left(\frac{b_s}{b}\right)^2 = \frac{4-3Q^2}{b^2} - H_{23} \leq \frac{4}{b^2} + K \leq 5K,
\end{equation*}
where the last inequality follows from (1). For the third bound we have
$$ \frac{Q^2}{b^2} = \frac{1}{3} \left( \frac{4}{b^2} - \left(\frac{b_s}{b}\right)^2 - H_{23} \right) \leq \frac{5}{3}K.$$
\end{proof}

\begin{lem}
\label{asa-bound-lem}
Let $(M_k ,g)$, $k \in \N$, and $K > 0$ such that $|Rm_g|_g \leq K \: \text{ on } \: M_k.$
Then everywhere on $M_k$ we have 
\begin{equation}
\label{asa-bound}
- 2 \sqrt{K} < \frac{a_s}{a} < \frac{1}{s} + \sqrt{K}
\end{equation}
\end{lem}
\begin{proof}
The quantity $\phi = \frac{a_s}{a}s$ satisfies the ODE
\begin{equation*}
\frac{d \phi}{ds} = s \frac{a_{ss}}{a} + \frac{\phi(1-\phi)}{s}
\end{equation*}
and by L'H\^opital's rule we have $\phi(0) = 1$. Note that the function $\phi$ can be extended to an even function on $\R$. Therefore $\frac{d \phi}{ds} (0) = 0$ and there exists a small $\epsilon > 0$ such that 
\begin{equation*}
\phi \leq 1 + \sqrt{K}s \: \text{ for } \:  0 \leq s \leq \epsilon.
\end{equation*}
Actually the inequality holds for all $s\geq 0$, since whenever $\phi(s) = 1 + \sqrt{K}s$ we have 
\begin{equation*}
\frac{d \phi}{ds} = s \left( \frac{a_{ss}}{a} - K\right) - \sqrt{K} < 0, 
\end{equation*}
since $|\frac{a_{ss}}{a}| = |R_{0101}| \leq K$. this proves the the upper bound of (\ref{asa-bound}).

To prove the lower bound assume that $a_s(s_0) < 0$. For every $s_1 > s_0$ there exists a $s^\ast \in (s_0, s_1)$ such that  
\begin{equation*}
a_s(s_1)  - a_s(s_0) = (s_1- s_0) a_{ss} (s^{\ast}) \leq (s_1- s_0) K |a(s^{\ast})|
\end{equation*}
by the mean value theorem. It follows that
\begin{equation*}
a_s(s) \leq \frac{1}{2} a_s(s_0) \: \text{ for } \:  s_0 \leq s \leq s_0 + \frac{1}{2K} \Big|\frac{a_s(s_0)}{a(s_0)}\Big|.
\end{equation*}
Therefore 
\begin{equation*}
0 \leq a\left(s_0 + \frac{1}{2K} \Big|\frac{a_s(s_0)}{a(s_0)}\Big|\right) \leq a(s_0) + \frac{1}{2K} \Big|\frac{a_s(s_0)}{a(s_0)}\Big| \frac{a_s(s_0)}{2}
\end{equation*}
which implies
\begin{equation*}
\frac{a_s(s_0)}{a(s_0)} \geq -2 \sqrt{K}.
\end{equation*}
This concludes the proof.
\end{proof}

Now we may proceed to proving the maximum principle of Theorem \ref{maximum-principle}.

\begin{proof}[Proof of Case 1 of Theorem \ref{maximum-principle}]
Let $K>0$ such that
\begin{equation*}
\sup_{M_k \times [0,T]} |Rm_{g(t)}|_{g(t)} \leq K.
\end{equation*}
Introduce the new variable $r := r(s, t)$ defined by 
\begin{equation*}
r(r+2) = s^2
\end{equation*}
and let 
\begin{equation*}
\overline{u}(r,t) = u(\sqrt{r(r+2)},t).
\end{equation*}
Note that $ 2r \sim s^2$ for $s \ll 1$ and $r \sim s$ for $s \gg 1$. We make this substitution to remove the apparent singularity at $s=0$ in the $(s,t)$ coordinate representation (\ref{Pop-st}) of the operator $P[u]$. The function $\overline{u}$ is smooth, because $u$ is extendable to an even function around the origin (see \cite{W43}). Rewriting (\ref{Pop-st}) in terms of $(r,t)$ coordinates we see that $\overline{u}$ satisfies the inequality 
\begin{align*}
\partial_t \overline{u} \big|_{r} & \geq  A(r) \partial_{rr} \overline{u} + B(r,t) \partial_r \overline{u} + C(r,t) \overline{u}
\end{align*}
where $A$, $B$ and $C$ are smooth functions defined by 
\begin{align*}
A(r) &= \frac{r(r+2)}{(r+1)^2} \\
B(r,t) &= \left(m\frac{a_s}{a} + n\frac{b_s}{b} - \frac{\partial s}{\partial t} \right) \frac{s}{r+1} + \frac{1}{(r+1)^3} \\
C(r,t) &= c(s,t)
\end{align*}
Note that above we regard $s$ as a function of $r$. Recall that by Lemma \ref{parity-lem} the functions $a$ and $b$ can be extended to an odd and even function, respectively, around the origin. Therefore the quantity
\begin{equation}
\label{first-order-coef}
\left(m\frac{a_s}{a} +  n\frac{b_s}{b} - \frac{\partial s}{\partial t} \right)s, 
\end{equation}
considered as a function of $s$, can be extended to an even function around the origin by \cite{W43}. Hence this expression depends smoothly on $r$, showing that $B(r,t)$ is smooth. Similarly, we see that $C(r,t)$ is smooth. From the expression (\ref{dsdt}) for $\frac{\partial s}{\partial t}$ and the curvature components listed in subsection \ref{con-lap-cur-subsec} it follows that 
\begin{equation}
\label{dsdt-bound}
\Big|\frac{\partial s}{\partial t}\Big| = \Big|\int_0^s -K_1 - 2 K_2 \, \mathrm{d}s\Big| \leq 3 K s
\end{equation}
By Lemma \ref{b-bound} and Lemma \ref{asa-bound-lem} we hence see that
$$|B(r,t)| \leq M (r + 1)$$
for some some positive constant $M$ depending on $K$. Finally, noting that $A(r)$ is bounded and positive for $r > 0$, we can apply \cite[Theorem 9, p.43]{F64} to deduce that the weak maximum principle holds. Note that for any compact $U \subset M_k \times [0,T]$ we may assume that $c < 0$ on $U$ by performing the transformation $\overline{u} \leftarrow \overline{u} e^{-\gamma t}$, for $\gamma= \gamma(U)$ chosen sufficiently large. Therefore the strong maximum principle follows from a slight adaptation of \cite[Theorem 1, p.34]{F64}.
\end{proof}

\begin{proof}[Proof of Case 2 of Theorem \ref{maximum-principle}]
We first prove the weak maximum principle. Taking $u' = u e^{-\gamma t}$ we see that $u'$ satisfies
$$\partial_t u' \geq \partial_{ss} u' + \left( m \frac{a_s}{a} + n \frac{b_s}{b}\right) u'_s + (c - \gamma) u'.$$
As $c$ is a smooth function of $(s,t)$, we can choose $\gamma$ sufficiently large such that in a neighborhood of $\{s=0\}\times[0,T]$ we have $c - \gamma < 0$. Since $m+1 >0$ we see that $u'$ cannot attain a negative minimum on $\{s=0\} \times (0,T]$, as otherwise
\begin{equation*}
0 \leq \partial_t u' = (1 + m) u'_{ss} + c u' > 0,
\end{equation*}
which is a contradiction. The weak maximum principle now follows by the proof of \cite[Theorem 9, p.43]{F64}.

In this paper we only apply the strong maximum principle for $m\in \N$ and therefore only prove this case here. For the general case refer to \cite[Theorem 5.17]{Fee13}. Given a Ricci flow $(M_k, g(t))$, $t\in [0,T]$, define the corresponding family of rotationally symmetric spaces $(\R^{m+1}, h(t))$, $t \in [0,T]$, by
\begin{equation*}
h = ds^2 + a^2(s,t) g_{S^m(\frac{1}{k})}, 
\end{equation*}
where $g_{S^{m}(\frac{1}{k})}$ is the round metric on $S^m$ of sectional curvature $k^2$. A sufficient condition for $h$ to be smooth at $s=0$ is that $a$ is extendable to an odd function around the origin and
\begin{equation*}
a_s(0) = k.
\end{equation*}
Both these conditions are satisfied and we conclude that $h$ is a smooth metric. The Laplacian of a rotationally symmetric function $f$ on $(\R^{m+1},h)$ is given by
\begin{equation*}
\Delta_{h} f = f_{ss} + m\frac{a_s}{a} f_s
\end{equation*}
and thus the condition $P[u] \leq 0$ may be written as
\begin{equation*}
\partial_t u \geq \Delta_{h(t)} u + n \frac{b_s}{b} u_s + c u 
\end{equation*} 
Note that for any bounded $U \subset M_k \times [0,T]$ we may assume that $c < 0$ on $U$ by performing the transformation $u \leftarrow u e^{-\gamma t}$, for $\gamma= \gamma(U)$ chosen sufficiently large. Hence the desired result follows from \cite[Theorem 12.40]{ChII}.
\end{proof}

\begin{remark}
It was crucial in our analysis that $u$ is extendable to an even function around the origin, as the following example demonstrates:

Consider the degenerate parabolic equation
\begin{equation}
\label{u-eqn}
u_t = u_{xx} -2 \frac{u_x}{x} + 2 \frac{u}{x^2}
\end{equation}
on $x, t \geq 0$ with initial data satisfying $u(x,0) \leq 0$. If we take 
\begin{equation*}
u = x v 
\end{equation*}
a computation shows that the above PDE corresponds to 
\begin{equation}
\label{v-eqn}
v_t = v_{xx}
\end{equation}
Now considering (\ref{v-eqn}) as the heat equation on all of $\R$ we can set up initial data $v(x,0)$ such that
\begin{equation*}
v(x,0) \leq 0 \: \text{ for } \: x \geq 0,
\end{equation*}
however the solution $v$ to the heat equation becomes positive at some later time $t>0$ and $x >0$. This shows that $u \leq 0$ is not necessarily preserved by (\ref{u-eqn}).
\end{remark}

In this paper we also rely on a maximum principle for a system of parabolic inequalities on $u_1, u_2 \in C^\infty_{U(2)}(M_k \times [0,T])$ of the form 
\begin{align}
\label{para-system1}
\partial_t u_1 &\geq (u_1)_{ss} +  \left( m \frac{a_s}{a} + n \frac{b_s}{b}\right) (u_1)_s + h_{11} u_1 + h_{12} u_2 \\ 
\label{para-system2}
\partial_t u_2 &\geq (u_2)_{ss} +  \left( m \frac{a_s}{a} + n \frac{b_s}{b}\right) (u_2)_s + h_{21} u_1 + h_{22} u_2,
\end{align}
where $h_{ij}\in C^\infty_{U(2)}(M_k \times [0,T])$, $i,j = 1,2$, are bounded and satisfy
$$ h_{12}, h_{21} \geq 0 \: \text{ on } \: M_k \times [0,T].$$
We prove the following Lemma:

\begin{lem}
\label{maximum-principle-coupled}
Let $(M_k, g(t))$, $t \in [0,T]$, be a Ricci flow with bounded curvature. Assume $u_1, u_2 \in C^\infty_{U(2)}(M_k \times [0,T])$ satisfy the above system (\ref{para-system1})-(\ref{para-system2}) of parabolic inequalities and for some constants $M, \sigma > 0$ 
$$ u_1(s,t), u_2(s,t) \geq -M \exp(\sigma s^2) \: \text{ for } t \in [0,T].$$
If
\begin{align*}
u_1(s,0), u_2(s,0) &\geq 0 \: \text{ for } s \geq 0 \\
u_1(0,t), u_2(0,t) &\geq 0 \: \text{ for } t \in [0,T]
\end{align*}
then $u_1, u_2 \geq 0$ on $M_k\times[0,T]$.
\end{lem}
\begin{proof}
Writing the equation in terms of $(r,t)$ as in the proof of Theorem \ref{maximum-principle} we obtain 
\begin{align*}
\partial_t u_1 \Big |_{r} &\geq A(r) (u_1)_{rr} +  B(r,t) (u_1)_r + H_{11} u_1 + H_{12} u_2 \\
\partial_t u_2 \Big |_{r} &\geq A(r) (u_2)_{rr} +  B(r,t) (u_2)_r + H_{21} u_1 + H_{22} u_2,
\end{align*}
where $A(r)$, $B(r,t)$ are as in the proof of Theorem \ref{maximum-principle} and 
$$H_{ij} (r,t) = h_{ij}(s(r),t), \quad i,j = 1,2.$$
After constructing a barrier function of the form
$$H(s,t) = \exp\left[ \frac{k |s|^2}{1 - \mu t} + \nu t \right], \quad 0 \leq t \leq \frac{1}{2\mu},$$
the result follows combining the arguments of \cite[Theorem 1, p.34]{F64} and \cite[Theorem 13, p. 190]{PW84}.
\end{proof}

\section{K\"ahler quantities and the Eguchi-Hanson space}
\label{kahler-E-H-section}
Recall that a complex structure $J$ on a Riemannian manifold $(M,g)$ satisfying
\begin{enumerate}
\item $g(V_1,V_2) = g(JV_1,JV_2)$ for all $V_1, V_2 \in TM$
\item $\nabla J = 0$
\end{enumerate}
defines a K\"ahler structure. On the manifolds $M_k$, $k \geq 1$, we define two complex structures $J_1$ and $J_2$ by
$$ J_1 e_1 = e_0 \qquad J_1 e_2 = e_3 $$
and
$$ J_2 e_0 = e_2  \qquad J_2 e_1 = e_3.$$
A computation shows that $(M_k, g, J_1)$ is K\"ahler if and only if 
$$b_s - Q = 0.$$ 
Similarly, $(M_k, g, J_2)$ is K\"ahler if and only if 
\begin{equation*}
a_s + Q^2 - 2 = 0 \quad \text{and} \quad b_s - Q = 0.
\end{equation*}
Note that being K\"ahler with respect to $J_2$ automatically implies K\"ahlerity with respect to $J_1$. This motivates the definition of the following \emph{scale-invariant} quantities
\begin{align*}
x &:= a_s + Q^2 - 2 \\
y &:= b_s - Q
\end{align*}
to measure the deviation of a metric from being K\"ahler with respect to the complex structures $J_1$ and $J_2$. For example, the FIK shrinker \cite{FIK03} is K\"ahler with respect to the complex structure $J_1$ and in our notation satisfies $y=0$. The Eguchi-Hanson space is the unique K\"ahler manifold with respect to $J_2$ as the following lemma shows.
\begin{lem}
\label{unique-Kahler-lem}
Amongst all Riemannian manifolds $(M_k, g)$,$k \geq 1$, equipped with $U(2)$-invariant metric $g$ of the form (\ref{metric2}), up to scaling the Eguchi-Hanson space is the unique K\"ahler manifold with respect to the complex structure $J_2$. Furthermore being K\"ahler with respect to $J_2$ is equivalent to $x = y = 0$.
\end{lem}

\begin{proof}
By the above discussion being K\"ahler with respect to $J_2$ is equivalent to 
\begin{equation}
\label{xynull}
x = y = 0.
\end{equation}
Notice at $s=0$ we have 
$$x = a_s -2 = 0$$ 
forcing the underlying manifold to be diffeomorphic to $M_2$ by the boundary conditions (\ref{boundary-cond-intro-s}). Then in terms of $a$ and $b$ the condition $x = y = 0$ is equivalent to the first order system of equations
\begin{align}
\label{EH-a}
a_s &= 2 - Q^2 \\
\label{EH-b}
b_s &= Q
\end{align}
Let $a^E$ and $b^E$ be a solution to this system of equations satisfying the initial conditions  
\begin{align*}
a^E &= 0 \\
b^E &= 1
\end{align*}
at $s=0$. Then by the scale-invariance of condition $(\ref{xynull})$, for every $\lambda >0$ the metric given by the warping functions $\lambda a^{E}(\lambda s)$ and $\lambda b^{E}(\lambda s)$ also satisfies $(\ref{xynull})$. Hence up to rescaling there is a unique K\"ahler manifold with respect to the complex structure $J_2$. From \cite{EH79} or \cite{Cal79} we see that the metric given by $a^E$ and $b^E$ is homothetic to the Eguchi-Hanson metric.
\end{proof}

In the rest of the paper we denote by $g_E$ the Eguchi-Hanson metric with warping functions $a^E$ and $b^E$ normalized such that $b^E = 1$ on $S^2_o$. Note that the normalization condition is equivalent to saying that the area of the exceptional divisor $S^2_o$ is equal to $2 \pi$. 

\begin{lem}
\label{E-H-properties-lem}
The warping functions $a^{E}$ and $b^{E}$ of the Eguchi-Hanson metric satisfy the following properties
\begin{align*}
a^{E} ,b^{E} &\sim s \: \text{ as } \: s \rightarrow \infty \\
a^{E}_{ss} &< 0 \: \text{ for } \: s \geq 0 \\
b^{E}_{ss} &> 0 \: \text{ for } \: s \geq 0 \\
\frac{a^{E}}{b^{E}} & < 1 \: \text{ for } \: s \geq 0 \\
\end{align*} 
\end{lem}

\begin{proof}
For brevity write $a$ and $b$ for $a^{E}$ and $b^{E}$, respectively. Note that on the Eguchi-Hanson background we have
$$Q_s = \frac{1}{b} \left(a_s - Q b_s \right) = \frac{2}{b} \left( 1- Q^2 \right),$$
where the last equality follows from (\ref{EH-a}) and (\ref{EH-b}). As $Q = 0$ at $s=0$ it follows that 
$$Q < 1 \: \text{ for } \:  \geq 0$$ 
and hence
$$Q_s > 0 \: \text{ for } \: s \geq 0.$$
As 
\begin{align*}
a_s &= Q_s b + Q b_s = 2 - Q^2 
\end{align*}
it follows that 
$$a_{ss} = -2 Q Q_s < 0.$$ 
Similarly
$$b_{ss} = Q_s > 0.$$
Therefore the limits
$$ a_\infty := \lim_{s\rightarrow \infty} a_s $$ 
and
$$ b_\infty := \lim_{s\rightarrow \infty} b_s $$
both exist. From the system of differential equations (\ref{EH-a}) and (\ref{EH-b}) we then see that
$$a_{\infty} = b_{\infty} = 1.$$ This concludes the proof.
\end{proof}

\section{Some preserved conditions}
\label{section-preserved-conditions}
In this section we derive various \emph{scale-invariant} inequalities that are preserved by a Ricci flow $(M_k, g(t))$, $t \in [0,T]$, $k \in \N$. The scale-invariance is crucial, as it ensures that the inequalities pass to blow-up limits and therefore also constrain their geometry. The preserved inequalities we list in this section will play an important role in all subsequent sections. 

\vspace{1em}
\noindent\textbf{Section Outline.} 
A central quantity in our analysis is 
$$Q = \frac{a}{b}.$$
In geometric terms, $Q$ measures the deviation of the cross-sectional $S^3 / \Z_k$ in $M_k$ from being round. That is, when $Q=1$ the cross-section is round and as $Q \rightarrow 0$ the cross-sectional $S^3 / \Z_k$ collapses along the $S^1$ Hopf fibres to a two-sphere. A computation shows that the evolution equation of $Q$ is
\begin{equation}
\label{Q-evol}
\partial_t Q = Q_{ss} + 3 \frac{b_s}{b} Q_s + \frac{4}{b^2} Q(1-Q^2).
\end{equation}
Therefore one expects that the inequality $Q \leq 1$ is preserved by Ricci flow, which in Lemma \ref{Qleq1} we prove to be the case. 

Apart from $Q$, the K\"ahler quantities $x$ and $y$ introduced in section \ref{kahler-E-H-section} are used throughout this paper and are one of the key ingredients in showing that certain Ricci flows on $M_2$ develop singularities modeled on the Eguchi-Hanson space. We show in Lemma \ref{xleq0-lem} and Lemma \ref{yleq0-lem} that the inequalities
$$ x\leq 0$$
and
$$ y \leq 0$$
are both preserved. Furthermore, using the maximum principle for systems of weakly coupled parabolic equations of Lemma \ref{maximum-principle-coupled}, we show in Lemma \ref{ab-mono} that
$$a_s, b_s \geq 0$$
is preserved. In Lemma \ref{da-bounded} we show that on a Ricci flow background satisfying $Q \leq 1$ and $y \leq 0$, for any $C > 2$ the inequality $a_s \leq C$ is preserved. In the following sections we will mainly consider Ricci flows satisfying $a_s, b_s \geq 0$, $y \leq 0$, $Q\leq 1$ and $a_s \leq C$. This gives us enough control on $a$ and $b$ to prove many interesting results. 

Finally, we show that whenever a subset of the inequalities $Q \leq 1$, $y\leq 0$ and $a_s, b_s \geq 0$ hold, the details of which are discussed below, the following inequalities
\begin{align*}
T_1 &= a_s + 2Q^2 -2 \geq 0 \\
T_2 &= Qy - x = -a_s + Qb_s + 2\left(1 - Q^2 \right) \geq 0 \\
T_3 &= a_s - Qb_s - Q^2 + 1 \geq 0 \\
\min(T_1, T_4) &\geq 0 
\end{align*}
where
$$T_4 = a_s - \frac{1}{2} Q b_s - \left(1 - Q^2\right)$$
are preserved by the Ricci flow. The precise statements and proofs of these preserved inequalities can be found in Lemmas \ref{T1-preserved-lem}, \ref{T2-preserved-lem}, \ref{T3-preserved-lem} and \ref{minT1T4-preserved-lem} below.

The main idea in constructing the above inequalities is to study the evolution equation of the scale-invariant quantities
\begin{equation}
\label{cons-cond-form}
T_{(\alpha, \beta, \gamma)} = \alpha a_s + \beta Q b_s + \gamma Q^2, \quad \alpha, \beta, \gamma \in \R.
\end{equation}
For this we need to compute the evolution equations of $a_s$, $Qb_s$ and $Q^2$. Recall Definition \ref{def:smoothU2} of $C^\infty_{U(2)}(M_k\times [0,T])$. To simplify the formulae slightly, define the linear operator 
$$L: C^\infty_{U(2)}(M_k\times [0,T]) \rightarrow C^\infty_{U(2)}(M_k\times [0,T])$$
by
\begin{equation*}
L[u] = u_{ss} + \left( 2\frac{b_s}{b} -\frac{a_s}{a} \right) u_s
\end{equation*}
away from the non-principal orbit $S^2_o$. As in section \ref{sec:maximum-principle} we may use L'H\^optital's rule to find a representation of $L$ on the non-principal orbit $S^2_o$. Then, as we show in the Appendix A, the evolution equations of $a_s$, $Qb_s$ and $Q^2$ can be written as
\begin{align}
\label{as-evol}
\partial_t a_s &= L[a_s] + \frac{1}{b^2}\left( -2 a_s b_s^2- 6 Q^2 a_s + 8 Q^3 b_s \right) \\
\label{Qbs-evol}
\partial_t Q b_s &= L[Qb_s] +\frac{1}{b^2}\left( 4 Q^2 a_s- 10 Q^3 b_s - 2 Q b_s^3 + 8 Q b_s\right) \\
\label{Q2-evol}
\partial_t Q^2  &= L[Q^2] + \frac{1}{b^2}\left(4 Q a_s b_s - 4 Q^2 b_s^2 - 8 Q^4 + 8 Q^2 \right).
\end{align}
Since the operator $L$ is linear, one sees that $T_{(\alpha, \beta,\gamma)}$ satisfies an evolution equation of the form
\begin{equation}
\label{T-schematic-evol}
\partial_t T_{(\alpha, \beta, \gamma)} = L[T_{(\alpha, \beta, \gamma)}] + \frac{1}{b^2} C_{(\alpha, \beta, \gamma)},
\end{equation}
where $C_{(\alpha, \beta, \gamma)}$ is a function of $a_s$, $b_s$ and $Q$. This evolution equation is very useful, as it allows us to systematically search for preserved inequalities. In particular, if we can find $\alpha, \beta, \gamma, \delta \in \R$ for which we can determine the sign of $C_{(\alpha, \beta, \gamma)}$ at a local extrema of $T_{(\alpha, \beta, \gamma)}$ at which $T_{(\alpha, \beta, \gamma)} = \delta$, it follows from the maximum principle of Theorem \ref{maximum-principle} that, depending on the sign, either 
$$T_{(\alpha, \beta, \gamma)} \geq \delta$$
or
$$T_{(\alpha, \beta, \gamma)} \leq \delta$$
is a preserved inequality.

We searched for real numbers $\alpha, \beta, \gamma$ and $\delta$ leading to preserved conditions that yield the most useful control of the geometry of the flow. This is how we found the quantities $T_1$, $T_2$, $T_3$ and $T_4$. In later sections we will make heavy use of each of their respective inequalities. For instance, we use the preserved inequalities $T_1 \geq 0$ to exclude shrinking solitons on $M_k$, $k \geq2$, in the next section. Finally, in section \ref{E-H-unique-ancient-section} we generalize the above idea to find a continuously varying family of conserved inequalities.

\vspace{1em}
\noindent\textbf{Statement and proof of results.} 
In this subsection we list the precise statements and proofs of the results stated in the section outline. Before we begin, we prove the following technical lemma, which we need for verifying the growth conditions of the maximum principle of Theorem \ref{curv-bound}.

\begin{lem}
\label{lem:growth-cond}
Let $(M_k, g)$, $k \in \N$, satisfy $|Rm_g|_g \leq K$. Then
$$ |a_s|, |Qb_s| , |Q^2| = O(\exp(2\sqrt{K} s)).$$
\end{lem}

\begin{proof}
By the curvature components listed in section \ref{con-lap-cur-subsec} we see that
$$\left|\frac{a_{ss}}{a}\right|, \left|\frac{b_{ss}}{b}\right| \leq K.$$
Integrating
$$b_{ss} \leq b K,$$
shows that
$$b = O(\exp(\sqrt{K} s)).$$
From Lemma \ref{b-bound} we have
$$Q^2 \leq \frac{5}{3} K b^2$$
from which we conclude that
$$Q^2 = O(\exp(2\sqrt{K} s)).$$
Similarly, Lemma \ref{b-bound} shows
$$ |b_s| \leq \sqrt{5K} b$$
from which we deduce that
$$ |Qb_s| \leq \sqrt{\frac{25}{3}} K b^2$$
and hence
$$ |Qb_s| = O(\exp(2\sqrt{K} s)).$$
Finally,
$$ |a_{ss}| \leq a K $$
shows that
$$ |a_s| = O(\exp(\sqrt{K} s)).$$ 
This concludes the proof. 
\end{proof}

Now we begin proving the conserved inequalities listed above.

\begin{lem}
\label{Qleq1}
Let $(M_k, g(t))$, $t\in[0,T]$, $k \geq 1$, be a Ricci flow with bounded curvature. Then the inequality
$$Q \leq 1$$
is preserved by the Ricci flow. 
\end{lem}

\begin{proof}
Define the quantity $\tilde{Q} = 1 - Q$. From the evolution equation (\ref{Q-evol}) of $Q$ we see
$$\partial_t \tilde{Q}  = \tilde{Q}_{ss} + 3 \frac{b_s}{b} \tilde{Q}_s + \tilde{Q} \left( -\frac{4}{b^2} Q(1+Q) \right).$$
As $Q \geq 0$ everywhere, the coefficient $-\frac{4}{b^2} Q(1+Q)$ is non-positive. Furthermore, by Lemma \ref{lem:growth-cond} we have $|\tilde{Q}| = o(\exp(s^2))$. Therefore we may apply the maximum principle of Theorem \ref{maximum-principle} to deduce that $\tilde{Q} \geq 0$ on $M_k \times [0,T]$. The desired result thus follows.
\end{proof}

\begin{lem}
\label{xleq0-lem}
Let $(M_k, g(t))$, $t\in[0,T]$, $k = 1,2$, be a Ricci flow with bounded curvature. Then the inequality
\begin{equation*}
x \leq 0
\end{equation*} 
is preserved by the Ricci flow.
\end{lem}

\begin{proof}
The evolution equation of $x$, as derived in the Appendix A, is
\begin{align}
\label{x-evol}
\partial_t x &= L[x] - \frac{2}{b^2}\left(2 Q^2 + y^2\right) x - \frac{2}{b^2}\left(Q^2 +2 \right) y^2 \\
						&\leq L[x] - \frac{2}{b^2}\left(2 Q^2 + y^2\right) x \nonumber
\end{align}
Note that $|x| = o(\exp(s^2))$ by Lemma \ref{lem:growth-cond}. Therefore applying the maximum principle of Theorem \ref{maximum-principle} yields the desired result.
\end{proof}

\begin{remark}
Note that $x = 2 -k$ at $s=0$ by the boundary conditions (\ref{boundary-cond-intro-s}). Therefore the result can only hold for $k = 1,2$. 
\end{remark}

\begin{lem}
\label{yleq0-lem}
Let $(M_k, g(t))$, $t\in[0,T]$, $k \geq 1$, be a Ricci flow with bounded curvature. Then the inequality
$$y \leq 0$$
is preserved by the Ricci flow.
\end{lem}

\begin{proof}
Let $K >0$ such that
$$\sup_{M_k \times [0,T]} |Rm_{g(t)}|_{g(t)} < K.$$
Since $y$ is an odd quantity, we consider the quantity $Qy = Qb_s - Q^2$ instead. Its evolution equation is
\begin{equation}
\label{Qy-evol}
\partial_t Qy = L[Qy] - 2 \frac{Qy}{b^2}\left(2(Q^2+x) + Qy + y^2 \right).
\end{equation}
Note that 
\begin{align*}
- \frac{2}{b^2}\left(2(Q^2+x) + Qy + y^2 \right) &=  - \frac{2}{b^2} \left(4Q^2 -4 + b^2 M_1 + Qb_s + b_s^2 \right) \\
												 &\leq \frac{8}{b^2} + 2K + 2 \frac{|Q||b_s|}{b^2} \\
												 &\leq \frac{8}{b^2} + 2K  + \frac{Q^2}{b^2}+ \frac{b^2_s}{b^2}, 
\end{align*}
where $M_1$ is one of the curvature components listed in section \ref{con-lap-cur-subsec}. By Lemma \ref{b-bound} we see that for some $C>0$
$$- \frac{2}{b^2}\left(2(Q^2+x) + Qy + y^2 \right) \leq C K \: \text{ on } \: M_k \times [0,T]$$
Furthermore $|Qy| = o(\exp(s^2))$ by Lemma \ref{lem:growth-cond}. Now the result follows from applying the maximum principle of Theorem \ref{maximum-principle}.
\end{proof}

\begin{lem}
\label{ab-mono}
Let $(M_k, g(t))$, $t\in[0,T]$, $k \geq 1$, be a Ricci flow with bounded curvature. If the initial metric $g(0)$ satisfies $a_s, b_s \geq 0$, then $a_s, b_s \geq 0$ for all times $t\in[0,T]$.
\end{lem}

\begin{proof}
The evolution equations (\ref{as-evol}) and (\ref{Qbs-evol}) of $a_s$ and $Qb_s$ can be written as a system of weakly coupled parabolic equations
\begin{align}
\label{as-evol-2}
\partial_t a_s  &= L[a_s] - \left(2 \left(\frac{b_s}{b}\right)^2 + 6\frac{Q^2}{b^2} \right) a_s + 8 \frac{Q^2}{b^2} (Qb_s) \\
\partial_t Qb_s &= L[Qb_s] + 4 \frac{Q^2}{b^2} a_s + \left( \frac{8-10Q^2}{b^2} - 2 \left(\frac{b_s}{b}\right)^2 \right) (Qb_s), 
\end{align}
By Lemma \ref{b-bound} and Lemma \ref{asa-bound-lem} the zeroth order coefficients of $a_s$ and $Qb_s$ are bounded. Lemma \ref{lem:growth-cond} shows that $|a_s|, |b_s| = o(\exp(s^2))$. Finally, note that the off-diagonal coefficients $8\frac{Q^3}{b^2}$ and $4 \frac{Q^2}{b^2}$ are non-negative. Thus the desired result follows by the maximum principle for weakly coupled parabolic equations of Lemma \ref{maximum-principle-coupled}.
\end{proof}

\begin{lem}
\label{da-bounded}
Let $(M_k, g(t))$, $t\in[0,T]$, $k \geq 1$, be a Ricci flow with bounded curvature satisfying $y \leq 0$, $Q \leq 1$ and $a_s, b_s \geq 0$. Then for $C\geq 2$ the inequality 
$$a_s \leq  C$$ 
is preserved by the Ricci flow.
\end{lem}

\begin{proof}
Define the quantity $A \coloneqq a_s - C$. Then from the evolution equation (\ref{as-evol-2}) of $a_s$ it follows that
\begin{align*}
\partial_t A &= L[A] - \left(2 \left(\frac{b_s}{b}\right)^2 + 6\frac{Q^2}{b^2} \right) A + \frac{1}{b^2}\left(8Q^3 b_s - CQ^2 - 2 Cb_s^2 \right).
\end{align*}
Fix $C \geq 2$. Then
$$8Q^3 b_s - 6CQ^2 - 2 Cb_s^2 \leq 8 Q^4 - C Q^2 \leq (8 - 6C)Q^2 \leq 0,$$
where we used $Q \leq 1$ and $y = b_s - Q \leq 0$. As $|A| = o(\exp(s^2))$ by Lemma \ref{lem:growth-cond} it follows from the maximum principle of Theorem \ref{curv-bound} that the inequality $A \leq 0$ is preserved by the Ricci flow. This proves the desired result.
\end{proof}

\begin{lem}
\label{T1-preserved-lem}
Let $(M_k, g(t))$, $t\in[0,T]$, $k \geq 1$, be a Ricci flow with bounded curvature satisfying $y \leq 0$, $b_s \geq 0$ and $Q \leq 1$. Then the inequality
\begin{equation*}
T_1 = a_s + 2Q^2 - 2 \geq 0
\end{equation*} 
is preserved by the Ricci flow.
\end{lem}

\begin{proof}
The evolution equation of $T_1$ is 
\begin{align}
\label{T1-evol}
\partial_t T_1  &= L[T_1] + \frac{1}{b^2} \left[- 4\left(1+Q^2\right)y^2 + 8Q\left(1-2Q^2 \right)y + 16Q^2\left(1-Q^2\right)\right] \\ \nonumber
			    & \qquad \qquad  +  T_1 \frac{2 y}{b^2}\left( 2Q- y\right),
\end{align}
which can be derived from the evolution equations (\ref{as-evol}), (\ref{Qbs-evol}) and (\ref{Q2-evol}) for $a_s$, $Qb_s$ and $Q^2$ listed above. Inspecting the quadratic expression
\begin{equation}
\label{quadexpr}
- 4\left(1+Q^2\right)y^2 + 8Q\left(1-2Q^2 \right)y + 16Q^2\left(1-Q^2\right)
\end{equation}
 we see that when $y=0$ it is equal to 
 \begin{equation*}
16Q^2\left(1-Q^2\right) \geq 0 
\end{equation*}
and when $y = - Q$ it is equal to
\begin{equation*}
4Q^2\left(1- Q^2\right) \geq 0 
\end{equation*}
As $y = b_s - Q \in [-Q,0]$ by the assumptions $y \leq 0$, $b_s \geq 0$ and $Q \leq 1$, and furthermore the quadratic expression (\ref{quadexpr}) is concave in $y$, we conclude that
\begin{equation*}
\partial_t T_1  \geq L[T_1] + \frac{2y}{b^2}\left( 2Q- y\right) T_1
\end{equation*}
Note that the zeroth order coefficient of $T_1$ is bounded by Lemma \ref{b-bound}. Furthermore $|T_1| = o(\exp(s^2))$ by Lemma \ref{lem:growth-cond}. Hence the result follows from applying the maximum principle of Theorem \ref{maximum-principle}.
\end{proof}

Below we prove some further preserved conditions. These can be skipped on the first reading of the paper.

\begin{lem}
\label{T2-preserved-lem}
Let $(M_k, g(t))$, $t\in[0,T]$, $k = 1, 2$, be a Ricci flow with bounded curvature satisfying $Q \leq 1$. Then the condition
$$T_2 = Qy - x = - a_s + Qb_s  + 2\left( 1 - Q^2 \right) \geq 0$$
is preserved by the Ricci flow.
\end{lem}
\begin{proof}
Note that $T_2 = 2 - k$ when $s = 0$ by the boundary conditions (\ref{boundary-cond-intro-s}). Therefore the result can only hold true for $k = 1,2$. The evolution equations of $T_2$ is
\begin{align}
\label{T2-evol}
\partial_t T_2 &= L[T_2] + \frac{4}{b^2}\left(1- Q^2\right) y^2  -2 \frac{T_2}{b^2} \left( \left(b_s-2Q\right)^2+ Q^2\right).
\end{align}
The coefficients are bounded by Lemma \ref{b-bound}. Furthermore $|T_2| = o(\exp(s^2))$ by Lemma \ref{lem:growth-cond}. Therefore applying the maximum principle of Theorem \ref{maximum-principle} yields the desired result.
\end{proof}

\begin{lem}
\label{T3-preserved-lem}
Let $(M_k, g(t))$, $t\in[0,T]$, $k \geq 1$, be a Ricci flow with bounded curvature satisfying $Q \leq 1$.
Then the inequality
$$T_3 = a_s - Qb_s - Q^2 + 1 \geq 0$$
is preserved by the Ricci flow.
\end{lem}
\begin{proof}
The evolution equations of $T_3$ is
\begin{align}
\label{T3-evol}
\partial_t T_3 &= L[T_3] + \frac{2}{b^2} \left(1 - Q^2\right) y^2 - 2 \frac{T_3}{b^2} \left( (b_s + Q)^2+ 4Q^2\right)
\end{align}
Note that the coefficients are bounded by Lemma \ref{b-bound}. Furthermore $|T_3| = o(\exp(s^2))$ by Lemma \ref{lem:growth-cond}. Applying the maximum principle of Theorem \ref{maximum-principle} yields the desired result.
\end{proof}

\begin{lem}
\label{minT1T4-preserved-lem}
Let $(M_k, g(t))$, $t\in[0,T]$, $k \geq 1$, be a Ricci flow with bounded curvature satisfying $y \leq 0$, $b_s \geq 0$ and $Q \leq 1$.
Then the inequality
\begin{align*}
\min(T_1, T_4) \geq 0
\end{align*}
is preserved by the Ricci flow. Here
\begin{align*}
T_1 &= a_s + 2Q^2 - 2\\
T_4 &= a_s - \frac{1}{2} Q b_s - \left(1 - Q^2\right).
\end{align*}
\end{lem}

\begin{proof}
By Lemma \ref{T1-preserved-lem} we already know that the inequality
$$T_1 = a_s - 2 + 2Q^2 \geq 0$$
is preserved. Thus we only need to show that $T_4\geq0$ is preserved whenever the Ricci flow satisfies $T_1 \geq 0$. The evolution equation of $T_4$ is
\begin{equation*}
\partial_t T_4  = L[T_4] + \frac{1}{b^2}\left(b_s\left(5Q^3-2b_s\right) -2 T_4\left(4Q^2 - 2Q b_s +b_s^2\right) \right).
\end{equation*}
A computation shows 
\begin{equation*}
\frac{1}{2}Qb_s = T_1 - T_4 + 1 - Q^2.
\end{equation*}
By the assumption $y \leq 0$ we have
$$ \frac{Q^2}{2} \geq \frac{1}{2}Qb_s$$
and hence it follows that
\begin{equation*}
Q^2 \geq \frac{2}{3}\left(1-T_4\right).
\end{equation*}
Therefore
\begin{equation*}
5Q^3 - 2 b_s \geq 5 Q^3 - 2Q \geq Q\left( \frac{4}{3} - \frac{10}{3} T_4 \right),
\end{equation*}
which implies that
\begin{equation*}
\partial_t T_4  \geq  L[T_4] - \frac{2 T_4}{b^2}\left(4 Q^2 - \frac{1}{3}Qb_s + b_s^2 \right)
\end{equation*}
since $b_s \geq 0$. Note that the zeroth order coefficient of $T_4$ is bounded by Lemma \ref{b-bound}. Furthermore $|T_4| = o(\exp(s^2))$ by Lemma \ref{lem:growth-cond}. Applying the maximum principle of Theorem \ref{maximum-principle} yields the desired result.
\end{proof}

\section{Exclusion of shrinking solitons}
\label{non-existence}
In this section we rule out $U(2)$-invariant shrinking solitons on $M_k$, $k\geq 2$, within a large class of metrics. In particular, we show

\begin{restatable}[No shrinker]{thm}{noshrinker}
\label{thm:no-shrinker}
On $M_k$, $k \geq 2$, there does not exists a complete $U(2)$-invariant shrinking Ricci soliton of bounded curvature satisfying the conditions
\begin{enumerate}
\item $\sup_{p \in M_k} |b_s| < \infty$ 
\item $T_1 = a_s + 2 Q^2 - 2 > 0$ for $s > 0$
\item $Q = \frac{a}{b} \leq 1$
\end{enumerate}
\end{restatable}

This theorem is the key ingredient in section \ref{E-H-sing-section}, where we show that certain Ricci flows on $M_k$, $k \geq 3$, develop Type II singularities in finite time. 

\vspace{1em}
\noindent\textbf{Soliton equations.} 
Recall that a shrinking Ricci soliton $(M,g(t))$ is a solution to the Ricci flow equation that up to diffeomorphism homothetically shrinks. Such a soliton solution may be written as
$$ g(t) = \sigma^2(t) \Phi^\ast_t g(0),$$
where
$$ \sigma(t) = \sqrt{1 - 2 \rho t}$$
for some $\rho > 0$ and $\Phi_t$ is a family of diffeomorphisms. The reader may consult \cite{Top06} for more details. Hence for a $U(2)$-invariant shrinking Ricci soliton $(M_k, g(t))$, $k \geq 1$, the corresponding warping functions can be written as 
\begin{align}
\label{a-sol-ricci-cor}
a(s,t) &= \sigma(t)a\left(\frac{s}{\sigma(t)},0\right) \\
\label{b-sol-ricci-cor}
b(s,t) &= \sigma(t)b\left(\frac{s}{\sigma(t)},0\right).
\end{align}
The above formulae are with respect to the radial coordinate $s$, which is equivalent to fixing a gauge. For this reason the family of diffeomorphisms $\Phi_t$ does not appear explicitly. Differentiating with respect to $t$ at time $0$ yields
\begin{align*}
\partial_t|_{t=0} a(s,t) &= a_s(s,0)\left( \frac{\partial s} {\partial t} + \rho s \right) - \rho a(s,0) \\
						   &= a_s(s,0) f_s - \rho a(s,0), 
\end{align*}
where $f: M_k \rightarrow \R$ is the potential function satisfying
\begin{equation*}
f_{ss} = \rho + \frac{a_{ss}}{a} + 2 \frac{b_{ss}}{b}
\end{equation*}
and we used the expression (\ref{dsdt}) for $\frac{\partial s}{\partial t}$ derived in section \ref{ricci-flow-equations-sec}. Similarly we obtain
$$\partial_t|_{t=0} b(s,t) = b_s(s,0) f_s(s) - \rho b(s,0).$$
Substituting the expressions $\partial_t a$ and $\partial_t b$ from the Ricci flow equations (\ref{a-evol}) and (\ref{b-evol}), respectively, we see that the soliton equations for the warping functions $a$ and $b$ at time $t=0$ read (c.f. \cite{A17})
\begin{align}
\label{p-soliton} f_{ss} &= \frac{a_{ss}}{a}+ 2 \frac{b_{ss}}{b} + \rho \\
\label{a-soliton} a_{ss} &=  2 \frac{a^3}{b^4} - 2 \frac{a_sb_s}{b} + a_s f_s - \rho a\\
\label{b-soliton} b_{ss} &=  \frac{4}{b} -  2\frac{a^2}{b^3} - \frac{a_sb_s}{a} - \frac{b^2_s}{b} + b_s f_s - \rho b
\end{align}
In a slight abuse of notation we will denote $a$ and $b$ as functions of $s$ only when we are considering Ricci solitons. In that case $a$ and $b$ should be interpreted as the initial data $a(s,0)$ and $b(s,0)$ at time zero that leads to a Ricci soliton solution, via the correspondence (\ref{a-sol-ricci-cor}) and (\ref{b-sol-ricci-cor}).

\begin{remark}
The above shows that all $U(2)$-invariant Ricci solitons on $M_k$ are automatically gradient Ricci solitons with potential function $f$.
\end{remark}

\vspace{1em}
\noindent\textbf{Evolution of $x$, $y$ and $Q$ on soliton background.} 
Since $x$, $y$ and $Q$ are \emph{scale-invariant} quantities, their evolution on a  Ricci soliton background can be expressed as follows:
\begin{align*}
x(s, t) &= x\left(\frac{s}{\sigma(t)},0\right) \\
y(s, t) &= y\left(\frac{s}{\sigma(t)},0\right) \\
Q(s, t) &= Q\left(\frac{s}{\sigma(t)},0\right) \\
\end{align*}
Differentiating, we therefore obtain
\begin{align*}
\partial_t|_{t=0} x(s,t) &= x_s(s,0) f_s(s) \\
\partial_t|_{t=0} y(s,t) &= y_s(s,0) f_s(s) \\
\partial_t|_{t=0} Q(s,t) &= Q_s(s,0) f_s(s).
\end{align*}
With help of the evolution equations (\ref{x-evol}), (\ref{y-evol}) and (\ref{Q-evol}) for $x$, $y$ and $Q$, this yields the following ordinary differential equations for $x$, $y$ and $Q$ at time zero on a soliton background
\begin{align}
\label{soliton-x}
0 &= x_{ss} + \left(2\frac{b_s}{b}-\frac{a_s}{a}- f_s\right)x_s - \frac{1}{b^2}\left( 2 Q^2 \left(2x + y^2\right) + 2 y^2\left(2 + x\right)\right)\\
\label{soliton-y}
0 &= y_{ss} + \left(\frac{a_s}{a} - f_s\right)y_s -\frac{y}{a^2} \left( \left(x+2\right)^2 + Q^2 \left(2x + y^2 \right)\right) \\
\label{soliton-Q}
0 &= Q_{ss} + \left(3\frac{b_s}{b} - f_s\right) Q_s + \frac{4}{b^2}Q\left(1-Q^2\right).
\end{align}
Alternatively these equations can be derived from the soliton equations (\ref{p-soliton})-(\ref{b-soliton}). In a slight abuse of notation we will often denote $x$, $y$ and $Q$ as functions of $s$ only when we are considering Ricci solitons. 

\vspace{1em}
\noindent\textbf{Exclusion of shrinking solitons.} By \cite{CZ10} we know that the potential function of a non-compact complete shrinking Ricci soliton grows quadratically with the distance to some fixed point. In our setting this translates into the following lemma:
\begin{lem}
\label{f-asymp}
Let $(M_k, g)$, $k \geq 1$, be a complete non-compact shrinking Ricci soliton of bounded curvature. Then
\begin{align*}
f &\sim \frac{\rho}{2} s^2 \\
f_s&\sim \rho s 
\end{align*}
as $s\rightarrow \infty$.
\end{lem} 
\begin{proof}
See Theorem 1.1, equation (2.3) and equation (2.8) of \cite{CZ10}.
\end{proof}

This allows us to prove the following lemma:
\begin{lem}
\label{dQ-Lemma}
Let $(M_k, g)$, $k \geq 1$, be a complete non-compact shrinking Ricci soliton of bounded curvature with $Q \leq 1$ on $M_k$. Then $Q_s \geq 0$ on $M_k$.
\end{lem}

\begin{proof}
First notice that for a complete shrinking Ricci soliton with $Q \leq 1$, the strong maximum principle applied to the evolution equation (\ref{Q-evol}) of $Q$ forces
\begin{equation*}
Q < 1 \: \text{ for } \: s \geq 0,
\end{equation*}
as otherwise we would have $Q=1$ everywhere, which cannot be the case. Similarly, 
\begin{equation*}
Q > 0
\end{equation*}
unless we are at the origin $s=0$. By equation (\ref{soliton-Q}) we have
\begin{equation}
\label{soliton-Q2}
Q_{ss} = \left(f_s - 3 \frac{b_s}{b}\right)Q_s - \frac{4}{b^2}Q\left(1-Q^2\right).
\end{equation}
We now argue by contradiction. Assume there exists an $s_{\ast} > 0$ such that $Q_s(s_{\ast}) < 0$. Then $Q_s(s) < 0$ for all $s> s_{\ast}$, because at any extremum of $Q$ we have $Q_s = 0$ and  
\begin{equation*}
Q_{ss} = - \frac{4}{b^2} Q \left(1 - Q^2 \right) < 0.
\end{equation*}
Lemma \ref{b-bound} shows that $\frac{b_s}{b}$ is bounded and from Lemma \ref{f-asymp} it follows that
$$f_s \rightarrow \infty \: \text{ as } \: s \rightarrow \infty.$$
Therefore eventually
$$ f_s - 3 \frac{b_s}{b} > 0$$
from which it follows by equation (\ref{soliton-Q2}) that 
$$Q_{ss} <0$$
for sufficiently large $s$. This, however, contradicts that $Q > 0$ unless $s=0$.
\end{proof}

In the lemma below we bound the term 
$$G := (x+2)^2 + Q^2(2x+y^2),$$ 
which appears in the evolution equation (\ref{soliton-y}) of $y$, away from zero.

\begin{lem}
\label{Gpos-lem}
Whenever $Q_s \geq 0$ and $Q, T_1 > 0$ we have $G > 0$.
\end{lem}

\begin{proof}
We have
\begin{equation*}
\frac{Q_s}{Q} = \frac{a_s}{a}-\frac{b_s}{b} = \frac{x}{a} - \frac{y}{b} + \frac{2}{a} - \frac{2a}{b^2}.
\end{equation*}
For $Q_s \geq 0 $ it follows that 
\begin{equation*}
x - Q y \geq 2\left(Q^2 - 1\right).
\end{equation*}
Recall the quantity
$$ T_1 = a_s + 2 Q^2 - 2$$
defined in section \ref{section-preserved-conditions}. Then
\begin{align*}
G &\geq \left(x+2\right)^2 + Q^2 \left(2\left(Qy + 2\left(Q^2 - 1\right) \right) + y^2\right) \\
  &= x^2 + 4x + 4 + 2 Q^3 y + 4Q^4 - 4Q^2 + Q^2 y^2 \\
  &=\left( a_s + Q^2 -2 \right)^2 + 4\left( a_s + Q^2 -2 \right) + 4 + 3 Q^4 - 4 Q^2 + Q^2\left(y + Q\right)^2 \\
  &=a_s^2 + 2 Q^2 a_s + 4\left(Q^4 - Q^2\right) + Q^2 \left(y + Q\right)^2 \\
  &=a_s^2 + 2Q^2 T_1 + Q^2 \left( y + Q \right)^2 \\
  &=a_s^2 + Q^2 b_s^2 + 2 Q^2 T_1 > 0
\end{align*}
\end{proof}

Now we prove the non-existence of shrinking solitons.

\begin{proof}[Proof of Theorem \ref{thm:no-shrinker}]
We argue by contradiction. Assume such a shrinking Ricci soliton exists. Applying L'H\^opital's Rule to the evolution equation (\ref{b-evol}) of $b$ shows that at $s=0$
\begin{align*}
\partial_t b \Big | _{s=0} &= 2 b_{ss} - \frac{4}{b} \\ \nonumber
						   &= 2 \left( y_s + \frac{k-2}{b}\right).
\end{align*} 
Clearly, every shrinking soliton satisfies 
$$\partial_t b \Big |_{s=0} < 0.$$
The boundary conditions (\ref{boundary-cond-intro-s}) of $a$ and $b$ at $s=0$ imply that
\begin{equation*}
y(0) = 0.
\end{equation*}
and thus we deduce from the above that 
$$y_s(0) < 0,$$
as $k \geq 2$ by assumption. The ordinary differential equation (\ref{soliton-y}) for $y$ can be written as
\begin{equation}
\label{soliton-y-simp}
y_{ss} =  \left(f_s - \frac{a_s}{a} \right)y_s + \frac{y}{a^2} G.
\end{equation}
Lemma \ref{dQ-Lemma} and Lemma \ref{Gpos-lem} imply that $$G > 0\:\text{ for }\: s > 0,$$  
which in turn shows that $y_s \leq 0$ everywhere, as at a negative local minimum of $y$ we would have
\begin{equation*}
y_{ss} = \frac{y}{a^2} G < 0.
\end{equation*}
The asymptotic properties of $f$ listed in Lemma \ref{f-asymp} and the bounds on $\frac{a_s}{a}$ proven in Lemma \ref{asa-bound-lem} show that eventually
$$f_s - \frac{a_s}{a} > 0$$
and hence from the equation (\ref{soliton-y-simp}) it follows that
$$ y_{ss} < 0$$
for $s$ sufficiently large. From this it follows that $$\lim_{s \rightarrow \infty} y = \lim_{s \rightarrow \infty} \left( b_s - Q \right)= - \infty,$$ which contradicts our assumptions on $b_s$ and $Q$. 
\end{proof}

\section{Curvature bound}
\label{sec:curv-bound}
The aim of this section is to prove that a Ricci flow $(M_k, g(t))$, $k \in \N$, $t\in [0,T)$, starting from an initial metric $g(0) \in \mathcal{I}$ --- where $\mathcal{I}$ is a class of metrics to be discussed below --- with $\sup_{p \in M_k} b(p,0) < \infty$ satisfies the curvature bound
$$ |Rm_{g(t)}|_{g(t)} \leq C_1 b^{-2} \: \text{ for } \:  t \in (0,T),$$
where $C_1>0$ is some constant. This allows us to control the geometry via the warping function $b$, which will be crucial for constructing blow-up limits in the following parts of the paper. Note that this bound was already derived in the compact case in \cite{IKS17} and we will follow their strategy to prove it in our non-compact setting.

Recall the following definition (see also \cite{ChI}[Definition 8.23]):
\begin{definition}[$\kappa$-non-collapsing]
Let $(M,g(t))$, $t \in [0,T)$, be a Ricci flow and $\kappa > 0$. We say that the Ricci flow is $\kappa$-non-collapsed at a point $(p_0, t_0)$ in spacetime at scale $\rho$ if the following two conditions hold for all $r \leq \rho$:
\begin{itemize}
\item (bounded normalized curvature) We have $|Rm(p,t)|\leq r^{-2}$ for every $(p,t) \in B_{g(t_0)}(p_0,r) \times [ t_0 - r^2, t_0]$. In particular we assume $ [ t_0 - r^2, t_0] \subset [0,T)$.
\item (non collapsed volume) At time $t_0$ the ball $B_{g(t_0)}(p_0, r)$ has volume at least $\kappa r^4$.
\end{itemize}
\end{definition}

We now define the class of metrics $\mathcal{I}$.
\begin{definition}
\label{def:I}
For $K>0$ let $\mathcal{I}_K$ be the set of all complete \emph{bounded curvature} metrics of the form (\ref{metric2}) on $M_k$, $k \geq 1$, with \emph{positive injectivity radius} that satisfy the following scale-invariant inequalities:
\begin{align}
\label{I1} Q &\leq 1 \\
a_s, b_s &\geq 0 \\
y &\leq 0 \\
\label{I4} \sup a_s &< K \\
\label{I5} \sup |b b_{ss}| &< K
\end{align}
Denote by $\mathcal{I}$ the set of metrics $g$ such that for sufficiently large $K>0$ we have $g\in \mathcal{I}_K$.
\end{definition}
Note that for any $k \in \N$ the set $\mathcal{I}$ of metrics on $M_k$ is non-empty, as for example the metric on $M_k$ defined by
\begin{align*}
a(s) &= Q = \tanh(k s), \quad k \in \N \\
b(s) &= 1
\end{align*}
is contained in $\mathcal{I}$. In Lemma \ref{I-preserved} below we show that if $g(0) \in \mathcal{I}_{K_0}$ for some $K_0>0$ then there exists a $K> K_0$ such that $g(t)\in \mathcal{I}_K$ for $t \in [0, T)$. Note that conditions (\ref{I1})- (\ref{I5}) are scale-invariant, and therefore pass to blow-up limits.

An adaptation of \cite[Theorem 8.26]{ChI} to our setting yields the following result:
\begin{thm}[No local collapsing]
\label{thm:no-local-collapsing}
Let $g(t)$, $t \in [0,T)$, $T< \infty$, be a Ricci flow starting from an initial metric $g(0) \in \mathcal{I}$. Then there exists a $\kappa >0$ depending on $T$, $\mathrm{inj}(g(0))$ and $\sup_{M \times [0, T/2]} Ric_{g(t)}$ such that $g(t)$ is $\kappa$-non-collapsed at every $(p,t) \in M\times (\frac{T}{2}, T)$ at every scale $\rho < \sqrt{T/2}$. 
\end{thm}

\begin{remark}
Recall that if a Ricci flow $g(t)$ is $\kappa$-non-collapsed at scale $\rho$, then the parabolically dilated Ricci flow $\alpha^2 g( \alpha^{-2} t)$ is $\kappa$-non-collapsed at scale $\alpha \rho$. As the $\kappa$-non-collapsedness property is preserved under Cheeger-Gromov limits, a blow-up limit of a Ricci flow $(M_k, g(t))$, $[0, T_{sing})$ is $\kappa$-non-collapsed at all scales. 
\end{remark}

Having set up the necessary terminology, we may now state the main theorem of this section:
\begin{thm}[Curvature bound]
\label{curv-bound}
Let $(M_{k}, g(t))$, $t \in [0, T)$, be a Ricci flow starting from an initial metric $g(0) \in \mathcal{I}$ (see Definition \ref{def:I}) with
\begin{equation*} 
\sup _{p \in M_k} b(p,0) < \infty.
\end{equation*}
Then there exists a constant $C_1 > 0$ such that
\begin{equation*}
|Rm_{g(t)}|_{g(t)}(p) \leq C_1 b(p,t)^{-2}
\end{equation*}
for $(p,t) \in M_k \times (0,T)$.
\end{thm}

A useful variant of Theorem \ref{curv-bound} is:

\begin{cor}
\label{cor:curv-bound-ancient}
Let $(M_{k}, g(t))$ with $g(t) \in \mathcal{I}$ (see Definition \ref{def:I}) for $t \in (-\infty, 0]$ be an ancient Ricci flow solution which is $\kappa$-non-collapsed at all scales. Then there exists a constant $C_1 > 0$ such that
\begin{equation*}
|Rm_{g(t)}|_{g(t)}(p) \leq C_1 b(p,t)^{-2}
\end{equation*}
for $(p,t) \in M_k \times (-\infty, 0]$.
\end{cor}

\begin{remark}
Corollary \ref{cor:curv-bound-ancient} follows immediately from Theorem \ref{curv-bound} for ancient $\kappa$-non-collapsed Ricci flows that arise as blow up limits of Ricci flows $(M_k, g(t))$, $t \in [0, T_{sing})$, $g(0) \in \mathcal{I}$, as the curvature bound is scale-invariant. Nevertheless, we give a proof of the general case.
\end{remark}

Let us now prove the assertions made above. We begin with the following lemma:

\begin{lem}
\label{bddb-bound-lem}
Let $K_0 > 0$ and assume that $(M_k,g(t))$, $k \geq 1$, $t\in[0,T)$, is a Ricci flow starting from an initial metric $g(0) \in \mathcal{I}_{K_0}$. Then there exists a constant $K \geq 0$, depending only on the initial metric $g(0)$, such that
\begin{equation}
\label{bddb-bound} 
|bb_{ss}| \leq K
\end{equation}
on $M_k \times [0,T)$.
\end{lem}

\begin{proof}
We follow the proof strategy of \cite[Lemma 7]{IKS17}. Consider the quantities
\begin{align*}
H_- &= b b_{ss} + a_s^2 - b_s^2  - C \\
H_+ &= b b_{ss} - a_s^2 - b_s^2  + C, 
\end{align*}
where $C > 0$ is a constant to be determined later. The goal is to show that the inequalities $H_+ \geq 0$ and $H_- \leq 0$ are preserved by the Ricci flow for sufficiently large $C> 0$.  The quantities $H_\pm$ satisfy the evolution equations
\begin{align*}
\partial_t  H_\pm &= [H_\pm]_{ss} + \left(\frac{a_s}{a} -2 \frac{b_s}{b}\right) [H_\pm]_s +H_\pm \left(-\frac{2 a_s^2}{a^2}-\frac{4 a^2}{b^4}-\frac{4b_s^2}{b^2}\right) \\
 & \pm C \left(\frac{2 a_s^2}{a^2}+\frac{4 a^2}{b^4}+\frac{4b_s^2}{b^2}\right) \\
 &\pm 2 a_{ss}^2  + a_{ss} \left(-\frac{2 b a_s b_s}{a^2} \mp \frac{8 a_s b_s}{b} \pm \frac{4 a_s^2}{a}+\frac{4a}{b^2}\right) \\
 &+\frac{2 b a_s^3 b_s}{a^3}-\frac{32 a a_sb_s}{b^3}\mp\frac{16 a^3 a_s b_s}{b^5}+\frac{4 a_s^2}{b^2} \pm \frac{8 a^2a_s^2}{b^4} \\
 &\mp \frac{2 a_s^4}{a^2} +\frac{32 a^2 b_s^2}{b^4}-\frac{16 b_s^2}{b^2}.
\end{align*}
In the Appendix A we carry out the derivation of the evolution equation. We now show that $H_- \leq 0$ is preserved. Using Young's inequality to bound the terms involving $a_{ss}$ and then disregarding non-positive terms not involving $C$, we obtain
\begin{align*}
\partial_t  H_- &\leq [H_-]_{ss} + \left(\frac{a_s}{a} -2 \frac{b_s}{b}\right) [H_-]_s +H_- \left(-\frac{2 a_s^2}{a^2}-\frac{4 a^2}{b^4}-\frac{4b_s^2}{b^2}\right) \\
 &-C \left(\frac{2 a_s^2}{a^2}+\frac{4 a^2}{b^4}+\frac{4b_s^2}{b^2}\right) \\
 &+\frac{1}{2}\left(\left(\frac{2 b a_s b_s}{a^2}\right)^2+ \left(\frac{8 a_s b_s}{b}\right)^2 +\left(\frac{4 a_s^2}{a}\right)^2+\left(\frac{4a}{b^2}\right)^2\right) \\
 &+\frac{2 b a_s^3 b_s}{a^3}+\frac{16 a^3 a_s b_s}{b^5}+\frac{4 a_s^2}{b^2}+\frac{2 a_s^4}{a^2} +\frac{32 a^2 b_s^2}{b^4}
\end{align*}
Recall that on $M_k \times [0,T)$ we have $y = b_s - Q \leq 0 $, $Q \leq 1$, $a_s,b_s \geq 0$ and $a_s \leq C'$ for some $C'>0$ by Lemma \ref{yleq0-lem}, Lemma \ref{Qleq1}, Lemma \ref{ab-mono} and Lemma \ref{da-bounded}, respectively. Therefore we obtain the following bounds away from the non-principal orbit $S^2_o$:
\begin{align*}
\left( \frac{2b a_s b_s}{a^2}\right)^2 &= \left( \frac{2 a_s b_s}{a Q} \right)^2 \leq \frac{4 a_s^2}{a^2} \\
\left(\frac{8 a_s b_s}{b} \right)^2 &= \left(8\frac{a_s}{a} Qb_s\right)^2 \leq 64 \frac{a_s^2}{a^2} \\
\left(\frac{4 a_s^2}{a}\right)^2 &= 16 C'^2 \frac{a_s^2}{a^2} \\
\frac{2 b a_s^3 b_s}{a^3} &= \frac{ 2a_s^3}{a^2} \frac{b_s}{Q} \leq 2 C' \frac{a_s^2}{a^2} \\
\frac{16 a^3 a_s b_s}{b^5} &= 16 Q^3 \frac{a_sb_s}{b^2} \leq 8Q^3 \left( \frac{a_s^2}{b^2} + \frac{b_s^2}{b^2} \right) \leq  8\left( \frac{a_s^2}{a^2} + \frac{b_s^2}{b^2} \right) \\
\frac{4 a_s^2}{b^2} &\leq 4 \frac{a_s^2}{a^2} \\
\frac{2 a_s^4}{a^2} &\leq 2 C'^2 \frac{a_s^2}{a^2} \\
\frac{32 a^2 b_s^2}{b^4} &= 32 Q^2 \frac{b_s^2}{b^2} \leq 32 \frac{b_s^2}{b^2}
\end{align*}
Hence for a sufficiently large $C>0$ it follows that 
\begin{align*}
\partial_t  H_- &\leq [H_-]_{ss} + \left(\frac{a_s}{a} -2 \frac{b_s}{b}\right) [H_-]_s + H_- \left(-\frac{2 a_s^2}{a^2}-\frac{4 a^2}{b^4}-\frac{4b_s^2}{b^2}\right)
\end{align*}
away from the non-principal orbit $S^2_o$. Switching to coordinates $(s,t)$ we see that for $s > 0$
\begin{align}
\label{H-evol-ineq}
\partial_t\Big|_{s}  H_- &\leq [H_-]_{ss} + \left(\frac{a_s}{a} -2 \frac{b_s}{b}\right) [H_-]_s + H_- \left(-\frac{2 a_s^2}{a^2}-\frac{4 a^2}{b^4}-\frac{4b_s^2}{b^2} - \frac{\partial s}{\partial t} \right).
\end{align}
On the non-principal orbit $S^2_o$, or equivalently when $s=0$, we have
$$H_- = bb_{ss} + k^2 - C \leq bQ_s +k^2 - C \leq k + k^2 -C,$$
where we used that $y = b_s - Q \leq 0$ with equality at $s = 0$. Choosing $C>k^2 + k$ we have $H_- < 0$ on $\{s = 0\} \times [0,T)$. Hence for every $T' \in [0, T)$ there exists a $s_0 > 0$ such that 
$$ H_-(s,t) \leq 0 \: \text{ on } \: [0,s_0] \times [0, T'],$$
as $H_-(s,t)$ is a smooth function on $\R_{\geq 0} \times [0, T']$. Furthermore note that
$$ |H_-|\leq |Rm_{g(t)}|_{g(t)} b^2 + C'^2 + 1 + C,$$
where we used the expression for the curvature component $R_{0202}$ derived in section \ref{con-lap-cur-subsec}. This shows that for each time $t < T'$ the function $H_-(s,t)$ grows subexponentially. Note that by Lemma \ref{b-bound} and Lemma \ref{asa-bound-lem} the coefficient
$$\frac{a_s}{a} -2 \frac{b_s}{b}$$
is bounded on $[s_0, \infty) \times [0, T']$. Similarly, we see from the bound (\ref{dsdt-bound}) on $|\frac{\partial s}{\partial t}|$ presented in the proof of the maximum principle of Theorem \ref{maximum-principle}, Case 1, that the coefficient  
$$-\frac{2 a_s^2}{a^2}-\frac{4 a^2}{b^4}-\frac{4b_s^2}{b^2} - \frac{\partial s}{\partial t}$$
grows at most linearly on every times slice of $[s_0, \infty) \times [0, T']$. Therefore, applying the weak maximum principle to the evolution equation (\ref{H-evol-ineq}) of $H_-$ on the parabolic neighborhood $[s_0, \infty) \times [0,T']$, we deduce that
$$H_- \leq 0 \: \text{ on } \: M_k \times [0, T'].$$
As $T' \in [0, T)$ was arbitrary it follows that $H_-\leq 0$ is preserved by the Ricci flow.

We repeat the same process to prove that $H_+ \geq 0$ is preserved. Applying Young's inequality to bound terms involving $a_{ss}$ and then disregarding non-negative terms not involving $C$, we see that
\begin{align*}
\partial_t  H_+ &\geq [H_+]_{ss} + \left(\frac{a_s}{a} -2 \frac{b_s}{b}\right) [H_+]_s +H_+ \left(-\frac{2 a_s^2}{a^2}-\frac{4 a^2}{b^4}-\frac{4b_s^2}{b^2}\right) \\
 & + C \left(\frac{2 a_s^2}{a^2}+\frac{4 a^2}{b^4}+\frac{4b_s^2}{b^2}\right) \\
 &-\frac{1}{2}\left(\left(\frac{2 b a_s b_s}{a^2}\right)^2+ \left(\frac{8 a_s b_s}{b}\right)^2 +\left(\frac{4 a_s^2}{a}\right)^2+\left(\frac{4a}{b^2}\right)^2\right) \\
 &-\frac{32 a a_sb_s}{b^3}-\frac{16 a^3 a_s b_s}{b^5}- \frac{2 a_s^4}{a^2} -\frac{16 b_s^2}{b^2}.
\end{align*}
Bounding the zeroth order terms via Young's inequality as above, we see that for $C>0$ sufficiently large 
\begin{align*}
\partial_t  H_+ &\geq [H_+]_{ss} + \left(\frac{a_s}{a} -2 \frac{b_s}{b}\right) [H_+]_s +H_+ \left(-\frac{2 a_s^2}{a^2}-\frac{4 a^2}{b^4}-\frac{4b_s^2}{b^2}\right)
\end{align*}
away from the non-principal orbit $S^2_o$. On the non-principal orbit $S^2_o$ we have
$$H_+ = bb_{ss} - k^2 + C \geq  - k^2  + C,$$
where we used that $b_s \geq 0$ with equality at $s =0$ to deduce that $b_{ss} \geq 0$ at $s = 0$. From here the above proof that $H_- \leq 0$ is preserved carries over and we may conclude that $H_+ \geq 0$ is preserved as well. Recalling the bounds on $a_s$ and $b_s$, the desired result now follows.

\end{proof}

Now we can prove that $\mathcal{I}$ (see Definition \ref{def:I}) is preserved by Ricci flow:

\begin{lem}
\label{I-preserved} Let $K_0 > 0$. Then there exists a $K \geq K_0$ such that the following holds: Let $(M_k,g(t))$, $k \geq 1$, $t \in [0,T)$, be a Ricci flow solution starting from an initial metric $g(0) \in \mathcal{I}_{K_0}$. Then $g(t) \in \mathcal{I}_K$ for every $t \in [0,T)$.
\end{lem}

\begin{proof}
By Lemma \ref{Qleq1}, Lemma \ref{ab-mono}, Lemma \ref{yleq0-lem}, Lemma \ref{da-bounded} we see that for $K > 2$ the conditions (\ref{I1}) - (\ref{I4}) are preserved. By Lemma \ref{bddb-bound-lem} we see that there exists a $K \geq K_0$ such that inequality (\ref{I5}) holds for $t \in [0,T)$.

Now we only need to prove that for every time $t \in [0, T)$ the metric $g(t)$ has bounded curvature and positive injectivity radius. As the curvature of $g(0)$ is bounded by the assumption that $g(0) \in \mathcal{I}_K$, it follows by Shi's Theorem \cite{Shi89} that for every time $T' \in [0,T)$ the Ricci flow $g(t)$ has bounded curvature on the time interval $[0,T']$. As $\mathrm{inj}_{g(0)} > 0$ it follows that the metric $g(0)$ is non-collapsed. By standard volume distortion estimates it follows that for each $t \in [0, T/2]$ the metric $g(t)$ is non-collapsed, and hence $\mathrm{inj}_{g(t)} > 0$. By Theorem \ref{thm:no-local-collapsing} there exists a $\kappa >0$ and $\rho>0$ such that for each $t \in [0,T)$ the metric $g(t)$ is $\kappa$-non-collapsed at scale $\rho < \sqrt{T/2}$. This shows that $\mathrm{inj}_{g(t)} > 0$ for all $t \in [T/2, T)$. 
\end{proof}

Before proving Theorem \ref{curv-bound}, we need to prove the following two lemmas in preparation:

\begin{lem}
\label{lem:dtbb-bound}
Let $(M_k,g(t))$, $k \geq 1$, $t\in[0,T)$, be a Ricci flow starting from an initial metric $g(0) \in \mathcal{I}$. Then there exists a constant $C_0 \geq 0$, depending only on the initial metric $g(0)$, such that
\begin{equation}
\label{dtbb-bound}
|\partial_t b^2| \leq C_0
\end{equation}

\end{lem}
\begin{proof}
By Lemma \ref{I-preserved} there exists a $K> 0$ such that $g(t) \in \mathcal{I}_K$ for $t \in [0, T)$. From the evolution equation (\ref{b-evol}) of $b$ and Definition \ref{def:I} of $\mathcal{I}_K$ it follows that 
\begin{align*}
|\partial_t b^2| &= \left|2 b b_{ss} - 8 + 4 Q^2 + 2 \frac{a_sb_s}{Q} + 2 b_s^2\right| \\ \nonumber
				 &\leq 2 K + 8 + 4 + 2 K + 2 \\ \nonumber
				 &= 4K + 14
\end{align*}
This concludes the proof.
\end{proof}

\begin{lem}
\label{bbounded-lem}
Let $(M_k, g(t))$, $k \geq 1$, $t \in [0,T)$, be a Ricci flow starting from an initial metric $g(0) \in \mathcal{I}$. Then 
$$\sup_{p \in M_k} b(p,t) \leq \sup_{p \in M_k} b(p,0)$$
for all $t \in [0,T)$.
\end{lem}

\begin{proof}
From the evolution equation (\ref{b-evol}) of $b$ and expression (\ref{laplacian}) for the Laplacian with respect to the background metric $g(t)$ it follows
\begin{align*}
\partial_t b^2 & = \Delta_{g(t)} b^2 - 8 + 4 Q^2 - 4 b_s^2 \\
				&\leq \Delta_{g(t)} b^2 - 4.
\end{align*}
Applying the maximum principle \cite[Theorem 12.14]{ChII} yields the desired result.
\end{proof}

We now proceed to proving Theorem \ref{curv-bound}.

\begin{proof}[Proof of Theorem \ref{curv-bound}]
We argue by contradiction. Assume there exists a sequence of points $(p_i, t_i)$ in spacetime and constants $D_i \rightarrow \infty$ as $i \rightarrow \infty$ such that
\begin{equation*}
|Rm_{g(t_i)}|_{g(t_i)}(p_i) = D_i b(p_i,t_i)^{-2} := K_i
\end{equation*}
and
\begin{equation*}
|Rm_{g(t)}|_{g(t)} \leq D_i b^{-2} \: \text{ on } \: M_k \times [0, t_i].
\end{equation*}
By the assumption that $g(0) \in \mathcal{I}$ the initial metric $g(0)$ has bounded curvature. Hence by Shi's theorem \cite{Shi89} we have that for every $T' \in [0,T)$ the metric $g(t)$ has bounded curvature on $M_k \times [0,T']$. As by Lemma \ref{bbounded-lem} the warping function $b$ is uniformly bounded on $M_k \times [0,T)$, we thus see that $D_i \rightarrow \infty$ forces $K_i \rightarrow \infty$ and therefore $t_i \rightarrow T$. 

Consider the rescaled Ricci flows
\begin{equation*}
g_i(t) = K_i g \left( t_i + K_i^{-1} t \right), \quad t \in [- K_i \Delta t_i, 0],
\end{equation*}
where $\Delta t_i > 0$ is to be determined below. As $K_i \rightarrow \infty$ we see that $g_i(t)$ are blow-ups rather than blow-downs, which is important for the following reason: By Theorem \ref{thm:no-local-collapsing} there exists a $\kappa>0$ such that $g(t)$ is $\kappa$-non-collapsed at every scale $p \leq \sqrt{T/2}$ at every spacetime point $(p,t) \in M_k \times [T/2, T)$. As $K_i \rightarrow \infty$ we see that $g_i(t)$ are $\kappa$-non-collapsed at scales tending to infinity as $i \rightarrow \infty$. 

By Lemma \ref{I-preserved} there exists a $K>0$ such that $g(t) \in \mathcal{I}_K$ for all $t \in [0,T)$. Furthermore, by Lemma \ref{lem:dtbb-bound} there exists a $C_0$ such that $|\partial_t b^2| \leq C_0$ on $M_k \times [0, T)$. Recall the Definition \ref{def:C} of $C_g(p,r)$. Set
$$\Delta t_i = \min\left( \frac{t_i}{2}, \frac{b^2(p_i, t_i)}{8 C_0} \right)$$
and consider the parabolic neighborhoods
$$\Omega_i = C_{g(t_i)}\left(p_i, \frac{b(p_i,t_i)}{2}\right) \times [t_i - \Delta t_i, t_i].$$
As $g(t) \in \mathcal{I}_K$ for $t \in [0,T)$ we have that $y = b_s - Q \leq 0$, $Q \leq 1$ and $b_s \geq 0$ everywhere on $M_k \times [0,T)$. Therefore 
$$b(p, t_i) \geq \frac{b(p_i, t_i)}{2} \text{ on } \Omega_i\cap \left\{t = t_i\right\}$$
By Lemma \ref{lem:dtbb-bound}
\begin{equation*}
b^2(p,t_i) - b^2(p,t) \leq C_0 (t_i - t) 
\end{equation*}
for all $(p,t) \in \Omega_i$ from which it follows that
$$ \frac{1}{4} b(p_i, t_i)^2 - b(p,t)^2 \leq b^2(p,t_i) - b^2(p,t) \leq C_0 (t_i - t) \leq C_0 \Delta t \leq \frac{1}{8} b(p_i, t_i)^2.$$
Thus we deduce that 
\begin{equation}
\label{eqn:b-lower}
b^2(p,t) \geq \frac{1}{8} b^2(p_i,t_i) \text{ on } \Omega_i.
\end{equation}
It follows that for $(p,t) \in \Omega_i$
\begin{align*}
|Rm_{g(t)}|_{g(t)}(p) &\leq D_i b(p,t)^{-2} \\ \nonumber
			 &\leq 8 D_i b(p_i,t_i)^{-2} \\ \nonumber
			 &= 8 K_i
\end{align*}
and hence the curvatures of the rescaled metrics $g_i(t)$ satisfy
$$
|Rm_{g_i(t)}|_{g_i(t)} \leq 8
$$
on the parabolic neighborhoods $\Omega'_i$
$$\Omega'_i \coloneqq C_{g_i(0)}\left(p_i, \sqrt{K_i} \frac{b(p_i,t_i)}{2} \right) \times [-K_i \Delta t_i, 0].$$
Note that
$$ K_i \Delta t_i \rightarrow \infty \:\text{ as }\: i \rightarrow \infty$$
and 
$$ \sqrt{K_i} \frac{b(p_i,t_i)}{2}  \geq \frac{\sqrt{D_i}}{2}  \rightarrow \infty \: \text{ as } \: i \rightarrow \infty.$$
Hence $\left(C_{g_i(t)}(p_i, \sqrt{D_i}/2), g_i(t), p_i\right)$, $t \in [-K_i \Delta t_i, 0]$, subsequentially converges, in the Cheeger-Gromov sense, to an ancient pointed Ricci flow $(M_{\infty}, g_{\infty}(t), p_\infty)$, $t \in (-\infty, 0]$.

\begin{claim}
The Ricci flow $(M_\infty, g_\infty(t), p_\infty)$, $t \in (-\infty, 0]$, splits as $(\R^2 \times N, g_{eucl} + g_N(t))$, $t \in (-\infty, 0]$, where $g_{eucl}$ is the flat euclidean metric, and $(N, g_N(t))$ is a non-compact ancient Ricci flow.
\end{claim}

\begin{claimproof}
Denote by $a_i$ and $b_i$ the warping functions of the rescaled metrics $g_i(t)$. Then by (\ref{eqn:b-lower}) we see that
\begin{equation}
\label{ineq:bi}
b_i(p, t) \geq \sqrt{\frac{D_i}{8}} \: \text{ on } \: \Omega'_i
\end{equation}
As $D_i \rightarrow \infty$, the warping functions $b_i$ tend to infinity uniformly. As $b_i$ describes the size of the base $S^2$ in the Hopf fibration,intuitively one can see that this claim is true. Nevertheless, we provide a formal proof below:

As $g(t)\in \mathcal{I}_K$ for $t\in [0,T)$ we have 
$$ \left|-\frac{b_{ss}}{b}\right| \leq \frac{K}{b^2} \: \text{ on } \: M_k \times [0, T).$$
Inspecting the curvature components listed in section \ref{con-lap-cur-subsec}, we see that all the curvature components of $g_i(t)$, apart from $R_{0101}$, tend to zero on $\Omega'_i$. Hence the curvature operator of $g_\infty(t)$ is of rank 1. Furthermore, as $g(0)$ has bounded curvature by the assumption that $g(0) \in \mathcal{I}$ we see that the scalar curvature $R_{g(t)}$ is pointwise bounded below by $\inf_{p \in M_k} R_{g(0)}(p) > - \infty$. Hence the blow-up limit $g_\infty(t)$ has non-negative scalar curvature, which in turn implies that the curvature operator is non-negative. By \cite[8.3. Theorem \& p. 178]{Ham86} we conclude that $(M_\infty, g_\infty(t))$ splits as a product $(\R^2 \times N, g_{eucl} + g_N(t))$. Note also that $N$ is diffeomorphic to the leafs of the distribution spanned by $e_0$ and $e_1$, as these are the only planes with non-flat sectional curvature. Recalling that $e_0=\frac{\partial}{\partial s}$ we see that the integral curves of $e_0$ are non-compact and therefore $N$ is non-compact as well.
\end{claimproof}

As $(M_\infty, g_{\infty}(t))$ is $\kappa$-non-collapsed at all scales, the above claim implies that $(N, g_N(t))$ is a 2d $\kappa$-solution. However, by Hamilton's work a two dimensional $\kappa$-solution is either the shrinking round sphere $S^2$ or its $\Z_2$ quotient \cite[\S 1 of Chapter 9]{CLN06}. Since $N$ is non-compact we have arrived at a contradiction. Therefore the desired result follows. 
\end{proof}

\begin{proof}[Proof of Corollary \ref{cor:curv-bound-ancient}]
The proof is the same as for Theorem \ref{curv-bound}. Since the Ricci flow is assumed to be $\kappa$-non-collapsed at all scales, we may also take blow-down limits and do not need to assume that $b$ is uniformly bounded. Furthermore, since ancient Ricci flows have non-negative scalar curvature, Claim 1 of the proof of Theorem \ref{curv-bound} also carries over. 
\end{proof}

\section{Compactness properties}
In this section we prove some compactness properties of $U(2)$-invariant cohomogeneity one Ricci flows. For general Ricci flows the compactness properties are well-known \cite[Chapter 3]{ChI}. Therefore the main technical difficulty is to show that the $U(2)$-symmetry passes to the limit. 

The main theorem of this section is Theorem \ref{thm:local-compactness} which roughly states the following compactness property: Let $(U_i, g_i(t), p_i)$, $[-\Delta t, 0]$, be a sequence of $U(2)$-invariant cohomogeneity one manifolds in the class $\mathcal{I}$ of metrics. Here the $U_i$ are open manifolds and assumed to compactly contain the sets $C_{g_i(0)}(p_i, r)$ (see Definition \ref{def:C}) for some fixed $r> 0$. This condition can be understood as requiring $U_i$ to have `radial diameter' of at least $r$. Furthermore the metrics $g_i(t)$ are normalized such that $b = 1$ at the points $(p_i, 0)$ in spacetime. We show that if the flows $g_i(t)$ are $\kappa$-non-collapsed and of uniformly bounded curvature, then $(U_i, g_i(t), p_i)$, $[-\Delta t, 0]$, subsequentially converges to a limiting $U(2)$-invariant Ricci flow $(\mathcal{C}_\infty, g_\infty(t), p_\infty)$. Moreover, if we correctly pick/normalize the coordinate $\xi$, the warping functions $u_i(\xi, t)$, $a_i(\xi,t)$ and $b_i(\xi,t)$ of the metrics $g_i(t)$ converge on compact parabolic sets in $C^\infty$ to the warping functions $u_\infty(\xi,t)$, $u_\infty(\xi, t)$ and $b_\infty(\xi,t)$ of $g_\infty(t)$. This in essence shows that when taking limits of $U(2)$-invariant Ricci flows, we may work with the warping functions only, without having to concern ourselves with the underlying manifold.

Theorem \ref{thm:local-compactness} has two important applications: Firstly, it implies the corresponding compactness result for complete Ricci flows. In particular, a sequence of uniformly bounded and non-collapsed $U(2)$-invariant cohomogeneity one Ricci flows $(M_k, g_i(t), p_i)$, $t \in [-t_i, 0]$, $t_i\rightarrow \infty$, normalized such that $b(p_i, 0) = 1$, subsequentially converges, in the Cheeger-Gromov sense, to a limiting Ricci flow $(M_\infty, g_\infty(t), p_\infty)$, $t \in [-\infty, 0]$, that is also $U(2)$-invariant and cohomogeneity one. Secondly, we prove a variant of Theorem \ref{thm:local-compactness} in Proposition \ref{blow-up-prop}, where we specialize to the case in which the `radial diameter' of the $U_i$ is equal to $\frac{1}{2}$. This will allow us to alter one assumption of Theorem \ref{thm:local-compactness} and yield a very useful tool for proving certain scale-invariant inequalities via a contradiction/compactness argument, as introduced in the outline of section 9 in section \ref{subsec:outline-paper} of this paper. 

Below we state the main results of this section. For this recall Definition \ref{def:C} of $C_g(p,r)$, $C_g^+(p,r)$ and $\Sigma_p$.

\begin{thm} [Local compactness]
\label{thm:local-compactness}
Let $k \in \N$ and $\kappa, \rho, K, r, \Delta t > 0$. Assume that
$$(U_i, g_i(t), p_i), \quad t \in \left[-\Delta t, 0\right],$$
is a sequence of pointed cohomogeneity one $U(2)$-invariant Ricci flows satisfying the following properties:
\begin{enumerate}
\item $U_i$ is an open $U(2)$-invariant manifold with principal orbit $S^3/\Z_k$.
\item For $t \in [-\Delta t, 0]$ we have $g_i(t) \in \mathcal{I}$ (see Definition \ref{def:I}). Denote by $u_i$, $a_i$ and $b_i$ the warping functions of $g_i(t)$.
\item The closed sets $\overline{C_{g_i(0)}\left(p_i, r\right)} \subset U_i$ are compact.
\item $b_i(p_i,0) = 1$.
\item The Ricci flow $(U_i, g_i(t))$ is $\kappa$-non-collapsed at $(p_i, 0)$ at scale $\min(\rho, r, \sqrt{\Delta t})$.
\item $|Rm_{g_i(t)}|_{g_i(t)} \leq K$ in $U_i \times \left[- \Delta t, 0\right]$.
\end{enumerate}
Then $(C_{g_i(0)}\left(p_i, r \right), g_i(t), p_i)$, $t \in [-\Delta t, 0]$, subsequentially converges, in the Cheeger-Gromov sense, to a pointed Ricci flow
$$(\mathcal{C}_\infty, g_\infty(t), p_\infty), \quad t \in \left[- \Delta t, 0\right],$$ 
satisfying the following properties:
\begin{enumerate}[label=(\alph*)]
\item $\mathcal{C}_\infty$ is a cohomogeneity one $U(2)$-invariant manifold such that either
\begin{enumerate}[label=(\roman*)]
\item All orbits are principal: In this case $\mathcal{C}_\infty$ is diffeomorphic to the cylinder $\R \times S^3/\Z_k$ and we equip $\mathcal{C}_\infty$ with a radial coordinate $\xi: \mathcal{C}_\infty \rightarrow \R$ defined by $\xi(p) = d_{g_\infty(0)}\left(p, \Sigma_{p_\infty} \right)$. 
\item There is exactly one non-principal orbit:  In this case $\mathcal{C}_\infty$ is diffeomorphic to $M_k$ and we equip $\mathcal{C}_\infty$ with the radial coordinate $\xi: \mathcal{C}_\infty \rightarrow \R$ defined by $\xi(p) = d_{g_\infty(0)}\left(p, S^2_o\right)$. 
\end{enumerate}
\item There exist warping functions 
$$u_\infty, a_\infty, b_\infty: C_{g_\infty(0)}\left(p_\infty, r \right) \times [-\Delta t, 0] \rightarrow \R_{\geq 0}$$
 such that the metric $g_\infty(t)$, $t\in \left[-\Delta t,0 \right]$, is of the form (\ref{metric1}) and in the class $\mathcal{I}$
\item Choosing the coordinate $\xi$ on $(U_i, g_i(t), p_i)$ corresponding to whether we are in case (i) or (ii) above, the warping functions $u_i(\xi,t)$, $a_i(\xi, t)$ and $b_i(\xi,t)$ converge on compact sets to $u_\infty(\xi,t)$, $a_\infty(\xi,t)$ and $b_\infty(\xi, t)$.
\item For every $r' < r$ the closed set $\overline{C_{g_\infty(0)}(p_\infty, r')} \subset \mathcal{C}_\infty$ is compact.
\end{enumerate}
\end{thm}

From Theorem \ref{thm:local-compactness} the following corollary follows immediately:

\begin{cor}[Compactness of complete Ricci flows] 
\label{cor:compactness-complete-flows}
Let $k \in \N,$ $\kappa, K >0$ and $r_i, t_i, \rho_i \rightarrow \infty$ as $i \rightarrow \infty$. Assume that $(\mathcal{M}_i, g_i(t), p_i)$, $t \in [-t_i, 0]$, is a sequence of pointed $U(2)$-invariant cohomogeneity one Ricci flows satisfying:
\begin{enumerate}
\item For $t \in [-t_i, 0]$ we have $g_i(t) \in \mathcal{I}$. Denote by $u_i$, $a_i$ and $b_i$ the warping functions of $g_i(t)$.
\item $\overline{C_{g_i(0)}(p_i, r_i) } \subset \mathcal{M}_i$ is compact.
\item $b(p_i,0) = 1$
\item $g_i(t)$ is $\kappa$-non-collapsed at scale $\rho_i$
\item $|Rm_{g_i(t)}|_{g_i(t)} \leq K$ on $\mathcal{M}_i \times [-t_i, 0]$
\end{enumerate}
Then $(\mathcal{M}_i, g_i(t), p_i)$, $t\in [-t_i, 0]$, subsequentially converges, in the Cheeger-Gromov sense, to a pointed complete ancient Ricci flow $(\mathcal{M}_\infty, g_\infty(t), p_\infty)$, $ t \in (-\infty, 0]$, with bounded curvature satisfying properties (a) - (d) of Theorem \ref{thm:local-compactness}, when taking $\mathcal{C}_\infty = \mathcal{M}_\infty$ and $r = \infty$.
\end{cor}

\begin{proof}
This follows from Theorem \ref{thm:local-compactness} by a diagonal argument. 
\end{proof}

The following proposition is a variant of Theorem \ref{thm:local-compactness} in the case we take $r=\frac{1}{2}$.

\begin{prop} 
\label{blow-up-prop}
Let $k \in \N$, $\kappa, \rho, C_1 > 0$, $r = \frac{1}{2}$ and $\Delta t \in (0, \frac{1}{48C_1}]$. Assume 
$$(U_i, g_i(t), p_i), \quad t \in \left[-\Delta t, 0\right],$$
is a sequence of pointed $U(2)$-invariant cohomogeneity one Ricci flows satisfying conditions (1)-(5) of Theorem \ref{blow-up-prop}. If, instead of condition (6) of Theorem \ref{thm:local-compactness}, we require 
\begin{enumerate}[label=(6')]
\item $|Rm_{g_i(t)}|_{g_i(t)} \leq \frac{C_1}{b^2}$ on $U_i \times [-\Delta t, 0]$ 
\end{enumerate}
then 
$$\left(C_{g_i(0)}\left(p_i, \frac{1}{2}\right), g_i(t), p_i\right), \quad t \in [-\Delta t, 0],$$ 
subsequentially converges, in the Cheeger-Gromov sense, to a pointed Ricci flow  
$$(\mathcal{C}_\infty, g_\infty(t), p_\infty), \quad t \in \left[- \Delta t, 0\right],$$
satisfying the same properties (a)-(d) listed in Theorem \ref{blow-up-prop}.
\end{prop}

\begin{proof}[Proof of Proposition \ref{blow-up-prop}]
For brevity we write $\Omega_ i = C_{g_i(0)}\left( p_i, \frac{1}{2}\right) \times [-\Delta t,0]$. Let $a_i$ and $b_i$ denote the warping functions of $g_i(t)$.
As $g_i(t) \in \mathcal{I}$ we have
$$0 \leq (b_i)_s \leq Q_i \leq 1 \: \text{ in } \Omega_i,$$
where $Q_i = \frac{a_i}{b_i}$ and thus
$$b_i(p, 0) \geq \frac{1}{2} \: \text{ for } p \in C_{g_i(0)}\left( p_i, \frac{1}{2}\right)$$
as $b_i(p_i, 0) = 1$ by assumption. By the Ricci flow equation we have
\begin{align*}
\partial_t b_i^2 &= - 2b^2 \left(R_{0202}+ R_{1212} + R_{2323}\right) \\
				 & \leq 6 C_1 \: \text{ on } \Omega_i.
\end{align*}
This implies
\begin{equation*}
b_i(p,t) \geq \frac{1}{\sqrt{8}} \: \text{ for } (p,t)\in \Omega_i,
\end{equation*}
as $\Delta t \leq \frac{1}{48C_1}$ by assumption. This yields the uniform curvature bound
\begin{equation*}
|Rm(g_i)|_{g_i} \leq 8C_1 \: \text{ on } \Omega_i.
\end{equation*}
The result now follows from Theorem \ref{thm:local-compactness}.
\end{proof}

The main proof idea of Theorem \ref{thm:local-compactness} is to construct a set of four Killing vector fields $\overline{X}_j$, $j = 1, 2, 3, 4$, generated by the $U(2)$-action on each $(U_i, g_i(t))$, and show that these Killing vector fields pass to the limit $(\mathcal{C}_\infty, g_\infty(t))$. This allows us to reconstruct the $U(2)$-action on $\mathcal{C}_\infty$, proving the desired result. The main difficulty, however, is to show that the orbits corresponding to the flows of the Killing vector fields do not degenerate in the limit and thereby ensure that the full $U(2)$ symmetry group is preserved. For this we will rely on Lemma \ref{lem:non-collapsed-Q} below, where we prove that $\kappa$-non-collapsedness implies a lower positive bound on $Q$ away from a non-principal orbit.

\begin{lem}
\label{lem:non-collapsed-Q}
Let $k \in \N$, $r_0 \in (0, 1]$ and $\kappa, C_1, c > 0$. Assume that $(M,g)$ is a $U(2)$-invariant cohomogeneity one manifold with principal orbit $S^3 /\Z_k$ equipped with a metric $g \in \mathcal{I}$. Take $p \in M$. If
\begin{enumerate}
\item The set $\overline{C^+_g(p,b(p)r_0)}$ (See Definition \ref{def:I}) is compactly contained in $M$
\item $|Rm_{g}|_{g} \leq \frac{C_1}{b^2}$ on $\overline{C^+_g(p,b(p)r_0)}$
\item $g$ is $\kappa$-non-collapsed at scale $cb(p)$: If for $r \leq c b(p)$ the ball $B_g(p,r)$ is compactly contained in $M$ and $|Rm_g|_g < r^{-2}$ on $B_g(p,r)$ then $vol(B_g(p,r)) \geq \kappa r^4$
\end{enumerate}
then there exists an $\epsilon > 0$ depending on $k, \kappa, C_1 , c$ and $r_0$ for which the following holds: If for $q \in C^+_g(p,b(p)r_0)$ the set $C_g(q, b(p) \frac{r_0}{4})$ is compactly contained in $C^+_g(p,b(p)r_0)$ then $Q(q) \geq \epsilon$.
\end{lem}
\begin{proof}
By rescaling we may assume without loss of generality $b(p) =1$ and that the metric $g$ is $\kappa$-non-collapsed at scale $c>0$. The latter follows from the fact that if $g$ is $\kappa$-non-collapsed at scale $\rho$ then $\alpha^2 g$ is $\kappa$-non-collapsed at scale $\alpha \rho$. Fix a $q \in C^+_g(p,b(p)r_0)$ such that the assumptions of the lemma hold. Take $U \coloneqq C_g(q, b(p) \frac{r_0}{4})$. Note that $U$ is a union of orbits of the $U(2)$-action. Recall that non-principal orbits are non-generic and characterized by $a = 0$. As $a_s \geq 0$ we see that all the orbits of $U$ are principal and therefore diffeomorphic to $S^3/\Z_k$. Because $0 \leq b_s \leq Q \leq 1 $ for metrics in $\mathcal{I}$ we see that
$$1 \leq b \leq 2 \text{ in } C^+_g(p,r_0)$$
and hence
$$|Rm_g|_g \leq C_1 \text{ in } C^+_g(p,r_0)$$
by assumption (2). From the expression
$$M_2 = \frac{1}{b^2}\left( a_s - Q b_s \right)$$
for the curvature component $R_{0231}$ derived in section \ref{con-lap-cur-subsec} and the fact that
$$Q_s = \frac{1}{b} \left( a_s - Q b_s\right)$$
we deduce that
$$ |Q_s| \leq C_1 \text{ in } C^+_g(p,r_0).$$
Thus for $r \leq r_1 \coloneqq \min\left(\frac{r_0}{4}, \frac{Q(q)}{C_1}\right)$ we have
$$ Q \leq 2Q(q) \: \text{ on } C_g(q, r_1).$$

\begin{claim}
For $r \leq r_2 \coloneqq \min\left(\frac{1}{100}, r_1\right)$ we have $$ vol(B_g(q,r)) \leq C r^3 Q(q) $$
for some constant $C> 0$ depending on $k$ only.
\end{claim}
\begin{claimproof} 
 Let $q'  \in C_g(q, r)$, $r < r_2$. Then $\Sigma_{q'}$ is isometric to $S^3/\Z_k$ equipped with a squashed Berger metric. In particular, if we denote by $\iota: \Sigma_{q'} \rightarrow M$ the inclusion, then
 $$\iota^\ast g = a(q')^2 \omega \otimes \omega + b(q')^2\pi^\ast(g_{FS}),$$
 where $g_{FS}$ is the Fubini-Study metric on $S^2$ normalized to have curvature equal to $\frac{1}{4}$ and $\pi: S^3 / \Z_k \rightarrow S^2$ is the Hopf fibration. Note that
$$\Sigma_{q'} \cap B_g(q,r) \subseteq \Sigma_{q'} \cap B_g(q',r) \subseteq \pi^{-1}(B_{g_{FS}}(\pi(q'), r)) \subseteq \Sigma_{q'}.$$ 
Furthermore, as the Hopf fibers of $\Sigma_{q'} \cong S^3/\Z_k$ have length $\frac{2\pi}{k} a(q')$, we see that
$$vol(\pi^{-1}(B_{g_{FS}}(\pi(q'), r))) = \frac{2\pi}{k} a(q') vol(B_{g_{FS}}(\pi(q'), r)) \leq C a(q') r^2,$$
for some constant $C>0$ depending on $k$ only. Since $Q = \frac{a}{b}$, $Q \leq 2 Q(q)$ and $b \in [1, 2]$ in $C_g(q, r_1)$ it follows that
$$ vol(\Sigma_{q'} \cap B_g(q,r)) \leq 4 C Q(q) r^2.$$
Integrating this inequality proves the claim.

\end{claimproof}
As $|Rm_g|_g \leq C_1$ on $C^+_g(p,r_0)$, the ball $B_g(q,\frac{r_0}{4})$ is compactly contained in $M$, and $g$ is $\kappa$-non-collapsed at scale $c$, we see that for $r \leq r_3 \coloneqq \min\left(\frac{1}{\sqrt{C_2}}, c, \frac{r_0}{4}\right)$ 
$$vol(B_g(q,r)) \geq \kappa r^4.$$
Setting $r_4 \coloneqq \min\left(r_2, r_3\right)$ we therefore obtain
$$C r_4^3 Q(q) \geq  vol(B_g(q,r_4)) \geq \kappa r_4^4.$$
Rearranging this inequality proves the lemma.
\end{proof}

Before proving the compactness theorems listed above, we construct a set of four Killing vector fields on a general $U(2)$-invariant cohomogeneity one manifold $M$ with principal orbit $S^3/\Z_k$. By passing to the universal cover we may assume without loss of generality that $k = 1$. Pick the basis
$$ X_0 = \begin{pmatrix} i & 0 \\ 0 & i \end{pmatrix} \qquad X_1 = \begin{pmatrix} i & 0 \\ 0 & -i \end{pmatrix} \qquad X_2 = \begin{pmatrix} 0 & i \\ i & 0 \end{pmatrix} \qquad X_3 = \begin{pmatrix} 0 & -1 \\ 1 & 0 \end{pmatrix} $$
for the Lie algebra of $U(2)$. Then $X_i$, $i = 1, 2, 3, 4$, satisfy the commutation relations
$$[X_1,X_2] = 2 X_3 \qquad [X_2,X_3] = 2 X_1 \qquad [X_3,X_1] = 2 X_2.$$
and
$$[X_0, X_i] = 0 \: \text{ for } \: i = 1, 2, 3.$$
Extend $X_i$, $i = 1, 2, 3, 4$, to left-invariant vector fields on $U(2)$. Note that the integral curves generated by $X_i$, $i = 1, 2, 3, 4$, have period $2\pi$. The $U(2)$-action generates four corresponding Killing vector fields $\overline{X}_i$, $i = 1, 2, 3, 4$, on $M_k$ by taking
$$\overline{X}_i(p) = \frac{d}{dt} \Big |_{t=0} \exp(t X_i) \cdot p, \quad p \in M_k,\:i = 1,2,3,4.$$
We now prove the following:
\begin{lem}
\label{lem:killing}
For $i = 1, 2, 3, 4$ we have
$$|\overline{X}_i|_{g} \leq \max(a,b).$$
\end{lem}
\begin{proof}
By the form (\ref{metric2}) of the metric we see that $|\overline{X}_0|_g = a$. Hence we only need to prove the result for $i = 1, 2, 3$. First note that the vector fields $\overline{X}_i$, $i = 1, 2, 3$, are orthogonal to $\frac{\partial}{\partial s}$ and therefore parallel to the orbits of the $U(2)$ action on $M_k$. Hence it suffices to study the metric $g$ restricted to these directions. Here we see that
$$a^2 \omega \otimes \omega + b^2 \pi^\ast(g_{FS}) \leq \max(a,b)^2 g_{S^3},$$
where $g_{S^3}$ is the round metric on $S^3$ with sectional curvatures equal to 1. Thus it suffices to show that
$$|\overline{X}_i|_{g_{S^3}} \leq 1.$$
If we identify $S^3$ with $SU(2)$, the vectors $\overline{X}_i$ correspond to right-invariant vector fields on $SU(2)$. Moreover, one can check that these vector fields are orthonormal with respect to the metric $g_{S^3}$. Hence the desired result follows. 
\end{proof}

\begin{remark}
In fact one can show that $\min(a,b) \leq |\overline{X}_i|_{g} \leq \max(a,b)$. Recalling that the isometry generated by the Killing vector field $\overline{X}_i$ descends to a rotation of the base $S^2$ in the Hopf fibration $\pi: S^3 \rightarrow S^2$, one can see that the upper bound is attained on $\pi^{-1}(\{\text{Equator of }S^2\})$ and the lower bound is attained on $\pi^{-1}(\{N,S\})$, where $N$, $S$ denote the north and south pole with respect to the rotation induced by $\overline{X}_i$. 
\end{remark}

Now we proceed to proving the main theorem of this section:

\begin{proof}[Proof of Theorem \ref{thm:local-compactness}]
As $g_i(t)$ is $\kappa$-non-collapsed at $(p_i,0)$ at scale $\min(\rho, r, \sqrt{\Delta t})$ it follows from \cite[Lemma 6.54]{ChI} that there exists a uniform $\delta >0$ such that
$$\mathrm{inj}_{g_i(0)}(p_i) > \delta.$$
By assumption $g_i(t)$ has bounded curvature on the parabolic neighborhood 
$$\Omega_i \coloneqq C_{g_i(0)}\left(p_i, r\right) \times [-\Delta t, 0].$$ 
By an adaptation of \cite[Theorem 3.16]{ChI} we therefore deduce that after passing to a subsequence 
$$\left(C_{g_i(0)}\left( p_i, r\right), g_i(t), p_i\right), \quad t \in [-\Delta t, 0],$$
converges, in the Cheeger-Gromov sense, to a pointed Ricci flow 
$$\left(\mathcal{C}_{\infty}, g_{\infty}(t), p_{\infty}\right), \quad t \in [-\Delta t, 0],$$
where $\mathcal{C}_\infty$ is an open manifold. 

\begin{claim}
$(\mathcal{C}_{\infty}, g_{\infty}(t))$, $ t \in [-\Delta t, 0]$, is $U(2)$-invariant.
\end{claim}

\begin{claimproof}
Recall the construction of the Killing vector fields $\overline{X}_j$, $j = 1, 2, 3, 4$ for a general $U(2)$-invariant manifold $M$ explained above. Let $\overline{X}_{ij}$, $i \in \N$, $j = 1, 2, 3, 4$, denote the corresponding Killing vector fields on the manifolds $C_{g_i(0)}\left(p_i, r\right)$. Recall that $g_i(t) \in \mathcal{I}$ implies that $0 \leq b_s \leq Q \leq 1$. Therefore $b \leq r+1$ on $C_{g_i(0)}\left(p_i, r\right)$. Note that from the evolution equation (\ref{b-evol}) of $b$ it follows that
$$\left|\frac{\partial_t b}{b}\right| \leq c |Rm_{g_i(t)}|_{g_i(t)},$$
where $c > 0$ is some universal constant. As $|Rm_{g_i(t)}|_{g_i(t)} \leq C_1$ on $\Omega_i$ by assumption, we see that there exists a $C>0$, depending on $r$ and $C_1$ only, such that $b \leq C$ on $\Omega_i$. From Lemma \ref{lem:killing} it hence follows that for $i \in \N$, $j = 1, 2, 3, 4$,
$$ |\overline{X}_{ij}|_{g_i(t)} \leq C \: \text{ on } \Omega_i.$$
Recall that in general a Killing vector field $X^a$ on a manifold satisfies the relation
$$ \nabla_a\nabla_b X^c = -R^c_{\:abd} X^d.$$
Therefore we see that the Killing vector fields $\overline{X}_{ij}$ are uniformly bounded in $C^2(\Omega_i)$, and converge to $C^1$ Killing vector fields $\overline{X}_{\infty, j}$, $j = 1, 2, 3,4$, on $\left(\mathcal{C}_{\infty}, g_{\infty}(t)\right)$ after passing to a subsequence. However, since the group of isometries of a smooth manifold is a smooth Lie group, the vector fields $\overline{X}_{\infty, j}$, $j = 1, 2, 3,4$ are in fact smooth. As the Killing vector fields $\overline{X}_{ij}$, $i \in \N$, $j = 1,2, 3,4$, are complete, so are $\overline{X}_{\infty, j}$, $j = 1, 2, 3,4$. Integrating the Killing vector fields $\overline{X}_{\infty, j}$, $j = 1,2,3,4$, then yields the desired $U(2)$-action on $(\mathcal{C}_\infty, g_\infty(t))$.  
\end{claimproof}

It remains to be shown that this action is faithful by proving that the Killing vector fields are non-zero at times $t \in [-\Delta t, 0]$. 

\begin{claim}
The $U(2)$-action on $(\mathcal{C}_{\infty}, g_{\infty}(t))$, $ t \in [-\Delta t, 0]$ is faithful.
\end{claim}

\begin{claimproof}
Take $r_1 > 0$ such that $C_{g_\infty(t)}(p_\infty, r_1)$ is compactly contained in $\mathcal{C}_\infty$ for all $t \in [-\Delta t, 0]$. This is possible by standard distance distortion estimates and the fact that $\mathcal{C}_{\infty} \times [-\Delta t, 0]$ has bounded curvature. Furthermore, since $C_{g_\infty(t)}(p_\infty, r_1) \times [-\Delta t, 0]$ is compactly contained in $\mathcal{C}_\infty \times [-\Delta t, 0]$, there exist constants $\rho', \kappa' > 0$ such that for each $t \in [-\Delta t, 0]$ the manifold $(C_{g_\infty(t)}(p_\infty, r_1), g_\infty(t))$ is $\kappa'$-non-collapsed at scale less or equal to $\rho'$. Since $(C_{g_i(t)}(p_i, r), g_i(t), p_i)$ converges in the Cheeger-Gromov sense to $(\mathcal{C}_\infty, g_\infty(t), p_\infty)$ we see that eventually $(C_{g_i(t)}(p_i, r_1), g_i(t))$ is $\kappa'/2$-non-collapsed at scales less or equal to $\rho'/2$. 

Fix $t' \in [-\Delta t, 0]$ and choose points $q_i \in \Sigma^+_{p_i}$ (see Definition \ref{def:C}) and $q_\infty \in \mathcal{C}_\infty$ with $d_{g_i(t')}(q_i, \Sigma^+_{p_i}) = \frac{1}{2}r_1$ and $q_i \rightarrow q_\infty$. Checking the conditions of Lemma \ref{lem:non-collapsed-Q}, we see that there exists an $\epsilon > 0$, independent of $i$, such that 
$$Q(q_i,t') \geq \epsilon.$$ 
As $g \in \mathcal{I}$ and therefore $0 \leq b_s \leq Q \leq 1$, we see that 
$$1 \leq b(q_i, t') \leq \frac{3}{2}$$ 
Therefore the geometry of the orbit $\Sigma_{q_i} \cong S^3 /\Z_k$ is controlled --- the curvature and diameter are uniformly bounded from above, and its volume and Hopf fiber lengths are uniformly bounded away from zero. Hence the norms of the Killing vector fields $\overline{X}_{ij}$, $j = 1, 2, 3, 4$, at the points $(q_i, t')$ in spacetime are uniformly bounded away from zero, proving that on the limiting space $(\mathcal{C}_\infty, g_\infty(t'))$ the Killing vector fields $\overline{X}_{\infty, j}$, $j = 1, 2,3,4$, are non-zero. As $t' \in [\Delta t, 0]$ was arbitrary, the desired result follows.
\end{claimproof}

By the slice theorem we we see that either (i) all orbits of $\mathcal{C}_\infty$ are principal and diffeomorphic to $S^3/\Z_k$ or (ii) there exists exactly one non-principal orbit, which is diffeomorphic to $S^2$ and as usual we denote by $S^2_o$. Below it will become clear why $\mathcal{C}_\infty$ cannot possess two non-principal orbits. In case (i) $\mathcal{C}_\infty$ is diffeomorphic to the manifold $\R \times S^3/\Z_k$ and in case (ii) it is diffeomorphic to $M_k$. In both cases there is a dense open set of the form $\R \times S^3/\Z_k \subset \mathcal{C}_\infty$. 

We now show that the metrics $g_\infty(t)$ can be expressed in the form (\ref{metric2}). Denote the warping functions of the metrics $g_i(t)$ by $a_i$ and $b_i$. In case (i) we define the radial coordinates  
$$\xi_i(p) = \pm d_{g_i(0)}(p, \Sigma_{p_i}),$$
and
$$\xi_\infty(p) = \pm d_{g_\infty(0)}(p, \Sigma_{p_\infty})$$
on $C_{g_i(0)}\left(p_i, r\right)$ and $\mathcal{C}_\infty$, respectively. We choose the sign of $\xi_i(p)$ depending on which side of the hypersurface $\Sigma_{p_i}$ the point $p$ lies, and in such a way that $\partial_{\xi_i} a_i , \partial_{\xi_i} b_i \geq 0$. The sign of $\xi_\infty(p)$ is chosen such that $\xi_i \rightarrow \xi_\infty$ as $i \rightarrow \infty$. In case (ii) we may assume without loss of generality that for all $i \in \N$ the open manifolds $C_{g_i(0)}\left(p_i, r\right)$ contain a point $o_i$ such that the orbit $\Sigma_{o_i}$ is non-principal and $o_i \rightarrow o_\infty \in \mathcal{C}_\infty$ as $i \rightarrow \infty$. Then define radial coordinates
$$\xi_i(p) = d_{g_i(0)}(p, \Sigma_{o_i})$$
and
$$\xi_\infty(p) = d_{g_\infty(0)}(p, \Sigma_{o_\infty})$$
on $C_{g_i(0)}\left(p_i, r\right)$ and $\mathcal{C}_\infty$. Note that the coordinates $\xi_i$ and $\xi_\infty$ are smooth away from a non-principal orbit and furthermore that $\xi_i \rightarrow \xi_\infty$ in $C^\infty$ away from a non-principal orbit. Hence we obtain the one-forms $d\xi_i$ and $d\xi_\infty$ away from a non-principal orbit, which are orthogonal to all orbits of $C_{g_i(0)}\left(p_i, r\right)$ and $\mathcal{C}_\infty$, respectively. For brevity we drop the subscript and write $\xi$ for the coordinates $\xi_i$ or $\xi_\infty$. 

Since the metric $g_\infty(t)$ is $U(2)$-invariant, as shown above, there exists warping functions $u_\infty, a_\infty, b_\infty: \mathcal{C}_\infty \times [-\Delta t, 0] \rightarrow \R_{\geq 0}$ such that the metric can be expressed as
$$g_\infty(t) = u^2_\infty(\xi, t) d\xi^2 + a^2_\infty(\xi,t) \omega\otimes\omega + b_\infty^2(\xi, t) \pi^\ast(g_{FS}),$$
where at time $0$ we have
$$u = 1 \: \text{ on } \: \mathcal{C}_\infty.$$
As
$$a_\infty(p,t) = |\overline{X}_{\infty, 0}|_{g_\infty(t)}(p)$$
and $\overline{X}_{i,o} \rightarrow \overline{X}_{\infty,0}$ as $i \rightarrow \infty$ by above, we see that away from a non-principal orbit $a_i \rightarrow a_\infty$ smoothly. Similarly, one can show with help of the remaining Killing vector fields $\overline{X}_{i,j}$, $j = 1, 2, 3$, that away from a non-principal orbit $b_i \rightarrow b_\infty$ smoothly. 

Hence away from a non-principal orbit,$a_i(\xi, t), b_i(\xi, t) \rightarrow a_\infty(\xi, t), b_\infty(\xi, t)$ in $C^\infty$ as $i \rightarrow \infty$. Furthermore, from the curvature bounds on $\Omega_i$ and the boundary conditions on $a_i$, $b_i$, $a_\infty$, $b_\infty$ at a non-principal orbit (see section \ref{manifold-metric-subsec} for the smoothness conditions on the warping functions at the non-principal orbit), one can show that in fact $a_i(\xi, t), b_i(\xi, t) \rightarrow a_\infty(\xi, t), b_\infty(\xi, t)$ smoothly everywhere. Hence the metric $g_\infty(t)$, $t \in [-\Delta t,0]$, is in the class $\mathcal{I}$. As $a_s \geq 0$ for metrics in $\mathcal{I}$ we see that $\mathcal{C}_\infty$ can possess at most one non-principal orbit. Finally, we note that by \cite[Theorem 3.16]{ChI} the closed set $\overline{C_{g_{\infty(0)}}(p_\infty, r')} \subset \mathcal{C}_\infty$ is compact for every $r' < r$.
\end{proof}

\section{Ancient Ricci flows Part I}
In this section we prove some properties of ancient Ricci flows $g(t) \in \mathcal{I}$, $-\infty < t \leq 0$, that are \emph{non-collapsed at all scales}. This yields important geometric information on the blow-up limits of singular Ricci flows, which we exploit and refine in later chapters. The main goal is to prove the following theorem:
\begin{thm}
\label{T-ancient}
Let $\kappa >0$ and $(M_k, g(t))$, $k \geq 2$, $t \in (-\infty,0]$, be an ancient Ricci flow, which satisfies the following properties:
\begin{enumerate}[label=(\roman*)]
\item $\kappa$-non-collapsed at all scales 
\item $g(t) \in \mathcal{I}$ for $t \in (-\infty,0]$. 
\end{enumerate}
Then if $k = 2$ the following inequalities hold:
\begin{align*}
T_1 &= a_s + 2Q^2 - 2 \geq 0 \\
T_2 &= Q y -x  \geq 0  \\
T_3 &= a_s - Q b_s - Q^2 + 1 \geq 0
\end{align*} 
If $k > 2$ we only have $T_1 \geq 0 $ and $T_3 \geq 0$. For all $k \geq 2$ we have $T_1(p,t) = 0$ if, and only if, $k =2$ and $p \in S_o^2$.
\end{thm}
Furthermore we show
\begin{thm}
\label{kahler-ancient-thm}
Let $\kappa >0$ and $(M_2,g(t))$, $t \in (-\infty,0]$, be an ancient Ricci flow, which satisfies the following properties: 
\begin{enumerate}[label=(\roman*)]
\item $\kappa$-non-collapsed at all scales  
\item $g(t) \in \mathcal{I}$ for $t \in (-\infty, 0]$
\item K\"ahler with respect to $J_1$, or equivalently $y=0$ everywhere
\end{enumerate}
Then $(M_2,g(t))$ is stationary and homothetic to the Eguchi-Hanson space. 
\end{thm}

\vspace{1em}
\noindent\textbf{Proof strategy.} 
In both of these theorems we are given an ancient Ricci flow $(M,g(t))$, $t \leq 0$, and want to show that a scale invariant quantity $T$ satisfies
$$T \geq 0 \: \text{ on } \: M \times (-\infty, 0].$$ 
We prove such statements by a contradiction/compactness argument. First we assume that
$$\iota \coloneqq \inf_{M_k \times (-\infty,0]} T < 0$$
and take a sequence of points $(p_i, t_i)$ in spacetime such that 
$$T(p_i, t_i) \rightarrow \iota \: \text{ as } \: i \rightarrow \infty.$$ 
Then we consider the rescaled Ricci flows 
$$g_i(t) = \frac{1}{b^2(p_i,t_i)} g\left( t_i + b^2(p_i, t_i) t \right), \quad t \in [-\Delta t, 0],$$
where $\Delta t > 0$ is chosen such that the conditions of Proposition \ref{blow-up-prop} are met. Then $(C_{g_i(0)}\left(p_i, \frac{1}{2}\right), g_i(t), p_i)$ subsequentially converges to a Ricci flow $(\mathcal{C}_\infty, g_\infty(t), p_\infty)$, $t \in [-\Delta t, 0]$. By construction 
$$T(p_\infty, 0) = \inf_{\mathcal{C}_\infty\times [-\Delta t, 0]} T = \iota < 0.$$
However, if the evolution equation of $T$ precludes the possibility of a negative infimum being attained, we have arrived at a contradiction and proven the desired result.
\vspace{1em}

\vspace{1em}
\noindent\textbf{Proof of main theorems of this section.} Before proving Theorem \ref{T-ancient} we need to state a technical lemma in preparation:
\begin{lem}
\label{away-from-origin-lem}
Let $(M_k, g(t))$, $k \in \N$, $t \leq 0$, be an ancient Ricci flow satisfying the conditions of Theorem \ref{T-ancient}. Then for every $\epsilon > 0$ there exists a $\delta > 0$ such that whenever at a point $(p, t)$ in spacetime one of the following inequalities holds
\begin{enumerate}[label=(\roman*)]
\item $T_1(p, t)  \leq - \epsilon$ and $k \geq 2$  
\item $T_2(p, t) \leq  -\epsilon$ and $k \leq 2$ 
\item $T_3(p, t) \leq - \epsilon$ and $k \in \N$
\item $|x(p, t)| \geq \epsilon$ and $k = 2$
\end{enumerate}
then $s(p, t) \geq \delta b(p,t)$.
\end{lem}
\begin{proof}
Recall that by Corollary \ref{cor:curv-bound-ancient} there exists a $C_1 > 0$ such that $|Rm_{g(t)}|_{g(t)} \leq \frac{C_1}{0}$ on $M_k \times (-\infty, 0]$. We first prove (i). We fix $\epsilon >0$ and argue by contradiction. Assume there exists a sequence of points $(p_i, t_i)$ in spacetime such that 
$$T_1(p_i, t_i) \leq - \epsilon$$
and 
\begin{equation}
\label{eqn:dist-from-origin}
\frac{s(p_i,t_i)}{b(p_i, t_i)} \rightarrow 0.
\end{equation} 
Define the rescaled metrics 
$$g_i = \frac{1}{b^2(p_i, t_i)} g\left(t_i + t b^{2}(p_i, t_i) \right), \: t \in [-\Delta t, 0].$$ 
For sufficiently small $\Delta t > 0$ the conditions of Proposition \ref{blow-up-prop} are satisfied and hence $( C_{g_i(0)}\left(p_i, \frac{1}{2}\right), g_i(t), p_i)$ subsequentially converges to a Ricci flow $(\mathcal{C}_\infty, g_\infty(t), p_\infty)$. By (\ref{eqn:dist-from-origin}) one sees that $p_\infty$ lies on the non-principal orbit $S^2_o$ of $\mathcal{C}_\infty$. By construction we have $T_1(p_\infty, 0) \leq - \epsilon$ as $T_1$ is a scale invariant quantity. This however contradicts the fact that $T_1 = a_s + 2\left(Q^2 -1\right) = k - 2 \geq 0$ on $S^2_o$.

Note that $T_2 = 2 - k$, $T_3 = k + 1$ and $x = k - 2$ on $S^2_o$. Therefore by the same argument applied to $T_2$, $T_3$ and $x$ the desired result holds true.
\end{proof}

Next we prove Theorem \ref{T-ancient}.

\begin{proof}[Proof of Theorem \ref{T-ancient}]
Recall that by Corollary \ref{cor:curv-bound-ancient} there exists a $C_1 > 0$ such that $|Rm_{g(t)}|_{g(t)} \leq \frac{C_1}{0}$ on $M_k \times (-\infty, 0]$.
We first show that $T_1 \geq 0$ in $M_k \times (-\infty, 0]$. We argue by contradiction. Assume that
\begin{equation*}
\iota := \inf_{M_k\times(-\infty, 0]} T_1 < 0.
\end{equation*}
As $g(t) \in \mathcal{I}$ we know that $\iota> - \infty$. Take a sequence of points $(p_i, t_i)$ in spacetime such that
\begin{equation*}
T_1(p_i, t_i) \rightarrow \iota \: \text{ as } \: i \rightarrow \infty.
\end{equation*}
From Lemma \ref{away-from-origin-lem} it follows that for sufficiently large $i$ 
\begin{equation}
\label{away-from-origin}
s(p_i, t_i) \geq \delta b(0,t_i)
\end{equation}
for some $\delta > 0$. Define the rescaled metrics 
$$g_i = \frac{1}{b^2(p_i, t_i)} g\left(t + t_i b^2(p_i, t_i)\right), \quad t \in [-\Delta t, 0].$$ 
For sufficiently small $\Delta t > 0$ the conditions of Proposition \ref{blow-up-prop} are satisfied and hence $(C_{g_i(0)}\left(p_i, \frac{1}{2}\right), g_i(t), p_i)$ subsequentially converges to a Ricci flow $(\mathcal{C}_\infty, g_\infty(t), p_\infty)$, $t \in [-\Delta t, 0]$, on which by construction 
\begin{equation*}
b(p_{\infty},0) = 1.
\end{equation*}
and
\begin{equation}
\label{T1-inf}
T_1(p_{\infty},0) = \inf_{\mathcal{C}_\infty \times [-\Delta t, 0]} T_1 = \iota < 0,
\end{equation}
as $T_1$ is a scale invariant quantity. Since $T_s(p_\infty, 0) = 0$, we see from the evolution equation (\ref{T1-evol}) of $T_1$ that
\begin{align*}
\partial_t T_1 \Big |_{(p_\infty, 0)} &= L[T_1] + \frac{1}{b^2} \left[- 4\left(1+Q^2\right)y^2 + 8Q\left(1-2Q^2 \right)y + 16Q^2\left(1-Q^2\right)\right] \\
									  & \qquad  + \frac{2 y T_1}{b^2} \left( 2Q- y\right) \\ \nonumber
 & \geq (T_1)_{ss} + \frac{4Q^2}{b^2}\left(1- Q^2\right) + \frac{2 y T_1}{b^2} \left( 2Q- y\right),
\end{align*}
where we bounded the zeroth order term from below as in the proof of Lemma \ref{T1-preserved-lem}. Hence 
\begin{equation*}
\partial_t T_1 \Big|_{(p_{\infty},0)} > 0
\end{equation*}
unless 
$$\text{case b)}: \quad Q(p_{\infty},0) = 0 \text{ and } y(p_{\infty},0) = 0$$
or
$$\text{case a)}: \quad Q(p_{\infty},0) = 1 \text{ and } y(p_{\infty},0) = 0$$
However by (\ref{T1-inf}) we have
\begin{equation*}
\partial_t T_1 \Big|_{(p_{\infty},0)}  \leq 0.
\end{equation*}
showing that either case a) or case b) must hold. We now show that both of these cases are impossible, thereby arriving at a contradiction. First note that by (\ref{away-from-origin}) we know that $p_\infty$ does not lie on the non-principal orbit $S^2_o$. Therefore the strong maximum principle applied to the evolution equation (\ref{Q-evol}) of $Q$ shows that in case a) $Q= 0$ everywhere on $\mathcal{C}_\infty \times [-\Delta t, 0]$. This, however, contradicts the non-collapsedness of $\mathcal{C}_\infty$ and therefore case a) cannot occur. In case b) the same argument shows that $Q=1$ everywhere on $\mathcal{C}_\infty\times [-\Delta t, 0]$. Then applying the strong maximum principle to the evolution equation (\ref{Qy-evol}) of $Qy$, which simplifies when $Q = 1$, shows that $y =0$ everywhere on $\mathcal{C}_\infty\times [-\Delta t, 0]$. This, however, implies that $T_1 = 1 > 0$ on $\mathcal{C}_\infty \times [-\Delta t, 0]$ contradicting our assumption that $\iota < 0$. 

It remains to be shown that $T_1(p,t) = 0$ if, and only if, $k=2$ and $p$ lies on the non-principal orbit $S^2_o$. We argue by contradiction. Assume there exists a point $(p,t)$ in spacetime such that $p \notin S^2_o$ and 
$$T_1(p, t) = 0.$$ 
Then arguing as above, we see that either case a) or case b) must hold true, both of which lead to the same contradiction.

By the same method we may prove that $T_2 \geq 0$ and $T_3 \geq 0$ on $M_k \times (-\infty, 0]$. Note that the evolution equations (\ref{T2-evol}) and (\ref{T3-evol}) show that $T_2$ and $T_3$ cannot attain a negative infimum, leading to the desired contradiction.
\end{proof}

Next we prove Theorem \ref{kahler-ancient-thm}.

\begin{proof}[Proof of Theorem \ref{kahler-ancient-thm}]
Recall that by Corollary \ref{cor:curv-bound-ancient} there exists a $C_1 > 0$ such that $|Rm_{g(t)}|_{g(t)} \leq \frac{C_1}{0}$ on $M_2 \times (-\infty, 0]$. Also recall Lemma \ref{unique-Kahler-lem}, which states that $(M_2, g(t)$, $t \in (-\infty, 0]$, is homothetic to the Eguchi-Hanson space if, and only if, 
$$x = y = 0 \: \text{ on } \: M_2 \times (-\infty, 0].$$
Therefore it suffices to show that $x = 0$. We follow the proof strategy of Theorem \ref{T-ancient} and argue by contradiction. Assume 
\begin{equation*}
\iota \coloneqq \inf_{M_2\times(\infty, 0]} x < 0
\end{equation*}
and take a sequence of points $(p_i, t_i)$ in spacetime such that
\begin{equation*}
x(p_i, t_i) \rightarrow \iota.
\end{equation*}
Note that $\iota > -\infty$ as $g(t) \in \mathcal{I}$ for $t \in (-\infty, 0]$. From Lemma \ref{away-from-origin-lem} it follows that
\begin{equation}
\label{away-from-origin2}
s(p_i, t_i) \geq \delta b(0,t_i)
\end{equation}
for some $\delta > 0$. Define the rescaled metrics 
$$g_i = \frac{1}{b^2(p_i, t_i)} g\left(t + t_i b^2(p_i, t_i)\right), \quad t \in [-\Delta t, 0],$$ 
where $\Delta t > 0$ is chosen such that the conditions of Proposition \ref{blow-up-prop} are satisfied. Then $(C_{g_i(0)}\left(p_i, \frac{1}{2}\right), g_i(t), p_i)$ subsequentially converges to a Ricci flow $(\mathcal{C}_\infty, g_\infty(t), p_\infty)$, $t \in [-\Delta t, 0]$, on which by construction 
\begin{equation*}
x(p_{\infty},0) = \iota < 0,
\end{equation*}
since $x$ is a scale-invariant quantity. Furthermore, we see by (\ref{away-from-origin2}) that $p_\infty$ does not lie on the non-principal orbit $S^2_o$. The evolution equation (\ref{x-evol}) for $x$ in the K\"ahler case $y=0$ simplifies to 
\begin{equation*}
\partial_t x = L[x] - \frac{4Q^2}{b^2} x
\end{equation*}
which implies that
\begin{equation*}
\partial_t x \Big|_{(p_{\infty}, 0)} = x_{ss} - \frac{4Q^2}{b^2} x > 0
\end{equation*} 
unless $Q(p_\infty,0) = 0$. This, however, cannot happen, as otherwise the strong maximum principle applied to the evolution equation (\ref{Q-evol}) of $Q$ would imply that $Q = 0$ on $\mathcal{C}_\infty \times [-\Delta t, 0]$. Hence we have arrived at a contradiction and conclude
\begin{equation*}
x \geq 0 \: \text{ on } M_2 \times (-\infty, 0].
\end{equation*}
By the same argument one shows that 
\begin{equation*}
x \leq 0 \: \text{ on } M_2 \times (-\infty, 0].
\end{equation*}
as well, which concludes the proof. 
\end{proof}

\section{Eguchi-Hanson and a family of Type II singularities}
\label{E-H-sing-section}
In this section we show that Ricci flow solutions $(M_k, g(t))$, $k \geq 2$, starting from a large class of initial metrics encounter a Type II singularity in finite time at the origin. In the case $k=2$ we show that the Eguchi-Hanson metric can occur as a blow-up limit. Below we state the precise result:

\begin{thm}[Type II singularities]
\label{E-H-sing-thm}
Let $(M_k, g(t))$, $k \geq 2$, be a Ricci flow starting from an initial metric $g(0) \in \mathcal{I}$ (see Definition \ref{def:I}) with
\begin{equation}
\label{bounded-b-initially}
\sup_{p \in M_2} b(p,0) < \infty.
\end{equation}
Then $g(t)$ encounters a Type II curvature singularity in finite time $T_{sing}>0$ and 
\begin{equation*}
\sup_{0 \leq t < T_{sing}} \left(T_{sing} - t\right) b^{-2}(o,t) = \infty.
\end{equation*}
Furthermore, there exists a sequence of times $t_i \rightarrow T_{sing}$ such that the following holds:
Consider the rescaled metrics 
\begin{equation*}
g_i(t) = \frac{1}{b^2(o,t_i)} g\left( t_i + b^2(o,t_i) t \right), \quad t \in \big[-b(o, t_i)^{-2} t_i, b(o, t_i)^{-2}\left(T_{sing} - t_i\right) \big).
\end{equation*}
Then $(M_k, g_i(t), o)$ subsequentially converges, in the pointed Gromov-Cheeger sense, to an eternal Ricci flow $(M_k, g_\infty(t), o)$, $t \in (-\infty, \infty)$. When $k=2$ the metric $g_\infty(t)$ is stationary and homothetic to the Eguchi-Hanson metric.
\end{thm}

In this paper we do not study the detailed geometry of the singularity models of the Type II singularities arising in the $k \geq 3$ case, however, as stated in Conjecture \ref{conj:sing} in the introduction, the author believes that these singularities are modeled on the non-collapsed steady Ricci solitons found in \cite{A17}. The author, in collaboration with Jon Wilkening, has carried out numerical simulations supporting this conjecture. A paper summarizing the results is in preparation \cite{AW19}.

\vspace{1em}
\noindent\textbf{Outline of proof.} Here we sketch the proof of Theorem \ref{E-H-sing-thm}. First we show in Lemma \ref{sing-time-finite} that the condition (\ref{bounded-b-initially}) forces a Ricci flow solution $(M_k, g(t))$, $k \geq 2$, to develop a singularity in finite time $T_{sing}>0$ at the origin. Then we take a sequence of times $t'_i \rightarrow T_{sing}$ and define the rescaled metrics
\begin{equation*}
g'_i(t) = \frac{1}{b^2(0,t'_i)} g\left( t'_i + b^2(0,t'_i) t \right).
\end{equation*}
These metrics subsequentially converge to a singularity model $(M_k, g'_\infty(t))$, $-\infty < t \leq 0$ --- an ancient solution of the Ricci flow. Now recall the dichotomy between Type I and Type II singularities and that every Type I singularity is modeled on a shrinking Ricci soliton \cite{EMT11}. Therefore we can prove that the singularity is of Type II by showing that $(M_k, g'_\infty(t))$ is not a shrinking Ricci soliton. For this we apply Theorem \ref{thm:no-shrinker}, which excludes shrinking solitons whenever (i) $\sup |b_s| < \infty$, (ii) $T_1>0$ for $s >0$ and (iii) $Q \leq 1$ hold. By definition, every metric in $\mathcal{I}$ satisfies conditions (i) and (iii). As these conditions are scale-invariant, they pass to the blow-up limit $(M_k, g'_\infty(t))$. From Theorem \ref{T-ancient} it follows that condition (iii) holds true as well, allowing us to conclude that $(M_k, g'_\infty(t))$ is not a shrinking soliton and that the singularity is of Type II. By the work of Hamilton we can then choose a sequence of times $t_i \rightarrow T_{sing}$, possibly different from the sequence $t'_i$, such that the corresponding blow-ups around the origin converge to an eternal Ricci flow $(M_k, g_\infty(t))$, $-\infty< t < \infty$. 

In the $k=2$ case we show that $(M_2, g_\infty(t))$ is stationary under Ricci flow and homothetic to the Eguchi-Hanson space. What makes $k=2$ special is that the second term of the right hand side of
\begin{equation*}
\partial_t b(0,t) = 2 \left( y_s + \frac{k-2}{b} \right) 
\end{equation*}
is zero and therefore 
\begin{equation}
\label{b0-evol-k2}
\partial_t b(0,t) = 2y_s(0,t) \leq 0,
\end{equation} 
as $y\leq 0$ with equality at $S^2_0$ for metrics in $\mathcal{I}$. It turns out that for the specific choice of $t_i \rightarrow T_{sing}$ from Hamilton's trick we have that on $(M_2, g_\infty(t))$ at $S^2_o$ at time 0 we have 
$$\partial_t b(0,t) = 2y_s(0,t) = 0.$$
An application of L'H\^opital's Rule shows that on $S^2_o$ we have
$$\frac{y}{Q} = \frac{y_s}{k}.$$
Therefore we can apply the strong maximum principle of Theorem \ref{maximum-principle}, Case 2, to the evolution equation (\ref{yQ-evol}) of $\frac{y}{Q}$ to show that $y = 0$ everywhere. From Theorem \ref{kahler-ancient-thm} it then follows that $(M_2, g_\infty(t))$ is homothetic to the Eguchi-Hanson space. 
\begin{remark}
A priori it may be possible that other sequences of times give rise to blow-up limits around $o$ that are not homothetic to the Eguchi-Hanson space. However in section \ref{E-H-unique-ancient-section} we show that the Eguchi-Hanson space is in fact the unique blow-up limit.
\end{remark}

\vspace{1em}
\noindent\textbf{Recap of some properties of singular Ricci flow solutions.}
Before proving Theorem \ref{E-H-sing-thm} we summarize some properties of curvature blow-up rates of Ricci flows encountering singularities and their respective singularity models. For this let $(M, g(t))$, $t \in [0,T)$ be a Ricci flow encountering a singularity at time $T$. Let
$$K_{max}(t) := \sup_{M} |Rm_{g(t)}|_{g(t)}.$$
By Shi's result \cite{Shi89} on the short time existence of Ricci flow we have
$$ \limsup_{t \nearrow T} K_{max}(t) = \infty.$$ 
In fact one can show with help of the evolution equation of $|Rm_{g(t)}|_{g(t)}^2$ that
\begin{equation}
\label{blow-up-rate-lower-bound}
\sup_{M} |Rm_{g(t)}|_{g(t)} \geq \frac{1}{8}\frac{1}{T-t}.
\end{equation}
Hamilton \cite{Ham95} introduced the notion of Type I and Type II Ricci flows, which are defined by the rate at which the curvature blows up as $t \nearrow T$. 
In particular, $(M_2, g(t))$ is of Type I if it satisfies if there exists a $C>0$ such that for $t \in [0, T)$ 
$$K_{max}(t) \leq \frac{C}{T-t},$$
In the case that such a constant $C>0$ does not exists, that is
$$\sup_{t \in [0,T)} (T-t) K_{max}(t) = \infty,$$
we say the singularity is of Type II. 

By the work of Naber \cite{N10} and Enders, M\"uller and Topping \cite{EMT11} every Type I singularity model is a non-flat Ricci shrinking soliton. Hamilton  showed how for Type II singularities one can extract a blow-up sequence converging to an eternal Ricci flow \cite[Theorem 16.4]{Ham95}. However it remains to be understood whether or not all Type II singularity models are steady solitons. So far all known examples are. 

 Below we recap the main result of \cite{EMT11}: First note the following definition:

\begin{definition}[see {\cite[Definition 1.2]{EMT11}} ]
\label{singular-point-def}
A spacetime sequence $(p_i, t_i)$ with $p_i \in M$ and $t_i \nearrow T$ in a Ricci flow is
called an \textbf{essential blow-up sequence} if there exists a constant $c>0$ such that
$$|Rm_{g(t_i)}|_{g(t_i)}(p_i) \geq \frac{c}{T-t}.$$
A point $p \in M$ in a Type I Ricci flow is called a (general) \textbf{Type I singular point} if there
exists an essential blow-up sequence with $p_i \rightarrow p$ on $M$.
\end{definition}
Now we state the main result of \cite{EMT11}, asserting that Type I singularities are modeled on shrinking Ricci solitons.
\begin{thm}[see {\cite[Theorem 1.4]{EMT11}} ]
\label{typeI-blow-up}
Let $(M, g(t))$ be a Type I Ricci flow on $[0, T)$ and suppose $p$ is a Type
I singular point as in Definition \ref{singular-point-def}. Then for every sequence $\lambda_j \rightarrow \infty$, the rescaled Ricci
flows $(M, g_j(t), p)$ defined on $[−\lambda_j T, 0)$ by 
$$g_j (t) := \lambda_j g\left(T + \frac{t}{\lambda_j} \right) $$
subconverge to a non-flat gradient shrinking soliton.
\end{thm}
We use Theorem \ref{typeI-blow-up} to exclude Type I singularities for Ricci flows satisfying the assumptions of Theorem \ref{E-H-sing-thm}.

\vspace{1em}
\noindent\textbf{Proof of the main theorem.}
First we show that a singularity must occur in finite time:
\begin{lem}
\label{sing-time-finite}
The maximal extension of a Ricci flow $(M_k, g(t))$, $k \geq 1$, starting from an initial metric $g(0) \in \mathcal{I}$ with 
\begin{equation*}
\sup_{p \in M_k} b(p,0) < \infty 
\end{equation*}
encounters a singularity at the $S^2_o$ in finite time $T_{sing}>0$. 
\end{lem}
\begin{proof}
By Shi's short time existence of Ricci flow \cite{Shi89} we have $T_{sing} > 0$. From the evolution equation (\ref{b-evol}) of $b$ under Ricci flow it follows that
\begin{equation*}
\partial_t b^2 \leq \Delta_{g(t)} b^2 - 4,
\end{equation*}
where we used expression (\ref{laplacian}) of the Laplacian. By the maximum principle (see for instance \cite[Theorem 12.14]{ChII}) we see that there exists a $T < \infty$ such that
$$\inf_{p\in M_k} b^2(p,t) \rightarrow 0 \quad \text{ as } t \rightarrow T.$$
As $b_s \geq 0$ we conclude that 
\begin{equation*}
\lim_{t\rightarrow T} b(o,t) = 0.
\end{equation*}
From Lemma \ref{b-bound} it follows that the curvature at $S^2_o$ blows up as $t \rightarrow T$. Hence $T = T_{sing}$. 
\end{proof}

Below we prove Theorem \ref{E-H-sing-thm}.

\begin{proof}[Proof of Theorem \ref{E-H-sing-thm}]
By Lemma \ref{sing-time-finite} the Ricci flow becomes singular in finite time $T_{sing} > \delta > 0$ and $b(o,t) \rightarrow 0$ as $t \nearrow T_{sing}$. Recall that by Theorem \ref{curv-bound} there exist a $C_1 > 0$ such that
$$|Rm_{g(t)}|_{g(t)} \leq \frac{C_1}{b^2} \: \text{ on } \: M_k \times [0, T_{sing}).$$
Moreover, by Theorem \ref{thm:no-local-collapsing} there exist constants $\kappa, \rho > 0$ such that $g(t)$ is $\kappa$-non-collapsed at scales less or equal to $\rho$.

Now take a sequence of times $t'_i \nearrow T_{sing}$ such that
\begin{enumerate}
\item $b(o, t'_i) \rightarrow 0$ as $i \rightarrow \infty$ 
\item $b(o, t) \geq b(o,t'_i)$ for $t \leq t'_i$
\end{enumerate}
\begin{claim}
The sequence of points $(o, t'_i)$ in spacetime is an essential blow-up sequence.
\end{claim}

\begin{claimproof}
We argue by contradiction. Assume, after passing to a subsequence, that
$$(T_{sing}-t'_i) |Rm_{g(t'_i)}|_{g(t'_i)} \rightarrow 0 \text{ as } i \rightarrow \infty.$$
Then by Lemma \ref{b-bound} and the fact that $b_s \geq 0$ for metrics in $\mathcal{I}$ we have
$$b^2(p,t) \geq b^2(o,t) = \frac{4}{R_{2323}} \geq \frac{4}{|Rm_{g(t'_i)}|_{g(t'_i)}(o)},$$
where we used the expression for the curvature component $R_{2323}$ derived in section \ref{con-lap-cur-subsec}. This shows that
$$|Rm_{g(t'_i)}|_{g(t'_i)}(p) \leq \frac{C_1}{b^2(p,t'_i)} \leq \frac{C_1}{4} |Rm_{g(t'_i)}|_{g(t'_i)}(o)  \: \text{ for } \: p \in M_k.$$
Therefore 
$$\lim_{i \rightarrow \infty} \left(T_{sing} - t'_i\right) \sup_{p\in M_2}|Rm_{g(t'_i)}|_{g(t'_i)}(p) = 0,$$
which contradicts (\ref{blow-up-rate-lower-bound}). This proves the claim.
\end{claimproof}

Define the rescaled metrics
\begin{equation*}
g'_i(t) = \frac{1}{b^2(0,t'_i)} g\left( t'_i + b^2(0,t'_i) t \right), \quad t \in [- b^{-2}(o,t'_i)t'_i, 0],
\end{equation*}
By property (2) above and the fact that $b_s\geq 0$ for metrics in $\mathcal{I}$ it follows that
$$|Rm_{g'_i(t)}|_{g'_i(t)} \leq C_1 \: \text{ on } \: M_k \times [- b^{-2}(o,t'_i)t'_i, 0].$$
Note also that the rescaled metrics $g'_i(t)$ are $\kappa$-non-collapsed at scales tending to infinity. Corollary \ref{cor:compactness-complete-flows} then implies that $(M_k, g'_i(t),o)$ subsequentially converges, in the Cheeger-Gromov sense, to an ancient Ricci flow $(M_\infty, g'_{\infty}(t),o)$, $t \in (-\infty,0]$, where $M_\infty \cong M_k$. By Theorem \ref{T-ancient} we have
$$T_1(p,t) > 0 \: \text{ on } \: M_\infty \setminus S^2_o \times (-\infty, 0]$$
on the blow-up limit $g'_\infty(t)$. Theorem \ref{thm:no-shrinker} shows that $g'_\infty(t)$ cannot be a shrinking soliton, which by the contrapositive of Theorem \ref{typeI-blow-up} proves that the singularity is of Type II. Therefore
\begin{equation*}
\sup_{M \times [0,T_{sing})} \left(T_{sing}- t\right) |Rm_{g(t)}|_{g(t)} = \infty
\end{equation*}
from which we see that
\begin{equation*}
\sup_{t \in [0, T_{sing})} \left(T_{sing} - t\right) b^{-2}(0,t) = \infty.
\end{equation*}
Now we mimic the proof of \cite[Theorem 16.4, Type II(a)]{Ham95} to construct an eternal blow-up limit. Pick a sequence of times $T_i < T_{sing}$ satisfying 
$$\left(T_{sing} - T_i\right) b^{-2}(o,t) \rightarrow \infty $$ as $i \rightarrow \infty$. Then we can choose $t_i < T_i$ such that
\begin{equation}
\label{bminus2-bound}
\left(T_i - t_i\right) b^{-2}(o,t_i) = \sup_{t \leq T_i} \left(T_i - t\right)b^{-2}(o,t)
\end{equation}
as the latter goes to zero as $t \rightarrow T_i$. Consider the rescaled Ricci flow solutions 
$$g_i(t) = \frac{1}{b^2(0,t_i)} g\left( t_i + b^2(0,t_i) t \right),$$
which exist for $-A_i \leq t \leq B_i$ with
\begin{align*}
A_i &= t_i b^{-2}(o,t_i) \rightarrow \infty\\
B_i &= \left(T_i- t_i\right) b^{-2}(o, t_i) \rightarrow \infty.
\end{align*}
If we write $a_i, b_i$ for the warping functions of the rescaled metric $g_i(t)$ we obtain from equation (\ref{bminus2-bound}) the following inequality
\begin{equation*}
\left(B_i - t \right)b_i^{-2} (0,t) \leq B_i.
\end{equation*}
Note that here $t$ is the time variable of the rescaled Ricci flow $g_i(t)$. Therefore for any fixed $t$ we have
\begin{equation}
\label{b-before-limit-everywhere}
b_i^{-2} (o,t) \leq \frac{B_i}{B_i -t} \rightarrow 1 \quad \text{as } i \rightarrow \infty
\end{equation}
and 
\begin{equation}
\label{b-before-limit-at-o}
b^{-2}_i(o,0) = 1.
\end{equation}
From this, the fact that $b_s \geq 0$ and the curvature bound of Theorem \ref{curv-bound}, we see that on bounded time intervals the curvatures of $g_i(t)$ eventually become bounded by $2C_1$. In addition to this the metrics $g_i(t)$ are $\kappa$-non-collapsed at larger and larger scales. Therefore Corollary \ref{cor:compactness-complete-flows} implies that $(M_2, g_i(t), o)$ subsequentially converges to an eternal Ricci flow $(M_2, g_\infty(t), o)$. Furthermore (\ref{b-before-limit-everywhere}) and (\ref{b-before-limit-at-o}) show that that
\begin{equation}
\label{b-limit-greater-1}
b_{\infty}(o,t) \geq 1 \: \text{ for } \: t \in (-\infty, \infty)
\end{equation}
and
\begin{equation}
\label{b-limit-at-o}
b_{\infty} (o,0) = 1,
\end{equation}
where we write $a_{\infty}$ and $b_{\infty}$ for the warping functions of the metric $g_{\infty}(t)$. Notice that (\ref{b-limit-greater-1}) and (\ref{b-limit-at-o}) imply that at time $0$ on $S^2_o$ we have
\begin{equation}
\label{origin-behavior}
\partial_t b_{\infty} = 2 \left( (y_{\infty})_s + \frac{k-2}{b_{\infty}}\right) = 0,
\end{equation}
where $y_\infty = (b_\infty)_s - \frac{a_\infty}{b_\infty}$ corresponds to the K\"ahler quantity $y$ on the $g_\infty$ background.  

Now it only remains to be shown that in the $k=2$ case $g_\infty(t)$ is stationary and homothetic to the Eguchi-Hanson metric. In the following we drop the $\infty$ subscript and let $a$, $b$, $Q$, $y$ be with respect to the metric $g_\infty(t)$. Note that equation (\ref{origin-behavior}) and an application of L'H\^opital's Rule show that at time $0$ on $S^2_o$ we have
\begin{equation*}
y_s = \frac{y}{Q} = 0
\end{equation*}
The evolution equation for $\frac{y}{Q}$ derived in the Appendix A is
\begin{equation}
\label{yoverQ-evol}
\partial_t \left(\frac{y}{Q}\right) = \left(\frac{y}{Q}\right)_{ss} + \left(3 \frac{a_s}{a} - 2 \frac{b_s}{b} \right) \left(\frac{y}{Q}\right)_s + \frac{2}{b^2}\frac{y}{Q}\left( 2 + \frac{y}{Q}\right) \left( Q b_s - 2 a_s \right).
\end{equation}
Because $g_\infty(t) \in \mathcal{I}$ is of bounded curvature we see that
\begin{equation*}
\frac{1}{b^2} \left( 2 + \frac{y}{Q}\right) \left( Q b_s - 2a_s \right) = \frac{1}{b^2}\left(-\frac{2 a_s b_s}{Q}-2 a_s+Q b_s+b_s^2 \right)
\end{equation*}
is bounded. Note that we applied Lemma \ref{b-bound} to show that $\frac{1}{b^2}$ is bounded. Therefore we may apply the strong maximum principle of Theorem \ref{maximum-principle}, Case 2, to deduce that
\begin{equation*}
\frac{y}{Q} = 0 \: \text{ on } \: M_2 \times (-\infty, 0],
\end{equation*}
yielding that $g_\infty(t)$ is K\"ahler in the the $k=2$ case. By Theorem \ref{kahler-ancient-thm} we then deduce that $g_\infty(t)$ is homothetic to the Eguchi-Hanson metric, which proves the desired result.
\end{proof}

\section{Ancient Ricci flows Part II: $k=2$ case}
\label{E-H-unique-ancient-section}
In this section we prove that every non-collapsed ancient Ricci flow in the class of metrics $\mathcal{I}$ is isometric to the Eguchi-Hanson metric:
\begin{thm}[Unique ancient flow]
\label{E-H-ancient-thm}
Let $\kappa > 0$ and $\left(M_2,g(t)\right)$, $t \in (-\infty,0]$, be an ancient Ricci flow that is $\kappa$-non-collapsed at all scales and $g(t) \in \mathcal{I}$, $t \in (-\infty,0]$ (see Definition \ref{def:I}). Then $g(t)$ is stationary and homothetic to the Eguchi-Hanson metric. 
\end{thm}
An immediate consequence of this theorem is that for \emph{every} sequence of times $t_i \rightarrow T_{sing}$ in Theorem \ref{E-H-sing-thm}, the rescaled Ricci flows
$$g_i(t) = \frac{1}{b^2(o,t_i)}g\left(t + t_i b^2(o,t_i) \right), \quad t \in \left[- b(o,t_i)^{-2} t_i, 0\right], $$
subsequentially converge to the Eguchi-Hanson space. In other words, the Eguchi-Hanson space is the \emph{unique} limit of blow-ups around the origin. With a little extra work one can show the slightly more general result, asserting that blow-up limits centered at points close to, but not necessarily on the tip of $M_2$, subsequentially converge to the Eguchi-Hanson space:
\begin{cor}
\label{E-H-blowup}
Let $(M_2, g(t))$, $t \in[0, T_{sing})$, be a Ricci flow starting from an initial metric $g(0) \in \mathcal{I}$ (see Definition \ref{def:I}) that develops a singularity at time $T_{sing}$. Let $(p_i, t_i)$ be a sequence of points in spacetime with $t_i \rightarrow T_{sing}$ satisfying 
$$ \sup_i \frac{b(p_i,t_i)}{b(o,t_i)} < \infty$$
and consider the rescaled metrics 
\begin{equation*}
g_i(t) = \frac{1}{b^2(p_i,t_i)} g\left(t_i + b^2(p_i,t_i) t\right), \quad t \in \left[- t_i b^{-2}(p_i,t_i) , 0\right].
\end{equation*}
Then $(M_2, g_i(t), p_i)$ subsequentially converges, in the Gromov-Cheeger sense, to a blow-up limit $(M_2, g_\infty(t), p_\infty)$, $t \leq 0$, which is homothetic to the Eguchi-Hanson space.
\end{cor}
We defer the proof of Corollary \ref{E-H-blowup} to the end of subsection \ref{ancient-main-thm-subsec}.

\subsection{Outline of Proof} 
\label{E-H-ancient-thm-outline}
Here we outline the proof of Theorem \ref{E-H-ancient-thm}. Below we take $(M_2, g(t))$, $t \in (-\infty,0]$, to be a non-collapsed ancient Ricci flow with $g(t) \in \mathcal{I}$, $t \in (-\infty,0]$. We construct a continuously varying one-parameter family of functions 
\begin{equation*}
f_{\theta}: [0,1] \rightarrow [0,1], \quad \theta \in (0,1],
\end{equation*}
satisfying the following five requirements:
\begin{enumerate}
\item For every $\theta \in (0,1]$ the condition
\begin{equation*}
Z_{\theta}(\xi, t) := \frac{x}{Q^2} + f_{\theta}(Q) = \frac{a_s + Q^2 - 2}{Q^2} + f_{\theta} (Q) \geq 0
\end{equation*}
is preserved on the $(M_2, g(t))$ background.
\item For every $0\leq Q < 1$ 
\begin{equation*}
f_{\theta} (Q) \longrightarrow 0 \: \text{ as } \: \theta \longrightarrow 0
\end{equation*}
\item For every $\theta \in (0,1]$ there exists a $Q_{\theta} \in [0,1)$ such that
$$f(Q) < 1 \: \text{ for } \: Q < Q_{\theta},$$
and
$$ f(Q) = 1 \: \text{ for } \: Q \geq Q_{\theta}.$$
Furthermore $Q_{\theta}$ depends continuously on $\theta$.
\item For $\theta = 1$ $$f_1 = 1 $$ everywhere.
\item For every $\theta \in (0,1]$ the function $f_{\theta}$ is extendable to a smooth even function around $0$.
\end{enumerate}

\begin{remark}
\label{f-constr-rem}
We briefly remark on some of the properties of $f_{\theta}$:
\begin{itemize}
\item In the expression for $Z_{\theta}$ of requirement (1) we take $x$, $Q$ and $a_s$ to be functions of spacetime. For brevity we do not express the dependence explicitly. 
\item The term $\frac{x}{Q^2}$ can be extended smoothly to the non-principal orbit $S^2_o$, as $x = x_s = 0$ at $s = 0$. Therefore $Z_{\theta}$ is well-defined on $M_2$. 
\item When $\theta = 1$ we already know that
$$Z_1 =  \frac{x + Q^2}{Q^2} \geq 0$$
in $M_2 \times (-\infty, 0]$, as from Theorem \ref{T-ancient} it follows that $T_1 = Q^2 Z_1\geq 0$ on $M_2 \times (-\infty, 0]$.
\item At any point $(p, t)$ in spacetime such that $Q(p,t) \geq Q_{\theta}$ we have $$Z_{\theta} (p, t) = Z_1(p, t) \geq 0.$$ 
\item The family $f_{\theta}(Q)$, $\theta \in (0,1]$, we construct below is smooth everywhere apart from when $Q = Q_{\theta}$. It will become clear later that this does not pose a problem.
\end{itemize}
\end{remark}

In subsection \ref{subsec-evol} we show that at points $(p, t)$ in spacetime at which $Q(p, t) \neq Q_{\theta}$, or equivalently at points where $f$ is smooth, the evolution equation of the corresponding $Z_{\theta}$ can locally be written as
\begin{align*}
\partial_t Z_{\theta}  &= [Z_{\theta}]_{ss} + \left(3 \frac{a_s}{a} - 2\frac{b_s}{b}\right) [Z_{\theta}]_s + \frac{1}{b^2}\left( W_{\theta} + Z_{\theta} \tilde{D}_{\theta}  \right),
\end{align*}
where $W_{\theta}$ and $\tilde{D}_{\theta}$ are bounded and scale-invariant expressions involving $b_s$, $Q$, $f_{\theta}(Q)$, $f'_{\theta}(Q)$ and $f''_{\theta}(Q)$. Again, all quantities in the evolution equation of $Z_{\theta}$ should be interpreted as functions of spacetime. 

In subsection \ref{subsec-constr-f} we construct a family of functions $f_{\theta}$, $\theta \in (0,1]$, by solving an initial value problem for a second order non-linear ordinary differential equation. Subsequently we show that the family satisfies requirements (1)-(5) listed above. In particular, we show in subsection \ref{subsec-mathcalQ-pos} that for the constructed family --- on a non-collapsed ancient Ricci flow background --- the following property holds true: For all points $(p, t)$ in spacetime such that $Q(p, t) < Q_{\theta}$, we have $W_{\theta}(p, t) \geq 0$, $\theta\in (0,1]$. This fact, in conjunction with the fourth bullet point of Remark \ref{f-constr-rem}, essentially shows that for each $\theta\in (0,1]$ the inequality  $Z_{\theta} \geq 0$ is preserved on the $g(t)$ background. 

Once we have shown that our family $f_{\theta}$, $\theta \in (0,1]$, satisfies requirements (1)-(5), we will use a blow-up argument in conjunction with the strong maximum principle to show that if for some $\theta_0 \in (0,1]$ the inequality $Z_{\theta} \geq 0$ holds for all $\theta \in [\theta_0 , 1]$, then there must exists an $\theta_1 < \theta_0$ such that the inequality also holds for all $\theta \in [\theta_1, 1]$. This shows that the set
$$ \mathcal{E} = \left \{ \theta \in (0,1] \: \big | \: Z_{\theta'} \geq 0  \text{ for all } \theta \leq \theta' \leq 1 \right\} \subseteq (0,1]$$
is open. As $\mathcal{E}$ is defined by a closed condition and therefore closed, it follows that $\mathcal{E} = (0,1]$ and therefore
$$Z_{\theta} \geq 0 \quad \text{ for all } \theta \in (0,1].$$
By property (2) of $f_{\theta}$ we deduce that at all points $(p,t)$ in spacetime such that $Q(p,t) < 1$ we have 
$$x(p,t) \geq 0.$$
Note that by the strong maximum principle applied to the evolution equation (\ref{Q-evol}) of $Q$ it follows that $Q < 1$ and hence $x \leq 0$ everywhere. Now recall Theorem \ref{T-ancient} which states that
$$ x \leq Qy \leq 0 \: \text{ on } \: M_2 \times (-\infty, 0].$$
Therefore
$$x = y = 0 \: \text{ on } \: M_2 \times (-\infty, 0]$$ 
and we conclude that the metric $g(t)$ is homothetic to the Eguchi-Hanson metric by Lemma \ref{unique-Kahler-lem}.

\subsection{Evolution equations}
\label{subsec-evol} The main difficulty in carrying out the proof is to find a family of functions $f_{\theta}$, $\theta \in (0,1]$, for which requirement (1) is satisfied. Our strategy is to first derive the evolution equation of $Z_{\theta}$ for a general $f_{\theta}$ and then reduce the problem to solving a second order ordinary differential equation in $f_\theta$. For this, first note that the evolution equation of $f_{\theta}(Q)$ away from the non-principal orbit $S^2_o$ can be written as
\begin{align*}
\partial_t f_{\theta}(Q)  = [f_{\theta}(Q)]_{ss} + \left(3 \frac{a_s}{a} - 2\frac{b_s}{b}\right) [f_{\theta}(Q)]_s  + \frac{1}{b^2} C_{f},
\end{align*}
where
\begin{align}
\label{Cf}
C_f &= \left( 8 a_s b_s - 3 \frac{a_s^2}{Q} - 5 Q b_s^2 + 4 Q \left(1 - Q^2\right) \right)f' -  \left( a_s - Q b_s \right)^2f''
\end{align}
The computation is carried out in the Appendix A.
\begin{remark} Some remarks on the evolution equation of $f_{\theta}(Q)$:
\begin{itemize}
\item We often omit the dependence of our quantities on spacetime, i.e. by $f_{\theta}(Q)$ we mean $f_{\theta}(Q(p,t))$. 
\item For brevity we often omit the dependence of $f$ on $\theta$ and $Q$, as in the expression for $C_f$ above. For instance, we write $f'$ for $f'_{\theta}(Q)$ and $f''$ for $f''_{\theta}(Q)$. 
\item Note that by Lemma \ref{parity-lem} the quantity $Q = \frac{a}{b}$ as a function of $s$ can be extended to an odd function around the origin. Therefore as long as $f_{\theta}: [0,1] \rightarrow [0,1]$ is extendable to an even function around the origin, the term $\frac{f'}{Q}$ and hence $C_f$ can be smoothly extended to all of $M_2$. 
\end{itemize}
\end{remark}

From equation (\ref{Cf}) and the evolution equations (\ref{as-evol}) and (\ref{Q2-evol}) of $a_s$ and $Q^2$, respectively, we see that the evolution equation of $Z_{\theta}$ can be written as
\begin{equation}
\label{Z-epsilon-evol-1}
\partial_t Z_{\theta}  = [Z_{\theta}]_{ss} + \left(3 \frac{a_s}{a} - 2\frac{b_s}{b}\right) [Z_{\theta}]_s  + \frac{1}{b^2}\left( C_{Z,0} + C_{Z,1} Z_{\theta} + + C_{Z,2} Z^2_{\theta} \right),
\end{equation}
after having eliminated any occurring $a_s$ by substituting 
\begin{equation*}
a_s = Q^2 Z_{\theta}- f Q^2  - Q^2 + 2.
\end{equation*}
A computation carried out in the Appendix A shows that
\begin{align*}
C_{Z,0} &= A_0 +\left[\frac{b_s}{Q}\right] A_1 + \left[\frac{b_s}{Q}\right]^2 A_2,
\end{align*}
where
\begin{align*}
A_0  =& - Q^4 f^2 f''-2 Q^4 f f''-Q^4 f''+4 Q^2 f f''+4 Q^2 f''-4 f''-3 Q^3 f^2 f' \\ \nonumber
			&-6 Q^3 f f'-7 Q^3 f'+12 Q f f'+16 Q f'-\frac{12 f'}{Q}-2 Q^2 f+8 f-2 Q^2-4 \\
A_1 =& -2 Q^4 f f''-2 Q^4 f''+4 Q^2 f''-8 Q^3 f f'-8 Q^3 f' \\ \nonumber
    &  \qquad+16 Q f'-4 Q^2 f^2-8 Q^2 f+8 f+4 Q^2+8 \\ 
A_2 =& - Q^4f''-5 Q^3 f'-2 Q^2 f-2 Q^2-4 \\
\end{align*}
Similarly, we compute the expressions for $C_{Z,1}$ and $C_{Z,2}$ in the Appendix A, however their exact forms are not important for our analysis. It is only important to note that when $f$ is extendable to an even function around $0$, the quantities $C_{Z,i}$, $A_i$, $i = 0, 1, 2$, are scale-invariant, bounded, and can be extended smoothly to $S^2_o$.

For reasons that will become clear below, we rewrite the equation (\ref{Z-epsilon-evol-1}) in the form
\begin{align}
\label{Z-epsilon-evol-2}
\partial_t Z_{\theta}  &= [Z_{\theta}]_{ss} + \left(3 \frac{a_s}{a} - 2\frac{b_s}{b}\right) [Z_{\theta}]_s + \frac{1}{b^2}\left( W_{\theta} + Z_{\theta} D_{\theta}  \right),
\end{align}
where
\begin{equation}
\label{mathcalQ}
W_{\theta} = A_0 +\left[\frac{b_s}{Q} - Z_{\theta} \right] A_1 + \left[\frac{b_s}{Q} - Z_{\theta} \right]^2 A_2
\end{equation}
and
\begin{equation*}
D_{\theta} =  C_{Z,1} + Z C_{Z,2} + A_1 - A_2\left(Z_{\theta} - 2 \frac{b_s}{Q}\right) 
\end{equation*}
Sometimes it will be useful to regard $W_{\theta}$ as a quadratic polynomial. Therefore we define
$$w_{\theta}(z) =A_0 + A_1 z + A_2 z^2$$
Then 
$$W_{\theta} = w_{\theta}\left(\frac{b_s}{Q} - Z_{\theta}\right).$$

In the proof of Theorem \ref{E-H-ancient-thm} we also need the evolution equation of 
$$Z_1 = \frac{x}{Q^2}+1,$$
which can be written as 
\begin{equation}
\label{Z1-evol}
\partial_t Z_1  = [Z_1]_{ss} + \left(3 \frac{a_s}{a} - 2\frac{b_s}{b}\right) [Z_1]_s + \frac{1}{b^2} \left(C_{Z_1,0} + C_{Z_1,1} Z_1 + C_{Z_1,2} Z^2_1 \right)
\end{equation}
where
$$C_{Z_1,0} = \frac{1}{Q^2}\left( - 4\left(1+Q^2\right)y^2 + 8Q\left(1-2Q^2 \right)y + 16Q^2\left(1-Q^2\right) \right)$$
and $C_{Z_1,1}$, $C_{Z_1,2}$ are a bounded scale-invariant functions of $a_s$, $b_s$ and $Q$. The derivation of this evolution equation is carried out in the Appendix A. Note the following lemma:

\begin{lem}
\label{lem:suc-con-start}
Let $(M_2, g(t))$, $t \in (-\infty, 0]$, be an ancient Ricci flow as in Theorem \ref{E-H-ancient-thm}. Then 
$$Z_1 \geq 0$$
and 
$$C_{Z_1,0} \geq 4\left(1-Q^2\right)$$
everywhere in $M_2 \times (-\infty, 0]$.
\end{lem}
\begin{proof}
By Theorem \ref{T-ancient} we know that $Z_1 = \frac{T_1}{Q^2} \geq 0$ in $M_2 \times (-\infty, 0]$. Moreover, notice the similarity of $C_{Z_1,0}$ to the zeroth order term in the evolution equation (\ref{T1-evol}) of $T_1$. Therefore we see by the proof of Lemma \ref{T1-preserved-lem} that $C_{Z_1,0} \geq 4\left(1-Q^2\right)$ for metrics in $\mathcal{I}$.
\end{proof}

In the proof of Theorem \ref{E-H-ancient-thm} we deform the inequality $Z_1 \geq 0$ along a path of conserved inequalities $Z_{\theta} \geq 0$, $\theta \in (0,1]$. Thus $Z_1 \geq 0$ is the starting point for successively constraining the ancient Ricci flow towards the Eguchi-Hanson space. Below we construct the $f_{\theta}$ leading to the conserved inequalities $Z_{\theta} \geq 0$.

\subsection{Construction of $f_{\theta}$, $\theta \in (0,1]$}
\label{subsec-constr-f}

The goal of the following discussion is to find a family of functions $f_{\theta}: [0,1] \rightarrow [0,1]$, $\theta \in (0,1]$, such that 
\begin{equation*}
W_{\theta} \geq 0
\end{equation*}
is non-negative on ancient Ricci flows satisfying $Z_{\theta} \geq 0$. For this we consider solutions to the ordinary differential equation
\begin{align}
\label{C0-ODE}
0 = -4 \left(1-Q^2\right)^2 f'' &- 4 \left(1-Q^2\right) \left(Q^2 f-5 Q^2+3\right) \frac{f'}{Q} \\ \nonumber
  & \qquad + 2f\left( f^2 Q^2+3 f Q^2-6 f-6 Q^2+8 \right),
\end{align}
which is equivalent to
\begin{equation}
\label{eqn:alpha-ode-equivalent}
w_{\theta} \left( - f + 1\right)  = 0.
\end{equation}
Note that we are now regarding $Q$ as an independent variable and not as a function of spacetime. Before we explain how we arrived at this differential equation, we list some of its properties below. For clarity of exposition we defer their proofs to subsection \ref{technical-lemmas-subsec}.

\begin{lem}
\label{f-existence-lem}
For every $f_0 \in \R$ the ordinary differential equation (\ref{C0-ODE}) possesses an even analytic solution around the origin with initial condition
\begin{equation*}
f(0) = f_0.
\end{equation*}
Furthermore, $f$ varies smoothly with $f_0$.
\end{lem}

\begin{lem}
\label{f-inc-lem}
Let $f: [0,Q_{max}) \rightarrow \R$, $Q_{max} \leq 1$, be the maximal solution to the ordinary differential equation (\ref{C0-ODE}) with initial condition $0 < f(0) < 1$. Then on any interval $(0, Q_\ast)$, $Q_\ast \leq Q_{max}$, on which $0 < f(Q) \leq 1$ we have $f'(Q) > 0$.
\end{lem}

\begin{lem}
\label{f-lem-2}
Let $\theta \in (0,1]$ and $f_{\theta}: 0 \in I \rightarrow \R$ be the maximal solution to the ordinary differential equation (\ref{C0-ODE}) with $f_{\theta}(0) = \theta$. Then there exists a $Q_{\theta} \in [0,1)$ such that 
\begin{equation*}
f_{\theta}(Q_{\theta}) = 1
\end{equation*}
and 
\begin{equation*}
f_{\theta}(Q) < 1 \: \text{ for } 0 \leq Q < Q_{\theta}.
\end{equation*}
Furthermore, 
\begin{enumerate}
\item $Q_{\theta}$ varies continuously with $\theta \in (0,1]$
\item $Q_{\theta} \rightarrow 1 \text{ as } \theta \rightarrow 0$
\item $Q_1 = 0$
\end{enumerate}
\end{lem}

For each $\theta \in (0,1]$ let
$$ \phi_{\theta}: [0, Q_{\theta}] \rightarrow [\theta, 1]$$
be the solution to the differential equation (\ref{C0-ODE}) with initial condition $$\phi_{\theta}(0) = \theta$$ and define $f_{\theta}$, $\theta \in (0,1]$, as follows:
 \begin{equation}
\label{f-def}
f_{\theta}(Q) =
    \begin{cases*}
      \phi_{\theta}(Q) & for $0 \leq Q \leq Q_{\theta}$ \\
      1 & for $ Q_{\theta} < Q \leq 1$ 
    \end{cases*}
\end{equation}
Note that $f_{\theta}$ is continuous but in general not smooth at $Q = Q_{\theta}$. This is not a problem, as will become clear later. In summary, we have:

\begin{prop}
\label{prop:f}
There exists a unique continuously varying family of continuous functions $f_{\theta}: [0,1] \rightarrow [0,1]$ and numbers $Q_{\theta} \in [0,1)$ for $\theta \in (0,1]$ satisfying the following properties:
\begin{itemize}
\item $f_{\theta}(Q)$ solves (\ref{C0-ODE}) or equivalently $w_{\theta}(-f_{\theta}(Q) + 1) = 0$ for $0 \leq Q \leq Q_{\theta}$
\item $f_{\theta}(0) = \theta$
\item $f_{\theta}(Q) < 1$ for $Q < Q_{\theta}$ and $f_{\theta}(Q) = 1$ for $Q \geq Q_{\theta}$
\item $f_{\theta}(Q)$ is strictly increasing in $Q$ when $0 < Q < Q_{\theta}$
\item $f_{\theta}(Q)$ is extendable to an even function around the origin
\item $Q_{\theta}$ varies continuously with $\theta$ 
\item $Q_{\theta} \rightarrow 1 \text{ as } \theta \rightarrow 0$ and $Q_1 = 0$
\item For every $Q \in [0,1)$ we have $f_{\theta}(Q) \rightarrow 0$ as $\theta \rightarrow 0$
\end{itemize}
\end{prop}

\subsection{Non-negativity of $W_{\theta}$}
\label{subsec-mathcalQ-pos}
For the choice of $f_{\theta}$, $\theta \in (0,1]$, defined above the following proposition holds true:
\begin{prop}
\label{Qcal-pos-prop}
Let $\theta \in (0,1]$ and $f_{\theta}$ be as defined in (\ref{f-def}). Assume $(M_2, g(t))$, $t \in (-\infty,0]$, is a non-collapsed ancient Ricci flow with $g(t) \in \mathcal{I}$ for $t\in (-\infty,0]$ and $Z_{\theta} \geq 0$ everywhere. Suppose at the point $(p, t)$ in spacetime $Q(p, t) < Q_{\theta}$. Then 
\begin{equation*}
W_{\theta}(p, t) \geq 0
\end{equation*}
with equality if, and only if, 
\begin{equation*}
T_2(p,t) = 0. 
\end{equation*}
\end{prop}
We prove this proposition in multiple steps. First note
\begin{lem}
\label{A2-negative}
Let $f_{\theta}$, $\theta\in(0,1]$, be the family of functions as defined in Proposition \ref{prop:f}. Then
\begin{equation}
\label{bs2-coeff}
 A_2 = - Q^4f''-5 Q^3 f'-2 Q^2 f-2 Q^2-4 < 0 
\end{equation}
for $0 \leq Q < Q_{\theta}$. Thus $w_{\theta}(z) = A_2 z^2 + A_1 z + A_0$ is concave in $z$ whenever $0 \leq Q < Q_{\theta}$.
\end{lem}
The proof of this technical lemma can be found in subsection \ref{technical-lemmas-subsec}. Furthermore we have
\begin{lem}
\label{lem:borders}
Let $\theta \in (0,1]$. Let $(M_2, g(t))$, $t \in (-\infty, 0]$, be a non-collapsed ancient Ricci flow with $g(t) \in \mathcal{I}$ for $ t \in (-\infty, 0]$ and $Z_{\theta} \geq 0$ everywhere. Then
\begin{equation*}
- f_{\theta}(Q) + 1 \leq \frac{b_s}{Q} - Z_{\theta} \leq \min\left(1, -f_{\theta}(Q) + \frac{3}{Q^2} - 2 \right).
\end{equation*}
and
\begin{equation}
\label{min}
\min\left(1 , -f_{\theta}(Q) + \frac{3}{Q^2} - 2 \right) =
    \begin{cases*}
      1 & if $f_{\theta}(Q) \leq 3 \frac{1-Q^2}{Q^2}$ \\
      -f + \frac{3}{Q^2} - 2  & otherwise
    \end{cases*}
\end{equation}
\end{lem}
\begin{proof}
By Theorem \ref{T-ancient} and since $g(t) \in \mathcal{I}$ we know that 
\begin{align*}
y   &=b_s - Q \leq 0 \\
T_2 &= Qy - x = -a_s + Q b_s + 2\left(1-Q^2\right) \geq 0 \\
T_3 &=a_s - Qb_s + 1 - Q^2 \geq 0
\end{align*}
on $M_2 \times (-\infty, 0]$. Therefore $y \leq 0$ implies
$$ \frac{b_s}{Q} - Z_{\theta} \leq 1 - Z_{\theta}$$
and $T_2 \geq 0$ implies
$$ \frac{b_s}{Q} - Z_{\theta} = \frac{Qb_s - a_s -Q^2 +2}{Q^2} - f_{\theta}(Q) \geq 1 - f_{\theta}(Q)$$
and finally $T_3 \geq 0$ implies
$$ \frac{b_s}{Q} - Z_{\theta} = \frac{Qb_s - a_s -Q^2 +2}{Q^2} - f_{\theta}(Q) \leq - f_{\theta}(Q) + \frac{3}{Q^2} - 2.$$
Now applying the assumption $Z_{\theta} \geq 0 $ proves the desired result. 
\end{proof}

Recalling that by definition $$W_{\theta} = w_{\theta}\left( \frac{b_s}{Q} - Z_{\theta}\right),$$ the above Lemma \ref{lem:borders} and concavity of $w_{\theta}(z)$ show that to prove Proposition \ref{Qcal-pos-prop} it suffices to check that for $\theta \in (0,1]$ and $0 \leq Q < Q_{\theta}$ we have 
\begin{align*}
\alpha &:= w_{\theta}\left(- f_{\theta}(Q) + 1 \right) \geq 0, \\
\beta &:= w_{\theta}\left( 1 \right)\geq 0 \: \text{ whenever } \: f_{\theta}(Q) \leq 3 \frac{1-Q^2}{Q^2},\\
\gamma &:= w_{\theta}\left(-f_{\theta}(Q) + \frac{3}{Q^2} - 2\right) \geq 0 \: \text{ whenever } \: f_{\theta}(Q) \geq 3 \frac{1-Q^2}{Q^2},
\end{align*}
where $f_{\theta}$ is as defined in Proposition \ref{prop:f}. Note that $\gamma$ is only defined for $Q^2> 0$. This however does not pose a problem as 
$$1 \geq f_{\theta}(Q) \geq 3 \frac{1-Q^2}{Q^2}$$
implies that $Q^2 \geq \frac{3}{4}> 0$. Recall that by the properties of $f_{\theta}(Q)$ summarized in Proposition \ref{prop:f} we have 
$$ w_{\theta}\left(- f_{\theta}(Q) + 1 \right) = 0 \: \text{ for } \: 0 \leq Q < Q_{\theta},$$ 
and therefore only need to investigate the sign of $\beta$ and $\gamma$ in their respective regimes. This explains why we chose to define $f_{\theta}(Q)$ via the ordinary differential equation (\ref{C0-ODE}). In the following technical lemma, the proof of which we defer to subsection \ref{technical-lemmas-subsec}, we show that for our choice of $f_{\theta}$ the functions $\beta$ and $\gamma$ are in fact positive in their respective regimes:
\begin{lem}
\label{quadratic-positive-lem}
Fix $\theta \in (0,1]$ and let $f_{\theta}(Q)$ as defined in Proposition \ref{prop:f}. Then for $0 \leq Q < Q_{\theta}$ we have
\begin{equation*}
\beta > 0 \: \text{ whenever } \:  f_{\theta}(Q) \leq 3 \frac{1-Q^2}{Q^2}, 
\end{equation*}
and
\begin{equation*}
\gamma > 0 \: \text{ whenever } \: f_{\theta}(Q) \geq 3 \frac{1-Q^2}{Q^2}.
\end{equation*}
\end{lem}

Now we can prove Proposition \ref{Qcal-pos-prop}.
\begin{proof}[Proof of Proposition \ref{Qcal-pos-prop}]
By Lemma \ref{A2-negative} and Lemma \ref{quadratic-positive-lem} we know that $W_{\theta} \geq 0$ whenever $0 \leq Q < Q_{\theta}$, with equality if and only if $$\frac{b_s}{Q} - Z_{\theta} = 1 - f_{\theta}(Q),$$ which by the definition of $Z_{\theta}$ is equivalent to $T_2 = 0$.
\end{proof}

\begin{remark}
The proof of Proposition \ref{Qcal-pos-prop} essentially implies that for every $\theta \in (0,1]$ the inequality $Z_{\theta} \geq 0$ is preserved on Ricci flow backgrounds in $\mathcal{I}$ satisfying $T_1 \geq 0$, $T_2 \geq 0$ and $T_3 \geq 0$. We do not prove this here, as our proof of Theorem \ref{E-H-ancient-thm} does not rely on this fact.  
\end{remark}

\subsection{Proof of main theorem} 
\label{ancient-main-thm-subsec}
Next, we prove that the Eguchi-Hanson space is the unique ancient Ricci flow in the class $\mathcal{I}$.

\begin{proof}[Proof of Theorem \ref{E-H-ancient-thm}]
Recall that by Corollary \ref{cor:curv-bound-ancient} there exists a $C_1 > 0$ such that 
$$|Rm_{g(t)}|_{g(t)} \leq \frac{C_1}{b^2}.$$
Moreover, by Theorem \ref{T-ancient}
$$T_1, T_2, T_3\geq 0 \: \text{ on } \: M_2\times (\infty , 0]$$
and by Lemma \ref{lem:suc-con-start}
$$ Z_1 \geq 0 \: \text{ on } \: M_2\times (\infty , 0].$$
Hence we may assume that there exists a $\theta_0 \in (0,1]$ such that for all $\theta \in [\theta_0, 1]$
\begin{equation*}
Z_{\theta} = \frac{x}{Q^2} + f_{\theta}(Q) \geq 0 \: \text{ on } \: M_2\times (\infty , 0].
\end{equation*}
\begin{claim}
For every $0 \leq Q_\ast < 1$ we have
\begin{equation*}
\inf \left\{ Z_{\theta_0}(p,t) \: \Big | \: (p,t) \in M_2 \times (\infty, 0] \text{ such that } Q(p,t) \leq Q_\ast \right\} > 0.
\end{equation*}
\end{claim}
\begin{claimproof}
We argue by contradiction. Assume there exists a sequence of points $(p_i, t_i)$ in spacetime such that
\begin{equation*}
Q(p_i, t_i) \leq Q_\ast
\end{equation*}
and
\begin{equation*}
Z_{\theta_0}(p_i, t_i) \rightarrow 0 \: \text{ as } \: i \rightarrow \infty.
\end{equation*}
Consider the rescaled metrics
\begin{equation*}
g_i(t) = \frac{1}{b^2(p_i,t_i)} g( t_i + b^2(p_i,t_i) t), \quad t\in[-\Delta t, 0].
\end{equation*}
For sufficiently small $\Delta t > 0$ the conditions of Proposition \ref{blow-up-prop} are satisfied and therefore $(C_{g_i(0)}\left(p_i, \frac{1}{2}\right), g_i(t), p_i)$, $t\in[-\Delta t, 0]$ subsequentially converges to a Ricci flow $(\mathcal{C}_\infty, g_\infty(t), p_\infty)$, $t \in [-\Delta t, 0]$. Write
$$\Omega = \mathcal{C}_\infty \times [-\Delta t, 0].$$
By construction
\begin{equation*}
Z_{\theta_0}(p_\infty, 0) = \inf_{\Omega } Z_{\theta_0} = 0.
\end{equation*}
Now we need to distinguish two cases: 
\subsubsection*{Case 1: $Q(p_\infty,0) < Q_{\theta_0}$} Then there exists an $r\in (0,\frac{1}{2})$ and $\Delta t' \in (0, \Delta t)$ such that on the parabolic set
$$\Omega' = C_{g_\infty(0)}(p_\infty, r) \times [-\Delta t', 0] \subset \Omega$$
we have $Q < Q_{\theta_0}$. By the strong maximum principle of Theorem \ref{maximum-principle} applied to the evolution equation (\ref{Z-epsilon-evol-2}) of $Z_{\theta_0}$ we have 
\begin{equation*}
Z_{\theta_0} = 0 \: \text{ on } \: \Omega'
\end{equation*}
and therefore 
$$(Z_{\theta_0})_s = (Z_{\theta_0})_{ss} = 0 \: \text{ on } \: \Omega'.$$
By the evolution equation (\ref{Z-epsilon-evol-2}) of $Z_{\theta_0}$ we see that that
\begin{equation*}
0 = \partial_t Z_{\theta_0} = \mathcal{Q}_{\theta_0} \: \text{ in } \Omega',
\end{equation*} 
which by Proposition \ref{Qcal-pos-prop} implies
\begin{equation*}
T_2 = Qy - x =  0 \: \text{ in } \Omega'.
\end{equation*}
However, the evolution equation (\ref{T2-evol}) of $T_2$ then implies 
$$y = 0 \: \text{ on } \Omega'.$$  
and thus also 
$$x = 0 \: \text{ on } \Omega'.$$
That in turn implies
\begin{equation*}
Z_{\theta_0}(p_\infty,0) = f(Q(p_\infty,0)) \geq \theta_0 > 0,
\end{equation*}
which is a contradiction.

\subsubsection*{Case 2: $ Q(p_\infty,0) \geq Q_{\theta_0}$}
Recall that at points $(p,t)$ in spacetime satisfying $Q(p,t) \geq Q_{\theta_0}$ we have $Z_{\theta_0}(p,t) = Z_1(p,t)$. In this case we therefore have
$$Z_1(p_\infty,0) = Z_{\theta_0}(p_\infty, 0) = 0.$$
By the strong maximum principle applied to the evolution equation (\ref{Z1-evol}) of $Z_1$  and Lemma \ref{lem:suc-con-start} we deduce
$$Z_1 = 0 \: \text{ on }\Omega.$$
Furthermore, we see that this is only possible when 
$$Q = 1 \: \text{ on } \: \Omega.$$
which contradicts
$$ Q(p_\infty,0) \leq Q_\ast  < 1.$$
This concludes the proof of the claim.
\end{claimproof}

Thus for every $Q_\ast \in (0,1)$ there exists a $\delta > 0$ such that for all points $(p,t)$ in spacetime satisfying $0 \leq Q(p,t) \leq Q_\ast$ we have
\begin{equation*}
Z_{\theta_0}(p,t) > \delta.
\end{equation*}
By the continuous dependence of $Z_{\theta}$ and $Q_{\theta}$ on $\theta$, and the fact that $Z_{\theta} = Z_{\theta'}$ at points $(p,t)$ in spacetime at which $Q(p,t) \geq \max(Q_{\theta}, Q_{\theta'})$, there exists an $\theta_1 < \theta_0$ such that for $\theta \in [\theta_1, 1]$
\begin{equation*}
Z_{\theta} \geq 0 \: \text{ on } \: M_2 \times (-\infty, 0].
\end{equation*}
Now consider the set
$$ \mathcal{E} = \left \{ \theta \in (0,1] \: \big | \: Z_{\theta'} \geq 0 \text{ for } \theta \leq \theta' \leq 1 \right\} \subseteq (0,1]$$
The above argument shows that $\mathcal{E}$ is an open subset of $(0,1]$. As the condition $Z_{\theta} \geq 0$ is closed and $f_{\theta}$ depends continuously on $\theta$, it follows that $\mathcal{E}$ is also a closed subset of $(0,1]$. Hence by connectedness of $(0,1]$ it follows that $\mathcal{E} = (0,1]$ and thus for all $\theta \in (0,1]$
\begin{equation*}
Z_{\theta} \geq 0 \: \text{ on } \: M_2 \times (-\infty, 0].
\end{equation*}
Note that by the strong maximum principle applied to the evolution equation (\ref{Q-evol}) of $Q$ 
$$Q < 1 \: \text{ on } \: M_2 \times (-\infty, 0],$$
as otherwise $Q=1$ everywhere, which is not true. As $Z_{\theta} = \frac{x}{Q^2} + f_{\theta}(Q)$ and by Proposition \ref{prop:f} for every $0 \leq Q < 1$ we have $f_{\theta}(Q) \rightarrow 0$ as $\theta \rightarrow 0$ it follows that 
$$x \geq 0 \: \text{ on } \: M_2 \times (-\infty, 0].$$ 
However, as 
$$T_2 = Qy - x \geq 0 \: \text{ on } \: M_2 \times (-\infty, 0]$$ 
and $y \leq 0$ by the assumption that $g(t) \in \mathcal{I}$ it follows that 
$$ x = y = 0 \: \text{ on } \: M_2 \times (-\infty, 0].$$
By Lemma \ref{unique-Kahler-lem} we conclude that $(M_2, g(t)), t\in (-\infty, 0],$ is stationary and homothetic to the Eguchi-Hanson space.
\end{proof}

Now we prove Corollary \ref{E-H-blowup}.
\begin{proof}[Proof of Corollary \ref{E-H-blowup}]
By Theorem \ref{curv-bound} and the fact that $b_s \leq 0$ for metrics in $\mathcal{I}$ there exists a $C_1 > 0$ such that
\begin{equation*}
|Rm_{g(t)}|_{g(t)}(p) \leq \frac{C_1}{b^2(p,t)} \leq \frac{C_1}{b^2(o,t)}.
\end{equation*}
This shows that 
$$b(o,t) \rightarrow 0 \text{ as } t \rightarrow T_{sing}.$$
As by assumption 
$$ C \coloneqq \sup \frac{b(p_i, t_i)}{b(o,t_i)} < \infty $$
and $y \leq 0$  by the fact that $g(t) \in \mathcal{I}$ it follows that 
\begin{equation*}
\partial_t b(o,t) \leq 0, \quad t \in [0,T_{sing}),
\end{equation*}
by (\ref{b0-evol-k2}). We deduce
$$b(p_i,t_i) \rightarrow 0 \text{ as } i \rightarrow \infty.$$
Consider the rescaled metrics 
$$g_i(t) = \frac{1}{b^2(p_i,t_i)} g\left(t_i + t b^2(p_i, t_i) \right), \quad t\in[-t_i b^{-2}(p_i,t_i), 0].$$ 
These satisfy the curvature bound
\begin{equation*}
|Rm_{g_i(t)}|_{g_i(t)} \leq C^2 C_1 \: \text{ on } \: M_2 \times [-t_i b^{-2}(p_i,t_i), 0].
\end{equation*}
By Theorem \ref{thm:no-local-collapsing} the rescaled metrics $g_i(t)$ are $\kappa$-non-collapsed at larger and larger scales. Hence by Corollary \ref{cor:compactness-complete-flows} the Ricci flows $(M_2, g_i(t), p_i)$ subsequentially converge to a pointed ancient Ricci flow $(M_2, g_{\infty}(t), p_\infty)$, $ - \infty< t < 0$, with $g_\infty(t) \in \mathcal{I}$. By Theorem \ref{E-H-ancient-thm} it follows that $g_\infty(t)$ is stationary and homothetic to the Eguchi-Hanson metric. 
\end{proof}

\subsection{Proof of technical lemmas}
\label{technical-lemmas-subsec}
In this subsection we collect the proofs of the technical lemmas we relied on above.

\begin{proof}[Proof of Lemma \ref{f-existence-lem}]
We apply \cite[Theorem 9.2]{A17} to prove this lemma. Define
$$r = Q^2.$$
Then 
\begin{align*}
f' &= 2 Q f_r \\
f'' &= 2 f_r + 4 r f_{rr},
\end{align*}
where $'$ denotes the derivative with respect to $Q$ and subscript $r$ denotes the derivative with respect to $r$. Rewriting the differential equation (\ref{C0-ODE}) with respect to the independent variable $r$, we obtain
\begin{equation}
\label{f-in-r-eqn}
r f_{rr} = \frac{1}{2(1-r)} \left( 6r - 4 - r f\right)f_r + \frac{f}{8 \left( 1- r\right)^2 }\left(f^2 r + 3 fr - 6f - 6r + 8 \right) 
\end{equation}
At $r=0$ the right hand side must equal zero, which can be ensured by requiring
$$ f_r (0) = \frac{1}{2} f_0 - \frac{3}{8} f_0^2 $$
Now define 
\begin{align*}
u_1 &= f - f_0 \\
u_2 &= f_r - f_r(0).
\end{align*}
Then (\ref{f-in-r-eqn}) can be written as a system of equations of the form
$$ r (u_i)_r = P_i(\vec{u},r,f_0), \quad i = 1, 2,$$
where
\begin{align*} 
P: \R^2 \times (-1,1) \times \R &\longrightarrow \R^2 \\
    (\vec{u}, r, f_0) &\longrightarrow P(\vec{u}, r, f_0) 
\end{align*}
is an analytic vector-valued function of several variables satisfying
$$P(\vec{0}, 0, f_0) = 0$$
for all $f_0 \in (-1, 1)$. A computation shows that
$$ \frac{\partial P} {\partial u} \Big |_{(\vec{0},0, f_0)} = \begin{pmatrix}  0 & 0 \\ 1- \frac{3}{2} f_0 &  -2 \end{pmatrix} $$
This matrix has no positive integer eigenvalues and furthermore
$$
B = \sup_{\substack{m\in \N \\ f_0 \in \R}} \norm{ \left( m I_2 - \frac{\partial P} {\partial u} \Big |_{(\vec{0},0, f_0)}\right)^{-1}} < \infty,
$$
where $I_2$ is the $2\times2$ identity matrix. By \cite[Theorem 9.2]{A17} the desired result follows.
\end{proof}

\begin{proof}[Proof of Lemma \ref{f-inc-lem}]
At $Q=0$, we have by L'H\^opital's Rule that
\begin{equation}
\label{ddf-origin}
f'' = \frac{1}{4} f \left( 4 - 3 f \right) > 0 \: \text{ for } 0 < f(0) < \frac{4}{3}.
\end{equation}
Furthermore, at an extremum of $f$ we have 
\begin{equation}
\label{df-extremum}
f'' = \frac{2 f}{4\left(1-Q^2\right)^2}\left( f^2 Q^2+3 f Q^2-6 f-6 Q^2+8 \right).
\end{equation}
Defining the polynomial
\begin{equation*}
p_1(f,Q^2) = f^2 Q^2+3 f Q^2-6 f-6 Q^2+8
\end{equation*}
we see that
\begin{equation*}
\partial_{Q^2} p_1 = f^2 + 3f - 6 < 0\: \text{ for } 0 < f \leq 1.
\end{equation*}
Therefore
\begin{equation*}
p_1(f, Q) > p_1(f, 1) = f^2 - 3f + 2 \geq 0 \: \text{ for } 0 < f \leq 1 \text{ and } 0 \leq Q < 1.
\end{equation*}
From (\ref{df-extremum}) it then follows that $f'>0$ for as long as $0 < f \leq 1$.
\end{proof}

\begin{proof}[Proof of Lemma \ref{f-lem-2}]
We argue by contradiction. Assume there does not exist such a $Q_{\theta} < 1$. Then by Lemma \ref{f-inc-lem} we have $f' >0$ on $Q \in (0,1)$ and hence 
\begin{equation*}
\lim_{Q \rightarrow 1^-} f(Q) = l \leq 1 
\end{equation*}
exists. By (\ref{C0-ODE}) we have
\begin{align*}
4 \left(1-Q^2\right)^2 f'' = &- 4 \left(1-Q^2\right) \left(Q^2 f-5 Q^2+3\right) \frac{f'}{Q} \\ \nonumber
  & \qquad + 2f\left( Q^2\left(f^2+3 f-6\right)+8-6 f \right).
\end{align*}
For $Q^2 > 1 - \frac{\theta}{4}$ and $0 < f < 1$ we have
\begin{equation*}
Q^2 f-5 Q^2+3 < 3 - 4 Q^2 < -1 + \theta
\end{equation*}
and
\begin{equation*}
Q^2\left(f^2+3 f-6\right)+8-6 f > (2-f)(1-f),
\end{equation*}
as for $0<f<1$
\begin{equation*}
f^2+3 f-6 < 0.
\end{equation*}
Hence for $ Q^2 > 1 - \frac{\theta}{4}$ and $0< f< 1$ we obtain the following inequality
\begin{align}
\label{ddf-ODI}
f'' &\geq \alpha \frac{f'}{1-Q} + \beta \frac{1-f}{(1-Q)^2},
\end{align}
where
\begin{align*}
\alpha &= \frac{1-\theta}{Q(1+Q)}\\
\beta &= \frac{f\left(2 -  f\right)}{2 \left(1+Q\right)^2}.
\end{align*}
Furthermore we observe that
\begin{align*}
\alpha &\rightarrow \frac{1}{2} \: \text{ as } \: f,Q \rightarrow 1 \\
\beta &\rightarrow \frac{1}{8} \: \text{ as } \: f,Q \rightarrow 1.
\end{align*}
If $l < 1$, then there would exists a $Q_\ast<1$ such that for $Q \geq Q_\ast$ we have
\begin{equation*}
f'' \geq \frac{1}{10} \frac{l(1-l)(2-l)}{(1-Q)^2}.
\end{equation*}
Here $\frac{1}{10}$ can be replaced by any positive number less that $\frac{1}{8}$. However, integrating this differential inequality shows that in this case $f$ would reach $1$ before $Q = 1$, leading to a contradiction of our assumption. Therefore we may assume that $l =1$. 

Defining
\begin{equation*}
g(Q) = 1 - f(Q)
\end{equation*}
we obtain the differential inequality
\begin{equation}
\label{ODI-g}
g''(Q) \leq \alpha \frac{g'(Q)}{1-Q} - \beta \frac{g(Q)}{(1-Q)^2}.
\end{equation} 
By Lemma \ref{f-inc-lem} we know that $g(Q) > 0$ and $g'(Q) < 0$ on $Q \in (0,1)$. 

\begin{claim}
The function $g(Q)$ reaches zero before $Q = 1$.
\end{claim}
 
\begin{claimproof}
By our assumption that $l=1$ we know that there exists a $Q_\ast < 1$ such that for $Q > Q_\ast$
\begin{equation*}
g(Q) < \theta.
\end{equation*}  
Furthermore, by choosing $Q_\ast < 1$ sufficiently close to 1, we may assume that for $Q \geq Q_\ast$
\begin{equation}
\label{ODI-g-2}
g''(Q) \leq \frac{3}{7} \frac{g'(Q)}{1-Q} - \frac{5}{49} \frac{g(Q)}{(1-Q)^2},
\end{equation}
as $ \frac{3}{7} < \frac{1}{2}$ and $\frac{5}{49} < \frac{1}{8}$. Now take the substitution
$$g(Q) = u\left( r \right)$$
for 
$$r = -\ln(1- Q).$$
Then the (\ref{ODI-g-2}) becomes
$$\frac{d^2u}{dr^2} + \frac{4}{7} \frac{du}{dr} + \frac{5}{49} u \leq 0.$$
The corresponding ordinary differential equation is of oscillatory type, which motivates the substitution
$$ u(r) = e^{-\frac{2}{7} r} v(r)$$
yielding the inequality
$$ \frac{d^2v}{dr^2} \leq  - \frac{1}{49} v.$$
Hence $v$ reaches $0$ in finite $r$, which tracing back the substitutions, shows that $g$ must reach zero before $Q = 1$. 
\end{claimproof}

Now it remains to prove the assertion (1), (2) and (3). We prove (1). First fix a $\theta \in (0,1)$. By Lemma \ref{f-inc-lem} we know that $f'_\theta(Q_{\theta}) > 0$. Now extend the solution $f_\theta$ of (\ref{C0-ODE}) to the interval $[0, Q_{\theta} + c]$, $c>0$, such that $f'_{\theta}(Q) > 0$ on $(0, Q_{\theta} + c]$. By the continuous dependence of $f_{\theta}(Q)$ on $\theta$ it follows that for every $\epsilon > 0$ there exists a $\delta > 0$ such that $|\theta - \theta'| \leq \delta$ implies $|Q_\theta - Q_{\theta'}| < \epsilon$. To prove the continuity of $Q_\theta$ at $\theta = 1$, note that $Q_1 = 1$, $f_1(0) = 1$ and $f'_1=0$. Then recall from the proof of Lemma \ref{f-inc-lem} that $f''(0) > 0$ when $0 < f(0) < \frac{4}{3}$. Now applying the same argument as above yields continuity of $Q_\theta$ at $\theta = 1$ and therefore proves (1). 

Assertion (2) follows from the fact that for the initial condition $f(0) = 0$ the corresponding solution to the ODE (\ref{C0-ODE}) is $f(Q) = 0$. By the continuous dependence of $f$ on $f(0) = 0$ we deduce that   
\begin{equation*}
Q_{\theta} \rightarrow 1 \text{ as } \theta \rightarrow 0.
\end{equation*}

Finally, assertion (3) follows by definition.
\end{proof}

\begin{proof}[Proof of Lemma \ref{A2-negative}]
First note that by Lemma \ref{f-inc-lem} we know that $f,f' \geq 0$ on $[0, Q_{\theta}]$. Solving the ODE (\ref{C0-ODE}) for $f''$, we obtain
\begin{align}
\label{ddf-exp}
f'' &= \frac{1}{2 Q \left(1-Q^2\right)^2} \Big( 2 Q^4 f f'-10 Q^4 f'-2 Q^2 f f'+16 Q^2 f' \\ \nonumber
    &\quad  -6 f' + Q^3 f^3+3 Q^3 f^2-6 Q^3 f-6 Q f^2+8 Q f \Big)
\end{align}
Substituting expression (\ref{ddf-exp}) into (\ref{bs2-coeff}) yields 
\begin{align*}
A_2&= - \frac{1}{2 \left(1-Q^2\right)^2} \Big ( 2 Q^3 \left(1-Q^2\right) \left( 2 - Q^2 f \right) f' \\ \nonumber 
					& \qquad + f^3 Q^6+3 f^2 Q^6-6 f^2 Q^4-2 f Q^6+4 f Q^2+4 Q^6-12 Q^2+8 \Big  )
\end{align*}
Defining
\begin{equation*}
p_2 (f, Q^2) = f^3 Q^6+3 f^2 Q^6-6 f^2 Q^4-2 f Q^6+4 f Q^2+4 Q^6-12 Q^2+8
\end{equation*}
we then only need to check that
\begin{equation*}
p_2 \geq 0 \: \text{ for } \: 0 \leq f, Q \leq 1.
\end{equation*}
Defining
\begin{equation*}
\tilde{p}_2 (F, Q^2) = p_2(f, Q^2)
\end{equation*}
for 
\begin{equation*}
F = f Q^2
\end{equation*}
we see that 
\begin{equation*}
\partial_{Q^2} \tilde{p}_2 \Big |_F = 3 F^2 - 4FQ^2 + 12 \left( Q^4 -1 \right) \leq 0 \: \text{ for } \: 0\leq F, Q \leq 1 
\end{equation*}
with equality only at $F= 0$, $Q = 1$. Therefore the minimum of $p_2$ is attained when $Q = 1$, in which case we have we have
\begin{equation*}
\tilde{p}_2(F, 1) = \left(F-2\right)  \left(F-1\right) F \geq 0 \: \text{ for } \: 0\leq F \leq 1.
\end{equation*}
As $ 0 \leq Q \leq Q_{\theta} < 1$, we actually have $p_2(f, Q^2) > 0$ on $(f,Q) \in [0,1] \times [0, Q_\theta]$ and the result follows.
\end{proof}

\begin{proof}[Proof of Lemma \ref{quadratic-positive-lem}]
A computation shows that
\begin{align}
\label{Cupper1}
\beta &= -\left(Q^2f+2 Q^2-2\right)^2 f'' \\ \nonumber
             &  + \left(-3 Q^4 f^2-14 Q^4 f+12 Q^2 f-20 Q^4+32 Q^2-12\right) \frac{f'}{Q} \\ \nonumber
             & -4 Q^2 f^2-12 Q^2 f+16 f
\end{align}
and
\begin{align}
\label{Cupper2}
\gamma &= -\left(1-Q^2\right)^2 f'' + (1- Q^2)(2 Q^2 f+11 Q^2-9)\frac{f'}{Q} \\ \nonumber
             & + 2 Q^2 f^3+6 Q^2 f^2-12 f^2-6 Q^2 f+\frac{30 f}{Q^2}-20 f-18 Q^2+\frac{54}{Q^2}-\frac{36}{Q^4},
\end{align}
where we omitted the dependence of $f$ on $\theta$ and $Q$ for brevity. We first show that $\beta > 0$ in the region 
\begin{equation*}
R_1 = \left \{(f,Q) \: \Big | \: 0 \leq Q \leq Q_{\theta}, 0 < f \leq \min\left(1, 3 \frac{1-Q^2}{Q^2}\right) \right\}
\end{equation*}
of the $f$-$Q$-plane. Plugging the expression (\ref{ddf-exp}) of $f''$ into the expression (\ref{Cupper1}) for $\beta$, we obtain
\begin{align*}
2 \left(1-Q^2 \right)^2 \beta &= 2 ff' Q(1 - Q^2)\left( Q^4 f^2+2 Q^4 f-4 Q^2 f-2 Q^4-2 Q^2+4 \right) \\ \nonumber
  							 & +  f^2 \Big ( -Q^6f^3-7 Q^6 f^2+10 Q^4 f^2-10 Q^6 f+36 Q^4 f \\ \nonumber
  							 & \quad\quad\quad\qquad -28 Q^2 f+4 Q^6+8 Q^4-36 Q^2+24 \Big)
\end{align*}
An application of L'H\^opital's Rule shows that $\beta = 12 f^2 > 0$ at $Q=0$ and therefore we may assume that $Q>0$. Recall that $f,f' > 0$ on $(0,Q_{\theta})$ by Lemma \ref{f-inc-lem}. Hence it suffices to show that the polynomials
\begin{equation*}
p_3(f,Q^2) = Q^4 f^2+2 Q^4 f-4 Q^2 f-2 Q^4-2 Q^2+4
\end{equation*}
and
\begin{align*}
p_4(f,Q^2)&= -Q^6f^3-7 Q^6 f^2+10 Q^4 f^2-10 Q^6 f+36 Q^4 f \\ \nonumber
          & \quad\quad-28 Q^2 f+4 Q^6+8 Q^4-36 Q^2+24
\end{align*}
are positive on $R_1 \cap \{ Q  > 0 \}$. A computation shows
\begin{equation*}
\partial_{Q^2} p_3 = 2\left(f^2 Q^2 - 1 \right) + 4f \left(Q^2 -1 \right) - 4 Q^2 \leq 0
\end{equation*}
and  hence for every $(f,Q) \in R_1$ we have
\begin{equation*}
p_3(f, Q^2) \geq p_3\left(f, \frac{3}{3 +f}\right) = \left(\frac{f}{3+f}\right)^2 > 0.
\end{equation*}
To show that $p_4 > 0$ on $R_1$ is more complicated. For this we introduce the variable 
\begin{equation*}
F= fQ^2
\end{equation*}
and polynomial
\begin{equation*}
\tilde{p}_4 (F, Q^2) = p_4( f, Q^2)
\end{equation*}
Then
\begin{align*}
\tilde{p}_4(F, Q^2) &= 24 -28 F+10 F^2 -F^3+\left(-7 F^2+36 F-36\right) Q^2 \\ \nonumber
            & \quad +(8-10 F) Q^4+4 Q^6
\end{align*}
which gives
$$\partial_{Q^2} \tilde{p}_4 = -7 F^2 - 4 F \left(-9 + 5 Q^2\right) + 4 \left(-9 + 4 Q^2 + 3 Q^4\right).$$
As this expression is concave in $F$ one can easily check that in the region
$$ 0 < F \leq \min\left( Q^2, 3\left(1-Q^2\right) \right),  0 \leq Q \leq 1$$
of the $Q^2$-$F$-plane we have
$$\partial_{Q^2} \tilde{p}_4 \leq 0$$
and thus
\begin{equation*}
\tilde{p}_4(F, Q^2) \geq \tilde{p}_4\left( F, \frac{3-F}{3}\right) = \frac{1}{27} F \left(2 F^2-3 F+18\right) > 0.
\end{equation*}
From this we conclude that $p_4 > 0$ on $R_1\cap\{Q > 0 \}$ and hence $\beta > 0$ on $R_1$.

We adopt the same procedure to show that $\gamma > 0$ in the region
\begin{equation*}
R_2 = \left \{(f,Q) \: \Big | \: 0 \leq Q \leq Q_{\theta}, 3 \frac{1-Q^2}{Q^2} \leq f \leq 1 \right\}.
\end{equation*}
Substituting the expression (\ref{ddf-exp}) for $f''$ into the expression (\ref{Cupper2}) for $\gamma$ we obtain
\begin{align*}
\gamma &=  3 \left(1-Q^2\right) \left(Q^2 f+2 Q^2-2\right) \frac{f'}{Q} +\frac{1}{Q^4} \Big( \frac{3 f^3 Q^6}{2}+\frac{9 f^2 Q^6}{2}-9 f^2 Q^4 \\ \nonumber
  & \quad -3 f Q^6-24 f Q^4+30 f Q^2-18 Q^6+54 Q^2-36\Big)
\end{align*}
First notice that for any point $(f,Q) \in R_2$
$$ 3\frac{1- Q^2}{Q^2} \leq 1$$
and hence
$$ \sqrt{\frac{3}{4}} \leq Q \leq Q_{\theta} < 1.$$
Then from
\begin{equation*}
f \geq 3 \frac{1-Q^2}{Q^2}
\end{equation*}
it follows that
\begin{equation*}
Q^2 f + 2 Q^2 - 2 \geq 1- Q^2 > 0 \: \text{ on } \: R_2.
\end{equation*}
Therefore the first term in the expression for $\gamma$ is positive and we only need to prove non-negativity of the second term. For this define the polynomials
\begin{align*}
p_5(f, Q^2) &= \frac{3 f^3 Q^6}{2}+\frac{9 f^2 Q^6}{2}-9 f^2 Q^4-3 f Q^6-24 f Q^4 \\ \nonumber
  & \quad +30 f Q^2-18 Q^6+54 Q^2-36 
\end{align*}
and
\begin{align*}
\tilde{p_5}(F,Q^2) &= \frac{3 F^3}{2}+\frac{9 F^2 Q^2}{2}-9 F^2-3 F Q^4-24 F Q^2 \\ \nonumber 
            & \quad +30 F-18 Q^6+54 Q^2-36,
\end{align*}
where we again took $F = f Q^2$. Computing the partial derivatives
\begin{align*}
\partial_{Q^2} \tilde{p}_5 &= \frac{9 F^2}{2}-6 F Q^2-24 F-54 Q^4+54 \\
\partial_{F} \tilde{p}_5 &= \frac{9 F^2}{2}+9 F Q^2-18 F-3 Q^4-24 Q^2+30
\end{align*}
We deduce that at an local extrema $\partial_{Q^2} \tilde{p}_3 = \partial_{F} \tilde{p}_4=0$ 
\begin{equation*}
F = \frac{-17 Q^4+8 Q^2+8}{5 Q^2+2}
\end{equation*}
and 
\begin{equation*}
80 + 144 Q^2 - 188 Q^4 - 200 Q^6 + 307 Q^8 = 0.
\end{equation*}
In Lemma \ref{lem:Q-pol-zeros} below we show that the equation for $Q^2$ has no zeros in the interval $Q^2 \in [\frac{3}{4}, 1]$. Therefore $p_5(F,Q^2)$ has no local extrema in the region $R_3$ of the $(F,Q^2)$-plane enclosed by the curves
\begin{align*}
L_1&: Q^2 = 1, 0\leq F\leq1 \\
L_2&: \frac{2}{3} \leq Q^2 \leq 1,  0\leq F\leq1 \\
L_3&: \frac{2}{3} \leq Q^2 \leq 1,  F = 3 \left(1-Q^2\right)
\end{align*}
As the set of the $(F,Q^2)$-plane corresponding to $R_2$ is a subset of $R_3$ it suffices to check non-negativity of $\tilde{p}_5$ on the boundary of the region $R_3$. There we have
\begin{equation*}
\tilde{p}_5(F, 1) = \frac{3}{2} F\left(1 - F\right) \left(2-F\right) \geq 0 \quad \text{ on $L_1$}
\end{equation*}
and 
\begin{equation*}
\tilde{p}_5(1, Q^2) = \frac{3}{2} (1-Q^2) \left(12 Q^4+14 Q^2-9\right) \geq 0 \quad \text{ on $L_2$}
\end{equation*}
and
\begin{equation*}
\tilde{p}_5\left(3(1-Q^2), Q^2\right) = \frac{9}{2} \left(1-Q^2\right) \left(2 Q^4-3 Q^2+3\right) \geq 0 \quad \text{ on $L_3$}
\end{equation*}
This concludes the proof.
\end{proof}

\begin{lem}
\label{lem:Q-pol-zeros}
The equation $$80 + 144 r - 188 r^2 - 200 r^3 + 307 r^4 = 0$$ has no roots in the interval $[0,1]$.
\end{lem}
\begin{proof}
For $r \in [0,1]$ we have
\begin{align*}
80 + 144 r - 188 r^2 - 200 r^3 + 307 r^4 &\geq  (80 - 6r) + 150r - 200r^2 - 200 r^3 + 300 r^4 \\
										 &\geq  24 + 50\left(1 + 3r - 4r^2 - 4r^3 + 6r^4 \right).
\end{align*}
Then we see that
\begin{align*}
1 + 3r - 4r^2 - 4r^3 + 6r^4 &= \left(1 - 2r^2\right)^2 +r \left( 2r^3 - 4r^2 + 3 \right) \\
							&\geq \left(1 - 2r^2\right)^2 +r \left( 2r^4 - 4r^2 + 3 \right) \\
							&\geq \left(1 - 2r^2\right)^2 +r \left( 2(r^2 -1)^2 + 1\right) \\
							&\geq 0
\end{align*}
This concludes the proof.
\end{proof}

\section{Discussion of blow-up limits in $k = 2$ case}
\label{sec:blow-up}
In this section we investigate the possible blow-up limits of a Ricci flow $(M_2, g(t))$, $t\in[0, T_{sing})$, starting from an initial metric $g(0) \in \mathcal{I}$ with $\sup_{p \in M_2} b(p,0) < \infty$. By Lemma \ref{sing-time-finite} and Corollary \ref{E-H-blowup} we know that such flows develop a Type II singularity modeled on the Eguchi-Hanson space as the area of the non-principal orbit $S^2_o$ shrinks to zero. One expects, however, that at larger distance scales from $S^2_o$ one could also see other blow-up limits. The goal of this section is to show that these are in fact limited to the following four possibilities: (i) the Eguchi-Hanson space, (ii) the flat $\R^4 / \mathbb{Z}_2$ orbifold, (iii) the 4d Bryant soliton quotiented by $\Z_2$ and (iv) the shrinking cylinder $\R \times \R P^3$. 

Next, we state the main theorem of this section:
\begin{thm}[Blow-up limits]
\label{blow-up-thm}
Let $(M_2, g(t))$, $t \in [0, T_{sing})$, be a Ricci flow starting from an initial metric $g(0) \in \mathcal{I}$ (see Definition \ref{def:I}) with $\sup_{p \in M_2} b(p,0) < \infty$. Let $(p_i,t_i)$ be a sequence of points in spacetime with $b(p_i, t_i) \rightarrow 0$. Passing to a subsequence, we may assume that we are in one of the four cases listed below. 
\begin{enumerate}[label=(\roman*)]
\item $\lim_{i \rightarrow \infty} \frac{b(p_i,t_i)}{b(o,t_i)} < \infty$
\item $\lim_{i \rightarrow \infty} \frac{b(p_i,t_i)}{b(o,t_i)} = \infty$ and $\lim_{i \rightarrow \infty} b_s(p_i, t_i) = 1$
\item $\lim_{i \rightarrow \infty} \frac{b(p_i,t_i)}{b(o,t_i)} = \infty$ and $\lim_{i \rightarrow \infty} b_s(p_i, t_i) \in (0,1)$
\item $\lim_{i \rightarrow \infty} \frac{b(p_i,t_i)}{b(o,t_i)} = \infty$ and $\lim_{i \rightarrow \infty} b_s(p_i, t_i) = 0$
\end{enumerate}
Consider the dilated Ricci flows 
$$g_i (t) = \frac{1}{b^2(p_i,t_i)} g\left(t_i + b^2(p_i, t_i) t \right), \quad t \in [- b(p_i,t_i)^{-2} t_i , 0].$$ 
Then $(M_2, g_i(t), p_i)$, $t \in [- b(p_i,t_i)^{-2} t_i , 0]$, subsequentially converges, in the Cheeger-Gromov sense, to an ancient Ricci flow $(M_\infty, g_\infty(t), p_\infty)$, $t \in (-\infty, 0]$. Depending on the limiting property of the sequence $(p_i, t_i)$ we have:
\begin{enumerate}[label=(\roman*)]
	\item $M_\infty \cong M_2$ and $g_\infty(t)$ is stationary and homothetic to the Eguchi-Hanson metric
	\item $M_\infty \cong \R^4\setminus\{0\} / \mathbb{Z}_2$ and $g_\infty(t)$ can be extended to a smooth orbifold Ricci flow on $\R^4/\Z_2$ that is stationary and isometric to the flat orbifold $\R^4/\Z_2$
	\item $M_\infty \cong \R^4\setminus\{0\} / \mathbb{Z}_2$ and $g_\infty(t)$ can be extended to a smooth orbifold Ricci flow on $\R^4/\Z_2$ that is homothetic to the 4d Bryant soliton quotiented by $\Z_2$
	\item $M_\infty \cong \R \times \R P^3$ and $g_\infty(t)$ is homothetic to a shrinking cylinder
\end{enumerate}
\end{thm}

\begin{remark}
Notice that in Theorem \ref{blow-up-thm} we do not claim that all blow-up limits (i)-(iv) actually occur. If the Eguchi-Hanson singularity is isolated one would only see (i) and (ii).
\end{remark}

\subsection{Outline of proof}
\label{blow-up-thm-outline}
Assume we are given a sequence of points $(p_i, t_i)$ in spacetime with $b(p_i, t_i) \rightarrow 0$. Consider the rescaled metrics
$$g_i(t) = \frac{1}{b(p_i,t_i)^2} g( t_i + b(p_i, t_i)^2 t), \quad t \in [- b(p_i, t_i)^{-2} t_i, 0],$$
normalized such that $b(p_i, 0) = 1$. By passing to a subsequence we may assume that either 
$$\text{ (I) } \sup_i \frac{b(p_i,t_i)}{b(o,t_i)} < \infty \quad \text{or} \quad \text{ (II) } \lim_{i \rightarrow \infty} \frac{b(p_i,t_i)}{b(o,t_i)} = \infty.$$ 
In case (I) we know by Corollary \ref{E-H-blowup} that $(M_2, g_i(t), p_i)$ subsequentially converges to the Eguchi-Hanson space, which is the blow-up limit (i) from above. Therefore we only need to investigate the behavior in case (II), i.e. at scales larger than the forming Eguchi-Hanson singularity. At these scales Lemma \ref{QT2bddb-limit-lem} yields very important geometric information. In particular, we show that for every $\epsilon > 0$ there exist constants $C,\delta > 0$ such that the following holds: For all points $(p,t)$ in spacetime at which $C b(o,t) \leq b(p,t) \leq \delta$ we have
\begin{itemize}
\item $Q \geq 1 - \epsilon$
\item $T_{F_1} \coloneqq b b_{ss} + 1- b_s^2 \geq -\epsilon$
\item $T_{F_2} \coloneqq b b_{ss} + 1- b_s^2 - \left(1-b_s^2\right)^2 \leq \epsilon$
\item $\partial_t b^2 \leq \epsilon$
\end{itemize}
Hence a blow-up limit $(M_\infty, g_\infty(t), p_\infty)$ in case (II) satisfies $Q = 1$, $T_{F_1} \geq 0$ and $T_{F_2} \leq 0$. Therefore $M_\infty$ is rotationally symmetric and satisfies
$$\frac{1 - b^2_s}{b^2} - \frac{(1-b_s^2)^2}{b^2} \leq -\frac{b_{ss}}{b} \leq \frac{1 - b_s^2}{b^2}.$$
As $-\frac{b_{ss}}{b}$ and $\frac{1-b_s^2}{b^2}$ are the only non-zero components of the curvature tensor of a rotationally symmetric metric, we see that blow-up limits of case (II) satisfy the curvature bound
$$|Rm_{g_\infty(t)}|_{g_\infty(t)} \leq c \frac{1-b_s^2}{b^2}$$
for some universal constant $c>0$. 

We now briefly explain some of the geometric intuition behind the quantities $T_{F_1}$ and $T_{F_2}$ for rotationally symmetric metrics. When $T_{F_1} = 0$ the underlying space is of constant curvature and therefore isometric to a sphere, the flat plane or hyperbolic space, depending on the sign of the scalar curvature.  On the other hand solving the ODE $T_{F_2} = 0$ one can show that $b_s \rightarrow 0$ as $s \rightarrow \infty$ and the underlying space is asymptotically cylindrical. Thus blow-up limits in case (II) are rotationally symmetric spaces that are `sandwiched' between a sphere and an asymptotically cylindrical space.

We need to divide case (II) into three subcases in order to investigate the possible blow-up limits: By passing to a subsequence we may assume that
$$ \text{(II.a) } b_s(p_i, t_i) \rightarrow 1 \: \text{ or } \:  \text{ (II.b) } b_s(p_i,t_i) \rightarrow \eta \in (0,1) \: \text{ or } \: \text{ (II.c) } b_s(p_i,t_i) \rightarrow 0.$$ 
For (II.a) and (II.c) we show in Lemma \ref{flat-orbifold-blow-up-lem} and Lemma \ref{cylinder-blow-up-lem} that $(M_2, g_i(t), p_i)$ subsequentially converges to the flat orbifold $\R^4 / \Z_2$ and the shrinking cylinder $\R \times \R P^3$, respectively. The main idea is that by the strong maximum principle applied to the evolution equation (\ref{db-evol}) of $b_s$ when $Q=1$ a minimum $b_s = 0$ or a maximum $b_s = 1$ can only be attained if $b_s$ is constant everywhere.

Proving that the blow-up limit in case (II.b) is an ancient orbifold Ricci flow, which is homothetic to the 4d Bryant soliton quotiented by $\Z_2$, is trickier. The construction is carried out in Lemma \ref{lem:orbifold-blowup}, the proof of which we sketch here: Fix a $T > 0$ and define 
$$E_{p,t,n} \coloneqq \left\{p' \in M_2 \: \Big | \: b(p',t) > \frac{b(p,t)}{n} \right\} \subseteq M_2.$$
Then consider the rescaled metrics $g_i(t)$, defined as above, on the parabolic neighborhoods 
$$\Omega_{i,n} \coloneqq E_{p_i, t_i, n} \times [-T - 1, 0]$$
in spacetime. By Lemma \ref{QT2bddb-limit-lem} we know that $\partial_t b^2 \rightarrow 0$ uniformly as $b \rightarrow 0$. Hence from the curvature bound 
\begin{equation}
\label{curv-bound-intro}
|Rm_{g(t)}|_{g(t)} \leq \frac{C_1}{b^{2}}
\end{equation}
of Theorem \ref{curv-bound}, we see that the curvature of $g_i(t)$ is bounded by $Cn^2$ on $\Omega_{i,n}$ for $i$ sufficiently large and $C>0$ some constant. The difficulty in constructing the limiting orbifold flow arises from the fact that the curvature bound (\ref{curv-bound-intro}) degenerates as $n \rightarrow \infty$. We get around this by exploiting the inequalities on $T_{F_1}$ and $T_{F_2}$ derived in Lemma \ref{QT2bddb-limit-lem}, to find a uniform curvature bound independent of $n$. From here it is then easy to construct the orbifold Ricci flow $g_\infty(t)$, $t \in [-T, 0]$, on $\R^4 \setminus \{0\} /\Z_2$ by taking the limit $n\rightarrow \infty$. Via Lemma \ref{lem:C11-regularity}, and Theorem \ref{RF-removable-singularity} in the Appendix B, we show that $g_\infty(t)$ can be extended to a smooth orbifold Ricci flow on $\R^4/\Z_2$. Apriori the curvature bound of $g_\infty(t)$ on $\R^4 / \Z_2 \times [0,T]$ may deteriorate as $T \rightarrow \infty$. Nevertheless we can use a diagonal argument to construct an ancient orbifold Ricci flow on $\R^4 /\Z_2$. Hamilton's trace Harnack inequality then implies that $g_\infty(t)$ has bounded curvature on $\R^4/\Z^2 \times (-\infty,0]$. Finally, we apply the result \cite{LZ18} by Xiaolong Li and Yongjia Zhang to deduce that $g_\infty(t)$ is homothetic to the 4d Bryant soliton quotiented by $\Z_2$.

\subsection{Proof of main theorem}

We begin by proving the central lemma of this section, which yields important geometric information on the high curvature regions of a Ricci flow $(M_2, g(t))$ as in the Theorem \ref{blow-up-thm}.

\begin{lem}
\label{QT2bddb-limit-lem}
Let $(M_2, g(t))$, $t \in [0, T_{sing})$, be a Ricci flow starting from an initial metric $g(0) \in \mathcal{I}$ with $\sup_{p \in M_2} b(p,0) < \infty$. Then for every $\epsilon \in (0,1)$ there exist constants $C, \delta > 0$ such that at all points $(p, t)$ in spacetime with $C b(o,t) \leq b(p,t) \leq \delta$ the following inequalities hold:
\begin{enumerate}[label=(\roman*)]
\item $Q \geq 1 - \epsilon$
\item $bb_{ss} \leq \epsilon$
\item $T_{F_1} \coloneqq b b_{ss} + 1- b_s^2 \geq -\epsilon$
\item $T_{F_2} \coloneqq b b_{ss} + 1- b_s^2 - \left(1-b_s^2\right)^2 \leq \epsilon$
\item $\partial_t b^2 \leq \epsilon$
\end{enumerate} 
\end{lem}

\begin{remark}
Inequality (ii) is implied by (iv) for metrics in $\mathcal{I}$. However, we need (ii) as an intermediate result before proving (iv).
\end{remark}

\begin{proof} 
Fix $\epsilon \in (0,1)$. Recall the following facts of the Eguchi-Hanson space $(M_2, g^E)$ derived in section \ref{kahler-E-H-section}:
\begin{enumerate}[label=(\alph*)]
\item $x = y = 0$ on $M_2$
\item $Q \rightarrow 1$ as $s \rightarrow \infty$
\item $g^E$ is normalized such that its warping function $b^E$ satisfies $b^E(0)=1$.
\end{enumerate} 
Using (a) a computation shows that
\begin{align*}
bb_{ss} &= 2\left(1-Q^2\right) \\
T_{F_1} &= 3\left(1 - Q^2\right) \\
T_{F_2} &= 3\left(1 - Q^2\right) - \left(1-Q^2\right)^2 \\
\partial_t b^2 &= 0
\end{align*}
on $(M_2, g^E)$. Pick $C > 10$ such that on $(M_2, g^E)$ we have $Q > 1 - \epsilon$, $bb_{ss} < \epsilon$, 
$T_{F_1} > -\epsilon$ and $T_{F_2} < \epsilon$  whenever $s > C$. This is possible by property (b).  

Take a path $\gamma: [0,T_{sing}) \rightarrow M_2$ such that 
$$s(\gamma(t), t) = C b(o,t),$$
where we recall that $s(p, t)$ is the distance of a point $p \in M_2$ from the non-principal orbit $S^2_o$ at the tip of $M_2$. By Corollary \ref{E-H-blowup} we know that at distance scales comparable to $b(o,t)$ away from $S^2_o$ we converge to the Eguchi-Hanson space as $t \rightarrow T_{sing}$. From the scale-invariance of $Q$, $bb_{ss}$, $T_{F_1}$, $T_{F_2}$ and $\partial_t b^2$ it follows that at spacetime points $(\gamma(t),t)$ inequalities (i)-(v) eventually hold as $t \rightarrow T_{sing}$.

Let $A$ be the set of all sequences of points $\{(p_i, t_i)\}_{i \in \N}$ in spacetime satisfying the following two properties: 
\begin{enumerate}
\item $b(p_i, t_i) \geq C b(o, t_i)$ 
\item $b(p_i, t_i) \rightarrow 0$ as $i \rightarrow \infty$
\end{enumerate} 
Note that property (2) implies that for such sequences $t_i \rightarrow T_{sing}$ as $i \rightarrow \infty$. 

We first prove inequality (i), arguing by contradiction. Assume that
$$\iota := \inf\left\{ \liminf_{i \rightarrow \infty} Q(p_i, t_i) \: \big | \: \{(p_i, t_i)\}_{i \in \N} \in A \right\} < 1 - \epsilon.$$
Then there exists a sequence $\{ A_n \}_{n \in \N}$ of sequences $A_n = \{(p_{n,i}, t_{n,i}) \}_{i \in \N}$ of points in spacetime satisfying properties (1) and (2) above, and
$$\lim_{n \rightarrow \infty} \liminf_{i \rightarrow \infty} Q(p_{n,i}, t_{n,i}) = \iota.$$
For each $n \in \N$ take $N(n) \in \N$ such that for $m \geq N(n)$ we have
$$\left| \liminf_{i \rightarrow \infty} Q(p_{m,i}, t_{m,i}) - \iota \right| \leq \frac{1}{n}.$$ 
For each $n \in \N$ take $I(n) \in \N$ such that 
$$\left|Q(p_{n,I(n)}, t_{n,I(n)}) - \liminf_{i \rightarrow \infty} Q(p_{n,i}, t_{n,i})\right| \leq  \frac{1}{n}$$
and
$$b(p_{n,I(n)}, t_{n, I(n)}) \leq \frac{1}{n}.$$
Let $(p_n, t_n) = (p_{N(n), I(N(n))}, t_{N(n), I(N(n))})$ for $n \in \N$. Then we see that 
$$Q(p_i, t_i) \rightarrow \iota \text{ as } i \rightarrow \infty$$
and both properties (1) and (2) from above hold. 

Before we carry on recall that by Theorem \ref{curv-bound} there exists a $C_1>0$ such that
$$|Rm_{g(t)}|_{g(t)} \leq \frac{C_1}{b^2} \: \text{ on } \: M_2 \times [0, T_{sing}).$$
Recall also Theorem \ref{thm:no-local-collapsing}, from which it follows that there exist constants $\kappa, \rho>0$ such that $g(t)$ is $\kappa$-non-collapsed at scale $\rho$. Next, consider the rescaled metrics
$$g_i(t) = \frac{1}{b^2(p_i, t_i)} g( t_i + t b^2(p_i, t_i)), \quad [-\Delta t, 0],$$ 
where $\Delta t > 0$ is chosen such that Proposition \ref{blow-up-prop} holds. Then $(C_{g_i(0)}\left(p_i, \frac{1}{2}\right), g_i(t), p_i)$ subsequentially converges to a pointed Ricci flow $(\mathcal{C}_\infty, g_\infty(t), p_\infty)$, $t \in [-\Delta t, 0]$. By construction
$$b(p_\infty, 0) = 1$$
and
$$ Q(p_\infty, 0) = \iota < 1 - \epsilon.$$

\begin{claim}
$\lim_{i \rightarrow \infty} \frac{b(p_i, t_i)}{b(o, t_i)} = \infty$
\end{claim}
\begin{claimproof}
We argue by contradiction. Assume there exists a $C'>0$ such that after passing to a subsequence of $(p_i, t_i)$
$$\frac{b(p_i, t_i)}{b(o, t_i)} < C'.$$
Consider the rescaled metrics 
$$g_i(t) = \frac{1}{b^2(p_i, t_i)} g( t_i + t b^2(p_i, t_i)), \quad [-b(p_i, t_i)^{-2} t_i, 0].$$
By Corollary \ref{E-H-blowup} we see that $(M_2, g_i(t), p_i)$ subsequentially converges to $(M_2, g_{\infty}(t), p_\infty)$, where $g_\infty(t)$ is stationary and homothetic to the Eguchi-Hanson metric. By construction
$$1 = b(p_\infty, 0) \geq C b(o, 0)$$
and 
$$Q(p_\infty, 0) = \iota < 1- \epsilon.$$ 
Furthermore, 
$$s(p_\infty, 0) \geq b(p_\infty, 0) \geq C b(o, 0),$$
where the first inequality follows from the fact that $1 \geq Q \geq b_s \geq 0$ for metrics in $\mathcal{I}$ and the second inequality follows from the definition of $C$. Thus
$$Q(p_\infty, 0) > 1 - \epsilon,$$
which is a contradiction and hence proves the claim.
\end{claimproof}

\begin{claim}
For every $N \in \N$ eventually $\frac{b(p,t)}{b(o,t)} > N$ everywhere in $(C_{g_i(0)}\left(p_i, \frac{1}{2}\right), g_i(t), p_i) \times [-\Delta t, 0]$
\end{claim}
\begin{claimproof}
Fix $N \in \N$. We argue by contradiction. After passing to a subsequence of $(p_i, t_i)$, we may assume that there exists a sequence of spacetime points $(p'_i, t'_i) \in (C_{g_i(0)}\left(p_i, \frac{1}{2}\right), g_i(t), p_i) \times [-\Delta t, 0]$ for which $\frac{b(p'_i,t'_i)}{b(o, t'_i)} \leq N$. Consider the rescaled metrics
$$g'_i(t) = \frac{1}{b^2(p'_i, t'_i)} g(t'_i + t b^2(p'_i, t'_i)), \quad t \in [-b(p'_i, t'_i)^{-2} t'_i, 0].$$
By Corollary \ref{E-H-blowup}, $(M_2, g'_i(t), p'_i)$, $t \in [-b(p'_i, t'_i)^{-2} t'_i, 0]$, converges to an ancient Ricci flow $(M_\infty, g'_\infty(t), p'_\infty)$, $t \in (-\infty, 0]$, which is stationary and homothetic to the Eguchi-Hanson space. Note that on the non-principal orbit $S^2_o$
$$0 \geq \partial_t b^2(o,t) = 4by_s \geq - 4bQ_s = -4k,$$
as $0 \geq y = b_s - Q \geq -Q$ and $y=Q=0$ at $o$ for metrics in $\mathcal{I}$. Hence for $\tau \in (0, \frac{1}{4k})$ the warping function $b_i$ of the metric $g_i(t)$ satisfies
$$b_i(o,t) \geq 1 - 4k \tau > 0 \: \text{ for } \: t \in [-b(p'_i, t'_i)^{-2} t'_i, \tau].$$
We deduce by Theorem \ref{curv-bound} that $g_i(t)$ has bounded curvature on $M_2 \times [-b(p'_i, t'_i)^{-2} t'_i, \tau]$. Hence $(M_2, g'_i(t), p'_i)$, $t \in [-b(p'_i, t'_i)^{-2} t'_i, \tau]$, also converges to the stationary Eguchi-Hanson space. In fact, inductively we can then show that for any $\tau > 0$ we converge to the Eguchi-Hanson space. As $(p'_i, t'_i)$ converges to a point $(p'_\infty, t'_\infty)$ in $\mathcal{C}_\infty \times [-\Delta t, 0]$, this implies that $\mathcal{C}_\infty \times [-\Delta t, 0]$ is a subset of a spacetime corresponding to the Eguchi-Hanson space, and therefore $\lim_{i \rightarrow \infty} \frac{b(p_i,t_i)}{b(o,t_i)} < \infty$. This, however, contradicts Claim 1.
\end{claimproof}

\begin{claim}
$ Q(p_\infty, 0) = \inf_{\mathcal{C}_\infty \times [-\Delta t, 0]} Q$
\end{claim}
\begin{claimproof}
We argue by contradiction. If $Q(p',t') < \iota$ at a point $(p', t') \in \mathcal{C}_\infty \times [-\Delta t, 0]$, one could pick a sequence of points $(p'_i, t'_i) \in C_{g_i(0)}\left(p_i, \frac{1}{2}\right) \times [-\Delta t, 0]$ with $(p'_i, t'_i) \rightarrow (p',t')$ as $i \rightarrow \infty$. Then shifting back to the time of the Ricci flow $(M_2, g(t))$ via $T'_i = t_i + t'_ib(p_i,t_i)^2$ we see that the sequence $(p'_i, T'_i) \in M_2 \times [0, T_{sing})$ satisfies properties (1) and (2). The former property holds because of Claim 2. This, however, would contradict the definition of $\iota$.
\end{claimproof}

By property (1) of the sequence $(p_i, t_i)$ we see that $p_\infty$ does not lie on a non-principal orbit of $\mathcal{C}_\infty$ and therefore $Q(p_\infty, 0) = a(p_\infty, 0) > 0$. However, the evolution equation (\ref{Q-evol}) of $Q$ shows that the only attainable minima are $0$ and $1$, yielding a contradiction. This concludes the proof of (i).

We prove (ii)-(v) by the same strategy. Below we first prove inequality (ii) by contradiction. Assume that
$$\iota := \sup\left\{ \limsup_{i \rightarrow \infty} bb_{ss}\big|_{(p_i, t_i)} \: \Big | \: \{(p_i, t_i)\}_{i \in \N} \in A \right\} >  \epsilon.$$
As before we can construct a sequence of points $(p_i, t_i)$ in spacetime satisfying properties (1) and (2), and such that
$$\lim_{i \rightarrow \infty} bb_{ss}\big|_{(p_i, t_i)} = \iota.$$
Consider the rescaled metrics
$$g_i(t) = \frac{1}{b^2(p_i, t_i)} g( t_i + t b^2(p_i, t_i)), \quad [-\Delta t, 0],$$ 
where $\Delta t > 0$ is chosen such that Proposition \ref{blow-up-prop} holds. Then $(C_{g_i(0)}\left(p_i, \frac{1}{2}\right), g_i(t), p_i)$ subsequentially converges to a pointed Ricci flow $(\mathcal{C}_\infty, g_\infty(t), p_\infty)$. By construction
$$ bb_{ss}\big|_{(p_\infty, 0)} = \iota > \epsilon.$$
Furthermore, by the same arguments as in Claim 1 \& 2 \& 3, we have
$$ bb_{ss}\big|_{(p_\infty, 0)} = \sup_{\mathcal{C}_\infty \times [-\Delta t, 0]} bb_{ss}$$ 
and hence 
$$\partial_t bb_{ss} \Big|_{(p_\infty, 0)} \geq 0.$$
By statement (i) of this lemma we know that $Q = 1$ on $\mathcal{C}_\infty \times [-\Delta t, 0]$. For $Q=1$ the evolution equation for $bb_{ss}$ is 
$$\partial_t (bb_{ss}) = (bb_{ss})_{ss} - \frac{b_s}{b} (bb_{ss})_s -4 \frac{b_s^2}{b^2} \left( 1- b_s^2\right) -2 \frac{b_{ss}}{b} \left(b b_{ss}+2 b_s^2\right).$$
The derivation is carried out in the Appendix A, in the subsection on the evolution equations when $Q=1$. From this we see that at the point $(p_\infty,0)$ in spacetime we have
$$\partial_t (bb_{ss}) \Big |_{(p_\infty,0)} < 0,$$
which is a contradiction. This proves (ii).

We prove inequality (iii) similarly. Assume that
$$\iota := \inf\left\{ \liminf_{i \rightarrow \infty} T_{F_1}(p_i, t_i) \: \big | \: \{(p_i, t_i)\}_{i \in \N} \in A \right\} < - \epsilon.$$
Pick $\{(p_i,t_i)\}_{i \in \N} \in A$ such that
$$\lim_{i \rightarrow \infty} T_{F_1}(p_i, t_i) = \iota.$$
As before, $(C_{g_i(0)}\left(p_i, \frac{1}{2}\right), g_i(t), p_i)$ subsequentially converges to a pointed Ricci flow $(\mathcal{C}_\infty, g_\infty(t), p_\infty)$. By construction
$$ T_{F_1} (p_\infty, 0) = \inf_{\mathcal{C}_\infty \times [-\Delta t, 0]} T_{F_1} =  \iota < - \epsilon$$
and hence
$$\partial_t T_{F_1} \big|_{(p_\infty, 0)} \leq 0.$$
By inequality (i) of this lemma $Q = 1$ on $\mathcal{C}_\infty \times [-\Delta t, 0]$. For $Q=1$ the evolution equation of $T_{F_1}$ can be written as 
$$\partial_t T_{F_1} = (T_{F_1})_{ss} - \frac{b_s}{b} (T_{F_1})_s - 8\frac{b^2_s}{b^2} T_{F_1}.$$
The derivation is carried out in the Appendix A. From this we see that at the point $(p_\infty,0)$ in spacetime we have
$$\partial_t T_{F_1} \Big |_{(p_\infty,0)} \geq 0,$$
with equality if, and only if, $b_s \big|_{(p_\infty,0)} = 0$. Therefore we conclude that $b_s = 0$ at $(p_\infty, 0)$. Applying the strong maximum principle to the evolution equation (\ref{db-evol}) of $b_s$ when $Q=1$, it then follows that $b_s = 0$, and hence $b_{ss}=0$, everywhere in $\mathcal{C}_\infty \times [-\Delta t, 0]$. This, however, implies $T_{F_1} = 1$ at $(p_\infty, 0)$, which is a contradiction and thus proves (ii).

We proceed to prove (iv) in the same fashion. Assume that
$$\iota := \sup\left\{ \limsup_{i \rightarrow \infty} T_{F_2}(p_i, t_i) \: \big | \: \{(p_i, t_i)\}_{i \in \N} \in A \right\} > \epsilon.$$
Pick $\{(p_i,t_i)\}_{i \in \N} \in A$ such that 
$$\lim_{i \rightarrow \infty} T_{F_2}(p_i, t_i) = \iota.$$
As before, $(C_{g_i(0)}\left(p_i, \frac{1}{2}\right), g_i(t), p_i)$ subsequentially converges to a pointed Ricci flow $(\mathcal{C}_\infty, g_\infty(t), p_\infty)$. By construction
$$ T_{F_2}(p_\infty, 0) = \sup_{\mathcal{C}_\infty \times [-\Delta t, 0]} T_{F_2} = \iota > \epsilon.$$
Therefore 
$$\partial_t T_{F_2} \Big|_{(p_\infty, 0)} \geq 0.$$
By statement (i) and (ii) of this lemma we have $Q=1$  and $bb_{ss} \leq 0$ on $\mathcal{C}_\infty \times [-\Delta t, 0]$. When $Q=1$ the evolution equation of $T_{F_2}$ can be written as
$$\partial_t T_{F_2} = (T_{F_2})_{ss} - \frac{b_s}{b} (T_{F_2})_s + \frac{1}{b^2} C_{F_2},$$
where $C_{F_2}$ is a polynomial expression in $bb_{ss}$ and $1 - b_s^2$. The derivation is carried out in the Appendix A. By Lemma \ref{lem:CF2-neg} in the Appendix A, $C_{F_2} < 0$ whenever $T_{F_2} > 0$ and $bb_{ss} \leq 0$. This, however, implies 
$$\partial_t T_{F_2}\Big|_{(p_\infty, 0)} < 0,$$
which is a contradiction and thus proves (iv).

Finally we prove (v), also by contradiction. Assume that
$$\iota := \sup\left\{ \limsup_{i \rightarrow \infty} \partial_t b^2 (p_i, t_i) \: \big | \: \{(p_i, t_i)\}_{i \in \N} \in A \right\} > \epsilon.$$
Pick $\{(p_i,t_i)\}_{i \in \N} \in A$ such that 
$$\lim_{i \rightarrow \infty} \partial_t b^2 (p_i, t_i) = \iota.$$
As before, $(C_{g_i(0)}\left(p_i, \frac{1}{2}\right), g_i(t), p_i)$ subsequentially converges to a pointed Ricci flow $(\mathcal{C}_\infty, g_\infty(t), p_\infty)$. By construction
$$\partial_t b^2\Big|_{(p_\infty, 0)} = \iota > \epsilon,$$
as $\partial_t b^2$ is a scale-invariant quantity.
By (i) we have $Q = 1$ on $\mathcal{C}_\infty\times[-\Delta t, 0]$ and the evolution equation (\ref{b-evol}) of $b$ simplifies to
$$\partial_t b^2 = 2bb_{ss} + 4 \left( b_s^2 -1\right).$$
By inequality (iii) of this lemma we have
$$bb_{ss} \leq b_s^2 -1 +\left(1 - b_s^2\right) \leq 0 \: \text{ on } \: \mathcal{C}_\infty \times [-\Delta t, 0],$$
as $b_s \in [0,1]$ for metrics in $\mathcal{I}$. This, however, implies that 
$$\partial_t b^2 \leq 0,$$
which is a contradiction and thus proves (v).
\end{proof}

\begin{lem}
\label{bbdot-bounded-lem}
Let $(M_2, g(t))$, $t \in [0, T_{sing})$, be a Ricci flow starting from an initial metric $g(0) \in \mathcal{I}$ with $\sup_{p \in M_2} b(p,0) < \infty$. Then for every $\epsilon \in (0,1)$ there exists a $\delta > 0$ such that at all points $(p,t)$ in spacetime at which $b(p,t) \leq \delta$ we have
$$ \partial_t b^2 \leq \epsilon.$$
\end{lem}

\begin{proof}
Fix $\epsilon > 0$. By Lemma \ref{QT2bddb-limit-lem} we only need to prove that there exists a $\delta > 0$ such that the result holds when $b(p,t)\leq C b(o,t) \leq \delta$, where $C>0$ is as in Lemma \ref{QT2bddb-limit-lem}. Note that
$$\partial_t b^2  = 0$$ 
on the Eguchi-Hanson space background. By Corollary \ref{E-H-blowup} we know that on the scale $b(p,t) \leq C b(o,t)$ we converge to the Eguchi-Hanson space as $t \rightarrow T_{sing}$. As $b(o,t) \rightarrow 0$ as $t \rightarrow T_{sing}$ we see that there exists a $\delta >0$ such that  
$$ \partial_t b^2 \leq \epsilon$$
at all points $(p,t)$ in spacetime at which $b(p,t) \leq C b(o,t) \leq \delta$. This completes the proof.
\end{proof}

Below we prove the simplest case of Theorem \ref{blow-up-thm}.

\begin{lem}
\label{cylinder-blow-up-lem}
Let $(M_2, g(t))$, $t \in [0, T_{sing})$, be a Ricci flow starting from an initial metric $g(0) \in \mathcal{I}$ with $\sup_{p \in M_2} b(p,0) < \infty$. Let $(p_i,t_i)$ be a sequence of points in spacetime satisfying
\begin{enumerate}
\item  $b(p_i, t_i) \rightarrow 0$
\item  $b_s(p_i,t_i) \rightarrow 0$
\item  $b(p_i,t_i) > 2 b(o, t_i)$ 
\end{enumerate}
Consider the rescaled metrics
\begin{equation*}
g_i(t) = \frac{1}{b^2(p_i,t_i)} g(t_i + b^2(p_i,t_i) t), \quad t \in [-b(p_i, t_i)^{-2}t_i, 0].
\end{equation*}
Then $(M_2, g_i(t), p_i)$ , $t \in [-b(p_i, t_i)^{-2}t_i, 0]$, subsequentially converges, in the Cheeger-Gromov sense, to the shrinking cylinder $\R \times \R P^3$.
\end{lem}

\begin{proof}
Fix $T>0$. By Lemma \ref{bbdot-bounded-lem}, the curvature bound of Theorem \ref{curv-bound} and the fact that $b_s \in [0,1]$ for metrics in $\mathcal{I}$, we see that the curvatures of $g_i(t)$ on the parabolic neighborhoods $C_{g_i(0)}(p_i, \frac{1}{2}) \times [-T, 0]$ are eventually uniformly bounded. By Theorem \ref{thm:no-local-collapsing} the Ricci flows $g_i(t)$ are $\kappa$-non-collapsed. Hence $(C_{g_i(0)}(p_i, \frac{1}{2}), g_i(t), p_i)$, $t \in [-T,0]$, subsequentially converges to a Ricci flow $(\mathcal{C}_\infty, g_\infty(t), p_\infty)$, $t \in [-T,0]$, where by construction $b_s = 0$ at the point $(p_\infty, 0)$ in spacetime. Lemma \ref{QT2bddb-limit-lem} implies $Q = 1$ on $\mathcal{C}_\infty \times [-T, 0]$. Applying the strong maximum principle to the evolution equation (\ref{db-evol}) of $b_s$ when $Q=1$ shows that 
$$b_s = 0 \: \text{ on } \: \mathcal{C}_{\infty} \times [-T, 0].$$
That is, the metric $g_\infty(t)$ is cylindrical. From here one can inductively show that for every $r > 0$ the Ricci flows $(C_{g_i(0)}(p_i, r), g_i(t), p_i)$, $t \in [-T,0]$ subsequentially converge to a limiting cylindrical Ricci flow. Hence $(M_2, g_i(t), p_i)$, $t \in [-T, 0]$, subsequentially converges to the shrinking cylinder $(\R \times \R P^3, g_\infty(t), p_\infty)$, $t \in [-T,0]$. As $T>0$ is arbitrary the desired result follows by a diagonal argument.
\end{proof}

Before we carry on constructing the blow-up limit (iii) of Theorem \ref{blow-up-thm}, we need to state two technical lemmas. Their proofs can be skipped on the first reading.

\begin{lem}
\label{lem:ODE-technical}
Let $\hat{\eta} > 0$. There exists a $K= K(\hat{\eta})>1$ such that the following holds: Let $s_0 > 0$ and $b: [s_0, \infty) \rightarrow \R$ satisfy the ordinary differential inequalities
\begin{equation}
\label{bddb-ODI}
b b_{ss} \leq b_s^2 -1 + \left( 1- b_s^2\right)^2
\end{equation}
and
$$b_s > 0.$$
If at $s_0$ we have
\begin{equation}
\label{init-cond-1}
\frac{1- b_s^2}{b^2}\Big|_{s_0} = K
\end{equation}
and 
\begin{equation}
\label{init-cond-2}
b_s\big|_{s_0} \in [\hat{\eta}, 1),
\end{equation}
then $b_s < \hat{\eta}$ when $b \geq1$.
\end{lem}

\begin{proof}
First note that $b(s_0) \leq \frac{1}{\sqrt{K}} < 1 $ by (\ref{init-cond-1}) and (\ref{init-cond-2}). Furthermore, as $b_s \in [\hat{\eta}, 1)$ at $s_0$ we see from (\ref{bddb-ODI}) that $b_s < 1$ on $[s_0, \infty)$.

Write $B = b_s$. Then the ODI becomes
$$b B_s \leq B^2 -1 + \left( 1- B^2\right)^2 = - B^2\left( 1- B^2\right).$$
Since $b_s > 0$ we may treat $b$ as the independent variable, yielding the following ODI
$$ \frac{\mathrm{d}B}{\mathrm{d}b}\leq - \frac{B^2 \left(1- B^2\right)}{bB}.$$
Note that as $B=b_s\in(0,1)$ we may rearrange the inequality and integrate to obtain 
$$ \int_{B_0}^B \frac{B \mathrm{d}B}{B^2\left(1 - B^2\right)} \leq -\int_{b_0}^b \frac{\mathrm{d}b}{b},$$
where we denote by $b_0$ and $B_0$ the values of $b$ and $B$ at $s_0$, respectively. Evaluating the integrals and rearranging we deduce
$$\frac{B^2}{1 - B^2} \leq \frac{B_0^2 b_0^2}{(1-B_0^2) b^2}.$$
By the initial conditions (\ref{init-cond-1}) and (\ref{init-cond-2}) we have 
$$ \frac{b_0^2 B_0^2}{1 - B_0^2} = \frac{B_0^2}{K} = \frac{B_0^4}{K B_0^2}\leq \frac{1}{K \hat{\eta}^2}$$
and therefore
$$ \frac{B^2}{1 - B^2} \leq \frac{1}{K\hat{\eta}^2b^2}.$$
Hence when $b \geq 1$ we have
$$\frac{B^2}{1 - B^2} \leq \frac{1}{K \hat{\eta}^2},$$
which can be rearranged to
$$B^2 \leq \frac{1}{K \hat{\eta}^2 + 1}.$$
Choosing $K$ sufficiently large the desired result follows.
\end{proof}

Now we may construct the orbifold Ricci flow blow-up:

\begin{lem}
\label{lem:orbifold-blowup}
Let $\eta \in (0,1)$ and $(M_2, g(t))$, $t \in [0, T_{sing})$, be a Ricci flow starting from an initial metric $g(0) \in \mathcal{I}$ with $\sup_{p \in M_2} b(p,0) < \infty$. Assume that $(p_i,t_i)$ is a sequence of points in spacetime satisfying
\begin{enumerate}
\item  $b(p_i, t_i) \rightarrow 0$ as $i \rightarrow \infty$
\item  $\frac{b(p_i,t_i)}{b(o,t_i)} \rightarrow \infty $ as $i \rightarrow \infty$
\item  $b_s(p_i,t_i) \rightarrow \eta $ as $i \rightarrow \infty$
\end{enumerate}
Consider the rescaled Ricci flows 
\begin{equation*}
g_i(t) = \frac{1}{b^2(p_i,t_i)} g(t_i + b^2(p_i,t_i) t), \quad t\in [-b(p_i, t_i)^{-2} t_i, 0].
\end{equation*}
Then $(M_2, g_i(t), p_i)$, $t\in [-b(p_i, t_i)^{-2} t_i, 0]$, subsequentially converges, in the Cheeger-Gromov sense, to an ancient Ricci flow $(M_\infty, g_\infty(t), p_\infty)$, $t \in (-\infty,0]$, where $M_\infty \cong \R^4\setminus\{0\}/\Z_2$. Moreover, $g_\infty(t)$ can be extended to a smooth orbifold Ricci flow on $\R^4 /\Z_2$ that is homothetic to the 4d Bryant soliton quotiented by $\Z_2$.
\end{lem}

\begin{proof}
Fix $T > 0$. By Lemma \ref{bbdot-bounded-lem} we have that for every $\epsilon > 0$ there exists a $\delta > 0$ such that at points $(p,t)$ in spacetime at which $b(p,t) < \delta$ we have $\partial_t b^2 \leq \epsilon$.
This shows that for every $N > 0$ there exists a $\delta'= \delta'(N) > 0$ such that whenever $b(p,t) < \delta'$ then 
\begin{equation}
\label{b-lower-bound}
b(p,t') > \frac{b(p,t)}{2} \: \text{ for } \: t' \in [t- N b^2(p,t), t].
\end{equation}
Consider the rescaled metrics
$$g_i(t) = \frac{1}{b^2(p_i,t_i)} g\left(t_i + b^2(p_i,t_i) t\right).$$
For $n \in \N_{\geq 2}$ and $(p,t) \in M_2 \times [0, T_{sing})$ define the open set
$$E_{p,t,n} := \left\{p' \in M_2 \: \Big | \: b(p',t) > \frac{b(p,t)}{n} \right\} \subseteq M_2$$
Furthermore, define the parabolic neighborhoods
$$\Omega_{i,n} = E_{p_i,t_i,n} \times [-T-1 , 0].$$
Recall that by Theorem \ref{curv-bound} there exists a $C_1 > 0$ such that
$$|Rm_{g(t)}|_{g(t)} \leq \frac{C_1}{b^2} \: \text{ on } \: M_2 \times [0, T_{sing}).$$
Hence for fixed $n$ and sufficiently large $i$ the curvatures of $g_i(t)$ on $\Omega_{i,n}$ are uniformly bounded:
$$|Rm_{g_i(t)}|_{g_i(t)} \leq 4n^2C_1  \: \text{ on } \Omega_{i,n}.$$
This follows from (\ref{b-lower-bound}), $b(p_i, t_i) \rightarrow 0$ and the fact that $b_s \geq 0$ for metrics in $\mathcal{I}$. By Theorem \ref{thm:no-local-collapsing} the Ricci flows $g_i(t)$ are $\kappa$-non-collapsed at larger and larger scales. By a slight adaptation of the local compactness Theorem \ref{thm:local-compactness} we see that for each $n \in \N$ the Ricci flows $(E_{p_i, t_i, n}, g_i(t), p_i)$, $t \in [-T-1,0]$, subsequentially converge to a Ricci flow $(\mathcal{E}_{\infty, n}, g_{\infty,n}(t), p_{\infty,n})$, $t \in [-T-1, 0]$. The manifolds $\mathcal{E}_{\infty, n}$ are diffeomorphic to $\R^4\setminus\{0\}/\Z_2$ and therefore incomplete. By a diagonal argument we may assume that $\mathcal{E}_{\infty, n} \subset \mathcal{E}_{\infty, n+1}$ and $g_{\infty,n}(t) = g_{\infty,n+1}(t)$ on $\mathcal{E}_{\infty, n}$. This allows us to drop the dependence on $n$ and write $g_\infty(t)$ and $p_\infty$ for brevity. By Lemma \ref{QT2bddb-limit-lem} we have $Q = 1$, $bb_{ss} \leq 0$, $T_{F_1} \leq 0$, $T_{F_2} \geq 0$ and $\partial_t b^2 \leq 0$ on $\mathcal{E}_{\infty, n}$.

\begin{claim}
There exists an $\hat{\eta} > 0$, independent of $n$, such that on the word line $(p_{\infty}, t)$, $t \in [-T, 0]$, in $\mathcal{E}_{\infty, n} \times [-T, 0]$ we have $b_s > \hat{\eta}$ uniformly.
\end{claim}
\begin{claimproof}
We argue by contradiction. Assume that $t' \in [-T,0]$ is such that for $t \in [t', 0]$ we have $b_s(p_\infty, t) \geq 0$ with equality if, and only if, $t = t'$. Applying the strong maximum principle to the evolution equation (\ref{db-evol}) of $b_s$ when $Q=1$, we obtain that $b_s = 0$ on $\mathcal{E}_{\infty, n} \times [-T-1, t']$. That is, the metric $g_\infty(t)$ is cylindrical for times $t \in [-T-1, t']$. We now show that this leads to a contradiction. Take times $t'_i = t_i - t' b^2(p_i,t_i)$. Then the spacetime points $(p_i,t'_i) \in M_2 \times [0, T_{sing})$ converge to the spacetime point $(p_\infty, t') \in \mathcal{E}_{\infty, n} \times [-T-1, t']$. Consider the rescaled metrics
$$g'_i(t) = \frac{1}{b(p_i,t'_i)^2} g(t'_i + t b(p_i, t'_i)^2), \quad t\in[-t'_i b(p_i, t'_i)^{-2}, 0].$$
Because $b_s(p_i, t'_i) \rightarrow 0$ as $i \rightarrow \infty$, Lemma \ref{cylinder-blow-up-lem} implies that after passing to a subsequence $(M_2, g'_i(t), p_i)$ converges to the shrinking cylinder $\R \times \R P^3$. For every $n\in \N$ take $N_n \in \N$ such that for $i \geq N_n$ the region $C_{g'_i(0)}(p_i, n) \subset M_2$ is close, in the Cheeger-Gromov sense, to a cylinder $\R \times \R P^3$ of length $2n$ and radius $1$. By Perelman's pseudolocality theorem there exists a $K > 0$ and $\tau > 0$ such that $g'_i(t)$ has bounded curvature on $C_{g'_i(0)}(p_i, n-1) \times [0, \tau]$. Hence $(M_2, g'_i(t), p_i)$, $t \in[-t'_i b(p_i, t'_i)^{-2}, \tau]$, subsequentially converges to a limiting Ricci flow $(M_\infty, g'_\infty(t),p_\infty)$, $ t \in (-\infty, \tau]$, where $M_\infty \cong \R \times \R P^3$ and $g'_\infty(0)$ is cylindrical. By the uniqueness of Ricci flow solutions \cite{CZ06}, we see that $g'_\infty(t)$ remains cylindrical for $t \in [0, \tau]$ and therefore $b_s= 0$ on $M_\infty \times (-\infty, \tau)$. Now we have arrived at a contradiction, as this implies that $t'$ is not the earliest time at which $b_s = 0$ on the wordline through the point $(p_\infty, 0)$ in the spacetime $\mathcal{E}_{\infty, n} \times [-T-1, 0]$. This proves the claim.
\end{claimproof}

As $T_{F_1} \geq 0$ we have
\begin{equation*}
-\frac{b_{ss}}{b} \leq \frac{1-b_s^2}{b^2} \: \text{ on } \: \mathcal{E}_{\infty, n} \times [-T, 0]
\end{equation*}
and hence 
\begin{equation}
\label{curv-bound-rotational}
|Rm_{g_{\infty}(t)}|_{g_{\infty}(t)} \leq c \frac{1-b_s^2}{b^2}
\end{equation}
for some universal constant $c >0$, as $\frac{1-b_s^2}{b^2}$ and $-\frac{b_{ss}}{b}$ are the only non-zero curvature components of a rotationally symmetric metric.

\begin{claim}
There exists a $K= K(\hat{\eta}) > 1$, independent of $n$, such that 
$$|Rm_{g_{\infty}(t)}|_{g_{\infty}(t)} < cK$$ 
uniformly on $\mathcal{E}_{\infty, n} \times [-T,0]$. 
\end{claim}
\begin{claimproof}
Fix $n \geq 2$. As $bb_{ss} \leq 0$ and $b_s \geq 0$ it follows from Claim 1 that $b_s \geq \hat{\eta}>0$ in the region
$$R = \left\{ (p,t) \in \mathcal{E}_{\infty, n} \times [-T,0] \: \Big| \: b(p,t) \leq b(p_\infty, t) \right\}.$$
As $\partial_t b^2 \leq 0$  we have $b(p_\infty, t) \geq b(p_\infty, 0) = 1$ for $t \in [-T, 0]$. 

Now choose a $K = K(\hat{\eta}) > 1$ such that Lemma \ref{lem:ODE-technical} holds true. If at some point $(p',t') \in \mathcal{E}_{\infty, n} \times [-T,0]$ we had
$$ \frac{1- b_s^2}{b^2} \geq K $$
then on the time slice $\{t= t'\} \subset \mathcal{E}_{\infty,n}$ the result of Lemma \ref{lem:ODE-technical} would imply that $b_s < \hat{\eta}$ when $b \geq 1$. This cannot be true, as by Lemma \ref{bbdot-bounded-lem} we have that $\partial_t b^2 \leq 0$ on $\mathcal{E}_{\infty,n} \times [-T, 0]$ and therefore $b(p_\infty, t) \geq 1$ for $t \in [-T,0]$. Hence we deduce by (\ref{curv-bound-rotational}) that the curvature is bounded by $cK$ on the region $R$. As on $\mathcal{E}_{\infty, n} \times [-T,0] \setminus R$ we have $b > 1$, it follows by (\ref{curv-bound-rotational}) and the fact that $b_s \in [0,1]$ for metrics in $\mathcal{I}$ that the curvature is uniformly bounded by $c$ there.
\end{claimproof}

Claim 2 shows that as $n \rightarrow \infty$ we may extract a limiting Ricci flow $(M_\infty, g_\infty(t), p_\infty)$, $t \in [-T, 0]$, with curvature bounded by $cK$. By construction $M_\infty$ is diffeomorphic to $(R^4\setminus\{0\})/ \Z_2$. Define the radial coordinate $\xi: M_{\infty} \rightarrow \R$ by
$$\xi(p) = d_{g_\infty(0)}(p, \Sigma_{p_\infty}) + \xi_0,$$
where $\xi_0 \in \R$ is chosen such that $\xi\rightarrow 0$ as $b \rightarrow 0$. 

Note that by the Ricci flow equation (\ref{b-evol}) for $b$ we have
$$|\partial_t b^2| \leq 3 b^2 |Rm_{g(t)}|_{g(t)} \leq 3 b^2 K \: \text{ on } \: M_\infty \times [-T,0].$$
Working in $(\xi, t)$ coordinates we see that
$$b^2(\xi, t) \leq b^2(\xi, 0) e^{3 K T}, \quad t \in [-T,0].$$
Hence for all $t \in [-T,0]$ we have $b(\xi, t) \rightarrow 0$ as $\xi \rightarrow 0$. As $M_\infty$ has bounded curvature, we see that $\frac{1-b_s^2}{b^2}$ is bounded as well and hence $b_s(\xi, t) \rightarrow 1$ as $\xi \rightarrow 0$. From Theorem \ref{RF-removable-singularity} in Appendix B it then follows that $g_\infty(t)$, $t \in (-T,0]$, can be extended to a smooth orbifold Ricci flow on $\R^4\times \Z_2$. Since $T$ was arbitrary, a diagonal argument produces an ancient orbifold Ricci flow $(\R^4\setminus/\Z_2, g_\infty(t), p_\infty)$, $t \in (-\infty, 0]$. Note that apriori $g_\infty(t)$ might have unbounded curvature as $t \rightarrow -\infty$.
 
As $Q = 1$, $b_s \in [0,1]$ and $b_{ss} \leq 0$ we see that $g_\infty(t)$ is rotationally symmetric and has positive sectional curvature. Furthermore, for each $t \in (-\infty, 0]$ the metric $g_\infty(t)$ is asymptotically cylindrical, as the following argument shows: Either $b$ is bounded, in which case $bb_{ss} \leq 0$ and $b_s \geq 0$ show that $\lim_{s\rightarrow \infty} b_s = 0$, or $b$ is unbounded, in which case the inequality $T_{F_2} \leq 0$ and the proof of Lemma \ref{lem:ODE-technical} show that on each time slice $b_s \rightarrow 0$ as $b \rightarrow \infty$. 

By the Hamilton's trace Harnack inequality (see for instance \cite[Theorem D.26]{ChII}) and the fact that for any $T > 0$ the metric $g_\infty(t)$ has bounded curvature on $\R^4/\Z_2 \times [-T,0]$, it follows that 
$$\partial_t R_{g_\infty(t)} \geq 0 \: \text{ on } \: \R^4/\Z_2 \times (-\infty,0].$$
Therefore $g_\infty(t)$ has bounded curvature on $\R^4/\Z_2 \times (\infty, 0]$. By the result of Li and Zhang \cite{LZ18} we conclude that $g_\infty(t)$ is homothetic to the four dimensional Bryant soliton quotiented by $\Z_2$.

\end{proof}

\begin{lem}
\label{flat-orbifold-blow-up-lem}
Let $(M_2, g(t))$, $t \in [0, T_{sing})$, be a Ricci flow starting from an initial metric $g(0) \in \mathcal{I}$ with $\sup_{p \in M_2} b(p,0) < \infty$. Let $(p_i,t_i)$ be a sequence of points in spacetime satisfying
\begin{enumerate}
\item  $b(p_i,t_i) \rightarrow 0$
\item  $b_s(p_i,t_i) \rightarrow 1$
\end{enumerate}
Consider the sequence of rescaled metrics
\begin{equation*}
g_i(t) = \frac{1}{b^2(p_i,t_i)} g(t_i + b^2(p_i,t_i) t), \quad t \in [-b(p_i,t_i)^2 t_i, 0].
\end{equation*}
Then $(M_2, g_i(t), p_i)$, $t \in [-b(p_i,t_i)^2 t_i, 0]$, subsequentially converges, in the Cheeger-Gromov sense, to an ancient Ricci flow $(M_\infty, g_\infty(t), p_\infty)$, $t\in (-\infty, 0]$, where $M_\infty \cong \R^4 \setminus\{0\} /\Z_2$ and $g_\infty(t)$ can be extended to a smooth orbifold Ricci flow on $\R^4/\Z_2$ that is stationary and isometric to the flat orbifold $\R^4/\Z_2$.
\end{lem}

\begin{proof}
First note that 
\begin{claim}
$$\frac{b(p_i, t_i)}{b(o,t_i)} \rightarrow \infty \: \text{ as } \: i \rightarrow \infty$$
\end{claim}
\begin{claimproof}
We argue by contradiction. Assume there exists a $C>0$ such that after passing to a subsequence of $(p_i, t_i)$ we have 
$$\frac{b(p_i, t_i)}{b(o,t_i)} \leq C.$$
Consider the rescaled metrics
$$g_i(t) = \frac{1}{b(p_i, t_i)^2} g \left(t_i + t b(p_i, t_i)^2 \right), \quad t \in [- b(p_i, t_i)^{-2}t_i, 0].$$
Then by by Corollary \ref{E-H-blowup} the sequence $(M_2, g_i(t), p_i)$ subsequentially converges to a blow-up limit $(M_2, g_\infty(t), p_\infty)$, which is homothetic to the Eguchi-Hanson space. By construction 
$$b(p_\infty, 0) = 1$$
and
$$b_s(p_\infty, 0) = 1.$$
The latter follows from the assumption that $b_s(p_i,t_i) \rightarrow 1$ as $i \rightarrow \infty$. However, by Lemma \ref{E-H-properties-lem} we have $b_s < 1$ everywhere on the Eguchi-Hanson space. This is a contradiction and the claim follows. 
\end{claimproof}

Fix $T>0$ and consider the rescaled metrics $g_i(t)$ on the parabolic sets $E_{(p_i,t_i,n)} \times [-T, 0]$ as in the proof of Lemma \ref{lem:orbifold-blowup}. By the same reasoning, we see that for all $n \in \N_{\geq 2}$ the flows $(E_{p_i,t_i,n}, g_i(t), p_i)$ subsequentially converges to a Ricci flow $(\mathcal{E}_{\infty, n}, g_{\infty, n}(t), p_{\infty, n})$, $t \in [-T,0]$. As in the proof of Lemma \ref{lem:orbifold-blowup}, we may assume that $\mathcal{E}_{\infty, n} \subset \mathcal{E}_{\infty, n+1}$ and $g_{\infty, n} = g_{\infty, n+1}$ on $\mathcal{E}_{\infty, n}$. Therefore we drop the dependence on $n$ and write $p_\infty$ and $g_{\infty}(t)$. By construction we have
$$b(p_\infty, 0) = 1$$
and 
$$b_s(p_\infty, 0) = 1,$$
where the latter follows from the assumption that $b_s(p_i, t_i) \rightarrow 1$ as $i \rightarrow \infty$. Furthermore, by Lemma \ref{QT2bddb-limit-lem} and Claim 1 we have $Q = 1$ on $\mathcal{E}_{\infty, n} \times [-T, 0]$. Applying the strong maximum principle to the evolution equation (\ref{db-evol}) of $b_s$ when $Q=1$ we deduce that $b_s = 1$ everywhere in $\mathcal{E}_{\infty, n} \times [-T, 0]$. Hence $g_\infty(t)$ is flat and $(\mathcal{E}_{\infty, n}, g_\infty(t), p_\infty)$, $t \in [-T, 0]$, converges to the flat orbifold $\R^4 / \Z_2$ as $n \rightarrow \infty$. As $T> 0$ was arbitrary the desired result follows by a diagonal argument.
\end{proof}

\section*{Appendix A}
Here we carry out some of the computations we rely on throughout the paper. Recall 
$$ \frac{\partial}{\partial s} = \frac{1}{u(\xi, t)} \frac{\partial}{\partial \xi}$$
and the commutation relation
$$ [\partial_t, \partial_s] = - \frac{a_{ss}}{a} - 2 \frac{b_{ss}}{b}$$
from subsection \ref{ricci-flow-equations-sec}. For the computations it will also be helpful to keep in mind that
$$ bQ_s = a_s - Q b_s$$
which follows from differentiating the expression $Q = \frac{a}{b}$. Finally recall the definition of the K\"ahler quantities $$x = a_s + Q^2 -2$$ and $$y = b_s - Q$$ from section \ref{kahler-E-H-section}. 

First we compute the evolution equation of $Q$: 
\begin{align*}
\partial_t Q &= \frac{\partial_t a}{b} - \frac{a \partial_t b}{b^2}
\end{align*}
Inserting the expressions for $\partial_t a$ and $\partial_t b$ from the evolution equations (\ref{a-evol}) and (\ref{b-evol}) for $a$ and $b$ we obtain
\begin{align*}
\partial_t Q = Q_{ss} + 3 \frac{b_s}{b} Q_s + \frac{4}{b^2} Q(1-Q^2).
\end{align*}

\subsection*{Evolution equations of $a_s$, $b_s$, $Qb_s$, $x$, $y$ and $\frac{y}{Q}$}
By the commutation relations above we have
\begin{align*}
\partial_t a_s &= \partial_s \partial_t a - \left(\frac{a_{ss}}{a} + 2 \frac{b_{ss}}{b}\right) a_s \\
\partial_t b_s &= \partial_s \partial_t b - \left(\frac{a_{ss}}{a} + 2 \frac{b_{ss}}{b}\right) b_s
\end{align*}
Hence plugging in the expressions for $\partial_t a$ and $\partial_t b$ from the evolution equations (\ref{a-evol}) and (\ref{b-evol}) for $a$ and $b$ we obtain 
\begin{align*}
\partial_t a_s &= (a_s)_{ss} + \left(2 \frac{b_s}{b} - \frac{a_s}{a}\right) (a_s)_s + \frac{1}{b^2}\left( -2 a_s b_s^2- 6 Q^2 a_s + 8 Q^3 b_s \right) \\
\partial_t b_s &= (b_s)_{ss} + \frac{a_s}{a} (b_s)_s + \frac{1}{b^2}\left( -\frac{a_s^2 b_s}{Q^2}+4 Q a_s-6 Q^2 b_s-b_s^3+4 b_s \right)
\end{align*}
From here we can compute the evolution equation of $Q b_s$:
\begin{align*}
\partial_t Qb_s &= (\partial_t Q )b_s + Q \partial_t b_s \\
				&=  (Qb_s)_{ss} + \left(2 \frac{b_s}{b} - \frac{a_s}{a}\right) (Qb_s)_s +\frac{1}{b^2}\left( 4 Q^2 a_s- 10 Q^3 b_s - 2 Q b_s^3 + 8 Q b_s\right) 
\end{align*}
Now we may compute the evolution equations of the K\"ahler quantities $x$ and $y$:
\begin{align*}
\partial_t x &= \partial_t a_s + 2 Q \partial_t Q \\
			 &=  x_{ss} + \left(2 \frac{b_s}{b} - \frac{a_s}{a}\right) x_s - \frac{2}{b^2}\left(2 Q^2 + y^2\right) x - \frac{2}{b^2}\left(Q^2 +2 \right) y^2,
\end{align*}
where in the last step we made the substitutions $a_s = x - Q^2 + 2$, $b_s = y + Q$ and $a = Q b$. Similarly,
\begin{align}
\label{y-evol}
\partial_t y &= \partial_t b_s - \partial_t Q \\ \nonumber
			 &= y_{ss} + \frac{a_s}{a} y - \frac{y}{a^2} \left( \left(x+2\right)^2 + Q^2 \left(2x + y^2\right) \right)
\end{align}
Then we can compute
\begin{align}
\label{yQ-evol}
\partial_t \left(\frac{y}{Q}\right) &=\frac{\partial_t y}{Q} - \frac{y \partial_t Q}{Q^2} \\ \nonumber
					   &= \left(\frac{y}{Q}\right)_{ss} + \left(3 \frac{a_s}{a} - 2 \frac{b_s}{b} \right) \left(\frac{y}{Q}\right)_s + \frac{2}{b^2}\frac{y}{Q}\left( 2 + \frac{y}{Q}\right) \left( Q b_s - 2 a_s \right)
\end{align}
where we substituted $a_s = x - Q^2 + 2$, $b_s = y + Q$ and $a = Q b$ in the last step. 

\subsection*{Evolution equation of $H_{\pm}$}
In section \ref{sec:curv-bound} we define the quantities
$$H_{\pm} \coloneqq bb_{ss} \mp a_s^2 - b_s^2 \pm C$$
for some constant $C> 0$. Below we derive its evolution equation.

First note that we have
$$\partial_t b_{ss} = \partial_s \partial_t b_{s} - \left(\frac{a_{ss}}{a} + 2 \frac{b_{ss}}{b}\right) b_s.$$
Substituting the evolution equation for $b_s$ derived above we obtain
\begin{align*}
\partial_t b_{ss} &= (b_{ss})_{ss} +\frac{a_s}{a}(b_{ss})_s + \frac{4 a a_{ss}}{b^3}-\frac{2 a_s^2 b_{ss}}{a^2}+\frac{2 a_s^3 b_s}{a^3} \\
                  & \qquad -\frac{24 a a_s b_s}{b^4}+\frac{4 a_s^2}{b^3}-\frac{2 a_s a_{ss} b_s}{a^2}-\frac{6 a^2 b_{ss}}{b^4}+\frac{24 a^2 b_s^2}{b^5}\\
                  & \qquad -\frac{2 b_{ss}^2}{b}+\frac{4 b_{ss}}{b^2}+\frac{2 b_s^4}{b^3}-\frac{8 b_s^2}{b^3}-\frac{3 b_s^2 b_{ss}}{b^2}
\end{align*}
Hence we can compute the evolution equation of $H$ via
$$\partial_t H = (\partial_t b)b_{ss} + b \partial_t b_{ss} \mp 2 a_s \partial_t a_s - 2 b_s \partial_t b_s$$
and substituting the expressions for $\partial_t b$, $\partial_t b_{ss}$, $\partial_t a_s$ and $\partial_t b_s$ derived above. Noting that
$$H_s = \mp 2 a_s a_{ss}+b (b_{ss})_s-b_{s} b_{ss}$$
and 
$$H_{ss} = \mp 2 a_{ss}^2 \mp 2 (a_{ss})_{s} a_s+b (b_{ss})_{ss}-b_{ss}^2$$
a longer computation shows that
\begin{align*}
\partial_t  H_\pm &= [H_\pm]_{ss} + \left(\frac{a_s}{a} -2 \frac{b_s}{b}\right) [H_\pm]_s +H_\pm \left(-\frac{2 a_s^2}{a^2}-\frac{4 a^2}{b^4}-\frac{4b_s^2}{b^2}\right) \\
 & \pm C \left(\frac{2 a_s^2}{a^2}+\frac{4 a^2}{b^4}+\frac{4b_s^2}{b^2}\right) \\
 &\pm 2 a_{ss}^2  + a_{ss} \left(-\frac{2 b a_s b_s}{a^2} \mp \frac{8 a_s b_s}{b} \pm \frac{4 a_s^2}{a}+\frac{4a}{b^2}\right) \\
 &+\frac{2 b a_s^3 b_s}{a^3}-\frac{32 a a_sb_s}{b^3}\mp\frac{16 a^3 a_s b_s}{b^5}+\frac{4 a_s^2}{b^2} \pm \frac{8 a^2a_s^2}{b^4} \\
 &\mp \frac{2 a_s^4}{a^2} +\frac{32 a^2 b_s^2}{b^4}-\frac{16 b_s^2}{b^2}.
\end{align*}

\subsection*{Evolution equation of $f_{\theta}(Q)$}
\begin{align*}
\partial_t f_{\theta}(Q)  &= f' \partial_t Q  \\
						   &= f' \left( Q_{ss} + 3 \frac{b_s}{b} Q_s + \frac{4}{b^2}Q \left(1 - Q^2\right) \right)
\end{align*}
by the evolution equation (\ref{Q-evol}) of $Q$. Note that we omitted the dependence the quantities on spacetime $(\xi, t)$ and and the dependence of $f$ on $Q$ and $\theta$. For example we wrote $f'$ for $f'_{\theta}(Q(\xi,t))$. Noting that
$$\left[f(Q)\right]_{s} = f'(Q)Q_s $$
and
$$ \left[f(Q)\right]_{ss} = f''(Q)Q^2_s + f'(Q) Q_{ss}$$ 
we obtain
\begin{align*}
\partial_t f(Q)  &= [f(Q)]_{ss} - f'' Q_s^2 + 3 \frac{b_s}{b} [f(Q)]_s + \frac{4}{b^2} f' Q \left(1 - Q^2 \right) \\
						   &= [f(Q)]_{ss}  + \left(3 \frac{a_s}{a} - 2 \frac{b_s}{b} \right)[f(Q)]_s  \\
						   & \qquad \qquad +  \left(5 \frac{b_s}{b} - 3 \frac{a_s}{a} \right) [f(Q)]_s + \frac{4}{b^2} f' Q \left(1 - Q^2 \right) - f'' Q_s^2\\
						   &= [f(Q)]_{ss}  + \left(3 \frac{a_s}{a} - 2 \frac{b_s}{b} \right)[f(Q)]_s + \frac{1}{b^2} C_f
\end{align*}
where
\begin{align*}
C_f &= \left(5 b_s - 3 \frac{a_s}{Q} \right) b[f(Q)]_s +  4 f' Q \left(1 - Q^2 \right) - f'' b^2 Q_s^2 \\
    &=  \left(5 b_s - 3 \frac{a_s}{Q} \right)f' \left(a_s - Qb_s\right) + 4 f' Q \left(1 - Q^2 \right)  -  \left(a_s - Q b_s\right)^2f'' \\
    &= \left( 8 a_s b_s - 3 \frac{a_s^2}{Q} - 5 Q b_s^2 + 4 Q \left(1 - Q^2\right) \right)f' -  \left( a_s - Q b_s \right)^2f''
\end{align*}

\subsection*{Evolution equation of $Z_{\theta}$}
We have
\begin{align*}
\partial_t Z_{\theta} = \partial_t \left(\frac{x}{Q^2}\right) + \partial_t f_{\theta}(Q)
\end{align*}
We computed the evolution equation for $f_{\theta}(Q)$ above. Therefore it remains to compute $\partial_t \frac{x}{Q^2}$. For this recall the evolution equation (\ref{x-evol}) of $x$
$$ \partial_t x = x_{ss} + \left(2 \frac{b_s}{b} - \frac{a_s}{a} \right) x_{s} + \frac{1}{b^2}C_x $$
where
$$ C_x = - 2 \left(2 Q^2 + y^2\right) x - 2 \left(Q^2 +2 \right) y^2.$$
Differentiation shows that
$$ \partial_s \left(\frac{x}{Q^2}\right) = \frac{x_s}{Q^2}- 2 \frac{xQ_s}{Q^3}$$
and
$$ \partial_{ss} \left(\frac{x}{Q^2}\right) = \frac{x_{ss}}{Q^2}- 4 \frac{x_sQ_s}{Q^3} - 2x \frac{Q_{ss}}{Q^3} + 6x \frac{Q_s^2}{Q^4} .$$
Therefore we get
\begin{align*}
\partial_t \frac{x}{Q^2}   &= \frac{1}{Q^2}\partial_t x  - 2\frac{x}{Q^3} \partial_t Q \\
									 &= \left(\frac{x}{Q^2}\right)_{ss}+ \left(3\frac{a_s}{a} - 2 \frac{b_s}{b} \right) \left(\frac{x}{Q^2}\right)_s + \frac{1}{b^2} C_{\frac{x}{Q^2}}
\end{align*}
where
\begin{align*}
C_{\frac{x}{Q^2}} &= 6 \frac{x a_s}{Q^4} (bQ_s) - 10 \frac{x b_s}{Q^3} (bQ_s) - 6\frac{x}{Q^4} (bQ_s)^2 + \frac{C_x}{Q^2} - \frac{8x}{Q^2} (1- Q^2) \\
				  &= -\frac{4 a_{s}^2 b_{s}}{Q^3}+\frac{2 a_{s} b_{s}^2}{Q^2}+\frac{8 a_{s}
   b_{s}}{Q^3}-\frac{8 a_{s}}{Q^2}+2 a_{s}-\frac{8 b_{s}^2}{Q^2}+8 Q
   b_{s}+\frac{16}{Q^2}-16
\end{align*}
In the last step, we used the expressions for $x$, $y$ and $Q_s$ in terms of $a_s$, $b_s$ and $Q$ to eliminate $x$, $y$ and $Q_s$ from the expression for $C_{\frac{x}{Q^2}}$. Hence we have
$$\partial_t Z_{\theta}  =  [Z_{\theta}]_{ss} + \left( 3\frac{a_s}{a} - 2 \frac{b_s}{b}\right) [Z_{\theta}]_s + \frac{1}{b^2} C_Z$$
where
$$C_Z = C_{\frac{x}{Q^2}} + C_f.$$
As
$$ Z_{\theta} = \frac{x}{Q^2} + f_{\theta}(Q) = \frac{a_s + Q^2 - 2}{Q^2} + f_{\theta}(Q)$$ 
by definition, we can solve for $a_s$ to obtain
$$a_s = Q^2 Z_{\theta} - Q^2 f_{\theta} + 2 - Q^2.$$
Using this substitution to eliminate all occurring $a_s$ from the expression $C_Z$ we obtain
\begin{align*}
C_Z = C_{Z,0} + C_{Z,1} Z_{\theta} + C_{Z,2} Z_{\theta}^2 
\end{align*}
where
\begin{align*}
C_{Z,0} &= A_0 + A_1 \left[\frac{b_s}{Q}\right] + A_2 \left[\frac{b_s}{Q}\right]^2 \\
C_{Z,1} &= 2 Q^3 b_s f''+8 Q^2 b_s f'+8 f Q b_s+8 Qb_s-\frac{8 b_s}{Q} +2 b_s^2+2 f Q^4 f'' \\
        &+2 Q^4 f''-4Q^2 f''+6 f Q^3 f'+6 Q^3 f'-12 Q f'+2 Q^2-8\\
C_{Z,2} &= -4 Q b_s -Q^4f''-3 Q^3 f' 
\end{align*}
and
\begin{align*}
A_0  =& - Q^4 f^2 f''-2 Q^4 f f''-Q^4 f''+4 Q^2 f f''+4 Q^2 f''-4 f''-3 Q^3 f^2 f' \\ \nonumber
			&-6 Q^3 f f'-7 Q^3 f'+12 Q f f'+16 Q f'-\frac{12 f'}{Q}-2 Q^2 f+8 f-2 Q^2-4 \\
A_1 =& -2 Q^4 f f''-2 Q^4 f''+4 Q^2 f''-8 Q^3 f f'-8 Q^3 f' \\ \nonumber
    &  \qquad+16 Q f'-4 Q^2 f^2-8 Q^2 f+8 f+4 Q^2+8 \\ 
A_2 =& - Q^4f''-5 Q^3 f'-2 Q^2 f-2 Q^2-4 \\
\end{align*}

\subsection*{Evolution equation of $Z_1$}
The evolution equation for $Z_1 = \frac{x}{Q^2} + 1$ follows quickly from the evolution equations for $Z_{\theta}$ by setting $f = 1$. One obtains
\begin{equation}
\partial_t Z_1  = [Z_1]_{ss} + \left( 3\frac{a_s}{a} - 2 \frac{b_s}{b}\right) [Z_1]_s + \frac{1}{b^2}\left( C_{Z_1, 0} + C_{Z_1, 1} Z_1 + C_{Z_1, 2} Z_1^2\right)
\end{equation}
where
\begin{align*}
C_{Z_1,0} &= \frac{1}{Q^2}\left( - 4\left(1+Q^2\right)y^2 + 8Q\left(1-2Q^2 \right)y + 16Q^2\left(1-Q^2\right) \right) \\
C_{Z_1,1} &= 16 Q b_s-\frac{8 b_s}{Q}+2 b_s^2+2 Q^2-8 \\
C_{Z_1,2} &= -4 Q b_s.
\end{align*}
Note that we wrote $C_{Z_1,0}$ in terms of $y = b_s - Q$ instead of $b_s$ in order to see the similarity with the zeroth order term of the evolution equation of $T_1$ presented in the proof of Lemma \ref{T1-preserved-lem}.

\subsection*{Evolution equations when $Q=1$}
When $Q = 1$ we have $a = b$ and the Ricci flow equations simplify. In particular, we obtain
\begin{align*}
\frac{\partial_t u}{u} &= 3\frac{b_{ss}}{b} \\
\partial_t b &= b_{ss} + \frac{2}{b}\left(b_s^2 -1 \right)
\end{align*}
Using the commutation relation of $\partial_t$ and $\partial_s$ we can also compute the evolution equation of $b_s$ and $b_{ss}$:
\begin{align}
\label{db-evol}
\partial_t b_s &= \partial_s \partial_t b - 3 b_s \frac{b_{ss}}{b} \\ \nonumber
			   &= (b_s)_{ss} + \frac{b_s}{b}(b_s)_s + 2 \frac{b_s}{b^2}\left( 1- b_s^2\right)
\end{align}
Similarly,
\begin{align}
\label{ddb-evol}
\partial_t b_{ss} &= \partial_s \partial_t b_s - 3 b_{ss} \frac{b_{ss}}{b} \\ \nonumber
			      &= (b_{ss})_{ss} +\frac{b_{s}}{b} (b_{ss})_s -\frac{2 b_{ss}^2}{b}+\frac{4 \left(b_{s}^2-1\right) b_{s}^2}{b^3} \\ \nonumber
			      &\qquad -\frac{5 b_{s}^2 b_{ss}}{b^2}-\frac{2 \left(b_{s}^2-1\right) b_{ss}}{b^2}
\end{align}
Let us introduce the notation
\begin{align*}
X &\coloneqq 1 - b_s^2 \\ 
Y &\coloneqq -b b_{ss}. 
\end{align*}
We need the evolution equations of \emph{scale-invariant} quantities of the form
$$T_F = -Y + F(X),$$
where $F: [0, 1] \rightarrow \R$ is a smooth function.
In particular, we see that
$$\partial_t T_F = (\partial_t b) b_{ss} + b \partial_t b_{ss} -2 b_s F'\left(1-b_s^2\right) \partial_t b_s$$
Expanding this expression, we obtain
$$\partial_t T_F = (T_F)_{ss} - \frac{b_s}{b}(T_F)_s + \frac{1}{b^2} C_F$$
where
\begin{align*}
C_F &= 4 X^2-4 X Y-4 X-2 Y^2+4 Y \\
	& +2\left( 2 X^2-2 X Y-2 X+Y^2+2 Y \right)F'(X) \\
	& +4 (X-1) Y^2 F''(X)
\end{align*}
In this paper we make use of three different choices of $F$:
\begin{align*}
F_0(X) &= 0 \\
F_1(X) &= X \\
F_2(X) &= X - X^2 
\end{align*}
Plugging these into the expression for $C_F$ above we compute
\begin{align*}
C_{F_0} &= -4 b_s^2 \left( 1- b_s^2\right) -2 b b_{ss} \left(b b_{ss}+2 b_s^2\right) \\
C_{F_1} &= -8 b_s^2 T_{F_1} \\
C_{F_2} &= 4 \left( (2-3 X) Y^2 +\left(2 X^2-4 X+2\right) Y -2 (X-1)^2 X  \right)
\end{align*}

We also prove the following lemma here:
\begin{lem}
\label{lem:CF2-neg}
Let $X\in [0,1]$. Then whenever $0 \leq Y < X - X^2$ we have
$$P(X,Y) \coloneqq (2-3 X) Y^2 +\left(2 X^2-4 X+2\right) Y -2 (X-1)^2 X < 0.$$
In other words, $C_{F_2} < 0$ whenever $T_{F_2} > 0$ and $bb_{ss} \leq 0$.
\end{lem}
\begin{proof}
Let $R$ be the region in the $X$-$Y$-plane satisfying the inequalities $X\in [0,1]$ and $0 \leq Y < X - X^2$. Note that $Y < X - X^2$ and $Y \geq 0$ implies that $X \in (0,1)$.
A computation shows
$$P(X, X- X^2) = -3 (X-1)^2 X^3 < 0 \: \text{ for } \: X \in (0,1)$$
Notice that for a fixed $X \in [0, \frac{2}{3}]$ the quadratic polynomial $P(X,Y)$ in $Y$ is convex. As 
$$P(X, 0) = -2 (X-1)^2 X < 0 \: \text{ for } \: X \in (0,1)$$
we deduce that $P(X,Y) > 0$ on $R \cap \{ X \leq \frac{2}{3} \}$. To prove that $P(X, Y) > 0$ in $R \cap \{ X \geq \frac{2}{3} \}$ is trickier. For this we prove the following claim:
\begin{claim}
$\partial_X P(X, Y) > 0$ on $R \cap \{ X \geq \frac{2}{3} \}$.
\end{claim}
\begin{claimproof}
A computation shows
$$\partial_X P(X, Y) = -6 X^2+8 X -2 +(4 X-4) Y-3 Y^2.$$
Hence for fixed $X$ is concave in $Y$. Then note that
$$\partial_X P(X, 0) = -6 X^2+8 X -2 > 0 \: \text{ for } \: X \in [\frac{2}{3},1)$$
and
$$\partial_X P(X, X-X^2) = (1-X) \left(3 X^3+X^2+2 X-2\right) > 0 \: \text{ for } \: X \in [\frac{2}{3},1).$$
The last inequality follows by the fact that the polynomial $3 X^3+X^2+2 X-2$ is increasing on $[0,1]$ and evaluates to $\frac{2}{3}$ at $X = \frac{2}{3}$. Hence the claim follows by concavity of $\partial_X P(X, Y)$ in $Y$.
\end{claimproof}
By above we know that $P(X,X^2-X) < 0$ for $X\in (0,1)$. Hence using the result of the claim, we see that $P(X,Y) < 0$ on $R \cap \{ X \geq \frac{2}{3} \}$.
\end{proof}

\section*{Appendix B: Removable singularity}
We prove the following theorem:

\begin{thm}[Removable singularity]
\label{RF-removable-singularity}
Let $(\R^4 \setminus \{0\}, g(t))$, $t \in [0,T]$, be a rotationally symmetric Ricci flow of bounded curvature, i.e. there exists a $K>0$ such that
$$|Rm_{g(t)}|_{g(t)} < K \: \text{ on } \: \R^4 \times [0,T].$$
Taking $\xi \in (0,\infty)$ to be a radial coordinate on $\R^4 \setminus \{0\}$ the metric $g(t)$ may be written as
$$g(t) = u^2(\xi, t) d\xi^2 + b^2(\xi,t) g_{S^3},$$
where $u,b : (0,\infty) \rightarrow \R$ are smooth warping functions, and $g_{S^3}$ is the round metric on $S^3$ with sectional curvatures equal to one. If for all $t \in [0,T]$ the warping function $b(\xi,t) \rightarrow 0$ as $\xi \rightarrow 0$, then $g(t)$ can be extended to a smooth Ricci flow on $\R^4 \times (0,T]$.
\end{thm}

Below we assume $(\R^4 \setminus \{0\}, g(t))$, $t \in [0,T]$, is a Ricci flow as in Theorem \ref{RF-removable-singularity}. The proof strategy will be as follows: First we prove in Lemma \ref{lem:C11-regularity} that for every $t_0 \in [0,T]$ there exist coordinates $x^i, i = 1,2,3,4$, of $\R^4$ for which the metric $g(t_0)$ can be extended to a $C^{1,1}$ metric on $\R^4$. Note, however, without control on the derivative of the curvature tensor the metric $g(t)$ at times $t \neq t_0$ may not to be $C^{1,1}$ with respect to the coordinates $x^i$. To get around this issue we show in Lemma \ref{lem:curv-deriv-bounded} and Lemma \ref{lem:higher-curv-deriv-bounded} that in fact all derivatives $\nabla^m Rm$, $m \in \N$, of the curvature tensor are bounded on $\R^4 \setminus \{0\}\times(\delta,T]$ for any $\delta > 0$. The proof utilizes Shi's interior derivative estimates and is based on a De Giorgi-Nash-Moser iteration argument. With these results in place, we use harmonic coordinates to prove Theorem \ref{RF-removable-singularity}. Let us begin by proving 

\begin{lem}
\label{lem:C11-regularity}
Let $g = ds^2 + b(s)^2 g_{S^3}$ be a smooth, rotationally symmetric metric with bounded curvature on $\R^4\setminus\{0\}$. Here $g_{S^3}$ is the round metric of curvature one on $S^3$ and $b: (0,\infty) \rightarrow \R$ is a smooth positive function. If
$$ b \rightarrow 0 \: \text{ as } \: s \rightarrow 0$$
then $g$ can be extended to a $C^{1,1}$ metric on $\R^4$. Furthermore, if we take standard Euclidean coordinates $x_i$, $i = 1, 2, 3,4$, on $\R^4$ we have $g_{ij} = \delta_{ij}$ and $\partial_k g_{ij} = 0$ at the origin, and $\partial_k \partial_l g_{ij}$ locally bounded on $\R^4\setminus \{0\}$.
\end{lem}

\begin{proof}
As $g$ has bounded curvature there exists a $K>0$ such that
$$\Big|\frac{b_{ss}}{b}\Big|, \Big|\frac{1-b_s^2}{b^2}\Big| \leq K,$$
because these are the only non-zero curvature components of a rotationally symmetric metric. In particular, this shows that
$$b_s \rightarrow 1 \: \text{ as } \: s \rightarrow 0^+$$
and
$$b_{ss} \rightarrow 0 \: \text{ as } \: s \rightarrow 0^+.$$ 
Let $x_i$, $i =1, 2,3,4$, be Euclidean coordinates of $\R^4$, normalized such that $\sum_i (x^i)^2 = s^2$. In these coordinates
$$g = \left[\delta_{ij} + \left(s^2 \delta_{ij} - x_i x_j \right) \Psi(s) \right] \textrm{d}x^i \textrm{d}x^j,$$
where 
$$\Psi(s) = \frac{\left(\frac{b}{s}\right)^2 -1}{s^2}.$$
Note that we used the Einstein summation convention. 

\begin{claim}
$\Psi$, $s \partial_s \Psi$, $s^2 \partial_{ss} \Psi = O(K)$ as $s\rightarrow 0$.
\end{claim}

\begin{claimproof}
Fix $s>0$. By Taylor's theorem there exist numbers $s_0, s_1 \in(0,s)$ such that
\begin{align*}
b(s) &= s + \frac{1}{2} b_{ss}(s_0) s^2 \\
b_s(s) &= 1 + b_{ss}(s_1)s.
\end{align*}
Hence
\begin{align*}
\Psi(s) &= \frac{b_{ss}(s_0)}{s} + \left(\frac{b_{ss}(s_0)}{2}\right)^2. 
\end{align*}
As $\big|\frac{b_{ss}}{b}\big| \leq K$, $b \rightarrow 0$ and $b_s \rightarrow 1$ as $s \rightarrow 0$, we see $\Psi(s) = O(K)$ as $s\rightarrow 0$. By similar reasoning one shows that
\begin{align*}
s \partial_s \Psi(s) &= \frac{ 2 - 4\left(\frac{b}{s}\right)^2 + 2 \left(\frac{b}{s}\right) b_s}{s^2} \\
					 &= -\frac{3 b_{ss}(s_0)}{s}+\frac{2 b_{ss}(s_1)}{s}-b_{ss}(s_0)^2+b_{ss}(s_1) b_{ss}(s_0)
\end{align*}
and
\begin{align*}
s^2 \partial_{ss} \Psi(s) &= \frac{2 b b_{ss}-16 \left(\frac{b}{s}\right) b_s + 2 b_s^2 +20 \left(\frac{b}{s}\right)^2 -6}{s^2}  \\
				   		  &= b_{ss}(s_0) b_{ss}(s)+\frac{2 b_{ss}(s)}{s}+\frac{12 b_{ss}(s_0)}{s}-\frac{12 b_{ss}(s_1)}{s} \\
				   		  & \quad\qquad +5 b_{ss}(s_0)^2-8 b_{ss}(s_1) b_{ss}(s_0)+2 b_{ss}(s_1)^2
\end{align*}
are of order $O(K)$ as $s \rightarrow 0$. 
\end{claimproof}

Next, extend $g$ to the origin by setting $g = \delta_{ij}$ there. As
$$ \left(s^2 \delta_{ij} - x_i x_j \right) = O(s^2)$$
it follows by Claim 1 that this defines a continuous extension of $g$ to the origin. 

A computation shows
$$\partial_k g_{ij} = \left(2 x_k \delta_{ij} - \delta_{ik} x_j - x_i \delta_{jk} \right) \Psi(s) +\left( s^2 \delta_{ij} - x_i x_j \right) \frac{x_k}{s}\partial_s \Psi(s).$$
As
$$\left(2 x_k \delta_{ij} - \delta_{ik} x_j - x_i \delta_{jk} \right) = O(s) $$
and
$$\left( s^2 \delta_{ij} - x_i x_j \right) \frac{x_k}{s} = O(s^2),$$
it follows that we may continuously extend $\partial_k g_{ij}$ to the origin by setting $\partial_k g_{ij} = 0$ there. Finally, note that
$$\partial_k \partial_l g_{ij} = O(1) \Psi + O(s)\partial_s \Psi + O(s^2) \partial_{ss} \Psi = O(K).$$
Hence $\partial_k \partial_l g_{ij}$ is bounded in a neighborhood around, but excluding the origin. This shows that $\partial_k g_{ij}$ is Lipshitz.
\end{proof}

Next, we prove the boundedness of the gradient of the curvature tensor. For this we recall some interior curvature estimates. Note the following differential inequalities for the evolution of the curvature tensor and its derivatives under Ricci flow (See for instance \cite[Chapter 7]{BC04}):
\begin{align}
\label{curv-evol} \big(\partial_t - \Delta \big) |Rm|^2 &\leq - 2|\nabla Rm|^2 + 16 |Rm|^3 \\ 
\label{dcurv-evol} \big(\partial_t - \Delta \big) |\nabla^m Rm|^2 &\leq- 2 |\nabla^{m+1} Rm|^2 \\ \nonumber
	 															& \qquad+ \sum_{j=0}^m c_{mj} |\nabla^j Rm| \cdot |\nabla^{m-j} Rm| \cdot |\nabla^m Rm|
\end{align}
Here $c_{mj}$ are positive constants depending on $j$, $m$ and the dimension of the manifold only. Note also that the laplacian is with respect to the evolving metric $g(t)$. Using these inequalities one can show the following interior derivative estimate (See for instance \cite[Theorem 1.4.2]{CZ06}).
\begin{thm}[Shi's interior estimates]
\label{shi}
There exist positive constants $\theta, C_m, m \in \N$, depending on the dimension $n$ only, such that the following holds: Let $M$ be a manifold of dimension $n$ and $0 < T \leq \frac{\theta}{K}$. Assume that $g(t)$, $t \in [0, T]$, is a solution to the Ricci flow on an open neighborhood $U$ of $M$ and
$$|Rm| < K \: \text{ on } \: B_{g(0)}\left(p, r\right) \times [0,T].$$
If for $p \in U$ and $r>0$ the closed set $\overline{B_{g(0)}(p, r)}$ is contained in $U$ then 
$$|\nabla^m Rm|^2<C_m K^2 \left(\frac{1}{r^{2m}} + \frac{1}{t^m} + K^m\right) \: \text{ on } \: B_{g(0)}\left(p, \frac{r}{2}\right) \times (0,T]$$
\end{thm}

Next, we prove that for all $\tau > 0$ the gradient $|\nabla Rm|$ is bounded on $\R^4 \setminus \{0\} \times [\tau, T]$. First note that due to Shi's estimates of Theorem \ref{shi} 
$$|\nabla Rm_{g(t)}|_{g(t)}(p) = O\left(\frac{1}{d_{g(t)}(p, 0)}\right) \: \text{ on } \: \R^4 \setminus \{0\}\times [\tau, T].$$
Furthermore, from (\ref{dcurv-evol}) and and an application of Kato's inequality to show that
$$|\nabla |\nabla Rm|| \leq |\nabla^2 Rm|$$
it follows that
$$\big( \partial_t - \Delta \big) |\nabla Rm| \leq C|Rm||\nabla Rm|.$$
Hence when curvature is bounded by $K$, the function $u \coloneqq e^{-CKt} |\nabla Rm|$ is a subsolution to the heat equation, i.e.
$$\big(\partial_t - \Delta \big) u \leq 0.$$ 
With help of a De Giorgi-Nash-Moser iteration argument, this is enough to prove that $u$ is bounded for $t > \tau$. We carry this out in the lemma below:

\begin{lem}
\label{lem:curv-deriv-bounded}
Let $(\R^4 \setminus \{0\},g(t))$, $ t\in [0,T]$, be a Ricci flow as in Theorem \ref{RF-removable-singularity}. Then for any $\tau > 0$ there exists a $C = C(K, \tau) > 0$ such that
$$|\nabla Rm| < C$$
on $\R^4 \times \setminus \{0\} \times [\tau, T]$. 
\end{lem}

\begin{proof}
As shown above, the function $u = e^{-CKt} |\nabla Rm|$ is a subsolution to the heat equation, i.e.  
$$\left(\partial_t - \Delta\right) u \leq 0.$$
By Lemma (\ref{lem:C11-regularity}) we may choose Euclidean coordinates $x^i$, $i = 1, 2, 3, 4$, on $\R^4$ for which $g(0)$ is $C^{1,1}$. Take $s^2 = \sum_i (x^i)^2$ and write $B_R(x)$ for the ball centered at $x$ with radius $R$ with respect to $g(0)$.

Since the curvature of $g(t)$ is bounded on $\R^4 \setminus \{0\} \times [0,T]$, there exists a $\lambda > 0$ such that
$$\frac{1}{\lambda} g(0) \leq g(t) \leq \lambda g(0) \: \text{ on } \: \R^4 \setminus \{0\} \times [0,T].$$
Therefore, Shi's interior estimates imply
$$ u( \cdot, t) = O\left(\frac{1}{s}\right) \: \text{ for } \: t \in [\tau, T].$$  
Hence it suffices to show that for some $R>0$ the function $u$ is bounded on $B_R(0) \times [\tau, T]$. We achieve this via a De Giorgi-Nash-Moser iteration argument. In the Claim below we derive the crucial estimate.

\begin{claim}
Let $\delta > 0$, $p \geq 2$, $R_0 \in [1,10]$ and $t_0 \in [\delta, T)$. Then there exists a constant $C = C(K, \delta, T) >0$ such that the following holds: If $u \in L^p(B_{R_0} \times [t_0, T])$, then for $R_1 \in [\frac{1}{2}, R_0)$ and $t_1 \in (t_0, T]$  
$$ \norm{u}_{L^{2p}(B_{R_1}(0) \times[t_1,T])} \leq \left[C\left(\frac{p^2}{(R_0 - R_1)^2} + \frac{1}{t_1-t_0} \right)\right]^{\frac{1}{p}} \norm{u}_{L^{p}(B_{R_0}(0) \times[t_0,T])}.$$
\end{claim}

\begin{claimproof}
In the following a constant $C$ is assumed to only depend on $K$, $\delta$ and $T$ and might vary from line to line. Fix a number $A > 1$ that we later take to $\infty$. Then choose a $C^2$ function $F: \R_{\geq 0} \rightarrow \R$ with the following properties:
\begin{enumerate}
\item $F(s) = s^p$ for $ s \leq A$
\item $F$ is linear for $s \geq A+1$ with slope $pA^{p-1} + 1$
\item On $[A,A+1]$ take $F(s)$ to be defined such that $F''\geq 0$
\end{enumerate}
We see that these properties imply that $F'(s) \leq ps^{p-1}$. Next, define the cut-off functions $\eta_\epsilon$, $\epsilon>0$, and $\phi: \R^4 \rightarrow \R$. For this take a smooth function $h: \R \rightarrow \R$ with $h = 0$ on $(-\infty,\frac{1}{2}]$ and $h = 1$ on $[1,\infty)$. Then define
$$\eta_\epsilon = h\left(\frac{s}{\epsilon}\right)$$
and
$$\phi = h\left(\frac{R_0-s}{R_0-R_1}\right).$$
That is, $\phi = 1$ on $B_{R_1}(0)$ and $\phi = 0$ on $\R^4 \setminus B_{R_0}(0)$. Clearly, $|\nabla \phi|_{g(t)} \leq \frac{c}{R_0 - R_1}$ and $|\nabla \eta_\epsilon|_{g(t)} \leq \frac{c}{\epsilon}$ for some universal constant $c$ depending on $h$ and $\lambda$ only. 

Since $u \in L^p(B_{R_0} \times [t_0, T])$ is a positive function there exists a $t' \in [t_0, t_1]$ such that 
\begin{equation}
\label{slice-p-bound}
\int_{B_{R_0}(0)} u^p(\cdot, t') \; \mathrm{d}x \leq \frac{1}{t_0-t_1} \int_{t_0}^T \int_{B_{R_0}(0)} u^p \; \mathrm{d}x \; \mathrm{d}t.
\end{equation}
Next, we compute via integration by parts
\begin{align*}
\frac{d}{dt} \int_{\R^4} \eta_\epsilon F(u) \phi^2 \; \mathrm{d} x &= \int_{\R^4} \eta_\epsilon F'(u) \Delta u \phi^2 \; \mathrm{d} x \\
				&= -\int_{\R^4} \nabla \eta_\epsilon F'(u) \nabla u \phi^2 \; \mathrm{d} x
				   -\int_{\R^4} \eta_\epsilon F''(u) |\nabla u|^2 \phi^2 \; \mathrm{d} x \\
				& \qquad\qquad   -\int_{\R^4} \eta_\epsilon F'(u) \nabla u \nabla \phi^2 \; \mathrm{d} x.
\end{align*}
Integrating with respect to time from $t'$ to $T$ and noting that
$$\int_{\R^4} \eta_\epsilon F(u(\cdot, T)) \phi^2 \; \mathrm{d} x \geq 0,$$
we obtain
\begin{align}
\label{crucial-ineq}
\int_{t'}^T \int_{\R^4} \eta_\epsilon F''(u) |\nabla u|^2 \phi^2 \; \mathrm{d}x \; \mathrm{d}t 
			&\leq \int_{\R^4} \eta_\epsilon F(u(\cdot, t')) \phi^2 \; \mathrm{d}x  	
			 -\int_{t'}^T\int_{\R^4} \nabla \eta_\epsilon F'(u) \nabla u \phi^2 \; \mathrm{d}x \; \mathrm{d}t \\ \nonumber
            &\qquad \qquad -\int_{t'}^T\int_{\R^4} \eta_\epsilon F'(u) \nabla u \nabla \phi^2 \; \mathrm{d}x \; \mathrm{d}t  \\ \nonumber
            &\coloneqq I_1 - I_2 - I_3.
\end{align}
Next we estimate each of these integrals $I_1$, $I_2$ and $I_3$ separately, in order to analyze their behaviors as $\epsilon \rightarrow 0$. For the first integral we have
$$I_1 \leq \int_{B_{R_0}} u^p(\cdot, t') \; \mathrm{d}x \leq \frac{1}{t_1-t_0} \norm{u}^p_{L^p(B_{R_0}(0)\times [t_0,T])}. $$
For the second integral $I_2$, note that Shi's estimates and Kato's inequality yield  
$$|\nabla u| = |\nabla \left( e^{-CKt}|\nabla Rm| \right)| \leq |\nabla^2 Rm| = O\left(\frac{1}{s^2}\right) \: \text{ as } \: s \rightarrow 0.$$
As $|\nabla \eta_\epsilon| \leq \frac{c}{\epsilon}$, $|F'| \leq pA^{p-1} +1$ and $\phi^2 =1$ in a neighborhood of $0$, we see that
$$|\nabla \eta_\epsilon| \cdot |F'(u)|\cdot |\nabla u|\cdot |\phi^2| \leq C \epsilon^{-3} (pA^{p-1}+1) \: \text{ on } \: B_{\epsilon}(0).$$ 
As $\textrm{vol}_{g(t)}(B_\epsilon(0)) \leq C \epsilon^4$ we obtain
$$|I_2| \leq C \epsilon (pA^{p-1}+1).$$
For the final integral $I_3$, recall that by definition $0 \leq F'(u) \leq pu^{p-1}$. Furthermore, $|\nabla \phi^2|$ has support in $B_{R_0}(0)\setminus B_{R_1}(0)$ and is bounded by $\frac{2c}{R_0-R_1}$. As $R_1 \geq \frac{1}{2}$ by assumption, Shi's estimates imply that on this set $|\nabla u|$ and $u$ are bounded by some constant $C$. Thus
\begin{align*}
|I_3| &\leq \frac{2c pC^{p}}{R_0-R_1} \textrm{vol}\left(B_{R_0}(0)\setminus B_{R_1}(0)\right) \\
&\leq pC^{p+1}, 
\end{align*}
where we used that 
$$\textrm{vol}(B_{R_0}(0)\setminus B_{R_1}(0)) \leq C\left( R^4_0 -R^4_1 \right) \leq C \left(R_0 - R_1\right),$$
as $\frac{1}{2} \leq R_0 \leq R_1 \leq 10$ by assumption. This shows that $I_3$ is convergent. Now split the integral $I_3$ as
$$I_3 = \int_{t'}^T \int_{\{u \leq A\}} \eta_\epsilon F'(u) \nabla u \nabla \phi^2 \; \mathrm{d}x \; \mathrm{d}t + I_4,$$
where
$$I_4 = \int_{t'}^T \int_{\{u \geq A\}} \eta_\epsilon F'(u) \nabla u \nabla \phi^2 \; \mathrm{d}x \; \mathrm{d}t.$$
As $I_3$ is convergent, we see that $I_4 \rightarrow 0$ as $A \rightarrow \infty$. Observe that by Young's inequality
$$p u^{p-1} |\nabla u| |\nabla \phi^2| = \phi u^{\frac{p-2}{2}} |\nabla u| \cdot 2p u^{\frac{p}{2}} |\nabla \phi| \leq \frac{1}{2} u^{p-2} |\nabla u|^2 \phi^2 + 2p^2 |\nabla \phi|^2 u^p.$$ 
Moreover,
$$|\nabla \phi|^2 \leq \left(\frac{c}{R_0-R_1}\right)^2.$$
Hence we obtain
\begin{align*}
|I_3| &\leq \int_{t'}^T\int_{\{u \leq A\}} \frac{1}{2} \eta_\epsilon u^{p-2} |\nabla u|^2 \phi^2 + 2 p^2 \eta_\epsilon |\nabla \phi|^2 u^p \; \mathrm{d} x \; \mathrm{d} t + I_4 \\
  &\leq \int_{t'}^T\int_{\{u \leq A\}} \frac{1}{2} \eta_\epsilon u^{p-2} |\nabla u|^2 \phi^2 \; \mathrm{d} x \; \mathrm{d} t + \frac{Cp^2}{(R_0-R_1)^2} \norm{u}^p_{L^p(B_{R_0}(0)\times [t_0,T])} + I_4.
\end{align*}
Next, note that $F''(u) = p(p-1)u^{p-2}$ for $u \leq A$ and $F'' \geq 0$ everywhere. Therefore
\begin{align*}
\int_{t'}^T \int_{\R^4} \eta_\epsilon F''(u) |\nabla u|^2 \phi^2 \; \mathrm{d} x \; \mathrm{d} t &\geq \int_{t'}^T \int_{\{u \leq A\}} \eta_\epsilon p(p-1)u^{p-2} |\nabla u|^2 \phi^2 \; \mathrm{d} x \; \mathrm{d} t.
\end{align*}
Substituting this inequality and the inequalities for $|I_1|$, $|I_2|$ and $|I_3|$ derived above into (\ref{crucial-ineq}), we deduce
\begin{align*}
\left(p(p-1) - \frac{1}{2} \right)& \int_{t'}^T \int_{\{u \leq A\}} \eta_\epsilon u^{p-2} |\nabla u|^2 \phi^2 \; \mathrm{d} x \; \mathrm{d} t \\
& \leq C\left(\frac{p^2}{(R_0-R_1)^2} + \frac{1}{t_1-t_0} \right)\norm{u}^p_{L^p(B_{R_0}(0)\times [t_0,T])} + C \epsilon (pA^{p-1}+1) + I_4
\end{align*}
Taking $\epsilon \rightarrow 0$ and then $A \rightarrow \infty$ yields
\begin{align*}
\left(p(p-1) - \frac{1}{2} \right) \int_{t'}^T \int_{\R^4}u^{p-2} &|\nabla u|^2 \phi^2 \; \mathrm{d} x \; \mathrm{d} t \\
& \leq  C\left( \frac{p^2}{(R_0-R_1)^2} + \frac{1}{t_1-t_0} \right) \norm{u}^p_{L^p(B_{R_0}(0)\times [t_0,T])}
\end{align*}
by the monotone convergence theorem. Then note that
$$ u^{p-2} |\nabla u|^2 \phi^2 = \frac{4}{p^2} |\nabla u^{\frac{p}{2}}|^2 \phi^2$$
and
\begin{align*}
|\nabla u^{\frac{p}{2}}|^2 \phi^2 &= \left| \nabla(\phi u^{\frac{p}{2}}) - u^{\frac{p}{2}} \nabla \phi \right|^2  \\
								  &\geq \left|\nabla(\phi u^{\frac{p}{2}})\right|^2 + \left|u^{\frac{p}{2}} \nabla \phi\right|^2 - 2\left|\nabla(\phi u^{\frac{p}{2}})\right| \cdot \left|u^{\frac{p}{2}} \nabla \phi\right| \\
								  &\geq \frac{1}{2}|\nabla(\phi u^{\frac{p}{2}})|^2 - u^p |\nabla\phi|^2,
\end{align*}
where in the last line we applied Young's inequality to bound the cross-term. Therefore
\begin{align*}
\int_{t'}^T \int_{\R^4} |\nabla (\phi u^{\frac{p}{2}})|^2 \; \mathrm{d} x \; \mathrm{d} t \leq C\left( \frac{p^2}{(R_0-R_1)^2} + \frac{1}{t_1-t_0} \right) \norm{u}^p_{L^p(B_{R_0}(0)\times [t_0,T])}
\end{align*}
and applying the Sobolev inequality proves Claim 1. 
\end{claimproof}

Now we may iterate the estimate of Claim 1 to prove the desired result. First note that due to Shi's estimates, for any $R_0> 0$ and $t_0>0$ we have that $u \in L^2(B_{R_0}(0)\times [t_0, T])$. We take $t_0 = \frac{\tau}{2}$, $R_0 = 2 + \sqrt{\frac{\tau}{2}}$, $\Delta t_i = (\Delta R_i)^2 = \frac{\tau}{2^{i+1}}$, $p_i = 2^{i+1}$ and 
\begin{align*}
R_{i+1} &= R_{i} - \Delta R_i \\
t_{i+1} &= t_{i} + \Delta t_i.
\end{align*} 
Then inductively applying the estimate of Claim 1 and taking the limit as $i \rightarrow \infty$, we obtain
$$\norm{u}_{L^{\infty}(B_{2}(0)\times [\tau, T])} \leq C_\infty \norm{u}_{L^{2}(B_{R_0}(0)\times [\frac{\tau}{2}, T])} < \infty,$$
where $C_\infty > 0$ is a positive constant. This proves the desired result.

\end{proof}

Next, we prove that the higher derivatives of the curvature tensor are also bounded at positive times. For this we need a generalization of Shi's estimates for the situation in which some of the derivatives of the curvature tensor are known to be bounded. In particular, we have

\begin{thm}[Shi's interior estimates with derivative bounds]
\label{shi-modified}
Let $n \geq 2$ and $m \geq 1$. Then for every choice of constant $K>0$ there exists constants $\theta > 0$ and $C>0$ such that the following holds: Let $M$ be an open manifold $M$ of dimension $n$ and $0 < T \leq \frac{\theta}{K}$. Assume that $g(t)$, $t\in [0,T]$, is a solution to the Ricci flow on an open subset $U$ of $M$ and  
\begin{align*}
|\nabla^l Rm| &\leq K \: \text{ on } \: U \times [0, T] \: \text{ and for } \: l \in \{0, 1, 2, \cdots, m\}
\end{align*}
If for $p \in U$ and $r>0$ the closed set $\overline{B_{g(0)}(p, r)}$ is contained in $U$ then 
$$|\nabla^{m+1} Rm|^2 \leq C \left( \frac{1}{r^2} + \frac{1}{t} + 1 \right) \: \text{ on } \: B_{g(0)}\left(p, \frac{r}{2}\right) \times (0,T]$$
\end{thm}

\begin{proof}
We follow the proofs of \cite[Theorem 1.4.2]{CZ06} and \cite[Theorem 14.16]{ChII}. In the following the constant $C$ depends on $m$ and $n$ only and may vary from line to line. Consider the quantity
$$ S= \left( BK^2 + |\nabla^m Rm|^2 \right) |\nabla^{m+1} Rm|^2,$$
where $B>0$ is to be fixed later. With help of the differential inequality (\ref{dcurv-evol}) we obtain
\begin{align*}
\partial_t S &\leq \Delta S - 2 \nabla |\nabla^m Rm|^2 \nabla |\nabla^{m+1} Rm|^2 - 2 |\nabla^{m+1} Rm|^4 \\
& + \sum_j c_{mj} \cdot |\nabla^j Rm| \cdot |\nabla^{m-j} Rm| \cdot |\nabla^m Rm| \cdot |\nabla^{m+1} Rm|^2 \\
&-2 \left(BK^2 + |\nabla^m Rm|^2 \right) |\nabla^{m+2} Rm|^2\\
& + \left( BK^2 + |\nabla^{m}Rm|^2 \right) \sum_j c_{m+1 j} \cdot |\nabla^j Rm| \cdot |\nabla^{m+1-j} Rm| \cdot |\nabla^{m+1} Rm|
\end{align*}
Using Cauchy's inequality and the assumption that $|\nabla^l Rm| \leq K$ for $l= 0, 1, 2, \cdots m$, we deduce
\begin{align*}
\partial_t S &\leq \Delta S + 8K |\nabla^{m+1} Rm|^2 |\nabla^{m+2} Rm| - 2 |\nabla^{m+1} Rm|^4 -2 BK^2 |\nabla^{m+2} Rm|^2 \\ 
& + CK^3 |\nabla^{m+1}Rm|^2 + CK^3(B+1) \left(|\nabla^{m+1} Rm|^2 + K |\nabla^{m+1} Rm| \right) 
\end{align*}
Noting that for all $x \in \R$ we have $x^2 + K x \leq 2 x^2 + \frac{1}{4} K^2$ we obtain with help of Young's inequality that
\begin{align*}
\partial_t S &\leq \Delta S  - |\nabla^{m+1}Rm|^4 + 2 \left( 32 - B\right)K^2 |\nabla^{m+2} Rm|^2\\
			& \qquad \qquad + CK^6 + CK^5(B+1) + CK^6(B+1)^2.
\end{align*}
Taking $B = 32$ and assuming without loss of generality that $K>1$, we obtain
\begin{align*}
\partial_t S \leq \Delta S - \frac{S^2}{C K^4} + C K^6 
\end{align*}
From here we may follow the proof of \cite[Theorem 1.4.2]{CZ06} to deduce the desired result.
\end{proof}

With help of Theorem \ref{shi-modified} we inductively prove that the higher derivatives of the curvature tensor are bounded.

\begin{lem}
\label{lem:higher-curv-deriv-bounded}
Let $(\R^4 \setminus \{0\},g(t))$, $ t\in [0,T]$, be a Ricci flow as in Theorem \ref{RF-removable-singularity}. Then for any $\tau > 0$ there exist constants $C_m = C_m(K, \tau) > 0$, $m \in \N$, such that
$$|\nabla^m Rm| < C_m$$
on $\R^4 \times \setminus \{0\} \times \left[\tau, T\right]$. 
\end{lem}
\begin{proof}
We prove this lemma by induction. By Lemma \ref{lem:curv-deriv-bounded} the result is true for $m=1$. Assume that the result is true for $m \leq N$. Then there exist constants $C_l>0$, $l = 1, 2, 3, \cdots, N$ such that 
$$|\nabla^l Rm| \leq C_l \: \text{ on } \: \R^4 \setminus \{0\} \times \left[\frac{\tau}{4}, T\right] \: \text{ and for } \: l = 1, 2, 3, \cdots, N.$$
As in the proof of Lemma \ref{lem:curv-deriv-bounded}, choose coordinates $x^i$, $i = 1, 2, 3,4$, such that $g(0)$ can be extended to a $C^{1,1}$ metric on $\R^4$, and write $s^2 = \sum_i (x^i)^2$. As the curvature of $g(t)$, $t \in [0,T]$, is bounded there exists a $\lambda > 0$ such that
$$\frac{1}{\lambda} g(0) \leq g(t) \leq \lambda g(0) \: \text{ on } \: \R^4 \setminus \{0\} \times [0, T].$$ 
By the modified Shi's estimates of Theorem \ref{shi-modified} we see that 
$$|\nabla^{N+1} Rm| \leq C\left(\frac{1}{s} + 1\right) \: \text{ on } \: \R^4 \setminus \{0\} \times \left[\frac{\tau}{2}, T\right],$$
for some $C$ that depends on $\tau, K$, and $C_l$, $l = 1, 2, \cdots, N$, only. In particular, this implies that for all $R>0$ the function $u \in L^2(B_R(0))\times [0,T]$. By the differential inequality (\ref{dcurv-evol}) for the evolution of the curvature derivatives we see that
$$\left(\partial_t - \Delta \right) |\nabla^{N+1}Rm|^2 \leq - 2 |\nabla^{N+2} Rm|^2 + CK^2|\nabla^{N+1} Rm| + CK |\nabla^{N+1} Rm|^2$$
and hence
$$\left(\partial_t - \Delta \right) |\nabla^{N+1}Rm| \leq CK \left( K + |\nabla^{N+1} Rm|\right).$$
Thus defining
$$u = e^{-CKt} \left(|\nabla^{N+1} Rm| + K\right)$$
we deduce that
$$\left(\partial_t - \Delta \right) u \leq 0.$$
Now we are in the same setup as in the proof of Lemma \ref{lem:curv-deriv-bounded}. Therefore we may use the same De Giorgi-Nash-Moser iteration argument to show that $u$ and hence $|\nabla^{N+1} Rm|$ are bounded in $\R^4 \times [\tau, T]$. This proves the desired result.
\end{proof}

Now we may prove the main Theorem \ref{RF-removable-singularity}:

\begin{proof}[Proof of Theorem \ref{RF-removable-singularity}]
By Lemma \ref{lem:C11-regularity} we can choose coordinates $x^i$, $i = 1, 2, 3, 4$, for $\R^4$ such that $g(T)$ can be extended to a $C^{1,1}$ metric on all of $\R^4$. Below we write $g = g(T)$ for brevity. By \cite[Lemma 1.2]{DK81} there exist $C^{2,\alpha}$ harmonic coordinates $y^i: U \rightarrow \R$, $i =1 , 2, 3, 4$, in an open neighborhood $U$ of $\R^4$ containing the origin and satisfying
\begin{enumerate}
\item $y^i = 0$
\item $\frac{\partial y^i}{\partial x^j} = \delta^i_j$ 
\end{enumerate}
at the origin. Furthermore, as $g$ is smooth on $U \setminus \{0\}$, it follows from interior elliptic regularity that $y^i$ are smooth on $U \setminus \{0\}$. Write 
$$g_{ij} = g\left(\frac{\partial}{\partial y^i}, \frac{\partial}{\partial y^j}\right) \: \text{ and } \: Ric_{ij} = Ric_g\left(\frac{\partial}{\partial y^i}, \frac{\partial}{\partial y^j}\right).$$
We have that $g_{ij} \in C^{1, \alpha}(U)$ with respect to the $y^i$ coordinates. Furthermore, $g_{ij}$ is smooth on $U \setminus \{0\}$. 

By \cite[Chapter 10, Lemma 49]{PP}) we have
\begin{align}
\label{elliptic-harmonic}
\frac{1}{2} \Delta g_{ij} + Q(g, \partial g) = - Ric_{ij} \: \text{ on } \: U \setminus \{0\},
\end{align} 
where $Q(g, \partial g)$ is some universal analytic expression that is polynomial
in the matrix $g$, quadratic in $\frac{\partial g}{\partial y^i}$, and has a denominator term depending on $\sqrt{det\,g_{ij}}$. The equation (\ref{elliptic-harmonic}) makes sense on all of $U$ if we interprete it in the weak sense. 

\begin{claim}
If $g_{ij}(y) \in C^{k}(U)$ for $k \in \N$ then $R_{ij}(y) \in C^{k-1, 1}(U)$. 
\end{claim}

\begin{claimproof}
Write
$$Y_i = \frac{\partial}{\partial y^i} \: \text{ for } \: i = 1,2,3,4.$$
on $U \setminus \{0\}$ we have
\begin{align*}
\frac{\partial^k}{\partial y^{i_1} \partial y^{i_2} \cdots \partial y^{i_k}} Ric_{ij} &= Y_{i_1} Y_{i_2} \cdots Y_{i_k} Ric( Y_i , Y_j) \\
															&= \nabla_{Y_{i_1}} \nabla_{Y_{i_2}} \cdots \nabla_{Y_{i_k}} Ric(Y_i, Y_j). 
\end{align*}
Since covariant differentiation commutes with contractions, we can use the product rule to express the above derivative as a sum of terms, which only involve $\nabla^m Ric$, $m = 1, 2, \cdots, k$, and $\nabla^m Y_{i_l}$, $m, l = 1, 2, 3, \cdots, k$, contracted with $Y_i$, $Y_j$ and $Y_{i_l}$, $l = 1, 2, \cdots, k$. As by Lemma \ref{lem:higher-curv-deriv-bounded} all the derivatives of the curvature tensor are bounded and $g_{ij}(y) \in C^{k}(U)$ we see that 
$$ \frac{\partial^k}{\partial y^{i_1} \partial y^{i_2} \cdots \partial y^{i_k}} Ric_{ij}$$
is bounded as well. Hence the $k$-th spatial derivatives $\partial^k Ric_{ij}$ are bounded, which implies that $\partial^{k-1} Ric_{ij}$ is a Lipshitz function and can be continuously extended to a all of $U$. Similarly, the lower order derivatives $\partial^m Ric_{ij}$, $m = 0, 1, 2, \cdots, k-2$, can be continuously extended to the origin.  
\end{claimproof}

First note that $g$ is a $C^{1,1}(U)$ weak solution of the elliptic equation (\ref{elliptic-harmonic}). Furthermore $Q(g, \partial g) \in C^{0,1}(U)$ and by Claim 1 we have $Ric_{ij} \in C^{0,1}(U)$ as well. Since such weak solutions are unique, and there exists a $C^{2,\alpha}(U)$ solution that agrees on the boundary $\partial U$, we see that $g$ is in fact $C^{2,\alpha}(U)$. Bootstrapping standard Schauder estimates and the result of Claim 1, we conclude that $g_{ij}$ is smooth with respect to the harmonic coordinates $y^i$, $i = 1, 2, 3, 4$. 

It remains to be shown that $g(t)$ can be extended to a smooth Ricci flow on $\R^4 \times (0,T]$. Recall that by Lemma \ref{lem:higher-curv-deriv-bounded}, for all $\tau >0$ the derivatives of the curvature tensor are bounded on $U \times [\tau, T]$. Moreover $g(T)$ is bi-lipshitz to the euclidean metric $\delta_{ij}$ on $U$ and by the previous paragraph the covariant derivatives of $g(T)$ with respect to $\delta_{ij}$ are all bounded. Therefore we may follow the proof of \cite[Lemma 3.11]{ChI} with $t_0 = T$ to deduce that 
$$\frac{\partial^m}{\partial t^m} \frac{\partial^n}{\partial y^{i_1} \cdots \partial y^{i_n}} \left(g(t)\right)_{ij} \leq K_{m,n} \: \text{ on } \: U\setminus \{0\}\times [\tau,T],$$
for some constants $K_{m,n}>0$. This shows that $g(t)$ can be smoothly extended to $U\times [\tau, T]$. As $\tau > 0$ was arbitrary the desired result follows.
\end{proof}

\end{document}